\documentclass[a4paper]{amsart}

\usepackage{amsmath,amssymb,amsthm,stmaryrd,mathrsfs}

\usepackage{color,graphicx}
\usepackage[dvipsnames]{xcolor}
\usepackage{accents}
\usepackage{enumerate}
\usepackage{varwidth}
\usepackage{enumitem}

\usepackage{fancyhdr, array,calc,graphicx,url,tabularx}
\usepackage{geometry}
\usepackage{latexsym}
\usepackage{amssymb}
\usepackage{amsmath}
\usepackage{enumerate}
\usepackage{verbatim}
\usepackage{tikz}
\usepackage{tikz-cd}
\usepackage[colorlinks=true,linkcolor=blue]{hyperref}
\usepackage[T1]{fontenc}

\usepackage{changepage}

\usepackage[utf8]{inputenc}
\usepackage{tikz}
\tikzset{
     block/.style={rectangle, draw, fill=red!40, text width=6em,
                   text centered, rounded corners, minimum height=3em},
     arrow/.style={-{Stealth[]}}
     }
\usetikzlibrary{positioning,arrows.meta, calc, decorations.pathreplacing, decorations.pathmorphing, shapes.misc}

\def\undertilde#1{\math{\sf{Ord}}{\vtop{\ialign{##\crcr
$\hfil\displaystyle{#1}\hfil$\crcr\noalign{\kern1.5pt\nointerlineskip}
$\hfil\tilde{}\hfil$\crcr\noalign{\kern1.5pt}}}}}

\makeindex

{
   \end{minipage}
   \vspace*{\stretch{3}}
   \clearpage
}

\def\undertilde#1{{\baselineskip=0pt\vtop
  {\hbox{$#1$}\hbox{$\scriptscriptstyle\sim$}}}{}}

\newcommand{\ADR}{\mathsf{AD}_{\mathbb{R}}}

\newcommand{\treg}{\mathsf{\Theta_{reg}}}
\newcommand{\rref}{\mathsf{ref}}

\newcommand{\tStr}{\mathtt{SLW}}
\newcommand{\tW}{\mathtt{Woodin}}

\newcommand{\bls}{\vspace{\baselineskip}}
\newcommand{\tc}{\mathrm{tc}}

\newcommand{\less}{\mathord{<}}

\newcommand{\NTbase}{\sf{NT_{base}}}

\newcommand{\IH}{\mathsf{IH}}

\newcommand{\uB}{\mathrm{uB}}

\renewcommand{\gg}{\gamma}

\newcommand{\barn}{\bar{\nu}}

\newcommand{\bR}{{\mathbb{R}}}

\newcommand{\rest}{\restriction}

\newcommand{\Add}{\mathrm{Add}}
\newcommand{\Coll}{\mathrm{Coll}}
\newcommand{\ord}{\mathsf{Ord}}
\newcommand{\rg}{\mathsf{RefGen}}
\newcommand{\rel}{\mathsf{rel}}
\newcommand{\its}{\mathsf{its}}

\newcommand{\ts}{\mathtt{slw}}

\newcommand{\rmd}{\mathrm{d}}
\newcommand{\rms}{\mathrm{s}}

\newcommand{\hd}{{\mathfrak{d}}}

\newcommand{\hp}{{\mathfrak{p}}}
\newcommand{\hn}{{\mathfrak{n}}}
\newcommand{\hq}{{\mathfrak{q}}}
\newcommand{\hr}{{\mathfrak{r}}}
\newcommand{\hs}{{\mathfrak{s}}}

\newcommand{\hw}{{\mathfrak{w}}}

\newcommand{\hh}{{\mathfrak{h}}}
\newcommand{\hm}{{\mathfrak{m}}}
\newcommand{\hk}{{\mathfrak{k}}}
\newcommand{\hx}{{\mathfrak{x}}}
\newcommand{\hy}{{\mathfrak{y}}}

\newcommand{\card}[1]{{\vert #1 \vert} }

\newcommand{\forces}{\Vdash}

\renewcommand{\models}{\vDash}
\newcommand{\powerset}{\mathscr{P}}
\newcommand{\hull}{{\mathsf{Hull}}}
\newcommand{\chull}{{\mathsf{cHull}}}

\newcommand{\dom}{{\rm dom}}
\newcommand{\rge}{{\rm rge}}

\newcommand{\cp}{{\rm crit }}

\newcommand{\cf}{{\rm cf}}
\newcommand{\lh}{{\rm lh}}

\newtheorem*{theorem*}{Theorem}
\newtheorem{theorem}{Theorem}[section]
\newtheorem{proposition}[theorem]{Proposition}
\newtheorem{definition}[theorem]{Definition}

\newtheorem{lemma}[theorem]{Lemma}
\newtheorem{corollary}[theorem]{Corollary}
\newtheorem{claim}[theorem]{Claim}
\newtheorem{subclaim}[theorem]{Subclaim}
\newtheorem{conjecture}[theorem]{Conjecture}
\newtheorem{question}[theorem]{Question}

\newtheorem{remark}[theorem]{Remark}
\newtheorem{notation}[theorem]{Notation}
\newtheorem{terminology}[theorem]{Terminology}
\numberwithin{figure}{section}

\newcommand{\rcl}[1]{Claim~\ref{#1}}

\newcommand{\rthm}[1]{Theorem~\ref{#1}}
\newcommand{\rlem}[1]{Lemma~\ref{#1}}

\newcommand{\rcor}[1]{Corollary~\ref{#1}}
\newcommand{\rdef}[1]{Definition~\ref{#1}}
\newcommand{\rfig}[1]{Figure~\ref{#1}}

\newcommand{\rsec}[1]{Section~\ref{#1}}

\newcommand{\rrem}[1]{Remark~\ref{#1}}
\newcommand{\rnot}[1]{Notation~\ref{#1}}
\newcommand{\rter}[1]{Terminology~\ref{#1}}

\newcommand{\ZFC}{\mathsf{ZFC}}

\newcommand{\AD}{\mathsf{AD}}
\newcommand{\AC}{\mathsf{AC}}
\newcommand{\GCH}{\mathsf{GCH}}
\newcommand{\AHA}{\mathsf{Axiom\ of\ Harmony}}
\newcommand{\AH}{\mathsf{AH}}
\newcommand{\nN}{\mathsf{N}}
\newcommand{\reg}{\mathsf{reg}}

\newcommand{\bbQ}{\mathbb{Q}}
\newcommand{\lmi}{\mathcal{M}^l_\infty}

\def\inseg{\trianglelefteq}

\def\k{\kappa}
\def\a{\alpha}
\def\b{\beta}
\def\d{\delta}

\def\l{\lambda}

\def\P{{\mathcal{P} }}
\def\W{{\mathcal{W} }}
\def\Q{{\mathcal{ Q}}}
\def\mH{{\mathcal{ H}}}

\def\R{{\mathcal R}}

\def\H{{\rm{HOD}}}

\def\M{{\mathcal{M}}}
\def\N{{\mathcal{N}}}

\def\T {{\mathcal{T}}}
\def\U{{\mathcal{U}}}
\def\S{{\mathcal{S}}}
\def\V{{\mathcal{V}}}
\def\X{{\mathcal{X}}}
\def\Y{{\mathcal{Y}}}

\def\card#1{\left|#1\right|}

\def\cof{\mathop{\rm cof}\nolimits}
\def\iff{\mathrel{\leftrightarrow}}

\def\and{\mathrel{\kern1pt\&\kern1pt}}

\def\inseg{\triangleleft}
\def\insegeq{\trianglelefteq}

\def\<#1>{\langle\,#1\,\rangle}

 \input xy
 \xyoption{all}

\newcommand{\DC}{\mathsf{DC}}

\newcommand{\cP}{\mathscr{P}}
\newcommand{\bbP}{\mathbb{P}}

\newcommand{\pmax}{\mathbb{P}_{\mathrm{max}}}

\newcommand{\ZF}{\mathsf{ZF}}

 \newcommand{\gen}{{\rm gen}}
\newcommand{\ind}{{\rm ind}}

 \def\c{{\mathrm{N}}}

\begin{document}

\title[The failure of square at all uncountable cardinals]{The failure of square at all uncountable cardinals is weaker than a Woodin limit of Woodin cardinals}

\author{Douglas Blue, Paul Larson, and Grigor Sargsyan}

\thanks{The second author's work is partially funded by National Science Foundation Grant DMS-2452139. The third author's work is funded by the National Science Center, Poland under the Maestro Call, registration number UMO-2023/50/A/ST1/00258.}

\begin{abstract}
    We force the Axiom of Choice over the least initial segment of a Nairian model satisfying ZF.
    In the forcing extension, \(\square_{\kappa}\) fails at all uncountable cardinals \(\kappa\), and every regular cardinal is \(\omega\)-strongly measurable in HOD, as witnessed by the \(\omega\)-club filter.
    Thus
    \begin{enumerate}
    	\item the failure of square everywhere is within the current reach of inner model theory, which resolves \cite[Question 5.6]{foreman2005some} and conjectures in \cite{steel2005pfa,zeman2017two}, and
	
	\item the HOD Hypothesis is not provable in ZFC.
    \end{enumerate}

\end{abstract}

\maketitle
\tableofcontents
\section{Introduction}

Given an ordinal \(\xi\), the Nairian model at \(\xi\) is the structure \(L_{\xi}(\mH, \bigcup_{\alpha<\xi}\alpha^{\omega})\), where \(\mH\) is a predicate for the HOD of a model of the Axiom of Determinacy. Let $\treg$ be the theory $\ADR+``\Theta$ is a regular cardinal$"$.

\begin{definition}\label{def: minnm} \normalfont Assume $\ZF+\treg+V=L(\powerset(\bR))$. Then $N$ is a \textbf{minimal Nairian Model} if, letting $\xi=\ord\cap N$, 
\begin{enumerate}
    \item $N\models \ZF$,

    \item  $N$ is the Nairian Model at $\xi$, 

    \item $\Theta^N<\Theta$ and $\Theta^N$ is a Solovay cardinal,\footnote{$\k$ is a Solovay cardinal if for every $\b<\k$, there is no ordinal definable surjection $f:\powerset(\b)\rightarrow \k$.}

    \item letting $\d$ be the least Solovay cardinal above $\Theta^N$, $\xi$ is an inaccessible cardinal of $\H$ and is a limit of ${<}\d$-strong cardinals of $\H$, and

    \item  if $\l<\xi$ is an inaccessible cardinal of $\H$, then $\l$ is not a limit of ${<}\xi$-strong cardinals of $\H$.
\end{enumerate}
\end{definition}

In this paper, we force over a $\pmax$ extension of a minimal Narian model to prove the following theorem.\footnote{See \textsection\ref{sec: square} for background on $\square(\kappa)$ and its relatives. It may help readers of this introduction to keep in mind that $\square(\kappa^{+})$ is a weak form of $\square_{\kappa}$.}

\begin{theorem}\label{square and measurability in hod}
   $\ZF+\treg+V=L(\powerset(\bR))$. Let \(N\) be a minimal Nairian model.
    There is a forcing extension of \(N\) satisfying
    \begin{enumerate}
        \item \(\mathrm{ZFC}\),
        
        \item for all ordinals $\gamma$ of cofinality greater than $\omega_{1}$, $\square(\gamma)$ fails, and

        \item every regular cardinal is \(\omega\)-strongly measurable in \(\mathrm{HOD}\).

    \end{enumerate}
\end{theorem}
The forcing is gentle or ``minimal'', the goal being to preserve as much of the determinacy structure of the underlying Nairian model as possible in the $\ZFC$ extension. It consists of $\pmax$ followed by a full-support iteration of partial orders successively wellordering the powersets of the members of a proper class of cardinals. 
This iteration is very similar to the standard one for forcing $\GCH$ to hold above $\omega_2$.

Towards proving \rthm{square and measurability in hod}, we prove some new fundamental properties of a minimal Nairian model assuming Steel's \textit{Hod Pair Capturing} (${\sf{HPC}}$, see \cite{SteelCom}).
Theorem \ref{thm: summary} summarizes the properties of the minimal Nairian Model proved in this paper as well as in \cite{blue2025nairian}.
Given a set $X$, let $\Theta_X$ be the supremum of the ordinals that are the surjective image of $X$. If $M$ is a transitive model of $\ZFC$, then $M|\k=V_\k^M$.

\begin{theorem}\label{thm: summary} Assume $\ZF+\ADR + \treg +{\sf{HPC}}+V=L(\powerset(\bR))$. Let $N$ be a minimal Nairian Model, $\xi=\ord\cap N$, and $\langle\k_\a: \a<\xi\rangle$ be the increasing enumeration of the ${<}\xi$-strong cardinals of $\H$ and their limits. Let $\k_{-1}=\omega$. Then the following holds.
\begin{enumerate}
\item $\k_0=\Theta^N$.
\item For every $\a\in [-1,\xi)$, $N\models \k_{\a+1}=\Theta_{\k_\a^\omega}$.
\item For every  $\a\in [-1, \xi]$, $N\models \k_{\a+2}=\k_{\a+1}^+$. 
\item For every  $\a\in [-1, \xi]$, $N\models ``\k_{\a+1}$ is a regular cardinal''. 
\item For every $\a\in [-1, \xi]$, $\cf(\k_{\a+1})=\k_0$.
\item For every $\a\in [-1, \xi]$, $\powerset(\k_\a^\omega)\cap N\subseteq N|\k_{\a+1}$.
\item For every limit ordinal $\a<\xi$, $(\k_\a^+)^N=(\k_\a^+)^W$ where $W=L(\H|\k_\a, \k_\a^{\cf(\k_\a)})$.
\item For every ordinal $\a\in [-1, \xi)$ and for every $\nu<\k_{\a+1}$, there is no cofinal $f:\nu^\omega\rightarrow \k_{\a+1}$ in $N$.
\item For every $\a\in [-1, \xi)$, for  every $\b\in [\a+1, \xi)$ and for every $X\in N|\k_{\b+1}$ there exist $\tau<\k_{\a+1}$ and  an elementary embedding $\sigma:N|\tau\rightarrow N|\k_{\b+1}$ such that 
\begin{enumerate}
    \item $\sigma\rest \kappa_{\a+1}=id$, and 
    \item $X\in \rge(\sigma)$.
\end{enumerate}
\item For every $\a\in [-1, \xi)$, for  every $\b\in [\a+1, \xi)$ and for every $X\in N|\k_{\b+1}$ there exist a transitive $W\in N|\k_{\a+2}$ and an elementary embedding $\sigma:W\rightarrow N|\k_{\b+1}$ such that 
\begin{enumerate}
    \item $W^{\k_\a^\omega}\subseteq W$,
    \item $\sigma\rest \kappa_{\a+1}=id$, and 
    \item $X\in \rge(\sigma)$.
\end{enumerate}
\item If $G\subseteq \pmax$ is generic, then $N[G]$ is $\omega_1$-closed in $V[G]$.
\item If $\k\leq\Theta^N$ is a Suslin cardinal, then $N\models ``\k$ is supercompact to ordinals''. \footnote{That is, for every $\a$ there is a normal fine $\k$-complete ultrafilter on $\powerset_{\k}(\a)$.}
\item $N\models ``\omega_1$ is supercompact''. \footnote{That is, for every set $X$ there is a normal fine countably complete ultrafilter on $\powerset_{\omega_1}(X)$.}
\end{enumerate}
\end{theorem}

In a subsequent publication, we will show that $W$ of clause (10) of \rthm{thm: summary} can be taken to be of the form $N|\nu$. 

Theorem \ref{square and measurability in hod} rules out approaches to two central conjectures in set theory.
The first we will discuss is the conjecture that the Proper Forcing Axiom is equiconsistent with the existence of a supercompact cardinal.
Baumgartner's consistency proof of PFA collapses a supercompact cardinal to \(\omega_2\), and Viale and Weiss \cite{VW11} showed that in any model resembling the known models of PFA, there is an inner model in which \(\omega_2\) is a strongly compact cardinal.
If the model is constructed by a proper forcing iteration, then \(\omega_2\) is supercompact in the inner model.
Thus if one could show that every model of PFA has the properties required by Viale and Weiss's theorem, then the PFA conjecture would be solved.
In the absence of such a proof, inner model theorists have used the core model induction to compute consistency strength lower bounds of PFA.
The best bound to date is in the short extender region of large cardinals, below a Woodin cardinal which is a limit of Woodin cardinals.

The lower bounds of PFA computed to date are really lower bounds of failures of square principles.
Such failures are characteristic consequences of both supercompact cardinals and PFA, which suggests they are quite strong indeed.


\begin{theorem}[Todor\v{c}evi\'c \cite{To84}]\label{thm: Todorcevic}
    \(\mathrm{PFA}\) implies \(\neg\square_{\kappa}\) for every uncountable cardinal \(\kappa\).
\end{theorem}

To derive consistency strength from PFA, one builds, assuming an anti-large cardinal hypothesis, a fine structural core model satisfying a weak form of covering.
Below the level of subcompact cardinals, which is presently beyond the reach of core model theory, \(\square_{\kappa}\) holds for all \(\kappa\) in core models by work of Schimmerling and Zeman \cite{schimmerling2001square}.
\rthm{thm: Todorcevic} then implies that the core model fails radically to approximate the universe, hence PFA is at least as strong as the negation of the anti-large cardinal hypothesis used to define the core model.

Schimmerling used \rthm{thm: Todorcevic} to show that PFA implies \(\Delta^1_2\)-Determinacy \cite{Sch95}.
Steel showed that PFA implies \(\mathrm{AD}^{L(\mathbb{R})}\) by applying the core model induction to \(\neg\square_{\kappa}\) for \(\kappa\) a singular strong limit \cite{steel2005pfa}.
Schimmerling showed that PD holds if $\neg\square_{\kappa} + \neg\square(\kappa)$ and \(\kappa\) is larger than the continuum \cite{Schim07}.
This in turn allows an improvement of Steel's theorem: the restriction of \(\mathrm{PFA}\) to partial orders of size at most \( (2^{\omega})^+ \) implies \(\mathrm{AD}^{L(\mathbb{R})}\).
Jensen, Schimmerling, Schindler, and Steel \cite{JSSS} showed that if \(\kappa\geq\omega_2\) is countably closed and \(\square(\kappa)\) and \(\square_{\kappa}\) each fail, then there is an inner model with arbitrarily large strong cardinals and Woodin cardinals.
Going further, Sargsyan showed that the failure of $\square_{\kappa}$ at a singular strong limit cardinal $\kappa$ implies the existence of a nontame mouse \cite{Sa14}, and Sargsyan and Trang showed that a certain configuration of failures of square implies that a model of the Largest Suslin Axiom exists \cite{LSA}.
The Sargsyan-Trang lower bound is the best to date.

The expectation has been that these lower bounds are far from optimal.
Steel
\cite{steel2005pfa} conjectured that the failure of $\square_{\kappa}$ at a singular strong limit cardinal $\kappa$ is at least as strong as a superstrong cardinal.
Question 5.6 of \cite{foreman2005some} asks to determine the consistency strength ordering of superstrong cardinals versus the failure of $\square_{\kappa}$ at a singular strong limit cardinal $\kappa$.
The best known upper bound for \(\neg\square_{\aleph_{\omega}}\) was computed by Zeman, who forced it from a measurable subcompact cardinal \cite{zeman2017two}.
(It is unclear whether Zeman's method can be used to make square fail at three consecutive regular cardinals, or at a limit of singular cardinals at which square also fails.)
This led Zeman to conjecture that ``subcompactness is the right candidate for the consistency strength of the failure of \(\square_{\lambda}\) at a singular cardinal \(\lambda\)'' \cite[p.~411]{zeman2017two}.

However, the consistent existence of minimal Nairian models follows from hypotheses weaker than the existence of a Woodin cardinal which is a limit of Woodin cardinals.

\begin{corollary}\label{failures of square are weak}
    The theory \(\mathrm{ZFC}\) + ``$\square(\gamma)$ fails for all ordinals $\gamma$ of cofinality greater than $\omega_{1}$'' is strictly weaker consistency-wise than the theory ZFC + ``there is a Woodin limit of Woodin cardinals.''
\end{corollary}

If PFA implies that there is an inner model with a long-extender large cardinal, then Corollary \ref{failures of square are weak} precludes showing this using failures of square as a proxy.\footnote{
Rado's Conjecture and the Product Measure Extension Axiom are apparently strong principles whose known consistency lower bounds are those of failures of square.}

Corollary \ref{failures of square are weak} solves \cite[Question 5.6]{foreman2005some}: the theory ZFC + ``there is a singular strong limit cardinal \(\kappa\) such that \(\square_{\kappa}\) fails'' does \emph{not} imply the consistency of ZFC + ``there is a superstrong cardinal.''
This should be contrasted with the following lower bound.

\begin{theorem}[Neeman-Steel \cite{ESC}]\label{ESC}
    Suppose that \(\delta\) is a Woodin cardinal such that \(\neg\square(\delta)\) and \(\neg\square_{\delta}.\)
    Suppose that the Strategic Branch Hypothesis holds at \(\delta\).\footnote{The \emph{Strategic Branch Hypothesis holds at} 
    \(\delta\) if for every countable elementary substructure \(M\) of \(V_{\delta}\), player II has a winning strategy in the iteration game on the transitive collapse of \(M\), allowing only short extenders mapping their critical points strictly above their strength and linear compositions of normal, non-overlapping, plus-2 iteration trees \cite{ESC}.}    
    Then there is an inner model with a subcompact cardinal.
\end{theorem}

Since the square principles fail at uncountable cardinals in the forcing extension of our minimal \(N\), either the Strategic Branch Hypothesis holding at \(\delta\) is false, or there are no Woodin cardinals in the forcing extension of \(N\).

As it happens, there are no Woodin cardinals in \(N[G]\), but the previous paragraph raises a possibility which to our knowledge has not been considered.
The results of this paper and \cite{blue2025nairian} strongly suggest that the \(K^c\) approach using partial backgrounds will not succeed in constructing mice with superstrong cardinals.
The only plausible hope is that fully backgrounded constructions will.
But the fully backgrounded approach will also break down if a Woodin cardinal can exist in a forcing extension of a Nairian model assuming  less than a superstrong cardinal.
It could be that mouse existence is eventually just a hypothesis, rather than a provable consequence of large cardinals.

\begin{question}
    Can there exist a Woodin cardinal in a forcing extension of a Nairian model?
\end{question}

Corollary \ref{failures of square are weak} also raises the possibility that there is a profound difference between bare PFA and PFA with large cardinals.
Apart from the Viale-Weiss theorem, the expectation that failures of square are strong has largely constituted the evidence that PFA is equiconsistent with a supercompact cardinal.
There could be a method of constructing models of PFA which is not subject to the Viale-Weiss theorem and whose consistency upper bound is closer to the upper bound of square failing everywhere.
In contrast, Neeman and Trang have independently shown that if PFA holds and there exists a Woodin cardinal, then there is an inner model with a Woodin limit of Woodin cardinals.

The following conjecture says that the expectation that failure of square everywhere reaches supercompactness is true in spirit, just in the determinacy sense of supercompactness rather than the ZFC sense.
Theory (2) holds in the Nairian model used to establish \rthm{square and measurability in hod} \cite{blue2025nairian}.

\begin{conjecture}
    The following theories are equiconsistent.
    \begin{enumerate}
        \item \(\mathrm{ZFC} + \forall\kappa \neg\square_{\kappa}\).

        \item \(\mathrm{ZF} + \mathrm{AD}_{\bR} + \text{``\(\omega_1\) is a supercompact cardinal.''}\)
    \end{enumerate}
\end{conjecture}

The second conjecture that Theorem \ref{square and measurability in hod} bears on is Woodin's HOD conjecture.
Let $S^{\kappa}_{\omega}$ denote the set of $\alpha<\kappa$ having countable cofinality.
A cardinal $\kappa$ is $\omega$-\emph{strongly measurable in} HOD if for some $\gamma<\kappa$ such that $(2^{\gamma})^{HOD}<\kappa$, there is no ordinal definable partition of $S^{\kappa}_{\omega}$ into $\gamma$-many stationary sets.
Woodin's HOD Dichotomy Theorem, as optimized by Goldberg \cite{goldberg2024strongly}, states that if there is a strongly compact cardinal $\delta$, then either (1) no regular cardinal greater than $\delta$ is $\omega$-strongly measurable in HOD or (2) every regular cardinal greater than $\delta$ is $\omega$-strongly measurable in HOD.
The \emph{HOD hypothesis} is that there are arbitrarily large regular cardinals which are not $\omega$-strongly measurable in HOD, and the \emph{HOD conjecture} is that ZFC + ``there is a supercompact cardinal'' proves the HOD hypothesis.
It has been conceivable that ZFC alone could do so.

When Woodin formulated the HOD conjecture, it was unknown how to obtain models with more than three regular cardinals which are \(\omega\)-strongly measurable in HOD or with a successor of a singular cardinal of uncountable cofinality which is \(\omega\)-strongly measurable in HOD \cite{woodin2016hod}.
Ben Neria and Hayut showed that it is consistent relative to a surprisingly mild large cardinal hypothesis that every successor of a regular cardinal is $\omega$-strongly measurable in HOD \cite{HayutBenNeria}.
If the HOD hypothesis were provable in ZFC, the proof would thus have to utilize the combinatorics of successors of singular cardinals.
Theorem \ref{square and measurability in hod} eliminates this possibility.

\begin{corollary}\label{cor: hod hypo not provable in zfc}
    The \(\H\) hypothesis is not provable in \(\mathrm{ZFC}\).
\end{corollary}

Corollary \ref{cor: hod hypo not provable in zfc} bears out the need---supposing that the HOD hypothesis is true---to posit a hypothesis beyond ZFC in order to prove it, presumably one strong enough to establish a version of the HOD Dichotomy, such as the existence of a supercompact cardinal.
We emphasize that the present work does not refute the HOD conjecture, as there is obviously no supercompact cardinal in our ZFC model.

Supercompact cardinals are at the heart of both the PFA and HOD conjectures.
Supercompactness in the measure sense is central to the present work.
We extend work of Becker and Jackson \cite{becker2001supercompactness} and Jackson \cite{jackson2001weak} which shows that the projective ordinals are supercompact to their supremum and Suslin cardinals are supercompact to ordinals below \(\Theta\), respectively.

\begin{theorem}
    In the Nairian model, \(\omega_2\) is supercompact to ordinals, and the witnessing measures are ordinal definable.
    In fact, all Suslin cardinals are supercompact to ordinals, and so is \(\Theta\).
\end{theorem}




Finally, Woodin has conjectured that Nairian models defined using \(\omega\)-sequences are in fact closed under \({<}\Theta\)-sequences, the intuition being that the set of reals of Wadge rank \(\theta_{\alpha}\) in the parent \(\mathrm{AD}^+\) model should be too complex to define from below even granting generous parameters.
We confirm this for the minimal Nairian model we force over.

\begin{theorem}\label{thm: closure under <Theta sequences}
    Let \(N\) be the Nairian model defined at the least inaccessible limit \(\varsigma_\infty\) of cardinals which are \({<}\varsigma_\infty\)-strong in \(\mH\).
    Let \(\zeta<\Theta^N\).
    Then for all \(\beta<\varsigma_\infty\), \[(N|\beta)^{\zeta} \subseteq N.\]
\end{theorem}

\subsection{Problems}
Of course, Theorem \ref{square and measurability in hod} motivates the following question.

\begin{question}
    Which natural combinatorial principles imply there is an inner model with a superstrong cardinal?
\end{question}

Here ``natural'' principles include failures of approachability, the tree property at all points in an interval of cardinals, saturated or dense ideals, etc.
The question is as much about whether useful principles are actually strong as it is about developing new arguments for proving consistency lower bounds at superstrong in light of the above discussion.
For any principle, like \(\neg\square_{\kappa}\), which turns out to be weaker than superstrong, it is of interest whether it is forceable over a ZFC ground model from optimal hypotheses.

In the following question, \(\uB\) denotes the family of universally Baire sets of reals, and \(\Theta^{\uB}\) is the supremum of the ordinals which the reals can be surjected onto in \(L(\uB,\bR)\).

\begin{question}
    Is there a natural theory which implies that \(\Theta^{\uB} <\omega_3\) and \(\delta^1_2 = \omega_2\)?
    Is there one which extends \(\mathrm{ZFC} + \mathrm{MM}^{++}\)?
\end{question}

Viale \cite{Vi2016} defined the axiom \(\mathrm{MM}^{+++}\) and showed it implies that the theory of the \(\omega_1\)-Chang model is absolute for stationary set preserving forcings which preserve it.
Woodin defined \(\mathrm{CM}^{+}\), the Chang model augmented with the club filter on \(\powerset_{\omega_1}(\lambda)\) for all ordinals \(\lambda\), and proved it exists assuming class many Woodin limits of Woodin cardinals \cite{Wo21}.
Assuming the determinacy of a certain class of long games, he also showed that there exists a sharp for \(\mathrm{CM}^+\) \cite[Theorems 7.37, 7.39]{Wo21}.

\begin{question}
    Does absoluteness of the theory of \(\mathrm{CM}^{+}\) imply that the theory of \(L(\ord^{\omega_1})\) is absolute for stationary set preserving forcings?
\end{question}

The cardinal structure above \(\Theta\) in Nairian models, and the control one can maintain over it in their ZFC forcing extensions, raise the prospect that there may be ``determinacy forcing axioms'' like Axiom  \((*)\) but for larger powersets than \(\powerset(\omega_1)\) or \(\powerset(\bR)\).

\begin{question}
    Does Axiom \((*)\) generalize to larger cardinals?
\end{question}

\subsubsection*{Acknowledgments}
The authors would like to thank W.~Hugh Woodin for conversations about the content of this paper and for his encouragement.

\subsection{Reader's guide}

In this paper we only directly investigate minimal Nairian Models (though our results hold for a broader class of Nairian Models).
We summarize the main technical contributions as follows.
\vspace{1em}

\noindent
(1) Part 1 of the paper isolates a natural theory $\NTbase$ (\rdef{nairian theory}) and shows how to force the Axiom of Choice over its models via a very simple poset obtained by iteratively well-ordering the powerset of cardinals (\rthm{thm:forcing choice}).

\begin{theorem}\label{thm: part 1}
    Suppose \(M\models\mathrm{ZF} + \mathsf{NT_{base}}\).
    Then there is a homogeneous countably closed class iteration \(\mathbb{P}\) such that if \(G\subseteq\mathbb{P}\) is \(M\)-generic, then \[M[G]\models\mathrm{ZFC}.\]
\end{theorem}

\noindent
(2) The $\sf{Axiom\ of\ Harmony}$ (\rdef{def: ah}), one component of $\NTbase$, has already been shown to hold in Nairian models in \cite{blue2025nairian}. The main new technical property of minimal Nairian Models established here is the reflection principle $\rref$ (\rdef{def: kappa reflection}), which is the second clause of $\NTbase$ (\rdef{nairian theory}). $\rref$ allows us to prove $\DC_\kappa$ using the results of \rsec{sec: lifting dc}.
\vspace{1em}

\noindent
(3) Part 2 of the paper uses heavy descriptive inner model theoretic machinery, and it assumes \(\mathsf{HPC}\). It starts by proving a simple reflection property (\rthm{thm: basic reflection}) that can be used to show that if $N$ is a minimal Nairian Model, then in the parent determinacy model, $N$ is closed under ${<}\Theta^N$-sequences (\rthm{thm: closure under theta sequences}). 
\vspace{1em}

\noindent
(4) The reflection principle established in \rthm{thm: basic reflection} also implies that if $G\subseteq \pmax$ is generic over $V$ (the parent determinacy model of the Nairian Model in question), then in $V[G]$, $N[G]$ is closed under ${<}\omega_1$-sequences (\rcor{cor: omega1 closure}). This fact is used to show that the square principles $\square(\gamma)$ fail in the $\ZFC$-extension of $N$ by applying \cite[Theorem 3.3]{blue2025nairian}, which in turn comes from \cite{CLSSSZ}.
\vspace{1em}

\noindent
(5) The final sections of the paper are devoted to proving $\rref$ in minimal Nairian Models. 

\begin{theorem}\label{thm: part 2}
    Let \(N\) be a minimal Nairian model. Then $N\models \mathsf{NT_{base}}$.
\end{theorem}

\section{Threading coherent sequences on limit cardinals}\label{sec: square}

Let \(\gamma\) be an uncountable cardinal.
A \(\square_{\gamma}\)-\textbf{sequence} is a sequence \(\langle C_{\alpha} : \alpha<\gamma^+ \rangle\) such that for each $\alpha < \gamma^{+}$, \(C_{\alpha}\) is a club in \(\alpha\), \(C_{\beta} = C_{\alpha}\cap\beta\) whenever \(\beta\) is a limit point of \(C_{\alpha}\), and \(C_{\alpha}\) has order type at most \(\gamma\).
The principle \(\square_{\gamma}\) is the assertion that there is a \(\square_{\gamma}\)-sequence.

It is easier to work with the nonthreadability principles.

\begin{definition}[\(\square(\gamma,\delta)\)]\normalfont
    Let $\gamma$ be an ordinal and $\delta$ a cardinal.
    Then there is a sequence $\langle \mathcal{C}_{\alpha} \mid \alpha < \gamma \rangle$ such that
    \begin{enumerate}
        \item For each $\alpha < \gamma$,
        \begin{enumerate}
            \item $0 < |\mathcal{C}_{\alpha}| \leq \delta$,
            \item each element of $\mathcal{C}_{\alpha}$ is club in $\alpha$, and
            \item for each member $C$ of $\mathcal{C}_{\alpha}$, and each limit point $\beta$ of $C$, \[C \cap \beta \in\mathcal{C}_{\beta}.\]
        \end{enumerate}
        \item There is no club $E \subseteq \gamma$ such that $E \cap \alpha \in \mathcal{C}_{\alpha}$ for every limit point $\alpha$ of $E$.
\end{enumerate}
\end{definition}

The principle \(\square(\gamma,1)\) is the principle \(\square(\gamma)\).
Since for any uncountable cardinal \(\gamma\), \(\square_{\gamma}\) implies \(\square(\gamma^+)\), it suffices to show that \(\square(\gamma^+)\) fails for all \(\gamma\) in the extension by the forcing in Theorem \ref{thm:forcing choice}.

Apply \cite[Theorem 3.3]{blue2025nairian} to \(V\), the ambient determinacy model, and a minimal Nairian model \(N\).
Let \(\kappa = \Theta^N\).
Let \(G\subseteq\mathbb{P}_{\mathrm{max}}\) be \(V\)-generic.
Let \(\mathbb{Q}\) be the poset from Theorem \ref{thm:forcing choice} as defined in \(N[G]\).
In \(N[G]\), \(\mathbb{Q}\) is \({<}\omega_2\)-directed closed.
\rcor{cor: omega1 closure} implies that every \(\omega_1\)-sequence from \(\mathbb{Q}\) in \(V[G]\) exists in \(N[G]\), hence \(\mathbb{Q}\) is \({<}\omega_2\)-directed closed in \(V[G]\).
Letting \(H\subseteq\mathbb{Q}\) be \(V[G]\)-generic, \cite[Theorem 3.3]{blue2025nairian} implies that \[N[G*H]\models\neg\square(\gamma,\omega)\] for every \(\gamma\in [\kappa,\Theta^V)\) with \(\cof^V(\gamma)>\omega_1\).
Each of the ordinals \(\kappa_{\alpha} =_{def} (\kappa^{+\alpha})^N\) is such a \(\gamma\).
The \(N\)-cardinals \(\kappa,\kappa_1,\dots,\kappa_{\alpha},\dots\) are the cardinals \(\aleph_3,\aleph_4,\dots,\aleph_{3+\alpha},\dots\) in \(N[G*H]\).
Finally, that \[N[G*H]\models\neg\square(\aleph_2,\omega)\] follows from the fact that \(\mathrm{MM}^{++}(c)\) holds in \(N[G*H]\), via the proof of Theorem \ref{thm: Todorcevic}.

\section{\(\omega\)-strongly measurable cardinals in HOD}

For any cardinal \(\kappa\), let $F_\kappa$ be the $\omega$-club filter on \(\kappa\). We write $F_{\kappa}(W)$ for $F_{\kappa} \cap W$ for any set or class $W$. The following is an easy lemma.

\begin{lemma}[ZF]\label{lem: easy lem 1} Let $W$ be a transitive class containing \(\H\) and let $\kappa$ be a regular uncountable cardinal such that 
\begin{enumerate}
    \item $F_\kappa(W) \in W$, and

    \item  in $W$, $F_\k(W)$ is a $\k$-complete normal ultrafilter on $\k$.
\end{enumerate}
Then $\kappa$ is $\omega$-strongly measurable in $\H$. Moreover, 
if $S\in \H\cap \powerset(\k)$ is stationary, then $\kappa \setminus S$ is not stationary, and \begin{center} $\H\models ``F_\k(\H)$ is normal $\k$-complete ultrafilter.''\end{center}
\end{lemma}

We now show that if $\bbP$ is any countably closed homogeneous poset that forces $\ZFC$ over the Nairian Model, then $\bbP$ forces that every uncountable regular cardinal is $\omega$-strongly measurable in $\H$. Recall that $\mH$ is the hod premouse representation of $\H$. The notation we use is that of \cite{blue2025nairian}; the reader can review some of it in \textsection\ref{sec: prelim}.

\begin{theorem}\label{thm: omega club over nm 1} Assume $\treg+{\sf{HPC}}+{\sf{NLE}}$, and suppose that $\xi<\Theta$ is a Nairian cardinal. Let $N$ be the Nairian Model at $\xi$, $\bbP\in N$ be a countably closed poset, and $G\subseteq \bbP$ be $N$-generic. Then, in $N[G]$, for every uncountable regular cardinal $\k$ that is not a superstrong cardinal in $\mH$, \[F_\k(\mH)=E_\k,\] where $E_\k\in \vec{E}^\mH$ is the least total Mitchell order 0 ultrafilter on $\k$.\footnote{$E_\k$ is really an extender, but we treat it as an ultrafilter.}
\end{theorem}

\begin{proof} 
We will need the following basic claim.

\begin{claim}\label{clm: tail branch} Suppose that $\M$ is a transitive fine structural model of some kind (e.g. a mouse, hod mouse, etc), $\k$ is a measurable cardinal of $\M$, and $\T$ is a normal iteration tree\footnote{As in \cite{blue2025nairian}, we use Jensen indexing.} of length $\gg+1$ on $\M$ whose main branch $b=[0, \gg)_\T$ does not have a drop. Let $\W$ be the last model of $\T$. Suppose that for some $\iota_0\in b$, for every $\iota\in b-\iota_0$, \begin{center}$\sup(\pi^\T_{\iota, \gg}[\k])<\pi^\T(\k)$.\end{center} Then there is $\a \in b-\iota_0$, such that for all $\b\in b-\a$ \begin{center} $\pi_{\b, \gg}(\cp(E_\xi^\T))=\pi^\T(\k)$\end{center} where $\xi$ is such that $\T(\xi+1)=\b$ and $\xi+1\in b$.
\end{claim}
\begin{proof}
Let $b = [0, \gamma]_\mathcal{T}$ be the main branch of $\mathcal{T}$. For $\iota \in b$, let $\kappa_\iota = \pi_{0, \iota}^\T(\kappa)$. For any $\iota \in b$, let $\xi(\iota)$ be such that $\mathcal{T}(\xi(\iota)+1) = \iota$ and $\xi(\iota)+1 \in b$. Let $\mu_\iota= \cp(E_{\xi(\iota)}^\mathcal{T})$. 

Since $\mathcal{T}$ is normal, the sequence $\langle \mu_\beta \mid \beta \in b \rangle$ is non-decreasing. Since for every $\iota\in b-\iota_0$, $\pi_{\iota, \gg}^\T$ is discontinuous at $\kappa_\iota$, we have that for every $\iota\in b-\iota_0$, $\mu_\iota\leq \k_\iota$. 

It follows that there must be an $\iota\in b-\iota_0$ such that $\mu_\iota=\k_\iota$, as otherwise if $\mu_\iota<\k_\iota$ for all $\iota\in b-\iota_0$, then $\pi_{\iota_0, \gg}^\T$ is continuous at $\k_{\iota_0}$. Let $\a\in b-\iota_0$ be the least such that $\mu_\a=\k_\a$.

We now claim that for all $\beta \in b-\a$, $\mu_\beta= \kappa_\beta$. Otherwise we have $\b\in b-\a$ such that $\mu_\b<\k_\b$. Assuming $\b$ is the least with this property, the condition $\mu_\b<\k_\b$ violates normality. Thus, for all $\b\in b-\a$, $\mu_\beta= \kappa_\beta$.
\end{proof}

Suppose now $\hq$ is a hod pair such that $\k_\hq$ is defined. Let $\T=\T_{\hq, \hq^\k}$. Applying \rcl{clm: tail branch} and letting $b$ be the main branch of $\T$, we get some $\a\in b$ such that for every $\b\in b-\a$, letting $\xi(\b)$ and $\mu_\b$ be defined as in \rcl{clm: tail branch}, \[\pi_{\b, b}^\T(\mu_\b)=\k.\] 
We now need to identify the stages in $b$ that use the Mitchell-order $0$ measure. Let  $\l=(\k^+)^\mH$ and $X=\pi_{\hq, \hq^\k}[\M^\hq|\l]$. Put $\tau\in D$ if and only if 
\begin{itemize}
\item $\tau=\k\cap\hull^{\mH|\l}(\tau\cup X)$, and
\item $\mH\models ``\tau$ is not a measurable cardinal".
\end{itemize}
Since non-measurable inaccessible cardinals of $\mH$ have cofinality $\omega$, we have that $D$ is an $\omega$-club. Notice also that \[\{ \mu_\b: \b\in b-\a\}\subseteq D\ \ \ \ \ \ \ (*)\]
The following is a key claim.
\begin{claim}\label{clm: d is in b} Suppose $\tau\in D-\mu_\a$. Then for some $\b\in b-\a$, $\tau=\mu_\b$.
\end{claim}
\begin{proof} The claim is proved by induction. We assume that the claim is true for all $\tau'\in D\cap \tau$ and show that the claim holds for $\tau$. Notice that because $b$ is a club, if $\tau$ is a limit point of $D$, then $\tau=\mu_\b$ for some $\b\in b-\a$. We can then assume that $\tau$ is a successor point of $D$. Let $\tau'=\sup(D\cap \tau)$ and $\b\in b$ be such that $\tau'=\mu_\b$. Let $\xi=\min(b-(\b+1))$. We want to show that $\tau=\mu_\xi$. Let $E$ be the extender used at $\b$ in $b$. We then have that \[\pi_{\b, \xi}^\T(\mu_\b)=\mu_\xi.\] 

Suppose now that $\tau<\mu_\xi$ (because $\mu_\xi\in D$ and $\tau'=\sup(D\cap \tau)$, $\mu_\xi<\tau$ is not possible). Let $F$ be the Jensen-completion of $E\rest \tau$. We have that $F$ is not of superstrong type, and so \[\pi_F^{\M_\b^\T}(\mu_\b)>\tau,\] which implies that there is some $f\in \mH\cap \powerset(\mu_\b^{\mu_\b})$ and some $a\in \tau^{<\omega}$ such that $\pi_F^{\M_\b^\T}(f)(a)>\tau$. Because $F$ is a subextender of $E$, letting $\l_\b=\pi_{0, \b}^\T(\l)$, we have \[k: Ult(\M_\b^\T, F)\rightarrow \M^{\hq^\k}\]
such that \[\pi_{\b, b}^\T=k\circ \pi_F^{\M^\T_\b}.\] Hence, since \[k(f)\in\hull^{\mH|\l}(\mu_\b\cup  X),\]
we have that \[\tau<\sup(\k\cap \hull^{\mH|\l}(\tau\cup  X)),\]
which is a contradiction. 
\end{proof}
It is now easy to finish the proof of \rthm{thm: omega club over nm 1}. Notice first that $D\in N$, as it is easily definable from $X\in N$ and $\l$. If now $\tau\in D-\mu_\a$, $\b\in b$ is such that $\mu_\b=\tau$, and $E$ is the extender used at $\b$ in $b$, then $E$ is a Mitchell-order $0$ ultrafilter, as otherwise we would have that $\mH\models ``\tau$ is a measurable cardinal.'' It now follows that for every $A\in \mH\cap \powerset(\k)$, \[A\in E_\k \iff \exists \nu (D-\nu\subseteq A).\]
Clearly, the above equivalence implies that in $N[G]$, $F_\k(\mH)=E_\k$.
\end{proof}

Applying \rlem{lem: easy lem 1}, \rthm{thm: omega club over nm 1}, and the fact that every regular cardinal of a Nairian model is a measurable cardinal in $\mH$ \cite{SteelCom}, we get the following corollary. 

\begin{corollary}\label{cor: omega-strmes} Assume $\treg+{\sf{HPC}}+{\sf{NLE}}$, and suppose $N$ is a minimal Nairian model. Suppose $\bbP=\pmax*\bbQ\subseteq N$  is the iteration described in \rthm{thm:forcing choice}. Then in $N^\bbP$, every uncountable regular cardinal $\k$ is $\omega$-strongly measurable in $\H$, and in fact, $$\H\models \text{``$F_\k(\H)$ is a $\k$-complete normal ultrafilter.''}$$
\end{corollary}

\part{Forcing choice over Nairian Models}

In \textsection\ref{NairianTheorySec} we introduce the Nairian Theory $\mathsf{NT_{base}}$ and show how to force the Axiom of Choice over models of this theory. The theory is stated in terms of the $\nN$-\emph{hierarchy}, a way of stratifying the universe of sets via the relation $|X| \lesssim |Y|$, which says that there is a surjection from $Y$ to $X$. The levels of this hierarchy are $H$-\emph{like sets}, as defined in \textsection\ref{sec: H-like sets}.

\section{$H$-like sets}\label{sec: H-like sets}

Since we will be working in models of set theory in which the Axiom of Choice fails, we will compare sets via the existence of surjections, as opposed to injections. 

\begin{definition}\label{def:cardinality}\normalfont
Suppose $X$ and $Y$ are sets.
Then
\begin{enumerate}
    \item $\card{X}\lesssim \card{Y}$ indicates that either $X = \emptyset$ or there is a surjection $f\colon Y\rightarrow X$.

    \item $\card{X} \approx \card{Y}$ indicates that $\card{X}\lesssim \card{Y}$ and $\card{Y}\lesssim \card{X}$.

    \item if $\kappa$ is a cardinal, then $\card{X}\lesssim \kappa$ indicates $\card{X}\lesssim \card{\kappa}$ and $\card{X} \approx \kappa$ indicates $\card{X} \approx \card{\kappa}$.
\end{enumerate}

\end{definition}


In this section we introduce the $H$-\emph{like sets}, analogues of the usual sets $H(\theta)$ via the relation $\lesssim$. First we define several related properties.

\begin{definition}\label{largest cardinal}\label{total}\label{regulardef}\normalfont Let \(M\) be a set.
\begin{enumerate}
    \item $M$ \textbf{has a maximal cardinality} if there is $X\in M$ such that $M\models \forall Y(\card{Y}\lesssim \card{X})$. In this case, we say that $X$ is a \textbf{maximal cardinality of} $M$.

    \item $M$ is of \textbf{successor type} if it has a maximal cardinality.

    \item $M$ is of \textbf{limit type} if it is nonempty and it does not have a maximal cardinality.
    \item $M$ is \textbf{transitively closed} if for each $Y\in M$ there is a transitive set $Z\in M$ such that $Z\subseteq M$ and $Y\in Z$.

     \item $M$ is \textbf{closed under products} if  $X\times Y\in M$ whenever $X, Y\in M$.

    \item $M$ is \textbf{closed under subsets} if for every $X\in M$, $\powerset(X)\subseteq M$.

    \item $M$ is \textbf{full} if $M$
    \begin{enumerate}
        \item is transitively closed and 
        \item is closed under products and subsets.
    \end{enumerate}

     \item $M$ is \textbf{strongly regular} if, whenever $X \in M$ and $f$ is a function from $X$ to $\ord \cap M$, the range of $f$ is bounded in $\ord \cap M$.
\end{enumerate} 
\end{definition}



%

Note that transitively closed sets are transitive, so being transitively closed is equivalent to being transitive and closed under the function sending each set $Y$ to the transitive closure of $\{Y\}$. Note also that if $X$ is a maximal cardinality of a set $M$ which is closed under products and subsets, then $M \models \card{X} \approx \card{X^{n}}$ for all positive $n \in \omega$. We record the following facts.

\begin{lemma}\label{lem: easy lem1} Suppose that $M$ is full and of successor type. Then there is a transitive $X\in M$ such that $X$ is a maximal cardinality of $M$.
\end{lemma}

\begin{proof} Let $Y$ be a maximal cardinality of $M$, and let $X\in M$ be a transitive set such that $Y\in X$. We then have that $M\models \card{Y}\lesssim \card{X}$ (because $X\times Y\in M$, so $\powerset(X\times Y)\subseteq M$). Thus, $X$ is a maximal cardinality of $M$. 
\end{proof}

\begin{lemma}\label{closure under cartesian products} Suppose that $M$ is a full set. Then for any $n\in \omega$ and $\langle Y_i: i\leq n \rangle\in M$, $\powerset(\prod_{i\leq n}Y_i)\subseteq M$. Furthermore, if $X$ is transitive and a maximal cardinal of $M$, then $\card{M}\approx\card{\powerset(X)}$.
\end{lemma}

\begin{proof}
The first part of the lemma follows almost immediately from the definitions. To prove the second part, note that if $Y\in M$ and $Z$ is the transitive closure of $\{Y\}$, then there exist $E,F \subseteq X^2$ such that $E$ is an equivalence relation and $(Z, \in)$ is isomorphic to $(X/E, F/E)$, where $X/E$ is the set of $E$-equivalence classes and $F/E$ is the relation on $X/E$ induced by $F$. This observation allows us to define a surjection from $\powerset(X)$ to $M$, using the fact that $\card{X}\approx\card{X^2}$.
\end{proof}

The following property will be used crucially in conjunction with strong regularity. 

\begin{definition}\label{hierarchical}\normalfont 
A set $M$ is \textbf{hierarchical} if $M$ is transitive and, letting $\k=\ord \cap M$, there exists a $\subseteq$-increasing continuous sequence $\langle M_\a: \a<\k\rangle$ of transitive sets such that
\begin{enumerate}
    \item $M=\bigcup_{\a<\k}M_\a$ and, 
    \item for each $\a<\k$, $M_\a\in M$. 
\end{enumerate}
The sequence $\langle M_\a: \a<\k\rangle$ is a \textbf{hierarchical decomposition} of $M$.
\end{definition}


The property of being \(H\)-\emph{like} is the main notion introduced in this section. 
Levels of the $\nN$-hierarchy (introduced in Section \ref{NairianTheorySec}) will be \(H\)-like.

\begin{definition}\label{an for m}\normalfont Suppose that $M$ is a transitive set. Then $M$ is $H$-\textbf{like} if
\begin{enumerate}
\item  $M$ is full and hierarchical, and
\item  if $M$ is of successor type, then $M$ is strongly regular.
\end{enumerate}
\end{definition}

The following statement is a consequence of our theory $\mathsf{NT_{base}}$ (see Remark \ref{firstAHrem}). 

\begin{definition}\label{hlike axiom}\normalfont ``$V$ is $H$-like'' is the axiom asserting that every set belongs to an $H$-like set of successor type.
\end{definition}

The successor levels of the $\nN$-hierarchy will have the form \(H_X\), as in the following definition.  

\begin{definition}\label{h sets} \normalfont
Let $X$ be a set.
Then $H_X$ is the set of $Y$ that have hereditary size at most $\card{X}$, i.e. such that there is a surjection from $X$ to the transitive closure of $\{Y\}$. 
\end{definition}

Note that $H_{X} = H_{Y}$ whenever $|X| \approx |Y|$. 

\begin{definition}\label{true cardinal}\normalfont 
We say that $X$ is a \textbf{true cardinal} if $H_X$ is $H$-like. 
\end{definition}

The remainder of this section concerns $H$-like sets of successor type (which the successor levels of the $\nN$-hierarchy will be). We first note the connection between $H$-like sets and sets of the form $H_{X}$. 

\begin{theorem}\label{M=HXlem} If $M$ is $H$-like and $X$ is a maximal cardinality of $M$, then $M=H_X$. 
\end{theorem}

\begin{proof} It follows almost immediately from the definitions of fullness and maximal cardinality that $M$ is contained in $H_{X}$. For the converse, it suffices by Lemma \ref{lem: easy lem1} to consider the case where $X$ is transitive and $\card{X} \approx \card{X^{n}}$ for all positive $n \in \omega$. Suppose then that $f$ is a surjection from $X$ to the transitive closure of $\{Y\}$, which we will call $Z$. Arguing by induction on the rank of the elements of $Z$, it suffices to consider the case where $Y \subseteq M$. Let $\langle M_{\alpha} : \alpha < \kappa \rangle$ witness that $M$ is hierarchical. Since there is a surjection from $X$ to $Y$, and since $M$ is strongly regular, there is an $\alpha < \kappa$ such that $Y \subseteq M_{\alpha}$. Since $M_{\alpha} \in M$ and $M$ is closed under subsets, it follows that $Y$ is in $M$. 
\end{proof} 

The proof of the previous theorem gives a bit more, which we record here. 

\begin{corollary}\label{cor: zf in s-like} Suppose that $M$ is an $H$-like set of successor type, $Y\in M$ and $\{ A_y: y\in Y \} \subseteq M$. Then $\langle A_y: y\in Y \rangle\in M$, and so $\bigcup_{y\in Y} A_y\in M$. 
\end{corollary}


\begin{theorem}\label{thrm: zf in s-like} If $M$ is an $H$-like set of successor type, then $M\models \ZF-\sf{Powerset}$. 
\end{theorem}

\begin{proof} The proof that  $M \models {\sf{Replacement}}$ is very similar to the proof of Theorem \ref{M=HXlem}. 
Let $\kappa=\ord\cap M$, and let $\langle M_\a: \a<\kappa \rangle$ be a hierarchical decomposition of $M$. 
Let $f$ be a function from $Y$ to $M$ for some $Y \in M$. 
Because $M$ is strongly regular, there exists a $\b<\kappa$ such that for all $y\in Y$, $f(y)\in M_\b$. Since $M_{\beta} \in M$ and $M$ is closed under subsets, the range of $f$ is in $M$. 

Corollary \ref{cor: zf in s-like} implies that $M$
 satisfies $\sf{Pairing}$ and $\sf{Union}$, and closure under subsets gives $\sf{Comprehension}$. 
The proofs for the other axioms are straightforward. 
\end{proof}

\rthm{preservation of normality}, the main theorem of this section, shows that forcing extensions of successor $H$-like sets are successor $H$-like sets. It follows that set-sized forcing preserves the statement ``$V$ is $H$-like''.

\begin{theorem}\label{preservation of normality} Suppose that $M$ is an $H$-like set, and that $X$ is transitive and a maximal cardinality of $M$. Suppose that $\mathbb{P}\in M$ is a poset and that $G\subseteq \bbP$ is a $V$-generic filter.
Then $M[G]=(H_X)^{V[G]}$ and \[V[G]\models \text{``$H_X$ is $H$-like''.}\]
\end{theorem}

\begin{proof} Let $\kappa = M \cap \ord$, and let $\langle M_\a: \a<\kappa\rangle$ be a hierarchical decomposition of $M$ with $\bbP \in M_{0}$. 

\begin{claim}\label{claim 1}
If $\sigma$ is a $\mathbb{P}$-name such that $\sigma_G$ is a subset of $X$, then there is $\tau\in M$ such that $\sigma_G=\tau_G$. Hence, $\powerset(X)^{V[G]}\subseteq M[G]$.
\end{claim}

\begin{proof}
Let $p\in G$ be such that $p\forces \sigma\subseteq \check{X}$, and let $\tau$ be the set of $(q, x)\in \bbP\times X$ such that $q\leq p$ and $q\forces \check{x}\in \sigma$. We then have that $\tau\in M$, and $\sigma_G=\tau_{G} =\{ x: \exists q \in G ((q, x)\in \tau)\}$. 
\end{proof}

\begin{claim}\label{claim 2} $M[G]=(H_X)^{V[G]}$. 
\end{claim}

\begin{proof} First, suppose that $Y = \sigma_{G}$, for some $\bbP$-name $\sigma$ in $M$. We want to see that $V[G]$ contains a surjection from $X$ onto the transitive closure of $\{Y\}$. The union of $\bbP \times \{\sigma\}$ with the transitive closure of $\{\sigma\}$ contains a $\bbP$-name for the transitive closure of $\{\sigma_{G}\}$. Since $X$ is a maximal cardinal of $M$, there is a surjection from $X$ onto this name, which induces a surjection from $X$ onto the realization of the name in $V[G]$. 

For the reverse inclusion, suppose that $Y\in (H_X)^{V[G]}$ is a transitive set and that $f \colon X\rightarrow Y$ is a surjection in $V[G]$. Let $E=\{ (a, b)\in X^2: f(a)\in f(b)\}$. Then $E\in M[G]$, by Claim \ref{claim 1}, so $E = \sigma_{G}$ for some $\bbP$-name $\sigma$ in $M$. Arguing by induction on $\alpha$ as in the proof of Theorem \ref{M=HXlem} (or using Theorem \ref{thrm: zf in s-like}), one can show that whenever $M$ contains a $\bbP$-name for a wellfounded extensional relation of rank $\alpha$, it contains a $\bbP$-name for a transitive set whose ordering by $\in$ is isomorphic to this relation. 
\end{proof}

It follows from Claim \ref{claim 2} that $\langle M_{\alpha}[G] : \alpha < \kappa\rangle$ is a hierarchical decomposition of $M[G]$. 

\begin{claim}\label{claim: preservation of strong regularity} In $V[G]$, $M[G]$ is strongly regular.
\end{claim}

\begin{proof} 
By Claim \ref{claim 2} it suffices to fix an $f\colon X\to \kappa$ in $V[G]$ and show that $f[X]$ is bounded below $\kappa$. Let $\dot{f}$ be a $\bbP$-name in $V$ such that $\dot{f}_{G} = f$, and let $p\in G$ be such that $p\forces \dot{f}\colon\check{X}\to \check{\kappa}$. 
Since $M$ is strongly regular and $\bbP \times X \in M$, there exists a $\tau < \kappa$ such that $\xi < \tau$ whenever some $r \leq p$ forces that $\dot{f}(\check{x}) = \xi$. 
\end{proof}
 

To show that $M[G]$ is $H$-like in $V[G]$, it remains only to show that it is full. 
Since each $M_\a[G]$ is transitive, we have that $M[G]$ is transitively closed. It follows from Claim \ref{claim 1} and Claim \ref{claim 2} that $M[G]$ is closed under subsets and products. 
\end{proof}

\section{On Dependent Choice}

Our forcing construction for producing a model of the Axiom of Choice will proceed by forcing successively stronger forms of the axiom of Dependent Choice. 

\begin{definition}\label{def: kappa closed}\normalfont Suppose that $T$ is a binary transitive relation on a set $X$. 
\begin{enumerate}
    \item Given an ordinal $\eta$, an $\eta$-\emph{chain through} $T$ is a function $f \colon \eta \to X$ such that \[(f(\alpha), f(\beta)) \in T\] for all $\alpha < \beta < \eta$.

    \item A $T$-\textbf{upper bound} for an $\eta$-chain $f$ is a $u \in X$ such that $(f(\alpha), u)$ for all $\alpha < \eta$.

    \item We say that $T$ is $\k$-\textbf{closed} if $f$ has a $T$-upper bound whenever $\eta<\k$ and $f\colon \eta\rightarrow X$ is an $\eta$-chain through $T$.
\end{enumerate}
\end{definition}

\begin{definition}\normalfont Let \(\kappa\) be a cardinal, and let $X$ be a set.
Then 
\begin{enumerate}
    \item $\kappa$-\textbf{Dependent Choice for} $X$ ($\DC_\k(X)$) is the statement that if $T$ is a $\k$-closed binary transitive relation on $X$, then there is a $\k$-chain through $T$, and

    \item $\kappa$-\textbf{Dependent Choice} ($\DC_{\kappa}$) is the statement that $\DC_{\kappa}(X)$ holds for all sets $X$. 
\end{enumerate}
\end{definition}

\begin{remark}\label{rem:dcac}\hspace{.1in}\normalfont
\begin{enumerate}
    \item It follows from the definitions that $\DC_{\lambda}$ implies $\DC_{\kappa}$ whenever $\lambda > \kappa$.

    \item\label{dcrem2} If $\kappa$ is a cardinal and $Y$ is a surjective image of $X$, then $\DC_{\kappa}(X)$ implies $\DC_{\kappa}(Y)$. 

    \item For a given cardinal $\kappa$, $\DC_{\kappa}$ implies that for every set $X$ there exists either a wellordering of $X$ or an injection from $\kappa$ to $X$ (consider the set of functions $f \colon \gamma \to X$ which are either surjective or injective, for some $\gamma < \kappa$, ordered by extension).

    \item When $\kappa$ is singular, $\DC_{\kappa}$ follows from $\bigwedge_{\lambda < \kappa}\DC_{\lambda}$.

     \item The Axiom of Choice is equivalent to the assertion that $\DC_{\kappa}$ holds for all cardinals $\kappa$, which is equivalent to the assertion that $\DC_{\kappa}$ holds for a proper class of $\kappa$.
\end{enumerate}
\end{remark} 

The next lemma will be used to show that $\DC_{\kappa}$ is preserved by terminal segments of our forcing iteration. 

\begin{lemma}[$\ZF$]\label{lem: dc through closure} If $\kappa$ is an infinite cardinal such that $\DC_\k$ holds, and $\bbP$ is a $\k^+$-closed poset, then $\bbP$ forces $\DC_\k$.  
\end{lemma}

\begin{proof} Let $\dot{T}$ be a $\bbP$-name for a tree on a set given by a $\bbP$-name $\dot{X}$. Let $p\in \bbP$ be a condition which forces that $\dot{T}$ is $\k$-closed. Let $Y$ be the set of pairs $(q, \sigma)$ such that $q\leq p$ and $q\forces \sigma\in \dot{X}$. Let $S$ be the transitive binary relation on $Y$ defined by setting $((q, \sigma), (q', \sigma'))$ to be in $S$ if and only if $q'\leq q$ and $q'\forces (\sigma, \sigma')\in \dot{T}$. 

We claim that $S$ is $\k$-closed. Indeed, if $\eta<\k$ and $h:\eta\rightarrow Y$ is a chain of length $\eta$ then, setting each value $h(\alpha)$ to be $(q_\alpha, \sigma_\alpha)$, we have that for all $\alpha < \beta < \eta$, $q_{\beta}\leq q_{\alpha}$. Since $\bbP$ is $\k$-closed, we can find $q$ such that $q \leq q_{\alpha}$ for all $\alpha < \eta$. We then have that for all $\alpha < \beta < \eta$, $q\forces (\sigma_\alpha, \sigma_{\beta})\in \dot{T}$. It follows from the choice of $p$ that there exist $\sigma\in V^\bbP$ and $r\leq q$ such that $r\forces \forall \alpha<\check{\eta}( \sigma_{\alpha},  \sigma)\in \dot{T}$. Hence, for all $\alpha < \eta$, $((q_\alpha, \sigma_\alpha), (r, \sigma))\in S$. This shows that $S$ is indeed $\k$-closed. 

Using $\DC_\k$, we can find a chain $((q_\alpha, \sigma_\alpha):\alpha<\k)$ through $S$. Because $\bbP$ is $\k^+$-closed, we can find some $q\in \mathbb{P}$ such that for every $\alpha < \eta$, $q\leq q_\alpha$. It then follows that $q\forces ``(\sigma_\alpha:\alpha<\check{\k})$ is a $\check{\k}$-chain through $\dot{T}"$. 
\end{proof}

Given sets $X$ and $Y$, $X^{Y}$ denotes the set of functions from $Y$ to $X$. Given an ordinal $\eta$ and a set $X$, we write $X^{\less\eta}$ for $\bigcup_{\alpha < \eta}X^{\alpha}$.

\begin{remark}\label{goodrem}\normalfont Let $\kappa$ be an infinite cardinal for which $\DC_{\kappa}$ holds. 
If $\langle X_{\alpha} : \alpha < \kappa\rangle$ is such that $\card{X_{\alpha}} \lesssim \kappa$ for all $\alpha < \kappa$, then $\card{\bigcup_{\alpha < \kappa}X_{\alpha}} \lesssim \kappa$. 
In particular, if $\card{\alpha^{\beta}} \lesssim \kappa$ for all $\alpha, \beta < \kappa$, then $\card{\bigcup_{\alpha < \kappa}\alpha^{\less\kappa}} \approx\kappa$.
\end{remark}

\begin{definition}\normalfont If $\kappa$ is a cardinal and $X$ is a set, then \(X\) is $\less\kappa$-\textbf{closed} if $X^{\less\kappa} \subseteq X$. 
\end{definition}

\begin{lemma}\label{lem: ls} Suppose that $\k$ is an infinite regular cardinal, $\ZF+\DC_\k$ holds, and for every $\a<\b<\k$, $\card{\a^\b}\lesssim \k$. Suppose that $W$ is a $\less\k$-closed transitive structure.
Then for each $A \subseteq \kappa$ with $\card{A} \lesssim \kappa$ there is an $X\preceq W$ such that $A \subseteq X$,  $\card{X}\lesssim\k$ and $X$ is $\less\k$-closed.\footnote{This lemma appears not to be used.}
\end{lemma}

\begin{proof} We build $X$ as the union of a $\subseteq$-chain $\langle X_\a: \a<\k\rangle$, which we will obtain by applying $\DC_\k$ to a $\k$-closed transitive binary relation $U$. The domain of $U$ consists of those pairs $(X, f)$ such that $A \subseteq X \subseteq W$ and $f \colon \kappa \to X$ is a surjection. 
Given two such pairs $(X, f)$ and $(X', f')$, we set $((X,f),(X',f'))$ to be in $U$ if and only if 
\begin{enumerate}
\item $X \cup X^{\less\kappa}\subseteq X'$ and
\item whenever $s \in X^{<\omega}$ and $\varphi$ is a $(|s|+1)$-ary formula such that $W \models \exists x\varphi[x, s]$, there exists a $b \in X'$ such that $W \models \varphi[b, s]$. 
\end{enumerate}
Whenever $\langle (X_{\alpha}, f_{\alpha}) : \alpha < \eta \rangle$ is a $U$-increasing sequence, for some limit ordinal $\eta \leq \kappa$, there is a function $f$ such that $(\bigcup_{\alpha < \eta}X_{\alpha}, f)$ is in the domain of $U$, and, by the Tarski-Vaught criterion, $\bigcup_{\alpha < \eta}X_{\alpha} \preceq W$. It suffices then to see that for each $(X, f) \in \dom(U)$ there exists a pair $(X', f)$ with $((X, f), (X', f')) \in \dom(U)$. For the first condition on $X'$, Remark \ref{goodrem} implies that $\card{X \cup X^{\less\kappa}} \lesssim \kappa$ (applying the regularity of $\kappa$). For the second, fixing an enumeration $\langle (s_{\alpha}, \varphi_{\alpha}) : \alpha , \kappa\rangle$ of the pairs $(s, \varphi)$ in question (using $f$), an application of $\DC_{\kappa}$ gives a sequence $\langle b_{\alpha} : \alpha < \kappa\rangle$ such that $W \models \varphi[b_\alpha, s_\alpha]$ for each $\alpha < \kappa$. The set $X'$ can then be $X \cup X^{\less\kappa} \cup \{b_{\alpha} : \alpha < \kappa\}$. 
\end{proof} 

\section{A reflection principle}

We use the following notion of closure relative to an arbitrary set.

\begin{definition}\normalfont
Suppose $X$ is a set. Then $W$ is $X$-\textbf{closed} if $X\cup \{X\} \cup W^{X} \subseteq W$.
\end{definition}

It follows from the proof of Theorem \ref{thrm: zf in s-like} that if $M$ is a successor $H$-like set and $X \in M$, then $M$ is $X$-closed. Our proofs of the principles $\DC_{\kappa}$ will use the following reflection principle.

\begin{definition}[Reflection]\label{def: kappa reflection}\normalfont For sets $X$, $Y$ and $Z$, ${\sf{ref}}(X, Y, Z)$ is the statement that whenever 
$W$ is a transitive $X$-closed structure with $Y \in W$, and $a$ is an element of $W$, there exist a transitive set $M$ and an elementary embedding $\pi\colon M\to W$ such that 
\begin{itemize}
\item $\{X, Y\}\in M$, 
\item $a\in \rge(\pi)$, 
\item $M\in Z$, 
\item $\pi\rest (X \cup Y\cup \{X,Y\}) =\rm{id}$, and 
\item $M$ is $X$-closed.
\end{itemize}
When $Y=\emptyset$, we write $\rref(X, Z)$.
\end{definition}

Our main application of $\rref(X, Y, z)$ is Lemma \ref{lem:lifting dc1}. The rest of this section concerns properties of this principle. We start with an observation on replacing the first coordinate. 

\begin{remark}\label{switchX}\normalfont Suppose that $X$, $X'$ and $Y$ are sets such that $X, X' \in Y$ and $Y$ is a transitive model of $\ZF - \mathsf{Powerset}$ + $|X| = |X'|$. Then for any transitive structure $W$ with $Y \in W$, $W$ is $X$-closed if and only it if is $X'$-closed. It follows that, for any set $Z$, $\rref(X, Y, Z)$ holds if and only if $\rref(X', Y, Z)$ holds. 
\end{remark}

\begin{definition}\label{closed under definability}\normalfont Suppose that $M$ is a transitive set. Then $M$ is \textbf{definably closed} if whenever $N\in M$ is a transitive set and $X\subseteq N$ is definable over $N$ with parameters in $N$, then $X\in M$.
\end{definition}

Lemma \ref{def closed refl} shows that if one adds to the statement of $\rref(X,Y,Z)$ the condition that $V$ and $W$ are $H$-like, then one gets the additional property that $M$ is hierarchical and definably closed. Note that an $H$-like set $W$ is closed under subsets, but this property is not passed down to every structure which embeds elementarily into it. 

\begin{lemma}\label{def closed refl} Suppose that $V$ is $H$-like and that $X$, $Y$ and $Z$ are sets such that $Z$ is transitive and $\rref(X, Y, Z)$ holds. Then whenever $W$ is a transitive $X$-closed $H$-like set
such that $Y\in W$, and $a$ is an element of $W$, there exist a transitive, hierarchical, definably closed $M$ and an elementary embedding $\pi\colon M\to W$ such that 
\begin{itemize}
\item $\{X, Y\}\in M$, 
\item $a\in \rge(\pi)$, 
\item $M\in Z$, 
\item $\pi\rest X\cup Y\cup \{X, Y\}=\rm{id}$ and 
\item $M$ is $X$-closed.
\end{itemize} 
\end{lemma}

\begin{proof} Fix $X$, $Y$, $Z$, $W$ and $a$ as in the statement of the lemma. We have that $W$ is hierarchical. Let $\kappa = W \cap \ord$ and let $\langle W_\a: \a<\kappa\rangle$ be a hierarchical decomposition of $W$. Applying the assumption that $V$ is $H$-like, let $U$ be a true cardinal such that $\{X, Y, W, \langle W_\a: \a<\l\rangle\}\in H_U$. We now apply $\rref(X, Y, Z)$ to $H_U$ and $b= \{X, Y, W, a, \langle W_\a: \a<\l \rangle\}$ (we can do this because $H_U$ is $X$-closed and transitive). Let $\pi\colon M'\rightarrow H_U$ be an elementary embedding such that $M'\in Z$, 
$b \in \rge(\pi)$, $\{X, Y\}\in M'$, $\pi\rest X\cup Y\cup\{X, Y\}=\rm{id}$ and $M'$ is $X$-closed. Let $M=\pi^{-1}(W)$. We claim that $\pi\rest M: M\rightarrow W$ is as desired. We have that $\pi^{-1}(\langle W_\a: \a<\l\rangle)$ witnesses that $M$ is hierarchical. Since $M'$ is $X$-closed and $H_U\models ``W$ is $X$-closed'', $M$ is $X$-closed.

It remains to show that $M$ is definably closed. 
Let $N$ be an element of $M$, let $z$ be an element of $N$, and let $A\subseteq N$ and $\varphi$ be such that $A$ is the set of $y \in N$ such that $N\models \varphi[z, y]$. Let $B\subseteq \pi(N)$ be the set of $y$ such that $\pi(N)\models \phi[\pi(z), y)]$. Then $B\in W$, since $W$ is $H$-like. The elementarity of $\pi$ then implies that $\{ y \in A : N \models \varphi[z, y]\}$, i.e., $A$, is in $M$. 
\end{proof} 


\rthm{preserrving reflection} (which is used in the proof of Theorem \ref{thm:forcing choice}) shows that reflection is preserved in some forcing extensions.

\begin{theorem}\label{preserrving reflection} Assume $\ZF$+{$V$ \sf{is} $H$-\sf{like}}. Suppose $M_0\in M_1$ are two $H$-like sets of successor type such that, letting $X$ be a maximal cardinality of $M_0$, $\rref(X, M_0, M_1)$ holds. Let $\mathbb{P}\in M_0$ be a poset, and let $G\subseteq \bbP$ be a $V$-generic filter. Then \[V[G]\models\rref(X, M_0[G], M_1[G]).\]
\end{theorem}

\begin{proof} Let $W\in V[G]$ be an $X$-closed transitive structure such that $M_0[G]\in W$, and let $a$ be an element of $W$. Let $Z$ be a true cardinal of $V$ such that there are $\bbP$-names $\sigma, \tau\in H_Z$ with the property that $\sigma_G=W$ and $\tau_G=a$, and moreover $M_0\in H_Z$. 

Since $H_Z$ is $X$-closed, applying $\rref(X, M_0, M_1)$ and  \rlem{def closed refl} we get an elementary map $\pi\colon N\to H_Z$ and $\bbP$-names $\sigma', \tau'\in N$ such that 
\begin{itemize}
    \item $N\in M_1$,

    \item $M_0\in N$,

    \item $\{\sigma, \tau\}\in \rge(\pi)$,

    \item $\pi^{-1}(\sigma, \tau)=(\sigma', \tau')$,

    \item $\pi\rest M_0\cup \{M_0\}=\rm{id}$, and

    \item $N$ is $X$-closed, transitive, hierarchical and definably closed.
\end{itemize}

Letting $\pi^{+}(\rho_{G}) = \pi(\rho)_{G}$, for each $\bbP$-name $\rho \in N$, 
we get that $\pi$ extends to an elementary $\pi^+ \colon N[G]\to H_Z[G]$. 
Since $W = \pi(\sigma')_{G}$ and $a = \pi(\tau')_{G}$, we have that $(W, a)\in \rge(\pi^+)$. Let $K = \sigma'_{G}$ and $b = \tau'_{G}$. 
Then $(K, b)\in N[G]$, $\pi^+(K)=W$ and $\pi^+(b)=a$. Since $\pi\rest M_0\cup \{M_0\}=\rm{id}$, $\pi^+\rest M_0[G]\cup \{M_0[G]\}=\rm{id}$. Since $M_{0}[G] \in W \cap N[G]$, we have that $M_{0}[G] \in K$. 

We now claim that $\pi^+\rest K\colon  K\rightarrow W$ is as desired. Since
\begin{itemize}
    \item $K\in M_1[G]$, 

    \item $M_0[G]\in K$, 

    \item $\pi^+(b)=a$, and 

    \item $\pi^+\rest M_0[G]\cup \{M_0[G]\}=\rm{id}$,
\end{itemize}
\noindent
we will finish if we show that $N[G]$ is $X$-closed in $V[G]$. Indeed, given this, if $h\colon X\rightarrow K$ is a function in $V[G]$, then $h\in N[G]$, and, since $W$ is $X$-closed in $V[G]$, we have that $\pi^+(h)\in W$, which implies that $h\in K$. 

To show that $N[G]$ is $X$-closed in $V[G]$, fix $p\in G$ and a $\bbP$-name $\dot{f}$ such that $p\forces \dot{f}\colon \check{X}\rightarrow N[G]$. Let $\langle N_\a: \a<\eta\rangle$ be a hierarchical decomposition of $N$ with $\eta=\ord\cap N$. 
For each $x\in X$, let $P_x\subseteq \bbP$ be the set of $q\leq p$ such that for some $\bbP$-name $\rho\in N$, $q\forces \dot{f}(\check{x})=\rho$. Since $M_{0}$ is closed under subsets, $P_x\in M_0$, so $P_x\in N$. For each $x\in X$ and $q\in P_x$, let
\begin{itemize}
\item $P_{x, q}'$ be the set of $\bbP$-names $\rho\in N$ for which $q\forces \dot{f}(\check{x})=\rho$,
\item $\a_{x, q}$ be the least ordinal such that $N_\a\cap P_{x, q}'\neq\emptyset$, and 
\item $P_{x, q}$ be $P'_{x, q}\cap N_{\a_{x, q}}$. 
\end{itemize} 

Since $N$ is definably closed, each $P_{x, q}$ is in $N$. To see this, fix $\rho_0\in P_{x, q}$ and $\b<\eta$ such that $\{N_{\a_{x, q}}, M_0\}\in N_\b$. Then $P_{x, q}$ is definable over $N_{\b}$ from $(\rho_0, N_{\a_{x, q}}, q, \bbP)$ as follows: \[\rho\in P_{x, q} \text{ if and only if } N_\b\models \rho\in N_{\a_{x, q}} \wedge q\forces \rho=\rho_0.\] 

Because $N$ is $X$-closed and definably closed, and because there exists a surjection from $X$ to $X \times \bbP$ in $M_{0} \subseteq N$, we have that the function $(x, q)\mapsto P_{x, q}$ is in $N$. Using this function, we can easily define a name $\dot{g}\in N$ such that $p\forces \dot{f}=\dot{g}$. Hence, $\dot{f}_G\in N[G]$. 
\end{proof}

\section{Lifting $\DC$}\label{sec: lifting dc}


Lemma \ref{lem:lifting dc1} is our main tool for proving $\DC_{\kappa}$ from reflection.

\begin{lemma}\label{lem:lifting dc1} Suppose that $\ZF$+{$V$ \sf{is} $H$-\sf{like}}. Let $\k$ is an infinite regular cardinal, and $Y$ and $M$ be sets such that ${\sf{ref}}(\k, Y, M)$ holds and $\DC_\k(N)$ holds for all $N \in M$. Then $\DC_\k$ holds. 
\end{lemma}

\begin{proof} Let $T$ is a $\k$-closed relation, and let $U$ be a true cardinal such that $\{T, Y, \k\}\in H_U$. Applying $\rref(\k, Y, M)$, fix an elementary embedding $\pi\colon  N\to H_U$ such that $N$ is transitive, $T\in \rge(\pi)$, $N$ is $\k$-closed and $N\in M$. Let $S=\pi^{-1}(T)$. Since $N$ is $\kappa$-closed, the elementarity of $\pi$ gives that $S$ is $\k$-closed. It follows from $\DC_\k(N)$ that there is a $\k$-chain $f\colon \kappa\rightarrow \dom(S)$. We then have that $\pi\circ f\colon \kappa\rightarrow \dom(T)$ is a $\k$-chain.
\end{proof}

Suppose that $\kappa$ is a cardinal and $M$ is a $\kappa$-closed structure. If $\DC_{\kappa}(N)$ holds for all $N \in M$, then $M \models \DC_{\kappa}$. Lemma \ref{lem:lifting dc2} shows that a version of the reverse implication holds. It requires the following $\GCH$-like assumption.

\begin{definition}\label{def: gch in m}\normalfont 
Let $M$ be a set. 
\begin{enumerate}
    \item Given $Y \in M$, $M$ is $Y$-\textbf{full} if, for all $Z \in M$, $Z^{Y} \in M$. 
    \item $M$ is \textbf{completely full} if $M$ is $Y$-full for each $Y \in M$ which is not a maximal cardinality of $M$.

    \item \(M\) is \textbf{completely ordinal-full} if for every ordinal $\eta\in M$, if $\eta$ is less than some cardinal of $M$, then $M$ is $\eta$-full. 
\end{enumerate}
\end{definition}

\begin{lemma}\label{lem:lifting dc2} Suppose that 
\begin{itemize}
\item $M$ is $H$-like of successor type, 
\item $\lambda$ is a regular cardinal of $M$, 
\item $M$ is $\eta$-full for each $\eta < \lambda$ and 
\item $M \models \DC_{\lambda}$. 
\end{itemize}
Let $Y$ be a surjective image of $M$. Then $\DC_\l(Y)$ holds. In particular, $H_{M} \models \DC_{\lambda}$. 
\end{lemma}

\begin{proof} By part (\ref{dcrem2}) of Remark \ref{rem:dcac}, it suffices to show that $\DC_{\lambda}(M)$ holds. Set $\k=\ord \cap M$ and let $\langle M_\a: \a<\k\rangle$ be a hierarchical decomposition of $M$. Let $T$ be a $\l$-closed relation on $M$. 
For each $\a<\k$ let $T_{\alpha}$ be the restriction of $T$ to $M_{\alpha}$. 
Since $M$ is closed under subsets and products, $S_\a\in M$. We have that for $\a<\b<\k$, $S_\a\subseteq S_\b$. We claim that there is $\a<\k$ such that $S_\a$ is $\l$-closed. This follows from the fullness assumption on $M$ and its the strong regularity. Indeed, for each $\a<\k$, let $h(\a)$ be the least ordinal below $\k$ such that whenever $\eta<\l$,
\begin{enumerate}
\item $\dom(S_\a)^{\eta}\in M_{h(\a)}$,
\item whenever $g\colon\eta\rightarrow \dom(S_\a)$ is an $S_\a$-chain, $S_{h(\a)}$ contains an upper bound for $g$. 
\end{enumerate}
Because $M$ is strongly regular and $\eta$-full for each $\eta < \lambda$, we have that each value $h(\a)$ is less than $\k$ as desired. Indeed, for each $\eta<\l$, since $\dom(S_\a)^{\eta}\in M$, it follows from \rcor{cor: zf in s-like} that there is $\b<\k$ such that $\bigcup_{\eta < \lambda}\dom(S_{\alpha})^{\eta}\in M_\b$. Next, because $T$ is $\eta$-closed, it follows from the strong regularity of $M$ that there is a $\gg\in [\b, \k)$ such that every chain $g\colon\eta \rightarrow \dom(S_\a)$ has an upper bound in $S_\gg$. Thus, $h(\a)<\k$.

Again applying the strong regularity of $M$, let $\nu<\k$ be such that $\cf(\nu)=\l$ and $h[\nu]\subseteq \nu$. It follows that $S_\nu =\bigcup_{\a<\nu}S_\a$ is $\l$-closed. 
It then follows from the fact that $M\models \DC_\l$ that there is a $\l$-chain $g\colon\l\rightarrow \dom(S_\nu)$ through $T$.  
\end{proof}

\section{A Nairian Theory}\label{NairianTheorySec}

In this section we introduce the $\nN$-hierarchy, 
and the corresponding $\sf{Axiom\ of\ Harmony}$, which says that the sets in the hierarchy are all $H$-like. We then introduce the theory $\mathsf{NT_{base}}$, which the Nairian models we produce will satisfy, and show how to force the Axiom of Choice over models of this theory.

\begin{definition}\label{h hierarchy}\normalfont  The $\nN$-\textbf{hierarchy} $\langle \nN_\a: \a\in \ord\rangle$ is defined as follows:
\begin{enumerate}
\item $\nN_0$ is the set of hereditarily finite sets.
\item $\nN_{1} = H_{\nN_{0}}$.
\item For every $\a$, $\nN_{\a+2}=H_{\nN_{\a+1}}$.
\item For every limit ordinal $\a$, $\nN_\a=\cup_{\b<\a}\nN_\b$, and 
$\nN_{\a+1}=H_{\powerset(\nN_\a)}$.
\end{enumerate}
\end{definition}


The $\nN_{\alpha}$'s are transitive and $\subseteq$-increasing in $\alpha$. Using the fact that $\omega \subseteq N_{0}$, for instance, it is not hard to see that whenever a set $X$ is a subset of some $\nN_{\alpha}$, $X$ is an element of $\nN_{\alpha + 1}$. It follows that every set is in some $\nN_{\alpha}$.
Observe also that $\nN_{1}$ is the collection of hereditarily countable sets, and that $\nN_{2} \cap \ord = \Theta$. 



\begin{terminology}\label{term: n term}\hspace{.1in}\normalfont
\begin{enumerate} 
\item A set $X$ is an \textbf{$\nN$-cardinal} if for some successor $\a$, $X$ is a maximal cardinality of $\nN_\a$. In this case, we also say that $X$ is an $\nN_{\alpha}$-cardinal.
\item When $\alpha$ is a successor ordinal, we write $\hn_{\alpha}$ for 
\begin{itemize}
    \item $\nN_{\alpha-1}$, if $\alpha-1$ is $0$ or a successor ordinal, and

    \item  $\cP(\nN_{\alpha - 1})$, if $\alpha-1$ is a limit ordinal.
\end{itemize}
We call $\hn_{\alpha}$ the \emph{canonical} $\nN$-cardinal of $\nN_{\alpha}$.  
\item We say $M$ is an \textbf{$\nN$-set} if for some $\a$, $M=\nN_\a$. 
\item We say that $M$ is a \textbf{successor-type} $\nN$-set if $M=\nN_{\a}$, for some successor ordinal $\alpha$, and a \textbf{limit-type} $\nN$-set if $M=\nN_\a$ for some limit ordinal $\a$. 
\item We say $M$ is a \textbf{limsuc} type $\nN$-set if $M=\nN_{\a+1}$ where $\a$ is a limit ordinal.

\item If $G$ is a generic filter and $\a\in \ord$, then we write $\nN_\a^G$ for $\nN_\a^{V[G]}$ and (if $\alpha$ is a successor ordinal) $\hn_{\alpha}^{G}$ for $\hn_{\alpha}^{V[G]}$.\footnote{We may not use this one or the next one.} 
\item Given any set $u$, we say $\a$ is the $\nN$-rank of $u$ if $\a$ is the least such that $u\in \nN_\a$.
\end{enumerate}
\end{terminology}


\begin{definition}[The $\sf{Axiom\ of\ Harmony}$\label{def: ah} ($\AH$)]\normalfont For every $\a$, $\nN_\a$ is $H$-like.
\end{definition}

\begin{remark}\label{firstAHrem}\normalfont
Since every set is in some $\nN_{\alpha}$, the $\AHA$ implies that $\sf{V}$ is $\sf{H}$-like. Since it also implies that each $\nN_{\alpha}$ is closed under products, 
it implies that $|\hn_{\alpha}|\approx |\hn_{\alpha}^{2}|$ for each successor ordinal $\alpha$. 
\end{remark}

In the proof of Theorem \ref{thm:forcing choice} we will force over the levels of the $\nN$-hierarchy. The following lemma shows that, assuming $\AH$, the levels of the $\nN$-hierarchy satisfy forms of ordinal fullness which are preserved in certain forcing  extensions. 




\begin{lemma}\label{preserving complet of1} Assume that $\AH$ holds. Suppose that $\alpha$ is an ordinal, $\bbP$ is a poset in $\nN_\a$ and $G\subseteq \bbP$ is a $V$-generic filter. Then for all $\b<\gg$ with $\a<\gg$, any $X\in \nN_\b$ and any $Y\in \nN_\gg$, \[(Y^{X})^{V[G]}\in \nN_\gg[G].\] In particular, $N_{\gamma}[G]$ is $\eta$-full for each $\eta < \kappa_{\beta}$. 
\end{lemma}

\begin{proof} It is enough to prove the claim in the case where $\alpha$, $\b$ and $\gg$ are successor ordinals. We may also assume that $\alpha \leq \beta$. Let $\gg=\iota+1$. It follows from \rthm{preservation of normality} that $\nN_\gg[G]$ is $H$-like, $\hn_{\gamma}$ is a maximal cardinality of $\nN_\gg[G]$ and $V[G] \models N_{\gamma}[G] = H_{n_{\gamma}}$. 

It suffices to consider the case where $X = \hn_{\beta}$ and $Y = \hn_{\gamma}$. 
Since $\bbP \in N_{\beta}$ and $|\hn_{\beta}| \approx |\hn_{\beta}^{2}|$, there is a surjection $h$ from $\hn_{\beta}$ to $\bbP \times \hn_{\beta}$ in $N_{\beta}$. 
Each (nice) $\bbP$-name $\rho$ for a function from $X$ to $Y$ is coded by the partial function $f_{\rho} = \{ ((p,x), y) : (p, (x,y)) \in \rho\}$ from $\bbP \times N_{\tau}$ to $N_{\iota}$ and therefore by the functions $h$ and $f_{\rho} \circ h$. 
It follows that each $\hn_{\beta}$-closed transitive model of $\ZF - \mathsf{Powerset}$ containing $Y$ (as a subset)---so in particular $N_{\iota}$---will contain each such $\bbP$-name $\rho$ (as an element).  
It follows that the set of such names is in $N_{\gamma}$, and the set of realizations of these names, i.e., $(Y^{X})^{V[G]}$, is then an element of $N_{\gamma}[G]$.
\end{proof}

\begin{definition}\label{nairian theory}\normalfont Let $\sf{NT_{base}}$ be the following theory: 
\begin{enumerate}
\item $\ZF+\AH$.
\item For every successor ordinal $\a$, whenever $X$ is a maximal cardinality of $\nN_\a$, then $\rref(X, \nN_\a, \nN_{\a+1})$ holds. 
\end{enumerate}
\end{definition}

Theorem \ref{thm:forcing choice} shows how to force the Axiom of Choice to hold over models of $\sf{NT_{base}}$.\footnote{In the main application in this paper, $\alpha$ will be $2$ and $\bbP$ will be $\pmax$. Another interesting case (also with $\alpha = 2$) is when $\bbP$ is $\Add(1, \omega_{1})$.} 
In its statement we use the following notation for (class-length) forcing iterations (with full support). Given an ordinal $\gg$, we let 
\begin{itemize}
    \item $\bbQ_{<\gg}=\bbQ\rest \gg$,

    \item $\bbQ_{\leq \gg}=\bbQ\rest(\gg+1)$,

    \item $\bbQ^{\geq \gg}$ be the iteration after and including stage $\gg$,

    \item $\bbQ^{>\gg}$ the iteration after and not including stage $\gg$, and

    \item $\bbQ(\gg)$ be the poset used at stage $\gg$.
\end{itemize}
Thus \[\bbQ=\bbQ_{<\gg}*\bbQ^{\geq\gg}=\bbQ_{< \gg}*\bbQ(\gg)*\bbQ^{>\gg}=\bbQ_{\leq \gg}*\bbQ^{>\gg}.\] 

For an ordinal $\gg$, we let 
\[H_{<\gg},\,
H_{\leq \gg},\,
H^{\geq\gg},\,
H^{>\gg}, \text{ and }
H(\gg)\]
be the natural factors of the generic filter $H$.
Thus \[H=H_{<\gg}*H^{\geq\gg}=H_{<\gg}*H(\gg)*H^{>\gg}=H_{\leq \gg}*H^{>\gg}.\]

Given a cardinal $\kappa$, the partial order $\Coll(\kappa, X)$ consists of the partial functions from $\kappa$ to $X$ of cardinality less than $\kappa$, ordered by inclusion. The partial order $\Add(1, \kappa)$ is the same as $\Coll(\kappa, 2)$. If $\cof(\gamma) < \kappa$ and $\DC_{\gamma}$ holds, then $\Add(1, \kappa)$ adds a surjection from $\kappa$ to $\cP(\gamma)$. 

\begin{theorem}\label{thm:forcing choice} Assume $\sf{NT_{base}}$ and that no regular $\nN$-cardinal is a limit of $\nN$-cardinals. Suppose 
that $\a$ is a successor ordinal, $\bbP\in \nN_\a$ is a poset, $G\subseteq \bbP$ is a $V$-generic filter, and $\nN_\a[G]$ is completely ordinal-full and satisfies $\AC$.

For each ordinal $\beta$, let $\k_\b$ denote $\ord\cap \nN_\b$. In $V[G]$, let $\mathbb{Q}$ be the full support class iteration defined as follows. 
\begin{itemize}
\item $\bbQ(0)=\Coll(\k_{\a},\nN_{\alpha}[G])$. 
\item For a limit ordinal $\b$, $\bbQ(\b)$ is the trivial poset and $\bbQ\rest \b$ is the full support iteration of $(\bbQ(\gg): \gg<\b)$. 
\item When $\beta$ is a successor of a nonlimit ordinal, $\bbQ(\b)=\Coll(\k_{\a+\b}, \nN_{\alpha + \beta})^{V^{\bbP*\bbQ\rest \b}}$.
\item When $\b$ is a successor of a limit ordinal, $\bbQ(\b)$ is the partial order \[\Coll(\eta, \cP(\nN_{\alpha + \beta})^{V})*\Coll(\k_{\alpha + \beta + 1}, \nN_{\alpha + \beta + 1}),\] as defined in $V^{\bbP*\bbQ\restriction \beta}$, where $\eta$ is the cardinal successor of $\kappa_{\alpha + \beta}$ in 
$V^{\bbP*\bbQ\restriction \beta}$.
\end{itemize}
Then the following hold in $V[G*H]$.
\begin{enumerate}
\item $\ZFC$.
\item\label{gchkappalpha} For each ordinal $\beta$, $\kappa_{\alpha + \beta}$ is a cardinal and $2^{\kappa_{\alpha + \beta}}= \kappa_{\alpha + \beta}^{+}$.
\item\label{succpres} For each nonlimit ordinal $\beta$, 
$\k_{\a+\b+1}= \k_{\a+\b}^{+}$. 
\item\label{type2card} For each limit ordinal $\beta$, 
$\k_{\a+\b+1}= \k_{\a+\b}^{++}$. 
\end{enumerate}
\end{theorem}

\begin{proof}  
To begin, we note the following. 
\begin{enumerate}[itemsep=0.3cm]
\item For each successor ordinal $\beta$, $\bbP*\bbQ_{<\beta}\in \nN_{\a+\beta}$, so, by \rthm{preservation of normality}, $N_{\alpha + \beta}[G*H_{<\beta}]$ is $H$-like and equal to $H_{\hn_{\alpha + \beta}}^{V[G*H_{<\beta}]}$, and therefore strongly regular in $V[G*H_{<\beta}]$.
\item For each successor ordinal $\beta$, $\bbQ(\beta)\not\in \nN_{\a+\beta}[G*H_{<\beta}]$.
\item For each limit ordinal $\beta$, since $\bbQ\restriction \beta \subseteq \nN_{\alpha + \beta}$, $(\kappa_{\alpha + \beta}^{+})^{V[G*H_{<\beta}]} < \kappa_{\alpha + \beta + 1}$. That is, the least nonzero ordinal which is not a surjective image of $\kappa_{\alpha + \beta}$ in $V[G*H_{<\beta}]$ is a surjective image of $\cP(\nN_{\alpha + \beta})$ in $V$.  
\end{enumerate}

For each successor ordinal $\beta$, let $(*)_{\beta}$ be the statement that the extension $V[G*H_{<\beta}]$ satisfies 
each of the following statements. 
\begin{itemize}[itemsep=0.3cm]
\item[$(*)^{\rmd}_{\beta}$] $\kappa_{\alpha + \beta-1}$ is a cardinal, and  $\DC_{\kappa_{\alpha + \beta-1}}$ holds, 
\item[$(*)^{\rms}_{\beta}$] if $\beta - 1$ is not a limit ordinal, and $\kappa$ is the largest cardinal of $N_{\alpha + \beta-1}[G]$, then $2^{\kappa} = \kappa_{\alpha + \beta-1}$ and  
there is a surjection from the set $\cP(\kappa_{\alpha + \beta - 1})$ to $\nN_{\alpha + \beta}[G*H_{<\beta}]$. 
\end{itemize}

Let $(*)$ be the statement that $(*)_{\beta}$ holds for all successor ordinals $\beta$. We will prove $(*)$ by induction on $\beta$. We show first that this will prove several of the conclusions of theorem (the other conclusions will be proved in the limit step of the induction proof). To begin with, the following claim, in conjunction with Remark \ref{rem:dcac} and Lemma \ref{lem: dc through closure}, shows that 
\begin{enumerate}[itemsep=0.3cm]
\setcounter{enumi}{3}
\item $(*)$ implies that $V[G*H] \models \AC$, and
\item for each successor ordinal $\gamma$, $(*)_{\gamma}$ implies that whenever $\kappa$ is a cardinal of $N_{\alpha + \gamma}[G*H_{\less\gamma}]$ and $\beta \geq \gamma$ is an ordinal, $\DC_{\kappa}$ holds in $V[G*H_{<\beta}]$. 
\end{enumerate}

\begin{claim}\label{claim: closure of q} Let $\beta$ be a successor ordinal, and suppose that $\k$ is a cardinal of $\nN_{\a+\beta}[G*H_{<\beta}]$ such that $V[G*H_{<\beta}]\models \DC_\k$. Then $\bbQ^{\geq \beta}$ is $\k^+$-closed in $V[G*H_{<\beta}]$.
\end{claim}

\begin{proof} It is enough to show that for each ordinal $\iota\geq \beta$, $\bbQ(\iota)$ is $\k^+$-closed in $V[G*H_{<\iota}]$. We show this by induction on $\iota$. By Lemma \ref{lem: dc through closure}, the induction hypothesis at $\iota$ will give that $\DC_{\kappa}$ holds in $V[G*H_{<\iota}]$. The limit steps of the induction are immediate.

If $\iota=\iota'+1$, where $\iota'$ is either $0$ or a successor ordinal, then since 
$\nN_{\alpha + \iota}[G*H_{<\iota}]$ is strongly regular by item (1) above, it follows that $\bbQ(\iota)=\Coll(\k_{\alpha + \iota}, \nN_{\alpha + \iota})$ is $\k^+$-closed in $V[G*H_{<\iota}]$. 

Suppose now that $\iota=\iota'+1$ where $\iota'$ is a limit ordinal, and that  $\bbQ_{<\iota}$ is $\k^+$-closed in the extension $V[G*H_{<\beta}]$.
Let $\eta$ be the cardinal successor of $\k_{\alpha + \iota'}$ (in $V$), which is less than $\kappa_{\alpha + \iota}$ by item (3) above.
Since $\DC_\k$ holds in $V[G*H_{<\iota}]$ by the induction hypothesis, we have that \[V[G*H_{<\iota}]\models \cf(\eta)>\k.\] Therefore, $\Coll(\eta,\cP(\nN_{\alpha + \iota'})^{V})$ is $\k^+$-closed in $V[H_{<\iota}]$. Since \[\bbP*\bbQ_{<\iota}*\Coll(\eta,\cP(\nN_{\alpha + \iota'})^{V})\in \nN_{\a+\iota},\] we once again have by Lemma \ref{lem: dc through closure} that, in $V^{\bbP*\bbQ_{<\iota}*\Coll(\eta,\cP(\nN_{\alpha + \iota'})^{V}}$, $\DC_{\kappa}$ holds and $\Coll(\k_{\a+\iota},\nN_{\alpha + \iota})$ is $\k^+$-closed. Therefore, $\bbQ(\iota)=\Coll(\eta,\cP(\nN_{\alpha + \iota'})^{V})*\Coll(\k_{\a+\iota}, \nN_{\alpha + \iota})$ is $\k^+$-closed in the extension $V[G*H_{<\iota'}]$.
\end{proof}

Claim \ref{claim: closure of q} and $(*)$ also imply that, for each successor ordinal $\beta$, and each cardinal $\kappa$ of 
the structure $N_{\alpha + \beta}[G*H_{<\beta}]$, 
$\cP(\kappa)^{V[G*H]} \subseteq V[G*H_{<\beta}]$. It follows that $(*)$ implies that for each ordinal $\beta$, $\kappa_{\alpha + \beta}$ is a cardinal in $V[G*H]$. It also implies (via the definitions of $\bbQ(\beta)$ and the $\nN$-hierarchy) that $2^{\kappa_{\alpha + \beta}} = \kappa_{\alpha + \beta + 1} = \kappa_{\alpha + \beta}^{+}$ holds in $V[G*H]$ whenever $\beta$ is not a limit ordinal, which establishes part (\ref{succpres}) of the theorem, and part (\ref{gchkappalpha}) in the case where $\beta$ is not a limit ordinal.  
We will prove part (\ref{type2card}), and part (\ref{gchkappalpha}) in the case where $\beta$ is a limit ordinal, during the corresponding limit step of our induction (i.e., in Claim \ref{claim: ind step limit}).

Before we start our inductive proof of $(*)$, we make some observations about $N_{\alpha}[G]$ (analogous to $(*)_{\beta}$ for $\beta = 0$). 
Note first that since $\bbP$ is in $\nN_{\alpha}$, which is strongly regular, $\kappa_{\alpha}$ is a regular cardinal in $V[G]$. 
The rest of our analysis of $N_{\alpha}[G]$ is given by the two following claims. 

\begin{claim}\label{ihalpha} For some $\nN_{\a}[G]$-cardinal $\k$, there is a surjection $$f:\powerset(\k)^{V[G]}\rightarrow \nN_{\a}[G]$$ in $V[G]$.\\
\end{claim}


\begin{proof} Let $X\in \nN_\a$ be a maximal cardinality of $\nN_\a$, and (applying the assumption that $N_{\alpha}[G] \models \AC$) let $\k$ be a cardinal of $\nN_\a[G]$ such that $\card{X}^{\nN_\a[G]}=\k$. By Lemma \ref{closure under cartesian products} there is a surjection $h\colon \powerset(X)\to \nN_\a$ in $V$, so we get a surjection $f\colon\powerset(\k)\to \nN_\a[G]$ in $V[G]$.
\end{proof}

\begin{claim}\label{ihalpha0} 
For every $N_{\alpha}[G]$-cardinal $\kappa$, $V[G] \models \DC_{\kappa}$. 
\end{claim}

\begin{proof}
Let $\kappa$ be the largest $\nN[G]$-cardinal (i.e., the cardinality of $\hn_{\alpha}$). It suffices to prove the claim for $\kappa$. We have the following facts. 
\begin{itemize}[itemsep=0.3cm]
\item $\nN_{\a}[G]\models \DC_\k$, since $\nN_{\a}[G]\models \AC$.
\item  \rlem{lem:lifting dc2} implies that $\nN_{\a+1}[G]\models \DC_\k$, since $\nN_{\a}[G]$ is completely ordinal-full by the hypotheses of the theorem. 
\item $\rref(\k, \nN_{\a}[G], \nN_{\a+1}[G])$ holds in $V[G]$, by \rthm{preserrving reflection} and Remark \ref{switchX}.
\end{itemize} 
The claim now follows from \rlem{lem:lifting dc1}. 
\end{proof}

Claim \ref{succstep} gives the following steps of our inductive proof. 
\begin{itemize}[itemsep=0.3cm]
\item $(*)_{1}$ (via  Claims \ref{ihalpha} and \ref{ihalpha0}); 
\item if $\beta$ is not a limit ordinal, and $(*)_{\beta+1}$ holds, then $(*)_{\beta + 2}$ holds (in this case the $\kappa$ from the claim is $\kappa_{\alpha + \beta}$).
\end{itemize}


\begin{claim}\label{succstep} Suppose that 
\begin{itemize}
\item $\beta$ is either $0$ or a successor ordinal, 
\item $\kappa$ is the largest cardinal of $N_{\alpha + \beta}[G*H_{<\beta}]$ and,
\item in $V[G*H_{\less\beta}]$, $\DC_{\kappa}$ holds and there exists a surjection from $\cP(\kappa)$ to $\nN_{\alpha + \beta}[G*H_{\less\beta}]$.
\end{itemize} 
Then $(*)_{\beta+1}$ holds and $\nN_{\alpha + \beta + 1}[G*H_{\leq\beta}]\models \AC$. 
\end{claim}

\begin{proof} Since $\beta$ is either $0$ or a successor ordinal, $\bbQ(\beta)$ is $\Coll(\kappa_{\alpha + \beta}, \nN_{\alpha + \beta})$ and $\nN_{\alpha + \beta + 1}$ is $H_{\nN_{\alpha + \beta}}$. Since $\DC_{\kappa}$ holds in $V[G*H_{<\beta}]$, forcing with $\bbQ(\beta)$ preserves $\kappa_{\alpha + \beta}$ as a cardinal and adds 
a wellordering of $\nN_{\alpha + \beta}[G*H_{\leq\beta}]$ in $V[G*H_{\leq\beta}]$ in ordertype $\kappa_{\alpha + \beta}$.
Since $\nN_{\alpha + \beta + 1}[G*H_{\leq\beta}]$ is $\kappa_{\alpha + \beta}$-closed (by Theorem \ref{preservation of normality}), and every element of $\nN_{\alpha + \beta + 1}$ is a surjective image of $\nN_{\alpha + \beta}$, this means that \[N_{\alpha + \beta + 1}[G*H_{\leq\beta}] \models \AC.\] 
Since there is a surjection from $\cP(\nN_{\alpha+\beta})$ to $\nN_{\alpha + \beta + 1}$ in $V$, there is a surjection from $\cP(\kappa_{\alpha + \beta})$ to $\nN_{\alpha + \beta + 1}[G*H_{\leq\beta}]$ in $V[G*H_{\leq\beta}]$.

Since $\rref(\nN_{\alpha + \beta}, \nN_{\alpha + \beta + 1}, \nN_{\alpha + \beta + 2})$ holds in $V$, by $\mathsf{NT_{base}}$, Remark \ref{switchX} and \rthm{preserrving reflection} imply that  \[V[G*H_{\leq\beta}]\models\rref(\k_{\alpha + \beta}, \nN_{\a+\beta+1}[G*H_{\leq\beta}], \nN_{\a+ \beta+2}[G*H_{\leq\beta}]).\] 
By Lemma \ref{preserving complet of1}, 
\[V[G*H_{<\beta}] \models \nN_{\alpha + \beta}^{\kappa} \in \nN_{\alpha + \beta + 1}[G*H_{<\beta}].\] 
Since $\bbQ(\beta)$ does not add $\kappa$-sequences, \[V[G*H_{\leq\beta}] \models \nN_{\alpha + \beta}^{\kappa} \in \nN_{\alpha + \beta + 1}[G*H_{\leq\beta}].\]
It follows that $\nN_{\alpha + \beta + 1}[G*H_{\leq\beta}]$ is $\eta$-full for all $\eta < \kappa_{\alpha + \beta}$. 
Since $N_{\alpha + \beta + 1}[G*H_{\leq\beta}]$ satisfies $\AC$, Lemma \ref{lem:lifting dc2} implies that \[V[G*H_{\leq\beta}]\models \DC_{\kappa_{\alpha + \beta}}(N_{\alpha + \beta + 1}[G*H_{\leq\beta}]).\] By Theorem \ref{preservation of normality}, \[V[G*H_{\leq\beta}]\models\text{ ``$V$ is $H$-like''.}\]  \rlem{lem:lifting dc1} then implies that  \[V[G*H_{\leq\beta}]\models\DC_{\kappa + \beta}.\]
\end{proof} 

For the limit step of our induction, we first note the following fact. 

\begin{claim}\label{limitindclaim} If $\beta$ is a limit ordinal and $(*)_{\gamma}$ holds for all successor ordinals $\gamma < \beta$, then the following hold in $V[G *H_{<\beta}]$. 
\begin{enumerate}
\item $\kappa_{\alpha + \beta}$ is a cardinal. 
\item $\DC_{\kappa_{\alpha + \beta}}$
\item There is a bijection between $\kappa_{\alpha + \beta}$ and $\nN_{\alpha + \beta}[G*H_{\less\beta}]$. 
\end{enumerate}
\end{claim} 

\begin{proof} Since $(*)_{\gamma}$ holds for all successor $\gamma < \beta$, $\kappa_{\alpha + \beta}$ is a limit of cardinals in $V[G*H_{<\beta}]$ and therefore a cardinal. By the hypotheses of the theorem, $\beta$ is singular. 
By Remark \ref{rem:dcac} and the second item just before the statement of Claim \ref{claim: closure of q}, we get \[V[G*H_{<\beta}]\models\DC_{\kappa_{\alpha + \beta}}.\] 
We have that for each successor ordinal $\gamma < \beta$, $H_{\gamma + 1}$ wellorders $\nN_{\alpha + \gamma + 1}[G*H_{\leq\gamma}]$ in ordertype $\kappa_{\alpha + \gamma + 1}$.
From $\DC_{\kappa_{\alpha + \beta}}$, we get that there is a bijection between $\kappa_{\alpha + \beta}$ and $\nN_{\alpha + \beta}[G * H_{<\beta}]$ in $V[G*H_{<\beta}]$. 
\end{proof} 

It follows from Claim \ref{limitindclaim} (using Lemma \ref{lem: dc through closure}) that when $\beta$ is a limit ordinal, 
\[\cP(\kappa_{\alpha + \beta})^{V[G*H_{<\beta}]} = 
\cP(\kappa_{\alpha + \beta})^{V[G*H_{\leq(\beta + 1)}]} = \cP(\kappa_{\alpha + \beta})^{V[G*H]},\]
so $\kappa_{\alpha + \beta}^{+}$ is the same whether computed in $V[G*H_{<\beta}]$ or $V[G*H]$. Claim \ref{limitindclaim} also gives the induction step for $(*)^{\rmd}_{\beta + 1}$ immediately.



Claim \ref{claim: ind step limit} shows that if $\beta$ is a limit ordinal such that $(*)_{\gamma}$ holds for all successor $\gamma < \beta$, then 
$(*)_{\beta + 2}$ holds, and, in $V[G*H_{\leq(\beta + 1)}]$, \[2^{\kappa_{\alpha + \beta}} = \kappa_{\alpha + \beta}^{+}\] and \[\kappa_{\alpha + \beta + 1} = \kappa_{\alpha + \beta}^{++}.\]  This finishes the proof of the theorem.

\begin{claim}\label{claim: ind step limit} 
Let $\beta$ be a limit ordinal, and suppose that $(*)_{\gamma}$ holds for all successor ordinals $\gamma <\beta$. 
Let $\eta=(\k_{\a+\beta}^+)^{V[G*H_{<\beta}]}$. Then in $V[G*H_{\leq(\beta + 1)}]$,
\begin{enumerate}
\item there is a surjection from $\cP(\kappa_{\alpha + \beta + 1})^{V[G*H_{\leq(\beta + 1)}]}$ to $\nN_{\alpha + \beta + 2}[G*H_{\leq(\beta + 1)}]$, 
\item $\eta$ is a regular cardinal and $\eta=\k_{\a+\beta}^+ = 2^{\kappa_{\alpha + \beta}}$, 
\item $\k_{\a+\beta+1}$ is a regular cardinal
and $\k_{\a+\beta+1}=\eta^+ = 2^{\eta}$, 
\item $\nN_{\a+\b+2}[G*H_{\leq (\b+1)}]\models \AC$, and
\item $\DC_{\kappa_{\alpha + \beta + 1}}$.
\end{enumerate}
\end{claim}

\begin{proof}
By \rthm{preservation of normality} (and the triviality of $\bbQ(\beta)$), we have that \[\nN_{\a+\b+1}[G*H_{<\b}] = (H_{\powerset(\nN_{\alpha + \beta})^{V}})^{V[G*H_{<\beta}]}\] 
and
\[\nN_{\a+\b+2}[G*H_{\less(\beta + 1)}] = (H_{\nN_{\alpha + \beta+1}})^{V[G*H_{\less(\beta+1)}]}\] are
both $H$-like in $V[G*H_{\leq\beta}]$. 

Let $K_{0}$ and $K_{1}$ be such that $H(\beta+1) = K_{0} * K_{1}$. Since \[V[G*H_{\leq\beta}] \models \DC_{\kappa_{\alpha + \beta}} \wedge \eta=\k_{\a+\beta}^+,\] 
we have that $\eta$ is a regular cardinal and the cardinal successor of $\kappa_{\alpha + \beta}$ in both $V[G*H_{\leq\beta}]$ and $V[G*H_{\leq\beta}*K_{0}]$. Since $\nN_{\alpha + \beta + 1}$ is strongly 
regular and $\Coll(\eta, \cP(\nN_{\alpha + \beta})^{V})$ is an element of $\nN_{\alpha + \beta + 1}[G*H_{<\beta}]$, $\kappa_{\alpha + \beta + 1}$ is a regular cardinal in $V[G*H_{\leq\beta}*K_{0}]$. 





The forcing $\Coll(\eta, \cP(\nN_{\alpha + \beta})^{V})$ adds a surjection 
from $\eta$ to $\cP(\nN_{\alpha + \beta})^{V}$. 
Since $\powerset(\nN_{\a+\beta})^V$ is a maximal cardinality of $\nN_{\a+\b+1}[G*H_{\leq\b}]$, we have (using Claim \ref{limitindclaim}) that $\eta$ is a maximal cardinality of $\nN_{\a+\b+1}[G*H_{\leq\b}*K_0]$. This then implies that  $\nN_{\a+\b+1}[G*H_{\leq\b}*K_0]\models \AC$, and that in $V[G*H_{\leq\b}*K_0]$, \[\eta = \k_{\a+\b}^{+}=2^{\kappa_{\alpha + \beta}}\] and \[\eta^{+} = \kappa_{\alpha + \beta + 1}.\] 



\begin{subclaim}
$V[G*H_{\leq\b}*K_0]\models \DC_\eta$.
\end{subclaim}

\begin{proof}
By the definition of $\NTbase$, $\rref(\cP(\nN_{\alpha + \beta}), \nN_{\alpha + \beta + 1}, \nN_{\alpha + \beta + 2})$ holds in $V$. By Lemma \ref{preserrving reflection}, \[V[G*H_{\leq\beta}*K_{0}]\models \rref(\cP(\nN_{\alpha + \beta})^{V}, \nN_{\alpha + \beta + 1}[G*H_{\leq\beta}*K_{0}], \nN_{\alpha + \beta + 2}[G*H_{\leq\beta}*K_{0}]).\]
By Remark \ref{switchX}, \[V[G*H_{\leq\beta}*K_{0}]\models\rref(\eta, \nN_{\alpha + \beta + 1}[G*H_{\leq\beta}*K_{0}], \nN_{\alpha + \beta + 2}[G*H_{\leq\beta}*K_{0}]).\] By Lemma \ref{lem:lifting dc1}, and the fact that 
\[\nN_{\alpha + \beta + 2}[G*H_{\leq\beta}*K_{0}] = 
(H_{\nN_{\alpha + \beta + 1}})^{V[G*H_{\leq\beta}*K_{0}]},
\]
it suffices to show that \[V[G*H_{\leq\beta}*K_{0}]\models\DC_{\eta}(\nN_{\alpha + \beta + 1}[G*H_{\leq\beta}*K_{0}]).\]  
Since \[\nN_{\alpha + \beta + 1}[G*H_{\leq\beta}*K_{0}]\models\AC,\]  and since $\kappa_{\alpha + \beta}$ is the largest cardinal of $N_{\alpha + \beta + 1}[G*H_{\leq\beta}*K_{0}]$ below $\eta$, it suffices by Lemma \ref{lem:lifting dc2} to show that the structure
 $\nN_{\alpha + \beta + 1}[G*H_{\leq\beta}*K_{0}]$ is $\kappa_{\alpha + \beta}$-full in $V[G*H_{\leq\beta}*K_{0}]$. This amounts to showing that, in $V[G*H_{\leq\beta}*K_{0}]$,  \[\cP(\nN_{\alpha + \beta})^{\kappa_{\alpha + \beta}} \in \nN_{\alpha + \beta + 1}[G*H_{\leq\beta}*K_{0}].\]
 Since $|\nN_{\alpha+\beta}| = \kappa_{\alpha + \beta}$ in $\nN_{\alpha + \beta + 1}[G*H_{\leq\beta}*K_{0}]$, and \[V[G*H_{\leq\beta}*K_{0}] \models \cP(\kappa_{\alpha + \beta}) \in \nN_{\alpha + \beta}[G*H_{\leq\beta}*K_{0}],\]
 and since functions from $\kappa_{\alpha + \beta}$ to $\cP(\kappa_{\alpha + \beta})$ can be coded by subsets of $\kappa_{\alpha + \beta + 1}$,
 we get that  $\nN_{\alpha + \beta + 1}[G*H_{\leq\beta}*K_{0}]$ is $\kappa_{\alpha + \beta}$-full in $V[G*H_{\leq\beta}*K_{0}]$, as desired. 
\end{proof}

Since
\begin{itemize}[itemsep=0.3cm]
    \item $\Coll(\k_{\a+\b+1}, \nN_{\alpha + \beta + 1})$ is $\k_{\a+\b+1}$-closed in $V[G*H_{\leq\b}*K_0]$,

    \item $(\eta^+)^{V[G*H_{\leq\beta}*K_{0}]}=\k_{\a+\b+1}$, and

    \item $V[G*H_{\leq\b}*K_0]\models \DC_\eta$,
\end{itemize}
\noindent
we get part (2) of the claim. 

Since $\cP(\nN_{\alpha + \beta + 1})^{V}$ maps onto $\nN_{\alpha + \beta + 2}$ (in $V$), we also get part (1).
Since $\nN_{\alpha + \beta + 1}[G*H_{\leq\beta}*K_{0}]$ is strongly regular, $\kappa_{\alpha + \beta + 1}$ is a regular cardinal in $V[G*H_{\leq(\beta + 1)}]$, and so we have part (3). 


Since $\eta$ is a maximal cardinality of $\nN_{\a+\b+1}[G*H_{\leq\b}*K_0]$, we have that \[\nN_{\alpha + \beta + 2}[G*H_{\leq(\b+1)}]\models\card{\nN_{\a+\b+1}[G*H_{\leq\b}*K_0]}=\k_{\a+\b+1}.\] Since $\nN_{\a+\b+1}[G*H_{\leq\b}*K_0]$ is a maximal cardinality of $\nN_{\a+\b+2}[G*H_{\leq (\beta + 1)}]$, we have that \[\nN_{\a+\b+2}[G*H_{\leq (\b+1)}]\models \AC.\]

It remains to show part (5).
Since $\rref(\nN_{\alpha + \beta + 1}, \nN_{\alpha + \beta + 2}, \nN_{\alpha + \beta + 3})$ holds in $V$, by $\mathsf{NT_{base}}$, Remark \ref{switchX} and \rthm{preserrving reflection} imply that \[V[G*H_{\leq(\beta+1)}]\models\rref(\k_{\alpha + \beta+1}, \nN_{\a+\beta+2}[G*H_{\leq(\beta+1)}], \nN_{\a+ \beta+3}[G*H_{\leq(\beta+1)}]).\]

By Lemma \ref{preserving complet of1}, 
\[V[G*H_{\leq\beta}*K_{0}] \models \nN_{\alpha + \beta+1}^{\eta} \in \nN_{\alpha + \beta + 2}[G*H_{\leq\beta*K_{0}}].\] 
Since $\Coll(\kappa_{\alpha + \beta + 1},\nN_{\alpha + \beta + 1})$ does not add $\eta$-sequences when applied to models of $\DC_{\eta}$, \[V[G*H_{\leq(\beta+1)}] \models \nN_{\alpha + \beta+1}^{\eta} \in \nN_{\alpha + \beta + 2}[G*H_{\leq(\beta+1)}].\]
It follows that $\nN_{\alpha + \beta + 2}[G*H_{\leq(\beta+1)}]$ is $\tau$-full for all $\tau < \kappa_{\alpha + \beta+1}$. 
Lemma \ref{lem:lifting dc2} then gives that \[V[G*H_{\leq(\beta+1)}]\models\DC_{\kappa_{\alpha + \beta + 1}}(N_{\alpha + \beta + 2}[G*H_{\leq(\beta+1)}]).\]

By Theorem \ref{preservation of normality}, \[V[G*H_{\leq(\beta+1)}]\models\text{ ``$V$ is $H$-like''.}\] 
\rlem{lem:lifting dc1} then implies that \[V[G*H_{\leq(\beta+1)}]\models \DC_{\kappa + \beta+1}.\]

\end{proof}
\end{proof}

\part{Nairian Models satisfy the Nairian Theory}
In this part of the paper, we develop the necessary machinery to prove Theorem \ref{thm: part 2}. We present the proof in \rsec{sec: the proof of ntbase}. 

\section{Preliminaries}\label{sec: prelim}
We import the set up of \cite{blue2025nairian} here with one difference: we will use $\varsigma$, $\hd$, and $\varkappa$ for the $\xi$, $\d$, and $\eta$ used in Chapter 10 of \cite{blue2025nairian} so that we can use $\xi$, $\d$, and $\eta$ elsewhere.

\begin{notation}[The set up]\label{not: set up}\normalfont
We assume as our working hypothesis that $\V\models {\sf{ZFC}}$ is a hod premouse\footnote{We tacitly assume that all large cardinal notions are witnessed by the extenders on the extender sequence of the hod premouse.} and $\varsigma\leq \hd<\varkappa$ are such that:
	\begin{enumerate}
		\item $\varkappa$ is an inaccessible limit of Woodin cardinals of $\V$,
		\item $\hd$ is a Woodin cardinal of $\V$,
		\item $\varsigma$ is the least inaccessible cardinal of $\V$ that is a limit of Woodin cardinals of $\V$ and also a limit of $<\hd$-strong cardinals of $\V$.
	\end{enumerate}
	Let $g\subseteq \Coll(\omega, {<}\varkappa)$ be $\V$-generic, and let $M$ be the derived model of $\V$ as computed by $g$. Set:
	\begin{enumerate}
		\item $\P=\V|(\varkappa^+)^\V$,
		\item $\Sigma\in \V$ be the iteration strategy of $\P$ indexed on the sequence of $\V$,
		\item $\hp=(\P, \Sigma)$,
		\item $\mathcal{F}=\mathcal{F}^g_\hp$ (see \cite[Notation 8.10]{blue2025nairian}),
		\item $\mH=\M_\infty(\hp)$,
		\item $\Omega=\Sigma_{\mH}$,
		\item $\varsigma_\infty=\pi_{\hp, \infty}(\varsigma)$ and $\hd_\infty=\pi_{\hp, \infty}(\hd)$.
	\end{enumerate}
	Let 
    \begin{center} 
    $\c^-=(\c^-_{\varsigma_\infty})^M,$ $\c=(\c_{\varsigma_\infty})^M,$ and $\c^+=(\c_{\varsigma_\infty}^+)^\M,$
    \end{center}
    as defined in Section 1.6 of \cite{blue2025nairian}. We treat $\hp$ as a hod pair in $\V$ and in $\V[g]$. Observe that if $\a<\varkappa$, then $\hp|\a\in M$.

When we say $\hq$ is a complete iterate of $\hp$, we mean that $\hq$ is a complete iterate of $\hp|\hd$. Also, we will often treat complete iterates of $\hp$ as members of $M$. The intended meaning here is that we treat $\hq$ as a complete iterate of $\hp|\hd$, and $\hp|\hd\in M$. 
\end{notation}

We use the following additional notation.

\begin{notation}\label{not: more not for reflection}\normalfont
Our notation will be that introduced in \cite[\textsection10.1]{blue2025nairian} together with the following.
\begin{enumerate}
\item A cardinal $\k$ is a:
\begin{itemize}
    \item ${<}\a$-\textbf{strong-limit-of-Woodins cardinal} (${<}\a$-\textbf{slw cardinal}) if it is a ${<}\a$-strong cardinal that is a limit of Woodin cardinals;

    \item $\a$-\textbf{strong-limit-of-Woodins cardinal} ($\a$-\textbf{slw cardinal}) if it is an $\a$-strong cardinal that is a limit of Woodin cardinals;

    \item \textbf{slw cardinal} if it is a ${<}\ord$-slw cardinal.
\end{itemize}

\item Suppose that $W$ is a transitive model of some fragment of $\ZFC$.
\begin{enumerate}
    \item We let $\ts(W)$ be the least slw cardinal of $W$.
    \begin{itemize}
        \item When we write $\ts(W)$, we mean that it exists.
        
        \item Here and elsewhere, if $\hp$ is a hod pair, we write $\ts(\hp)$ instead of $\ts(\M^\hp)$.
        \item We will use the same terminology with all of our other notations, using $\hp$ as supercript or within the parenthesis. 
    \end{itemize}

    \item We let $\tW(W)$ be the set of Woodin cardinals of $W$.

    \item If $\zeta_0\leq\zeta_1\leq \ord\cap W$, then we let $\tStr(W, {\leq}\zeta_0, \zeta_1)$ be the set of $\nu\leq\zeta_0$ such that:
    \begin{itemize}
        \item $\nu=0$, or

        \item $\nu$ is a ${<}\zeta_1$-slw cardinal of $W$, or

        \item $\nu$ is a limit of ${<}\zeta_1$-slw cardinals of $W$.
    \end{itemize}
    
    We define $\tStr(W, {<}\zeta_0, \zeta_1)$ similarly. If $\zeta_0=\zeta_1$, then we write $\tStr(W, {\leq} \zeta_0)$ or $\tStr(W, {<}\zeta_0)$ for $\tStr(W, {\leq}\zeta_0, \zeta_0)$ and $\tStr(W, {<}\zeta_0, \zeta_0)$ respectively. If $\zeta_0=\ord\cap W$, then we write $\tStr(W)$ for $\tStr(W, {<}\zeta_0)$.
\end{enumerate}

\item $\d$ is an \textbf{lsa cardinal} if $\d\in \tW(W)$, no cardinal $<\d$ is $\d+1$-strong, and $\ts(V_\d)$ is defined.\footnote{Equivalently, the least ${<}\d$-strong cardinal is a limit of Woodins, or it ends an lsa block in the sense of \cite[Terminology 9.8]{blue2025nairian}.}

\item Suppose that $\hq$ is any hod pair, $\nu$ is a Woodin cardinal of $\M^\hq$, and $\k<\nu$ is the least ${<}\nu$-strong cardinal of $\M^\hq$. We let
\[\Delta_{\hq, \nu}=\{ A\subseteq \bR: w(A)<\pi_{\hq|\nu, \infty}(\k)\},\]
where $w(A)$ is the Wadge rank of $A$.

\item Suppose that $\hq$ is any hod pair, $\k$ is an inaccessible cardinal of $\M^\hq$, and $g\subseteq \Coll(\omega, {<}\k)$ is $\M^\hq$-generic. Working in $\M^\hq$, we let $\lmi(\hq, \k)$, the local direct limit at $\k$, be the direct limit of the system $\mathcal{F}(\hq, \k)$ consisting of all complete iterates $\hr$ of $\hq$ such that:
\begin{enumerate}
    \item $\T_{\hq, \hr}\in \M^\hq[g]$ (although the limit is independent of $g$),
    \item $\lh(\T_{\hq, \hr})<\k$, and
    \item $\T_{\hq, \hr}$ is based on $\hq|\k$.
\end{enumerate}
We let $\hm(\hq, \k)=(\lmi(\hq, \k), (\Sigma^\hq)_{\lmi(\hq, \k)})$.
Setting $\gg=\pi_{\hq, \lmi(\hq, \k)}(\k)$, we let
\[N^l(\hq, \k)=(L_{\gg}(\lmi(\hq, \k)|\gg)^\omega))^{\M^\hq[g]}\]
be the local Nairian Model at $\k$.\footnote{As in \cite{blue2025nairian}, $N^l(\hq, \k)$ is defined in such a way that the ordinal height of $N^l(\hq, \k)$ is $\gg$.}
\end{enumerate}
\end{notation}

Here is some terminology we will use in this paper.

\begin{terminology}\label{term: cutpoints etc}\normalfont
Suppose $\hq=(\Q, \Lambda)$ is a hod pair and $\d$ is a cutpoint Woodin cardinal of $\Q$.

\begin{enumerate}
\item A normal iteration $\T$ of $\hq$ is \textbf{above} $\gamma$ if all extenders used in $\T$ have critical point at least $\gamma$, and \textbf{strictly above} $\gamma$
if for all $\a<\lh(\T)$, $\cp(E_\a^\T)>\gamma$.

\item We say that $\d$ \textbf{ends an lsa block} of $\Q$ if the least ${<}\d$-strong cardinal of $\Q$ is a limit of Woodin cardinals of $\Q$.\footnote{This terminology is justified by \cite{LSA}.}

\item Given $\nu<\d$, we say that $\nu$ is in the \textbf{$\d$-block} of $\Q$ if, letting $\k$ be the least ${<}\d$-strong cardinal of $\Q$, $\nu\in [\k, \d]$.

\item Assuming that $\d$ ends an lsa block of $\Q$, given $\nu$ in the $\d$-block of $\Q$, we say that $\nu$ is \textbf{big} in $\Q$ (relative to $\d$) if there is a $\tau\leq \nu$ such that $\Q\models ``\tau$ is a measurable cardinal that is a limit of ${<}\d$-strong cardinals.''

\item For $\nu<\d$ as above, we say that $\nu$ is \textbf{small} in $\Q$ (relative to $\d$) if $\nu$ is not big in $\Q$.

\item A $\Q$-cardinal $\nu$ is \textbf{properly overlapped} in $\Q$ if there is a $\d$ such that:
\begin{itemize}
    \item $\d$ is a Woodin cardinal of $\Q$,

    \item $\d$ ends an lsa block of $\Q$,

    \item $\nu$ is in the $\d$-block of $\Q$, and

    \item for all $\k\leq \nu$, $\Q|\nu\models ``\k$ is a strong cardinal'' if and only if $\Q\models ``\k$ is a ${<}\d$-strong cardinal.''
\end{itemize}
In this case, we say that $\nu$ is \textbf{properly overlapped relative to} $\d$.

\item A $\Q$-cardinal $\nu$ is a \textbf{proper cutpoint} in $\Q$ (relative to $\d$) if $\nu$ is properly overlapped (relative to $\d$) and is not a critical point of a total extender of $\Q$.
\end{enumerate}
\end{terminology}

We will use the following statement throughout this portion of the paper. $\rg$ stands for \textit{reflection generator}.

\begin{definition}\label{def: ref gen}\normalfont Let $\rg_0(\hq, \k, \nu)$ be the conjunction of the following statements:
\begin{enumerate}
\item $\hq$ is a complete iterate of $\hp$,\footnote{By this we usually mean countable iterate, i.e., $\T_{\hp, \hq}$ is countable in $M$.}
\item $\k<\nu\leq \hd_\hq$,
\item $\k<\varsigma$ is a strong cardinal of $M^\hq|\varsigma$,
\item $\nu$ is an inaccessible proper cutpoint of $\M^\hq$ (relative to $\hd_{\hq}$). 
\end{enumerate}
Let $\rg_1(\hq, \k, \nu)$ be the conjunction of the following statements:
\begin{enumerate}
\item $\rg_0(\hq, \k, \nu)$,
\item $(\k, \nu)\cap \tStr(\hq|\hd_\hq)=\emptyset$,
\item $\nu\in \tW(\Q)$ and $\sup(\tW(\Q|\nu))<\nu$,\footnote{Thus, $\nu$ is not a limit of Woodin cardinals.} and
\item $(\k, \nu)\cap \tW(\Q)\neq\emptyset$.
\end{enumerate}
Let $\rg_2(\hq, \l, \k, \nu)$ be the conjunction of the following statements:
\begin{enumerate}
\item $\rg_0(\hq, \k, \nu)$,
\item $\l\leq \k<\nu\leq \hd_\hq$,
\item $\l$ is a strong cardinal of $\M^\hq|\varsigma$,
\item $(\k, \nu)\cap \tStr(\hq|\hd_\hq)=\emptyset$,
\end{enumerate}
Let $\rg_3(\hq, \l, \k, \nu)$ be the conjunction of the following statements:
\begin{enumerate}
\item $\rg_1(\hq, \k, \nu)$ and
\item $\rg_2(\hq, \l, \k, \nu)$.
\end{enumerate}
\end{definition}

We will evaluate $\rg$ in $M$.

\begin{notation}[{\cite[Definitions 2.4.9 and 2.4.10]{LSA}}]\label{not: restriction}\normalfont Suppose that $W$ is a transitive set and $W'$ is a rank initial segment of $W$ such that $\ord\cap W'$ is an inaccessible cardinal of $W$.
\begin{enumerate}
    \item If $\T$ is an iteration tree on $W$, then we say that $\T$ is \textbf{based} on $W'$ if for every $\a<\lh(\T)$, either $[0, \a)_\T\cap D^\T\neq\emptyset$ or $E_\a^\T\in \pi_{0, \a}^\T(W')$.\footnote{Equivalently, $\T$ is below $\ord\cap W$.}

    \item Given $\T$ based on $W'$, let $\downarrow(\T, W')$ be the unique iteration tree $\U$ on $W'$ that has the same tree structure as $\T$ and uses the same extenders as $\T$.

    \item Conversely, if $\T$ is an iteration tree on $W'$, then we let $\uparrow(\T, W)$ be the unique iteration tree $\U$ on $W$ that has the same tree structure as $\T$ and uses the same extenders as $\T$.\footnote{Alternatively, $\uparrow(\T, W)$ is the copy of $\T$ onto $W$ via the identity map. Here, we are simply introducing a notation. Clearly, without extra iterability assumptions, $\uparrow(\T, W)$ may not have the same length as $\T$.}

    \item If $\hs$ is a hod pair and $\hs'\trianglelefteq \hs$, then we will use $\downarrow(\T, \hs')$ and $\uparrow(\T, \hs)$ instead of $\downarrow(\T, \M^{\hs'})$ and $\uparrow(\T, \M^\hs)$.\footnote{Notice that the results of \cite{SteelCom} show that $\downarrow(\T, \hs')$ and $\uparrow(\T, \hs)$ make sense as they are according to $\Sigma^{\hs'}$ and $\Sigma^{\hs}$ respectively.}
\end{enumerate}
\end{notation}

\begin{notation}\label{not: bounded powerset}\normalfont If $X$ is a set and $\beta = \ord\cap X$, then we write $\powerset_{\omega_1}^b(X)$ for the set of countable $Y\subseteq X$ such that, for some $\b'<\b$, $Y\subseteq X\cap \b'$.
\end{notation}

\begin{terminology}\label{term: long extender}\normalfont Suppose $j: K\rightarrow N$ is an elementary embedding between two transitive sets or classes, and $\nu\in \ord\cap N$. Then $E$ is the long extender derived from $j$ with space $\nu$ if $E$ consists of pairs $(a, A)$ such that $a\in \nu^{<\omega}$, $A\in K$ and $a\in j(A)$.

If $E$ is a $K$-extender, then let $K_E=Ult(K, E)$ and $\pi_E: K\rightarrow K_E$ be the ultrapower embedding.
\end{terminology}

We will also need the terminology introduced in \cite[Notation 8.4]{blue2025nairian}. In this paper, when we say that \textit{$\hq$ is a pair on the iteration tree $\T$} we mean that for some $\a<\lh(\T)$, $\hq=\hm_\a^\T$ (see clause 3 of \cite[Notation 8.4]{blue2025nairian}).

\section{Corollaries of \cite{blue2025nairian}}\label{sec: useful corollary}

This section contains technical lemmas that are used later on in the paper. The reader may skip it and return to it at a later point.

Both the present paper and \cite{blue2025nairian} investigate minimal Nairian Models assuming ${\sf{HPC}}$. We are able to analyze this portion of the universe because of certain important technical facts, namely \cite[Lemma 9.9, Corollary 9.10, Lemma 9.12 and Theorem 11.1]{blue2025nairian}.
We require consequences of \cite[Lemma 9.12]{blue2025nairian} that are implicit in \cite{blue2025nairian}.\footnote{E.g.~in the portion of the proof of \cite[Theorem 11.1]{blue2025nairian} showing that Clause 1 implies Clause 2.}
These consequences are isolated as Corollary \ref{cor: technical corollary}, which says roughly that if $\hs$ is an iterate of a hod pair $\hq$ above some cardinal $\k$ and $\hr$ is any other iterate of $\hq$, then in some circumstances we may argue that, essentially, in the comparison of $\hs$ with $\hr$, only $\hs$ moves.\footnote{Note that we follow the notation used in this part of the paper as well as in \cite{blue2025nairian}: in particular, $\hp, M, \mH$ have the same meaning that we have been using throughout this paper.}

\begin{corollary}\label{cor: technical corollary} Suppose that
\begin{itemize}
\item $\rg_0(\hq, \k, \nu)$ holds\footnote{We do not need that $\nu$ is a proper cutpoint of $\M^\hq.$}




    \item $\hr$ and $\hs$ are complete iterates of $\hq$ such that both $\T_{\hq, \hr}$ and $\T_{\hq, \hs}$ are based on $\hq|\nu$ and $\T_{\hq, \hs}$ is above $\k$ (but not necessarily strictly above), and

    \item $\zeta< \k_\hs$ is an inaccessible cardinal of $\M^\hs$ such that
    \[
    \sup(\pi_{\hr|\k_\hr, \infty}[\k_\hr])<\pi_{\hs|\zeta, \infty}(\ts(\hs|\zeta))
    \]
    and $\zeta$ is a proper cutpoint in $\hs|\k_\hs$.
\end{itemize}
      Suppose that $\hw'$ is the common iterate of $\hr$ and $\hs$ obtained via the least-extender-disagreement-coiteration.
      Then
\begin{enumerate}
\item $\T_{\hs, \hw'}$ is based on $\hs|\nu_{\hs}$ and $\T_{\hr, \hw'}$ is based on $\hr|\nu_\hr$,
\end{enumerate}
and there is a hod pair $\hw$ such that (see \rfig{fig:w interpolation})
\begin{enumerate}
\item[(2)] $\hw$ is an iterate of $\hs$,
\item[(3)] $\T_{\hs, \hw}$ is based on $\hs|\zeta$,
\item[(4)] $\hw|\zeta_{\hw}=\hw'|\zeta_{\hw'}$,
\item[(5)] $\T_{\hs, \hw}\insegeq \T_{\hs, \hw'}$,
\item[(6)] $(\T_{\hs, \hw'})_{\geq \hw}$ is strictly above $\zeta_{\hw}$, and
\item[(7)] $\sup(\pi_{\hr, \hw'}[\k_\hr])\leq \zeta_\hw$.
\end{enumerate}
\end{corollary}
\begin{proof}
\begin{figure}
\[\begin{tikzcd}[row sep=tiny, column sep=small]
    &&&&&&& \bullet \\
    &&&& {\kappa_\hw} \\
    &&&&&&& \bullet \\
    \bullet \\
    &&&& \bullet && {\T_{\hw, \hw'}} \\
    {\nu_\hs} \\
    &&&&&&& \bullet \\
    {\kappa_\hs} &&&& {\zeta_\hw} \\
    &&&& \bullet \\
    \zeta \\
    \bullet &&& {\T_{\hs, \hw}} \\
    {} \\
    &&&& \bullet \\
    \\
    \hs &&&& \hw &&& {\hw'}
    \arrow[dash, from=5-5, to=3-8]
    \arrow[dash, from=5-5, to=7-8]
    \arrow[dash, from=6-1, to=4-1]
    \arrow[dash, from=8-1, to=6-1]
    \arrow[dash, from=8-5, to=2-5]
    \arrow[dash, from=10-1, to=8-1]
    \arrow[dash, from=11-1, to=9-5]
    \arrow[dash, from=11-1, to=13-5]
    \arrow[dash, from=12-1, to=10-1]
    \arrow[dash, from=15-1, to=11-1]
    \arrow[dash, from=15-5, to=8-5]
    \arrow[dash, from=15-8, to=1-8]
\end{tikzcd}\]
\caption{Corollary \ref{cor: technical corollary}}
\label{fig:w interpolation}
\end{figure}

Clause 1 follows because $\T_{\hq, \hr}$ and $\T_{\hq, \hs}$ are based on $\hq|\nu$. Now, \cite[Lemma 9.12]{blue2025nairian} says that for each $\l\leq \k_{\hw'}$ that is a strong cardinal of $\M^{\hw'}|\hd_{\hw'}$, only one of $\T_{\hr, \hw'}$ and $\T_{\hs, \hw'}$ uses extenders on its main branch whose critical points are preimages of $\l$. Intuitively, because $\T_{\hq, \hs}$ is above $\k$, for each $\l<\k_{\hw'}$ that is a strong cardinal of $\M^{\hw'}|\hd_{\hw'}$, we do not use extenders on the main branch of $\T_{\hr, \hw'}$ with critical point a preimage of $\l$. Thus, the following must be true:
\begin{claim}\label{clm: only on one side} Suppose $\l<\k_{\hw'}$ is a strong cardinal of $\M^{\hw'}|\hd_{\hw'}$ and $\gamma+1=\lh(\T_{\hr, \hw'})$. Then if $\a\in [0, \gamma)_{\T_{\hr, \hw'}}$ is such that $\l\in \rge(\pi^{\T_{\hr, \hw'}}_{\a, \gamma})$ and $\b+1\in [0, \gamma)_{\T_{\hr, \hw'}}$ is such that $\T_{\hr, \hw'}(\b+1)=\a$, then
\[
\pi^{\T_{\hr, \hw'}}_{\a, \gamma}(\cp(E_\b^{\T_{\hr, \hw'}}))\neq\l.
\]
\end{claim}
\begin{proof} 
\begin{figure}[h]
\centering
\begin{tikzpicture}[xscale=1.2, yscale=0.8, font=\small]

    \begin{scope}[xshift=-3.5cm]
        \draw[thick] (0,0) node[below] {$\hs'$} -- (0,4.5);
        
        \draw[thick] (2.5,0) node[below] {$\hw'$} -- (2.5,4.5);
        
        \draw[dashed] (0,1.5) -- (2.5,1.5);
        \filldraw (0,1.5) circle (1pt) node[left] {$\tau$};
        \filldraw (2.5,1.5) circle (1pt) node[right] {$\tau$};
        \node[above, font=\scriptsize] at (1.25, 1.5) {$\hs'|\tau = \hw'|\tau$};
        
        \filldraw (2.5,3.5) circle (1pt) node[right] {$\lambda$};

        \draw[->, thick, blue, >=stealth] (0.1, 2.5) to[bend left=20] 
            node[midway, above, align=center, font=\scriptsize] {$(\mathcal{T}_{\hs, \hw'})_{\geq \hs'}$ \\ strictly above $\tau$} 
            (2.4, 2.5);
    \end{scope}

    \begin{scope}[xshift=0.5cm]
        \draw[thick] (0,0) node[below] {$\hr'$} -- (0,4.5);
        
        \draw[thick] (2.5,0) node[below] {$\hw'$} -- (2.5,4.5);
        
        \draw[dashed] (0,1.5) -- (2.5,1.5);
        \filldraw (0,1.5) circle (1pt) node[left] {$\tau$};
        \filldraw (2.5,1.5) circle (1pt) node[right] {$\tau$};
        \node[above, font=\scriptsize] at (1.25, 1.5) {$\hr'|\tau = \hw'|\tau$};
        
        \filldraw (2.5,3.5) circle (1pt) node[right] {$\lambda$};

        \draw[->, thick, red!80!black, >=stealth] (0.1, 2.5) to[bend left=20] 
            node[midway, above, align=center, font=\scriptsize] {$(\mathcal{T}_{\hr, \hw'})_{\geq \hr'}$ \\ strictly above $\tau$} 
            (2.4, 2.5);
    \end{scope}

\end{tikzpicture}
\caption{The iterations leading from $\hs'$ and $\hr'$ to $\hw'$.}
\label{fig:s_prime_r_prime_to_w_prime}
\end{figure}

Towards a contradiction, assume not. Let $\l<\k_{\hw'}$ be the least strong cardinal of $\M^{\hw'}|\k_{\hw'}$ violating the claim. It follows from \cite[Lemma 9.12]{blue2025nairian} that
 \vspace{0.3cm}
    \begin{enumerate}[label=(1), itemsep=0.3cm]
        \item $\T_{\hs, \hw'}$ does not use extenders on its main branch whose critical point is a preimage of $\l$.
    \end{enumerate}
    \vspace{0.3cm}
Let $\tau$ be the supremum of all strong cardinals of $\M^{\hw'}|\l$, and $\xi<\lh(\T_{\hs, \hw'})$ be the least such that setting $\hs'=\hm_{\xi}^{\T_{\hs, \hw'}}$, $\hs'|\tau=\hw'|\tau$ (see \rfig{fig:s_prime_r_prime_to_w_prime}). We then have that
\vspace{0.3cm}
    \begin{enumerate}[label=(2), itemsep=0.3cm]
        \item $(\T_{\hs, \hw'})_{\geq \hs'}$ is a normal iteration tree on $\hs'$ that is strictly above $\tau$.
    \end{enumerate}
    \vspace{0.3cm}
Similarly, let $\iota<\lh(\T_{\hr, \hw'})$ be the least such that letting $\hr'=\hm_\iota^{\T_{\hr, \hw'}}$, $\hr'|\tau=\hw'|\tau$. Again, we then have that
\vspace{0.3cm}
    \begin{enumerate}[label=(3), itemsep=0.3cm]
        \item $(\T_{\hr, \hw'})_{\geq \hr'}$ is a normal iteration tree on $\hr'$ that is strictly above $\tau$.
    \end{enumerate}
    \vspace{0.3cm}
Moreover, because $\l$ is a small cardinal, we have that
\vspace{0.3cm}
    \begin{enumerate}[label=(4), itemsep=0.3cm]
        \item for every $\iota'<\iota$ with $\iota'$ on the main branch of $\T_{\hr, \hw'}$, if $E$ is the extender used at $\iota'$ on the main branch of $\T_{\hr, \hw'}$, $\pi_{\iota'}^{\T_{\hr, \hw'}}(\cp(E))\neq\l$.
    \end{enumerate}
    \vspace{0.3cm}
Let now $\hq'$ be the normal iterate of $\hq$ such that $\hq'|\tau=\hw'|\tau$, $\gen(\T_{\hq, \hq'})\subseteq \tau$ and the strong cardinals of $\M^{\hq'}|\tau$ and $\M^{\hw'}|\tau$ coincide (see \rfig{fig:q_qprime_sprime_wprime}). 
\begin{figure}[h]
\centering
\begin{tikzpicture}[xscale=1.5, yscale=0.9, font=\small]

    \filldraw (0,0.5) circle (2pt) node[left=5pt] {$\hq$};

    \draw[thick] (2.5,0) node[below] {$\hq'$} -- (2.5,4.5);

    \draw[thick] (4.5,0) node[below] {$\hs'$} -- (4.5,4.5);

    \draw[thick] (6.5,0) node[below] {$\hw'$} -- (6.5,4.5);

    \draw[->, very thick, >=stealth, shorten >=2pt] (0.2,0.5) to[bend left=15] 
        node[midway, above] {$\mathcal{T}_{\hq, \hq'}$} 
        (2.4,0.2);

    \draw[dashed] (2.5,2) -- (6.5,2);
    
    \filldraw (2.5,2) circle (1pt) node[above left] {$\tau$};
    \filldraw (4.5,2) circle (1pt) node[above right] {$\tau$};
    \filldraw (6.5,2) circle (1pt) node[above right] {$\tau$};
    
    \node[below, font=\scriptsize] at (3.5, 2) {$\hq'|\tau = \hs'|\tau$};
    \node[below, font=\scriptsize] at (5.5, 2) {$\hs'|\tau = \hw'|\tau$};

\end{tikzpicture}
\caption{The relationship between $\hq$, $\hq'$, $\hs'$, and $\hw'$.}
\label{fig:q_qprime_sprime_wprime}
\end{figure}
Because $\T_{\hq, \hs}$ is above $\k$, we have that $\hq'|\l_{\hs'}=\hs'|\l_{\hs'}$. 
It follows that for some $\phi\leq \l_{\hr'}$, $\hr'|\phi$ is a complete iterate of $\hs'|\l=\hq'|\l$. Let then $\hs''$ be the last pair of $\uparrow(\T_{\hs'|\l, \hr|\phi}, \hs')$. We have that $\l_{\hs''}=\phi$ and $\hs''$ is on the main branch of $\T_{\hs, \hw'}$. Now, if $\phi<\l_{\hr'}$ then the extender used at $\hs''$ on the main branch of $\T_{\hs, \hw'}$ has a critical point $\phi$, which contradicts the fact that $\T_{\hs, \hw'}$ does not use any extender on its main branch whose critical point is a preimage of $\l$. Therefore, $\phi=\l_{\hr'}$. But then \cite[Corollary 9.11]{blue2025nairian} implies that neither $\T_{\hs'', \hw'}$ nor $\T_{\hr', \hw'}$ use extenders with critical point $\phi$. Since $\T_{\hr, \hw'}$ uses such an extender, the first use of such an extender must occur before stage $\iota$, and this contradicts (4).
\end{proof}

Now let $\hw$ be the least node on $\T_{\hs, \hw'}$ such that $\hw|\zeta_{\hw'}=\hw'|\zeta_{\hw'}$. It follows that $(\T_{\hs, \hw'})_{\geq \hw}$ is above $\zeta_{\hw'}=\zeta_{\hw}$.\footnote{This is because \(\k\) is small.}
Thus, $\hw$ satisfies all the clauses above except possibly clause 3 and 7 (see \rfig{fig:s_w_wprime_partial}).
\begin{figure}[h]
\centering
\begin{tikzpicture}[xscale=1.5, yscale=1.0, font=\small]

    \coordinate (Stop) at (0.2, 2.6);
    \coordinate (Sbot) at (0.2, 2.4);
    \coordinate (Wtop) at (3.5, 4.0);
    \coordinate (Wbot) at (3.5, 1.0);

    \filldraw (0,2.5) circle (2pt) node[left=5pt] {$\hs$};

    \draw[thick, gray, ->, >=stealth] (Stop) -- ($(Stop)!0.6!(Wtop)$);
    \draw[thick, gray, dotted] ($(Stop)!0.6!(Wtop)$) -- ($(Stop)!0.85!(Wtop)$);
    
    \draw[thick, gray, ->, >=stealth] (Sbot) -- ($(Sbot)!0.6!(Wbot)$);
    \draw[thick, gray, dotted] ($(Sbot)!0.6!(Wbot)$) -- ($(Sbot)!0.85!(Wbot)$);

    \node[gray] at (1.2, 2.5) {$\mathcal{T}_{\hs, \hw'}$};

    \draw[thick] (3.5, 1.0) node[below] {$\hw$} -- (3.5, 4.0);

    \draw[thick] (6.0, 1.0) node[below] {$\hw'$} -- (6.0, 4.0);

    \draw[dashed] (3.5, 2.0) -- (6.0, 2.0);
    
    \filldraw (3.5, 2.0) circle (1pt) node[above left] {$\zeta_\hw$};
    \filldraw (6.0, 2.0) circle (1pt) node[above right] {$\zeta_{\hw'}$};
    
    \node[below, font=\scriptsize] at (4.75, 2.0) {$\hw|\zeta_\hw = \hw'|\zeta_{\hw'}$};

    \draw[->, thick, blue, >=stealth] (3.6, 3.2) to[bend left=15] 
        node[midway, above, align=center, font=\scriptsize] {$(\mathcal{T}_{\hs, \hw'})_{>\hw}$ \\ strictly above $\zeta_\hw$} 
        (5.9, 3.2);

\end{tikzpicture}
\caption{The iteration $\mathcal{T}_{\hs, \hw'}$ projecting towards $\hw$, and the subsequent agreement between $\hw$ and $\hw'$.}
\label{fig:s_w_wprime_partial}
\end{figure}
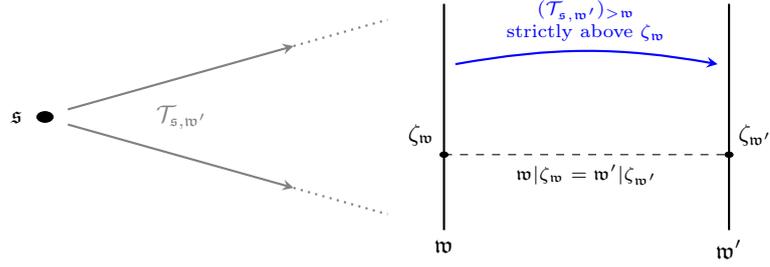

Let $\a<\lh(\T_{\hs, \hw'})$ be the largest ordinal such that, setting $\hx=\hm_\a^{\T_{\hs, \hw'}}$, $\T_{\hs, \hx}$ is based on $\hs|\zeta$, and let $\b$ be the largest ordinal such that, setting $\hy=\hm_\b^{\T_{\hr, \hw'}}$, $\T_{\hr, \hy}$ is based on $\hr|\k_\hr$. Then if $\X$ and $\Y$ are the iteration trees produced on $\hs|\zeta$ and $\hr|\k_\hr$ via the least-extender-disagreement-coiteration, then $\X\insegeq \downarrow(\T_{\hs, \hx}, \hs|\zeta)$ and $\Y=\downarrow(\T_{\hr, \hy}, \k_\hr)$.\footnote{See \rnot{not: restriction}.} $\X\neq\downarrow(\T_{\hs, \hx}, \hs|\zeta)$ is possible. Because $\sup(\pi_{\hr|\k_\hr, \infty}[\k_\hr])<\pi_{\hs|\zeta, \infty}(\ts(\hs|\zeta))$, we must have that
\vspace{0.3cm}
    \begin{enumerate}[label=(5), itemsep=0.3cm]
        \item $\hy|\k_\hy \insegeq\hx|\zeta_\hx$.
    \end{enumerate}
    \vspace{0.3cm}
The following is the crucial claim.
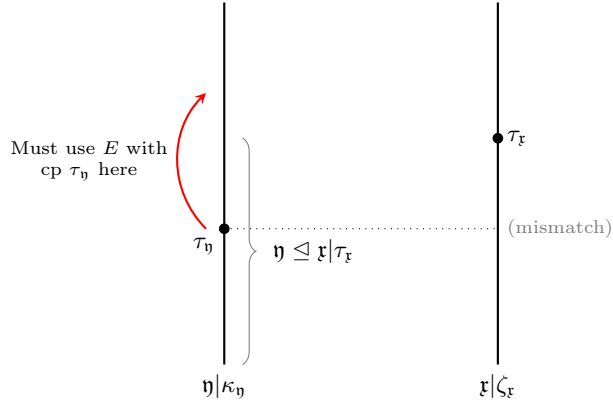
\begin{figure}[h]
\centering
\begin{tikzpicture}[scale=1.2, font=\small]
    \node[font=\scriptsize, blue, anchor=south, align=center] at (1.5, 4.2) {Definition: $\tau_\hx := \ts(\hx|\zeta_\hx)$ and $\tau_\hy := \ts(\hy|\kappa_\hy)$};

    \draw[thick] (0,0) node[below] {$\hy|\kappa_\hy$} -- (0,4);
    
    \draw[thick] (3,0) node[below] {$\hx|\zeta_\hx$} -- (3,4);
    
    \filldraw (0, 1.5) circle (1.5pt) node[below left] {$\tau_\hy$};
    \draw[dotted] (0, 1.5) -- (3, 1.5);
    \node[right, gray, font=\scriptsize] at (3, 1.5) {(mismatch)};
    
    \filldraw (3, 2.5) circle (1.5pt) node[right] {$\tau_\hx$};
    
    \draw[decorate, decoration={brace, amplitude=5pt, mirror}, gray] (0.2, 0) -- (0.2, 2.5)
        node[midway, right=8pt, align=left, black] {$\hy \trianglelefteq \hx|\tau_\hx$};

    \draw[->, red, thick, >=stealth] (-0.2, 1.5) to[bend left=45] 
        node[midway, left, align=center, font=\scriptsize, black] {Must use $E$ with\\cp $\tau_\hy$ here} 
        (-0.2, 3.0);
        
\end{tikzpicture}
\caption{The contradiction in Claim \ref{clm: crucial}: If the first strong cardinals differ ($\tau_\hx > \tau_\hy$), $\hy$ becomes a proper segment of $\hx$, forcing an extender usage that was previously ruled out.}
\label{fig:strong_cardinal_mismatch}
\end{figure}
\begin{claim}\label{clm: crucial} Suppose $\l<\zeta_{\hw'}$ is either a strong cardinal of $\M^{\hw'}|\k_{\hw'}$ or a limit of strong cardinals of $\M^{\hw'}|\k_{\hw'}$. Then $\hx|\l=\hw'|\l=\hy|\l$.
\end{claim}
\begin{proof}
The proof follows the inductive argument in \cite[Subsection 11.1.1]{blue2025nairian} so we only outline it. Suppose $\l=\ts(\hw'|\k_{\hw'})$.\footnote{So it is the least strong cardinal.} Because $\sup(\pi_{\hr|\k_\hr, \infty}[\k_\hr])<\pi_{\hs|\zeta, \infty}(\ts(\hs|\zeta))$, we must have that $\zeta>\ts(\hs|\k_\hs)$.

First we show that $\ts(\hx|\k_\hx)=\ts(\hy|\k_\hy)$. If not, then $\ts(\hx|\k_\hx)>\ts(\hy|\k_\hy)$ (see (5)), implying that 
\[
\hy|\k_\hy\inseg \hx|\ts(\hx|\k_{\hx}).
\]
But then we must use an extender with critical point $\ts(\hy|\k_\hy)$ on the main branch of $\T_{\hr, \hw'}$, which contradicts \rcl{clm: only on one side} (see \rfig{fig:strong_cardinal_mismatch}).

Now, since $\ts(\hx|\k_\hx)=\ts(\hy|\k_\hy)$, neither $(\T_{\hs, \hw'})_{\geq \hx}$ nor $(\T_{\hr, \hw'})_{\geq \hy}$ use an extender with critical point $\ts(\hx|\k_\hx)=\ts(\hy|\k_\hy)$ (see \cite[Corollary 9.11]{blue2025nairian}). Hence, we must have that 
\[
\ts(\hx|\k_\hx)=\ts(\hy|\k_\hy)=\l.
\]
One can repeat this inductively, as in \cite[Subsection 11.1.1]{blue2025nairian}, to conclude that the claim holds for any $\l$.
\end{proof}

Since $\zeta$ is a proper cutpoint, the claim above implies that $\zeta_\hx=\zeta_{\hw'}$, and so in fact, $\hx=\hw$. The claim also implies that $(\T_{\hr, \hw'})_{\geq \hy}$ is above $\k_\hy$. Hence, 
\[
\sup(\pi_{\hr, \hw'}[\k_\hr])=\sup(\pi_{\hr, \hy}[\k_\hr]).
\]
\end{proof}

The following is another version of \rcor{cor: technical corollary} that will be very useful in our calculations.

\begin{corollary}\label{cor: tech cor ii} 
Suppose that
\begin{itemize}
\item $\rg_2(\hq, \l, \k, \nu)$,
\item $\hr$ and $\hs$ are two complete iterates of $\hq$,
    \item $\T_{\hq, \hs}$ is above $\k$, 
    \item $\gen(\T_{\hq, \hr})\subseteq \l_\hr$,
    \item there is $\zeta<\k_\hs$ such that $\zeta$ is a proper cutpoint and an inaccessible cardinal of $\M^\hs$, $\gen(\T_{\hq, \hs})\subseteq \zeta$, and 
    \begin{center}$\sup(\pi_{\hr|\k_\hr, \infty}[\k_\hr])<\pi_{\hs|\zeta, \infty}(\ts(\hs|\zeta))$.\end{center}
\end{itemize}
Let $\hw$ be a complete iterate of $\hr$ and $\hs$ obtained via the least-extender-disagreement coiteration. Let $\hx$ be the least pair on $\T_{\hr, \hw}$ such that $\T_{\hr, \hx}$ is based on $\hr|\nu_\hr$. 

Then the following conditions hold (see \rfig{fig:cor_tech_ii_overview} and \rfig{fig:cor_tech_ii_details}):

\begin{figure}[ht]
\centering
\begin{tikzpicture}[scale=0.7, >=stealth, font=\small]
    \node (Q) at (0, 0) {$\hq$};
    \node (R) at (-3.0, 2.5) {$\hr$};
    \node (S) at (3.0, 2.5) {$\hs$};
    \node (W) at (0, 5.0) {$\hw$};

    \draw[->] (Q) -- (R) node[midway, left, font=\scriptsize] {$\gen \subseteq \lambda_\hr$};
    \draw[->] (Q) -- (S) node[midway, right, font=\scriptsize] {Above $\kappa$};

    \draw[->] (R) to[bend left=15] node[midway, left, font=\scriptsize] {$\T_{\hr, \hw}$} (W);
    \draw[->] (S) to[bend right=15] node[midway, right, font=\scriptsize] {$\T_{\hs, \hw}$ based on $\hs|\nu_\hs$} (W);
\end{tikzpicture}
\caption{The overview of the coiteration of $\hr$ and $\hs$ to $\hw$.}
\label{fig:cor_tech_ii_overview}
\end{figure}

\begin{figure}[ht]
\centering
\begin{tikzpicture}[scale=0.85, >=stealth, font=\small]
    \node (R) at (-4.0, 0) {$\hr$};
    \node (Xp) at (-1.0, 2.0) {$\hx'$}; 
    \node (X) at (-3.5, 4.0) {$\hx$};   
    \node (XE) at (2.0, 3.5) {$\hx'_E$};
    \node (W) at (4.5, 5.5) {$\hw$};    

    \draw[->] (R) to[bend right=10] (Xp);
    \draw[->, dotted, thick] (Xp) -- (X);
    \draw[->, red, thick] (Xp) to[bend right=20] node[midway, above, font=\scriptsize, black] {$E$} (XE);
    \draw[->] (XE) to[bend right=20] (W);

    \node[draw, align=left, font=\scriptsize, anchor=north west, rounded corners] at (0.5, 1.5) {
        \textbf{Key Properties:}\\[3pt]
        $\bullet$ $\T_{\hr, \hx}$ is based on $\hr|\nu_\hr$\\[3pt]
        $\bullet$ $\hx'$ and $\hx'_E$ are on the main branch of $\T_{\hr, \hw}$\\[3pt]
        $\bullet$ $\T_{\hx', \hx}$ is above $\k_{\hx'}$\\[3pt]
        $\bullet$ $\cp(E) = \kappa_{\hx'}$\\[3pt]
        $\bullet$ $(\T_{\hr, \hw})_{\geq \hx'_E}$ is above $\kappa_{\hx'_E}$
    };
\end{tikzpicture}
\caption{The structure of the $\hr$-side iteration: splitting at $\hx'$ and using $E$.}
\label{fig:cor_tech_ii_details}
\end{figure}
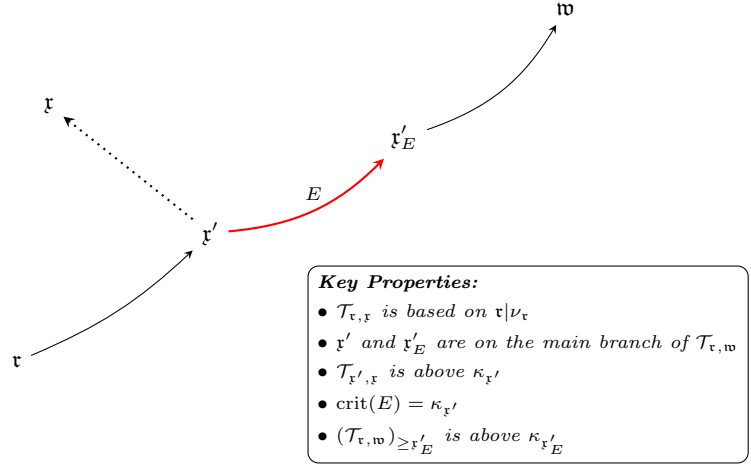

\begin{enumerate}
    \item $\gen(\T_{\hs, \hw})\subseteq \nu_\hw$ and $\gen(\T_{\hr, \hw})\subseteq \nu_\hw$,
    \item $\T_{\hs, \hw}$ is based on $\hs|\nu_\hs$,
    \item $\hx$ is a complete iterate of $\hr$,
    \item $\hx|\nu_\hx=\hw|\nu_\hx$,
    \item if $\hx'$ is the least node of $\T_{\hr, \hx}$ such that $\hx'|\k_{\hx'}=\hx|\k_\hx$, then $\hx'$ is on the main branch of $\T_{\hr, \hw}$,
    \item if $E$ is the extender used on the main branch of $\T_{\hr, \hx}$ at $\hx'$, then $\cp(E)=\k_{\hx'}$,
    \item $(\T_{\hr, \hw})_{\geq \hx'_E}$ is above $\k_{\hx'_E}$. 
\end{enumerate}
\end{corollary}

\begin{proof}
We begin with the simpler clauses:\\

\noindent\textbf{Clauses (1), (3), (4), and (7).} 
Clause (1) follows from full normalization and the fact that $\gen(\T_{\hq, \hs})\subseteq \nu_\hs$ and $\gen(\T_{\hq, \hr})\subseteq \nu_\hr$. Specifically, if $\U$ is the normal tree on $\hq$ via the coiteration with $\hw|\nu_\hw$, then $\hw$ is the last model of $\U$, and $\U$ fully normalizes the stacks $(\T_{\hq, \hs})^\frown \T_{\hs, \hw}$ and $(\T_{\hq, \hr})^\frown \T_{\hr, \hw}$. 
Clause (3) holds because the main branch of $\T_{\hr, \hw}$ does not drop, and $\hx$ is chosen specifically as a node on that tree. 
Clause (4) is an immediate consequence of (3). 
Clause (7) follows from the smallness of $\k$ ($\k<\hd_\hq$).

\medskip
\noindent\textbf{Clauses (2), (5), and (6).}
Here the key point that leads to a contradiction is that we show that the $\hr$-side of the iteration must use an extender on its main branch whose critical point is an image of $\k_\hr$, implying that no extender used on the main branch of $\hs$-side can have a critical point that is an image of $\k_\hs$. 

Let $\a < \lh(\T_{\hr, \hw})$ be such that $\hm_\a^{\T_{\hr, \hw}}=\hx$. Let $\hy=\hm_\a^{\T_{\hs, \hw}}$ be the corresponding model on the $\hs$-side\footnote{Recall from \cite{blue2025nairian} that in the least-extender-disagreement coiterations we allow padding.}. 
Note that $\T_{\hs, \hy}$ is based on $\hs|\zeta$ (where $\zeta < \k_\hs$) and $\T_{\hr, \hx}$ is based on $\hr|\nu_\hr$. Since these are produced via least-extender-disagreement coiteration, we have
\begin{center}$\hy|\nu_\hx = \hx|\nu_\hx$\ \ \ \ \ \ (*).\end{center}
The diagram \rfig{fig:four_lines_comparison} might be helpful to the reader. We note that
 \vspace{0.3cm}
    \begin{enumerate}[label=(1), itemsep=0.3cm]
        \item since $\T_{\hq, \hs}$ avoids extenders with critical points below $\k$, any extender $E$ with a critical point mapping to a strong cardinal of $\hw|\k_{\hw}$ can only appear in $\T_{\hs, \hw}$.
    \end{enumerate}
    \vspace{0.3cm}
Applying \cite[Lemma 9.12]{blue2025nairian} yields the following properties:
 \vspace{0.3cm}
    \begin{enumerate}[label=(2.\arabic*), itemsep=0.3cm]
        \item $\k_{\hx} < \k_{\hw}$.
\item  If $\gamma=\lh(\T_{\hr, \hw})+1$, and $\tau$ is a strong cardinal of $\hw|\k_\hw$, then for any $\b \in [0, \gamma]_{\T_{\hr, \hw}}$ such that $\tau \in \rge(\pi^{\T_{\hr, \hw}}_{\b, \gamma})$, we have:
    \[ \sup(\pi^{\T_{\hr, \hw}}_{\b, \gamma}[(\pi^{\T_{\hr, \hw}}_{\b, \gamma})^{-1}(\tau)]) = \tau. \]
\item For every $\tau < \k_\hw$, $\tau$ is a strong cardinal of $\hw|\k_\hw$ if and only if $\tau < \k_\hx$ and $\tau$ is a strong cardinal of $\hx|\k_\hx$.
    \end{enumerate}
    \vspace{0.3cm}

\begin{figure}[ht]
\centering
\begin{tikzpicture}[scale=1.1, >=stealth, font=\small]
    \def\xW{0}         
    \def\xX{-3.0}      
    \def\xXprime{-6.0} 
    \def\xY{3.5}       
    
    \def\hBase{0}
    \def\hKappaX{1.8}    
    \def\hNuPrime{2.6}   
    \def\hNu{3.4}        
    \def\hZeta{4.4}      
    \def\hKappaY{5.4}    
    \def\hKappaW{6.4}    
    \def\hTop{7.2}

    \draw[thick] (\xXprime, \hBase) -- (\xXprime, \hTop);
    \draw[thick] (\xX, \hBase) -- (\xX, \hTop);
    \draw[thick] (\xW, \hBase) -- (\xW, \hTop);
    \draw[thick] (\xY, \hBase) -- (\xY, \hTop);

    \node[below] at (\xXprime, \hBase) {\large $\hx'$};
    \node[below] at (\xX, \hBase) {\large $\hx$};
    \node[below] at (\xW, \hBase) {\large $\hw$};
    \node[below] at (\xY, \hBase) {\large $\hy$};

    \filldraw (\xXprime, \hNuPrime) circle (1.5pt) node[left] {$\nu_{\hx'}$};

    \draw[dotted, gray] (\xXprime, \hKappaX) -- (\xY, \hKappaX);

    \filldraw (\xXprime, \hKappaX) circle (1.5pt) node[left] {$\kappa_\hx$};
    
    \filldraw (\xX, \hKappaX) circle (1.5pt) node[right, xshift=2pt, font=\scriptsize] {$\kappa_\hx$};
    
    \filldraw (\xW, \hKappaX) circle (1.5pt) node[right, font=\scriptsize] {$\kappa_\hx$};
    
    \filldraw (\xY, \hKappaX) circle (1.5pt) node[left, font=\scriptsize] {$\kappa_\hx$};

    \draw[dashed, gray] (\xXprime, \hKappaX) -- (\xX, \hKappaX);
    
    \draw[->, very thick, violet] (\xXprime, \hKappaX+1.2) to[bend left=25] 
        node[midway, above, font=\scriptsize, align=center, violet] {$\T_{\hx', \hx}$\\(above $\kappa_\hx$)} 
        (\xX, \hKappaX+1.2);

    \draw[decorate, decoration={brace, amplitude=6pt}, thick] (\xW-0.15, \hBase+0.2) -- (\xW-0.15, \hKappaX)
        node[midway, left=8pt, align=right, font=\scriptsize] {Strong cardinals\\ of $\hw|\kappa_\hw$};

    \draw[dashed] (\xX, \hNu) -- (\xY, \hNu);
    \filldraw (\xX, \hNu) circle (1.5pt) node[left] {$\nu_\hx$};
    \filldraw (\xY, \hNu) circle (1.5pt) node[right] {$\nu_\hx$};
    \node[fill=white, inner sep=2pt, font=\scriptsize] at (\xW, \hNu) {Agreement $\le \nu_\hx$};

    \draw[densely dotted, thick] (\xY-0.4, \hZeta) -- (\xY+0.4, \hZeta);
    \node[left] at (\xY-0.1, \hZeta) {$\zeta_\hy$};
    
    \draw[decorate, decoration={brace, amplitude=4pt, mirror}, thick] (\xY+0.5, \hBase+0.2) -- (\xY+0.5, \hZeta)
        node[midway, right=6pt, align=left, font=\scriptsize] {$\T_{\hs, \hy}$ based\\on $\hs|\zeta_\hy$};

    \draw[gray!50] (\xY+0.1, \hZeta) -- (\xY+0.5, \hZeta);

    \filldraw (\xY, \hKappaY) circle (1.5pt) node[right] {$\kappa_\hy$};

    \filldraw (\xW, \hKappaW) circle (1.5pt) node[left] {$\kappa_\hw$};

\end{tikzpicture}
\caption{The ordinal hierarchy showing the iteration from $\hx'$ to $\hx$ acting above $\kappa_\hx$. The strong cardinals of $\hw$ are contained below $\kappa_\hx$.}
\label{fig:four_lines_comparison}
\end{figure}
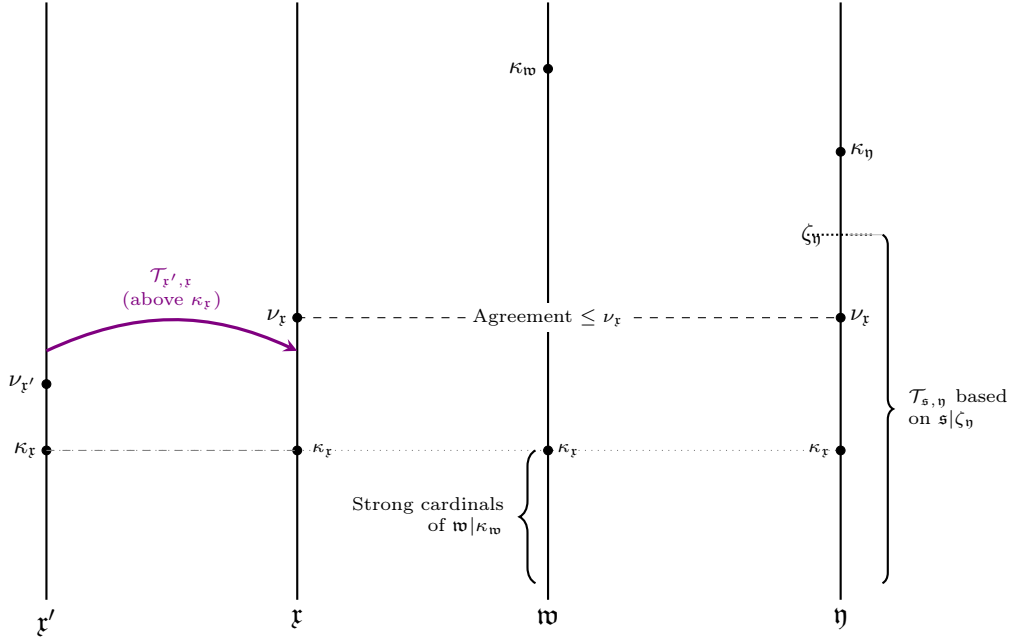
(2.1) follows because $\T_{\hs, \hy}$ is based on $\hs|\zeta$, $\zeta<\k_\hs$ and $(*)$ ((*) implies that $\k_\hx<\nu_\hx<\zeta_\hy$). (2.2) is a consequence of  \cite[Lemma 9.12]{blue2025nairian}, and follows from (1). (2.3) follows because we have that $\hx|\nu_\hx=\hw|\nu_\hx$ (see Clause 4 of the current theorem) and (2.2). 

Combining (2.2) and (2.3), we see that the extender used on the main branch of $\T_{\hr, \hw}$ immediately after stage $\a$ must have critical point $\k_{\hx}$. This confirms Clauses (5) and (6).

Finally, for Clause (2): Since $\T_{\hr, \hw}$ uses extenders with critical points in the image of $\k_\hr$, \cite[Lemma 9.12]{blue2025nairian} implies that $\T_{\hs, \hw}$ cannot use an extender whose critical point is an image of $\k_\hs$. Therefore, (2.3) and $(*)$ imply that $\T_{\hs, \hw}$ is based on $\hs|\nu_\hs$.
\end{proof}

We finish with the following corollary which is in the same spirit as those presented above.
\begin{corollary}\label{cor: comparing tau bounded iterates} Suppose $\rg_0(\hq, \l, \nu)$ holds, and $\hr$ and $\hs$ are two $\pi_{\hq, \infty}(\l)$-bounded complete iterates of $\hq$ such that
 $\T_{\hq, \hs}$ is based on $\hq|\nu$. Let $\hw$ be the least-extender-disagreement comparison of $\hr$ and $\hs$. Then $\T_{\hr, \hw}$ is based on $\hr|\nu_\hr$.
\end{corollary}

\begin{figure}[htbp]
    \centering
    \begin{tikzpicture}[>=stealth, thick]
        \node (hq) at (0, 0) {$\hq$};
        \node (hr) at (3, 2) {$\hr$};
        \node (hs) at (3, -2) {$\hs$};
        \node (hw) at (6, 0) {$\hw$};

        \draw[->] (hq) -- node[above left] {$\T_{\hq, \hr}$} (hr);
        \draw[->] (hq) -- node[below left] {$\T_{\hq, \hs}$} (hs);
        \draw[->] (hr) -- node[above right] {$\T_{\hr, \hw}$} (hw);
        \draw[->] (hs) -- node[below right] {$\T_{\hs, \hw}$} (hw);

        \node[draw=black, thick, fill=gray!10, rounded corners, inner sep=6pt, anchor=north] 
        at (3, -3.5) {
        \begin{minipage}{9cm}
            \centering
            \textbf{\underline{Key Properties}}
            \begin{itemize}[label={\tiny$\bullet$}, itemsep=1.2pt, topsep=3pt, leftmargin=1em, font=\scriptsize]
                \item $\hw$ is the least-extender-disagreement comparison of $\hr$ and $\hs$.
                \item $\M_\infty(\hs|\nu_\hs)\insegeq \M_\infty(\hr|\nu_\hr)$.
                \item $\T_{\hq, \hs}$ is based on $\hq|\nu$.
                \item \textbf{Goal:} Show $\T_{\hr, \hw}$ is based on $\hr|\nu_\hr$.
            \end{itemize}
        \end{minipage}
        };
    \end{tikzpicture}
    \caption{Comparison of bounded iterates $\hr$ and $\hs$.}
    \label{fig:comparing_tau_bounded}
\end{figure}

\begin{proof} The proof is very similar to the proofs we have already given, and so we will only outline the proof. We will again use \cite[Lemma 9.12]{blue2025nairian}. We have that  
\[\M_\infty(\hs|\nu_\hs)\insegeq \M_\infty(\hr|\nu_\hr).\] 
Let $\a$ be the largest such that $(\T_{\hr, \hw})_{\leq \a}$ is based on $\hr|\nu_\hr$, and let $\hm=\hm_\a^{\T_{\hr, \hw}}$. The goal is to verify the following.

\vspace{0.3cm}
\begin{enumerate}[label=(\alph*), itemsep=0.3cm]
    \item[(a)] $\nu_\hm=\nu_\hw$.
\end{enumerate}
\vspace{0.3cm}

To prove (a), we first establish the following.

\vspace{0.3cm}
\begin{enumerate}[label=(\alph*), itemsep=0.3cm]
    \setcounter{enumi}{1}
    \item[(b)] The set of strong cardinals of $\M^\hm|\nu_\hm$ and of $\M^\hw|\nu_\hw$ coincide.
\end{enumerate}
\vspace{0.3cm}

It is not hard to see that (b) implies (a). Indeed, let $\hx$, $\hy$ and $\hn$ be the least nodes of $\T_{\hr, \hw}$, $\T_{\hs, \hw}$, and $\T_{\hq, \hw}$ such that \[\hn|\l_\hw=\hx|\l_\hw=\hy|\l_\hw=\hw|\l_\hw.\]
Because 
$\gen(\T_{\hq, \hr})\subseteq \l_\hr$, $\gen(\T_{\hq, \hs})\subseteq \l_\hs$, $\gen(\T_{\hr, \hx})\subseteq \l_\hx$, and $\gen(\T_{\hs, \hy})\subseteq \l_\hy$, we have that $\hx=\hn=\hy$ (also, note that $\l_\hx=\l_\hy=\l_\hn=\l_\hw$). Hence, \[\hw=\hx=\hy=\hn.\] Because $\l_\hw<\nu_\hm$ (by (b)), we have that $\T_{\hr, \hw}$ is based on $\hr|\nu_\hr$. Thus, it is enough to prove (b).

The proof of (b) is by induction on strong cardinals of $\M^\hw|\nu_\hw$. Let $\zeta$ be a strong cardinal of $\M^\hw|\nu_\hw$ and set $\zeta'=\sup(\tStr(\M^\hw|\zeta))$. We assume that

\vspace{0.3cm}
\begin{enumerate}[label=(IH), itemsep=0.3cm]
    \item[(IH)] $\zeta'<\nu_\hm$,
\end{enumerate}
\vspace{0.3cm}

and want to show that $\zeta<\nu_\hm$. We have that 

\vspace{0.3cm}
\begin{enumerate}[label=(\arabic*), itemsep=0.3cm]
    \item[(1)] If $\zeta''\leq \zeta'$ is a strong cardinal of $\M^\hw|\nu_\hw$, then $\zeta''$ is a strong cardinal of $\M^\hm|\nu_\hm$.
\end{enumerate}
\vspace{0.3cm}

Let $\hm'$ be the least node of $\T_{\hr, \hw}$ such that $\hm'|\zeta'=\hw|\zeta'$, let $\hs'$ be the least node of $\T_{\hs, \hw}$ such that $\hs'|\zeta'=\hw|\zeta'$, and let $\hq'$ be the least node of $\T_{\hq, \hw}$ such that $\hq'|\zeta'=\hw|\zeta'$. We thus have that: 

\vspace{0.3cm}
\begin{enumerate}[label=(2.\arabic*), itemsep=0.3cm]
    \item If $\zeta''\leq \zeta'$ is a strong cardinal of $\M^\hw|\nu_\hw$, then $\zeta''$ is a strong cardinal of $\M^{\hs'}|\nu_{\hs'}$.
    \item If $\zeta''\leq \zeta'$ is a strong cardinal of $\M^\hw|\nu_\hw$, then $\zeta''$ is a strong cardinal of $\M^{\hq'}|\nu_{\hq'}$.
    \item $\T_{\hm', \hm}$ is based on $\hm'|\nu_{\hm'}$. 
    \item $\T_{\hq', \hm'}$ is above $\zeta$ (the iteration is defined because of full normalization).
    \item $\T_{\hq', \hs'}$ is above $\zeta'$ and $\T_{\hq', \hs'}$ is based on $\hq'|\nu_{\hq'}$. 
\end{enumerate}
\vspace{0.3cm}

It follows from (2.5) that 

\vspace{0.3cm}
\begin{enumerate}[label=(\arabic*), itemsep=0.3cm]
    \setcounter{enumi}{2}
    \item[(3)] $\M_\infty(\hs'|\nu_{\hs'})\insegeq\M_\infty(\hm'|\nu_{\hm'}).$
\end{enumerate}
\vspace{0.3cm} 

It follows from \cite[Lemma 9.12]{blue2025nairian} that in the least-extender-disagreement coiteration of $\hm'$ and $\hs'$ only one side uses extenders whose critical point is a preimage of $\zeta$. To determine which side, we analyze the least-extender-disagreement coiteration of $\hm'|\zeta_{\hm'}$ and $\hs'|\zeta_{\hs'}$. Letting $\hx$ and $\hy$ be the least nodes of $\T_{\hm', \hw}$ and $\T_{\hs', \hw}$ such that $\T_{\hm', \hx}$ and $\T_{\hs', \hy}$ are based on $\hm'|\zeta_{\hm'}$ and $\hs'|\zeta_{\hs'}$ respectively, we have three cases. 

\textbf{Case 1.} Suppose $\hx|\zeta_{\hx}=\hy|\zeta_{\hy}$. It then follows from \cite[Corollary 9.11]{blue2025nairian} that $\zeta=\zeta_{\hx}=\zeta_{\hy}$, and since $\T_{\hm', \hx}$ is based on $\hm'|\nu_{\hm'}$, we have that $\zeta<\nu_\hm$. 

\textbf{Case 2.} $\hx|\zeta_{\hx}\inseg \hy|\zeta_{\hy}$. In this case, we must use an extender whose critical point is a preimage of $\zeta$ on $\T_{\hx, \hw}$ but not on $\T_{\hy, \hw}$. It then follows that $\hw|\zeta$ is a complete iterate of $\hy|\zeta_\hy$, and (3) implies that the least-extender-disagreement coiteration of $\hx$ and $\hy|\zeta_\hy$ is based on $\hx|\nu_\hx$. Therefore, again $\zeta<\nu_\hm$.

\textbf{Case 3.} $\hy|\zeta_{\hy}\inseg \hx|\zeta_{\hx}$. In this case, we must use an extender whose critical point is a preimage of $\zeta$ on $\T_{\hy, \hw}$ but not on $\T_{\hx, \hw}$. It then follows that $\hw|\zeta$ is a complete iterate of $\hx|\zeta_\hx$, and therefore, $\zeta<\nu_\hm$. 
\end{proof}

\section{$\Delta$-genericity iterations and derived model representations}

We introduce $\Delta$-genericity iterations and use them to establish that some Nairian Models are realizable as derived models.\footnote{The reader may find reviewing \rnot{not: restriction} and \rdef{def: ref gen} helpful. This section is inspired by the Derived Model representations used in \cite{SargMu}, \cite{ESSealing}, \cite{SaTrB} and \cite{GenericGenerators}.}

\begin{definition}\label{def: delta-gen it}\normalfont In $M$, suppose $\rg_1(\hq, \k, \nu)$ holds. Let $G\subseteq \Coll(\omega_1, \bR)$ be $M$-generic, let $\vec{B}=(B_i: i<\omega_1)$ be an enumeration of $\Delta_{\hq, \nu}$ in $M[G]$, and $\vec{a}=(a_i: i<\omega_1)\in M[G]$ be an enumeration of $\bR$. Then $(\hq_i, \hq_i', \T_i, E_i: i<\omega_1)\in M[G]$ is a $\Delta_{\hq, \nu}$-\textbf{genericity iteration relative to} $(\vec{B}, \vec{a})$ if the following conditions hold:
\begin{enumerate}
\item $\hq_0=\hq$.
\item For $i<\omega_1$, $\T_i$ is an iteration tree on $\hq_i$ that is strictly above $\k_{\hq_i}$ and is based on $\hq_i|\nu_{\hq_i}$.
\item For $i<\omega_1$, $\hq_i'$ is the last pair of $\T_i$.
\item For $i<\omega_1$, $E_i\in \vec{E}^{\M^{\hq_i'}}$ is such that $\cp(E_i)=\k_{\hq_i}$ and $\gen(\T_i)\subseteq \lh(E_i)$.
\item For $i<\omega_1$, $\hq_{i+1}=Ult(\hq_i, E_i)$. 
\item For limit $i<\omega_1$, $\hq_i$ is the direct limit of $(\hq_j, \pi_{\hq_j, \hq_k}: j<k<i)$.
\item For $i<\omega_1$, for some $\zeta<\k_{\hq_{i+1}}$, $B_i$ is projective in ${\sf{Code}}(\Sigma^{\hq_{i+1}|\zeta})$, and for some $g\subseteq \Coll(\omega, \zeta)$ generic over $\M^{\hq_{i+1}}$, $a_i\in \M^{\hq_{i+1}}[g]$.
\end{enumerate}
Let $\hq_{\omega_1}$ be the direct limit of $(\hq_i, \pi_{\hq_i, \hq_j}: i<j<\omega_1)$. 
\end{definition}
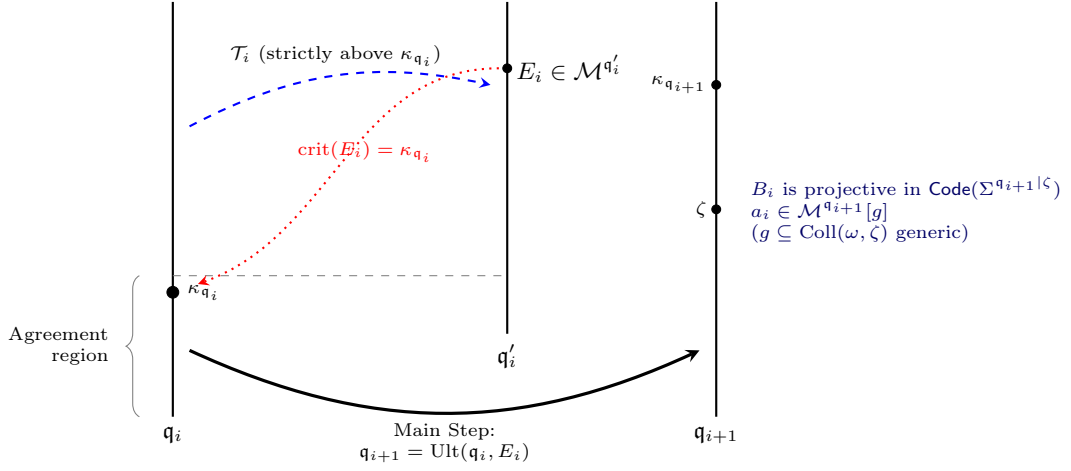
\begin{figure}[h]
\centering
\begin{tikzpicture}[scale=1.1, font=\small]

    \draw[thick] (0,0) node[below] {$\hq_i$} -- (0,5);
    \filldraw (0, 1.5) circle (2pt) node[right=2pt, font=\scriptsize] {$\k_{\hq_i}$};

    \draw[thick] (4,1.0) node[below] {$\hq_i'$} -- (4,5);
    
    \draw[->, thick, blue, >=stealth, dashed] (0.2, 3.5) to[bend left=20] 
        node[midway, above, align=center, font=\scriptsize, black] {$\T_i$ (strictly above $\k_{\hq_i}$)} 
        (3.8, 4.0);

    \draw[dashed, gray] (0, 1.7) -- (4, 1.7);
    \draw[decorate, decoration={brace, amplitude=5pt}, gray] (-0.4, 0) -- (-0.4, 1.7)
        node[midway, left=8pt, align=right, font=\scriptsize, black] {Agreement\\region};

    \filldraw (4, 4.2) circle (1.5pt) node[right] {$E_i \in \M^{\hq_i'}$};
    \draw[dotted, thick, red, ->, >=stealth] (3.9, 4.2) to[out=180, in=20] (0.3, 1.6);
    \node[red, font=\scriptsize] at (2.3, 3.2) {$\cp(E_i) = \k_{\hq_i}$};

    \draw[thick] (6.5,0) node[below] {$\hq_{i+1}$} -- (6.5,5);
    
    \filldraw (6.5, 4.0) circle (1.5pt) node[left, font=\scriptsize] {$\k_{\hq_{i+1}}$};
    \filldraw (6.5, 2.5) circle (1.5pt) node[left, font=\scriptsize] {$\zeta$};

    \draw[->, very thick, >=stealth] (0.2, 0.8) to[bend right=25] 
        node[midway, below, align=center, font=\scriptsize] {Main Step:\\$\hq_{i+1} = \operatorname{Ult}(\hq_i, E_i)$} 
        (6.3, 0.8);

    \node[right, align=left, font=\scriptsize, blue!40!black] at (6.8, 2.5) {
        $B_i$ is projective in $\mathsf{Code}(\Sigma^{\hq_{i+1}|\zeta})$\\
        $a_i \in \M^{\hq_{i+1}}[g]$\\
        ($g \subseteq \operatorname{Coll}(\omega, \zeta)$ generic)
    };

\end{tikzpicture}
\caption{The relationship between the models and trees in a single step of a $\Delta_{\hq, \nu}$-genericity iteration. Note that $B_i$ and $a_i$ are captured by $\hq_{i+1}$ below $\k_{\hq_{i+1}}$.}
\label{fig:genericity_step}
\end{figure}

We remark that the entire iteration \[(\hq_i, \hq_i', \T_i, E_i: i<\omega_1)\] is a normal iteration of $\hq$.
\begin{theorem}\label{thm: delta-gen it} In $M$, suppose $\rg_1(\hq, \k, \nu)$ holds. Let 
\begin{enumerate}
    \item $G\subseteq \Coll(\omega_1, \bR)$ be $M$-generic,

    \item $\vec{B}=(B_i: i<\omega_1)$ be an enumeration of $\Delta_{\hq, \nu}$ in $M[G]$,

    \item and $\vec{a}=(a_i: i<\omega_1)\in M[G]$ be an enumeration of $\bR$.
\end{enumerate}
Then there is $(\hq_i, \hq_i', \T_i, E_i: i<\omega_1)\in M[G]$ that is a $\Delta_{\hq, \nu}$-genericity iteration relative to $(\vec{B}, \vec{a})$.
\end{theorem}
\begin{proof} The construction below is visualized in \rfig{fig:recursive_step}. \rthm{thm: delta-gen it} is a corollary of \cite[Theorem 11.1]{blue2025nairian}. 

Suppose we have defined $(\hq_i, \hq_i', \T_i, E_i: i<\b)$, and we want to define $(\hq_\b, \hq_\b', \T_\b, E_\b)$. If $\b=\gg+1$ then $\hq_\b=\operatorname{Ult}(\hq_\gg, E_\gg)$. If $\b$ is a limit ordinal then $\hq_\b$ is the direct limit of \[(\hq_\gg, \pi_{\hq_\gg, \hq_{\gg'}}: \gg<\gg'<\b).\] 
We now describe the steps for obtaining $\hq_\b'$. \\\\
\textbf{The construction of $\hq_\b'$.}\\

\textbf{Step 1.} Let $\barn\in (\k_{\hq_\b}, \nu_{\hq_\b})$ be a Woodin cardinal of $\M^{\hq_\b}$. Let $\hs$ be a complete iterate of $\hq_\b$ such that 
\begin{itemize}[itemsep=0.3cm]
\item $\T_{\hq_\b, \hs}$ is based on $\hq_\b|\barn$, 
\item $\T_{\hq_\b, \hs}$ is strictly above $\k_{\hq_\b}$, and 
\item $a_\b$ is generic over $\M^{\hs}$ for a poset of size $\barn_\hs$.
\end{itemize}

\textbf{Step 2.} Next apply Clause 1 of \cite[Theorem 11.1]{blue2025nairian} to $B_\b$, $\hs|\nu_\hs$, and to the interval $(\barn_\hs, \nu_\hs)$\footnote{This means that the iteration tree is above $\barn_\hs$ and below $\nu_\hs$.}, and get a complete iterate $\hw$ of $\hs$ and an $\M^{\hw}$-inaccessible cardinal $\zeta$ such that
\begin{itemize}[itemsep=0.3cm]
\item $\T_{\hs, \hw}$ is based on $\hs|(\barn_\hs, \nu_\hs)$, 
\item $\gen(\T_{\hs, \hw})\subseteq \zeta$, and 
\item $B_\b$ is projective in ${\sf{Code}}(\Lambda_{\hw|\zeta})$. 
\end{itemize}
Set $\hq_\b'=\hw$.\\

Let now $E\in \vec{E}^{\M^{\hw}}$ be the extender with least index such that $\zeta<\lh(E)$, and set
\begin{itemize}[itemsep=0.3cm]
\item $\T_\b=\T_{\hq_\b, \hw}$, 
\item $E_\b=E$, 
\item $\cp(E)=\k_{\hq_\b}$, and
\item $\hq_{\b+1}=\operatorname{Ult}(\hq_\b, E)$.
\end{itemize}
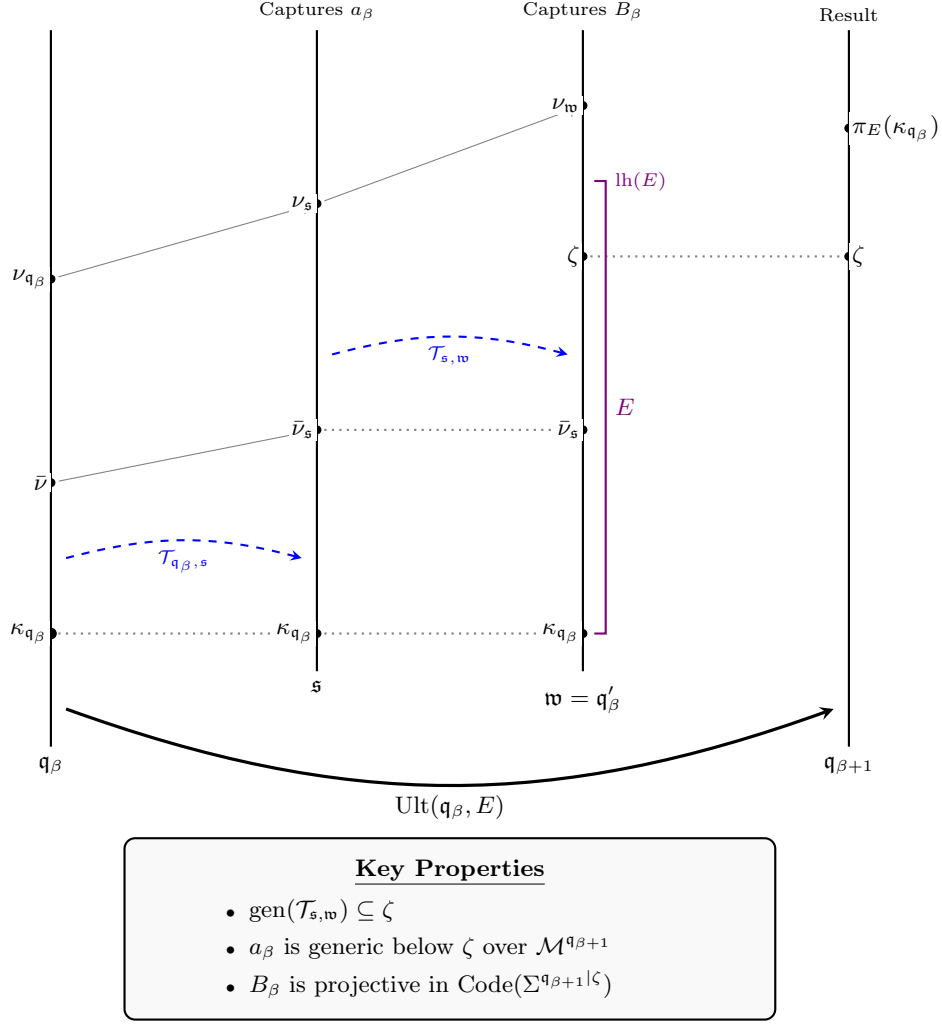
\begin{figure}[ht]
\centering
\begin{tikzpicture}[scale=1.0, font=\small]

    \def\xSep{3.5}   
    
    \def\xQ{0}
    \def\xS{\xSep}
    \def\xW{2*\xSep}
    \def\xNext{3*\xSep}

    \def\hBase{0}       
    \def\hStartSW{1.0}  
    \def\hTop{9.5}      
    \def\hKappa{1.5}    
    \def\hTreeOne{2.5}  
    \def\hBarn{3.5}     
    \def\hBarnS{4.2}    
    \def\hTreeTwo{5.2}  
    \def\hNuQ{6.2}      
    \def\hNuS{7.2}      
    \def\hNuW{8.5}      
    \def\hZeta{6.5}     
    \def\hLen{7.5}      
    \def\hKappaNext{8.2}     

    \draw[thick] (\xQ, \hBase) node[below] {$\hq_\beta$} -- (\xQ, \hTop);
    \draw[thick] (\xS, \hStartSW) node[below] {$\hs$} -- (\xS, \hTop);
    \draw[thick] (\xW, \hStartSW) node[below, yshift=-2pt] {$\hw = \hq_\beta'$} -- (\xW, \hTop);
    \draw[thick] (\xNext, \hBase) node[below] {$\hq_{\beta+1}$} -- (\xNext, \hTop);

    \node[above, font=\scriptsize, align=center] at (\xS, \hTop) {Captures $a_\beta$};
    \node[above, font=\scriptsize, align=center] at (\xW, \hTop) {Captures $B_\beta$};
    \node[above, font=\scriptsize, align=center] at (\xNext, \hTop) {Result};

    \draw[dotted, thick, gray] (\xQ, \hKappa) -- (\xS, \hKappa);
    \draw[dotted, thick, gray] (\xS, \hKappa) -- (\xW, \hKappa);
    \draw[dotted, thick, gray] (\xS, \hBarnS) -- (\xW, \hBarnS);
    \draw[dotted, thick, gray] (\xW, \hZeta) -- (\xNext, \hZeta);

    \draw[->, gray, thin, shorten <=3pt, shorten >=3pt] (\xQ, \hBarn) -- (\xS, \hBarnS);
    \draw[->, gray, thin, shorten <=3pt, shorten >=3pt] (\xQ, \hNuQ) -- (\xS, \hNuS);
    \draw[->, gray, thin, shorten <=3pt, shorten >=3pt] (\xS, \hNuS) -- (\xW, \hNuW);

    \draw[->, thick, blue, dashed, >=stealth] (\xQ+0.2, \hTreeOne) to[bend left=15] 
        node[midway, below, font=\scriptsize] {$\T_{\hq_\beta, \hs}$}
        (\xS-0.2, \hTreeOne);

    \draw[->, thick, blue, dashed, >=stealth] (\xS+0.2, \hTreeTwo) to[bend left=15] 
        node[midway, below, font=\scriptsize] {$\T_{\hs, \hw}$}
        (\xW-0.2, \hTreeTwo);

    \filldraw (\xQ, \hKappa) circle (2pt) node[left, fill=white, inner sep=1pt] {$\k_{\hq_\beta}$};
    \filldraw (\xQ, \hBarn) circle (1.5pt) node[left, fill=white, inner sep=1pt] {$\barn$};
    \filldraw (\xQ, \hNuQ) circle (1.5pt) node[left, fill=white, inner sep=1pt] {$\nu_{\hq_\beta}$};

    \filldraw (\xS, \hKappa) circle (1.5pt) node[left, fill=white, inner sep=1pt] {$\k_{\hq_\beta}$};
    \filldraw (\xS, \hBarnS) circle (1.5pt) node[left, fill=white, inner sep=1pt] {$\barn_\hs$};
    \filldraw (\xS, \hNuS) circle (1.5pt) node[left, fill=white, inner sep=1pt] {$\nu_\hs$};

    \filldraw (\xW, \hKappa) circle (1.5pt) node[left, fill=white, inner sep=1pt] {$\k_{\hq_\beta}$};
    \filldraw (\xW, \hBarnS) circle (1.5pt) node[left, fill=white, inner sep=1pt] {$\barn_\hs$};
    \filldraw (\xW, \hZeta) circle (1.5pt) node[left, fill=white, inner sep=1pt] {$\zeta$};
    \filldraw (\xW, \hNuW) circle (1.5pt) node[left, fill=white, inner sep=1pt] {$\nu_\hw$};

    \filldraw (\xNext, \hZeta) circle (1.5pt) node[right, fill=white, inner sep=1pt] {$\zeta$};
    \filldraw (\xNext, \hKappaNext) circle (1.5pt) node[right, fill=white, inner sep=1pt] {$\pi_E(\k_{\hq_\beta})$};

    \draw[thick, violet] (\xW+0.15, \hKappa) -- (\xW+0.3, \hKappa) -- (\xW+0.3, \hLen) -- (\xW+0.15, \hLen);
    \node[right, violet] at (\xW+0.3, {(\hKappa+\hLen)/2}) {$E$};
    \node[right, violet, font=\scriptsize] at (\xW+0.3, \hLen) {$\lh(E)$};

    \draw[->, very thick, >=stealth, preaction={draw=white, line width=4pt}] (\xQ+0.2, 0.5) to[bend right=20] 
        node[midway, below] {$\operatorname{Ult}(\hq_\beta, E)$}
        (\xNext-0.2, 0.5);

    \node[draw=black, thick, fill=gray!5, rounded corners, inner sep=8pt, anchor=north] 
        at ({(\xQ+\xNext)/2}, -1.2) {
        \begin{minipage}{8cm}
            \centering
            \textbf{\underline{Key Properties}}
            \begin{itemize}[label={\tiny$\bullet$}, itemsep=3pt, topsep=4pt, align=left]
                \item $\operatorname{gen}(\T_{\hs, \hw}) \subseteq \zeta$
                \item $a_\beta$ is generic below $\zeta$ over $\mathcal{M}^{\hq_{\beta+1}}$
                \item $B_\beta$ is projective in $\operatorname{Code}(\Sigma^{\hq_{\beta+1}|\zeta})$
            \end{itemize}
        \end{minipage}
    };

\end{tikzpicture}
\caption{The recursive step construction. The diagram illustrates the movement of cardinals and the agreement between models.}
\label{fig:recursive_step}
\end{figure}
\end{proof}

The next theorem blows up a local versions of $N|\k$ to $N|\k$ through a $\Delta_{\hq, \nu}$-genericity iteration. \rfig{fig:nairian_limit} shows how this happens. 
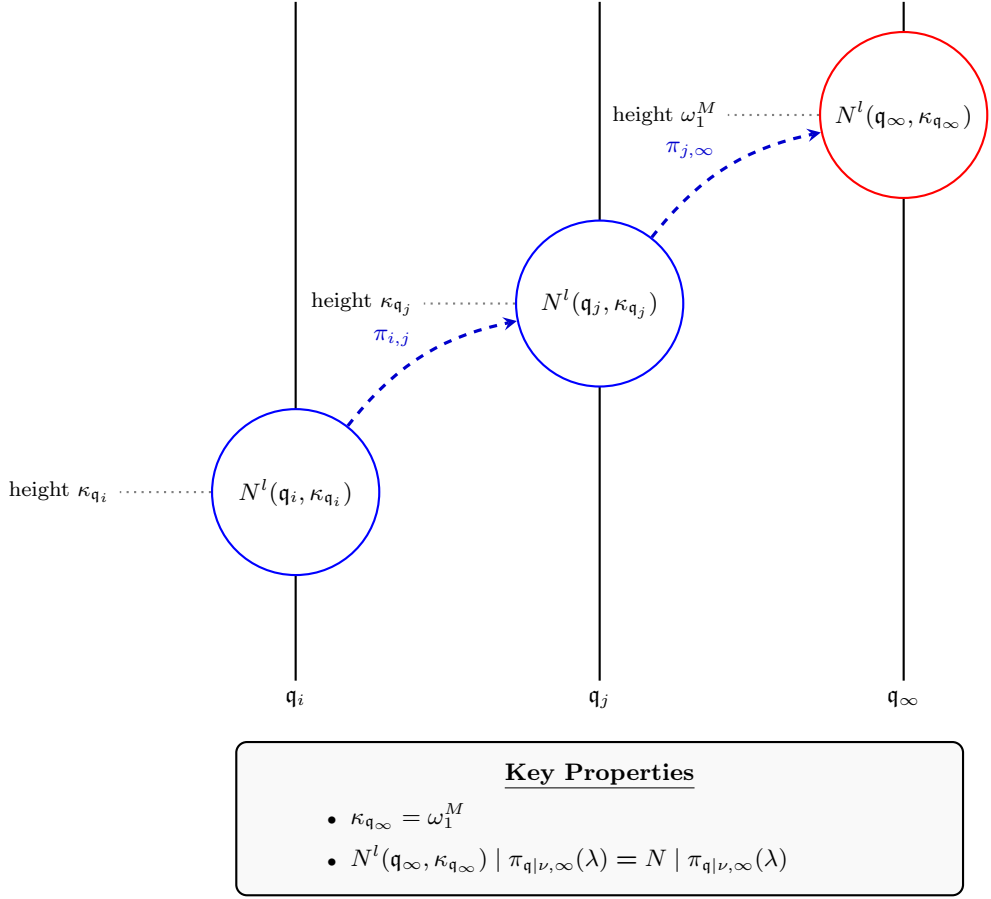
\begin{figure}[ht]
\centering
\begin{tikzpicture}[scale=1.0, font=\small]

    \def\xSep{4.0} 
    
    \def\xI{0}
    \def\xJ{\xSep}
    \def\xInf{2*\xSep}

    \def\yBase{0}
    \def\yTop{9.0}

    \def\hI{2.5}
    \def\hJ{5.0}
    \def\hInf{7.5}

    \draw[thick] (\xI, \yBase) node[below] {$\hq_i$} -- (\xI, \yTop);
    \draw[thick] (\xJ, \yBase) node[below] {$\hq_j$} -- (\xJ, \yTop);
    \draw[thick] (\xInf, \yBase) node[below] {$\hq_\infty$} -- (\xInf, \yTop);

    \node[circle, draw=blue, thick, fill=white, align=center, inner sep=1pt, minimum size=2.2cm] 
        (Ni) at (\xI, \hI) 
        {$N^l(\hq_i, \kappa_{\hq_i})$};

    \node[circle, draw=blue, thick, fill=white, align=center, inner sep=1pt, minimum size=2.2cm] 
        (Nj) at (\xJ, \hJ) 
        {$N^l(\hq_j, \kappa_{\hq_j})$};

    \node[circle, draw=red, thick, fill=white, align=center, inner sep=1pt, minimum size=2.2cm] 
        (Ninf) at (\xInf, \hInf) 
        {$N^l(\hq_\infty, \kappa_{\hq_\infty})$};

    \draw[->, very thick, blue!80!black, >=stealth, dashed] (Ni) to[bend left=20] 
        node[midway, above left] {$\pi_{i,j}$}
        (Nj);

    \draw[->, very thick, blue!80!black, >=stealth, dashed] (Nj) to[bend left=20] 
        node[midway, above left] {$\pi_{j,\infty}$}
        (Ninf);

    
    \draw[dotted, gray, thick] (Ni.west) -- ++(-1.2, 0) 
        node[left, black, font=\footnotesize] {height $\kappa_{\hq_i}$};

    \draw[dotted, gray, thick] (Nj.west) -- ++(-1.2, 0) 
        node[left, black, font=\footnotesize] {height $\kappa_{\hq_j}$};

    \draw[dotted, gray, thick] (Ninf.west) -- ++(-1.2, 0) 
        node[left, black, font=\footnotesize] {height $\omega_1^M$};

    \node[draw=black, thick, fill=gray!5, rounded corners, inner sep=8pt, anchor=north] 
        at ({(\xI+\xInf)/2}, -0.8) {
        \begin{minipage}{9cm}
            \centering
            \textbf{\underline{Key Properties}}
            \begin{itemize}[label={\tiny$\bullet$}, itemsep=5pt, topsep=4pt]
                \item $\k_{\hq_\infty} = \omega_1^M$
                \item $N^l(\hq_\infty, \k_{\hq_\infty}) \mid \pi_{\hq|\nu, \infty}(\l) \boldsymbol{=} N \mid \pi_{\hq|\nu, \infty}(\l)$
            \end{itemize}
        \end{minipage}
    };

\end{tikzpicture}
\caption{Visualization of Theorem \ref{thm: realizing as a nairian model}. The way the local N models get blown up to the full one.}
\label{fig:nairian_limit}
\end{figure}

\begin{theorem}\label{thm: realizing as a nairian model} 
Suppose $M\models \rg_1(\hq, \l, \k, \nu)$ holds. Let
\begin{itemize}
    \item $G\subseteq \Coll(\omega_1, \bR)$ be $M$-generic, and
    \item $(\hq_i, \hq_i', \T_i, E_i: i<\omega_1)\in M[G]$ be a $\Delta_{\hq, \nu}$-genericity iteration with respect to some $\vec{B}=(B_i: i<\omega_1)$ and $\vec{a}=(a_i: i<\omega_1)$.
\end{itemize} 
Let $\hq_{\infty}$ be the direct limit of $(\hq_i, \pi_{\hq_i, \hq_j}: i<j<\omega_1)$. Then 
\begin{enumerate}
    \item $\k_{\hq_\infty}=\omega_1^M$, 
    \item $\lmi(\hq_\infty, \k_{\hq_\infty})=\M_\infty(\hq|\nu)|\pi_{\hq|\nu, \infty}(\k)$,
    \item for every $i<\omega_1$, for every $\M^{\hq_i}$-generic $h\subseteq \Coll(\omega, <\k_{\hq_i})$, and for every $\hr\in \mathcal{F}_{\hq_i, \k_{\hq_i}, h}$, 
    \[ \pi_{\hr, \infty}\rest \k_\hr=\pi_{\hq_i, \hq_\infty}\circ\pi_{\hr, \hm(\hq_i, \k_{\hq_i})} \rest \k_\hr,\] and
    \item $N^l(\hq_\infty, \k_{\hq_\infty})|\pi_{\hq|\nu, \infty}(\l)=N|\pi_{\hq|\nu, \infty}(\l)$.
\end{enumerate}
\end{theorem}

\begin{proof} 
The first clause is a standard consequence of the fact that for each $i<\omega_1$, $\cp(E_i)=\k_{\hq_i}$. 

We introduce the following notation to be used in the rest of the proof. Working in $M[G]$, fix a $\M^{\hq_\infty}$-generic $g\subseteq \Coll(\omega, <\omega_1^M)$ such that $\bR^{\M^{\hq_\infty}[g]}=\bR$, and let
\begin{notation}\label{not: notation for the proof}\normalfont
\begin{enumerate}
    \item[]
    \item[(1.1)] $\tau=\omega_1^M=\k_{\hq_\infty}$, 
    \item[(1.2)] $\hq_\tau=\hq_\infty$,
    \item[(1.3)] for $i\leq \tau$, $\k_i=\k_{\hq_i}$,
    \item[(1.4)] for $i\leq \tau$, $\hq_i=(\Q_i, \Lambda_i)$,
    \item[(1.5)] for $i<m\leq \tau$, $\pi_{i, m}=\pi_{\hq_i, \hq_m}$,
    \item[(1.6)] for $i<m\leq \tau$, $\T_{i, m}=\T_{\hq_i, \hq_m}$,
    \item[(1.7)] for $i\leq \tau$, $\hm_i=\hm(\hq_i, \k_i)$ (see \rnot{not: more not for reflection}),
    \item[(1.8)] for $i\leq \tau$, $\M_i=\M^{\hm_i}=\lmi(\hq_i, \k_i)$,
    \item[(1.9)] for $i\leq \tau$, $g_i=g\cap \Coll(\omega, <\k_i)$, and
    \item[(1.10)] for $i\leq \tau$ such that $g_i$ is $\Q_i$-generic, $\mathcal{F}_i=\mathcal{F}(\Q_i, \k_i, g_i)$. 
\end{enumerate}
\end{notation}

We first establish the following useful lemma. Notice that, in $M[G]$,\\\\
(1) for a club of $i<\tau$, $i= \k_i$ and $g_i=g\cap \Coll(\omega, <\k_i)$ is generic over $\Q_i$.\\\\
Let $C$ be the above club.

\begin{lemma}\label{lem: lifting iterations} 
Suppose $i\in C$, $k\in C$ or $k=\tau$, and $i<k$. Then $\pi_{i, k}$ can be lifted to 
\[ \pi_{i, k}^+: \Q_i[g_i]\rightarrow \Q_k[g_k], \]
and if $\hr\in \mathcal{F}_i$ then $\pi^+_{i, k}(\hr)\in \mathcal{F}_k$ and $\hw=_{def}\pi^+_{i, k}(\hr)$ is an iterate of $\hr$ such that $\T_{\hr, \hw}$ is above $\k_\hr$. 
\end{lemma}

\begin{proof} 
The reader can consult \rfig{fig:lifting_lemma} and \rfig{fig:lifting_lemma_box_spaced}. Let $\zeta<\k_i$ be such that
\vspace{0.3cm}
\begin{enumerate}[label=(\alph*), itemsep=0.3cm]
    \item[(a1)] $\T_{\hq_i, \hr}$ is based on $\hq|\zeta$,
\item[(a2)] $\zeta$ is a proper cutpoint (see \rter{term: cutpoints etc}), and
\item[(a3)] $\zeta$ is an inaccessible cardinal of $\Q_i$.
\end{enumerate}
\vspace{0.3cm}

Let $\U=\downarrow(\T_{\hq_i, \hr},\hq|\zeta)$ (see \rnot{not: restriction}). We have that $\pi_{i, k}^+(\U)=\U$. Let $\X=\uparrow(\U,\hq_k)$.

\begin{figure}[h]
    \centering
    \begin{tikzpicture}[>=Stealth, node distance=3.5cm]
        \node (Qi) {$\hq_i$};
        \node (Qk) [right of=Qi] {$\hq_k$};
        \node (R) [above of=Qi] {$\hr$};
        \node (W) [above of=Qk] {$\hw$};

        \draw[arrow] (Qi) -- node[below] {$\pi_{i,k}$} (Qk);

        \draw[arrow] (Qi) -- node[left] {$\pi_{\hq_i, \hr}$} (R);

        \draw[arrow, dashed] (R) -- node[above] {$\pi_{i,k}^+$} (W);

        \draw[arrow, dashed] (Qk) -- node[right] {$\pi_{\hq_k, \hw}=\pi^+_{i, k}(\pi_{\hq_i, \hr})$} (W);
    \end{tikzpicture}
    \caption{$\hr$-to-$\hw$ diagram.}
    \label{fig:lifting_lemma}
\end{figure}

\begin{figure}[h]
    \centering
    \begin{tikzpicture}[
        scale=1.1,
        >=Stealth,
        x={(1cm,0cm)}, 
        y={(0cm,1cm)}, 
        z={(0.6cm,-0.35cm)} 
    ]

    \def\h{9.5}      
    \def\w{7.0}      
    \def\d{5}        
    
    \def\base{0}        
    \def\treeStart{1.5} 
    \def\treeEnd{3.5}   
    \def\zetaLow{3.0}   
    \def\zetaHigh{5.0}  
    
    \def\kiH{6.5}       
    \def\kkH{8.0}       
    
    \def\topTreeStart{7.5} 
    \def\topTreeEnd{9.0}    

    \coordinate (Qi_B) at (0,\base,0);      \coordinate (Qi_T) at (0,\h,0);
    \coordinate (Qk_B) at (\w,\base,0);     \coordinate (Qk_T) at (\w,\h,0);

    \coordinate (R_B) at (0,\base,\d);      \coordinate (R_T) at (0,\h,\d);
    \coordinate (W_B) at (\w,\base,\d);     \coordinate (W_T) at (\w,\h,\d);

    \draw[thick] (Qi_B) -- (Qi_T);
    \draw[thick] (Qk_B) -- (Qk_T);
    \draw[thick] (R_B) -- (R_T);
    \draw[thick] (W_B) -- (W_T);

    \node[below=8pt] at (Qi_B) {$\hq_i$};
    \node[below=8pt] at (Qk_B) {$\hq_k$};
    \node[below=8pt] at (R_B)  {$\hr$};
    \node[below=8pt] at (W_B)  {$\hw$};

    \draw[->, blue, thick] (0, \treeStart, 0) -- node[above, sloped, font=\small] {$\T_{\hq_i, \hr}$} (0, \treeEnd, \d);
    \draw[->, blue, thick, dashed] (\w, \treeStart, 0) -- node[above, sloped, font=\small] {$\T_{\hq_k, \hw}$} (\w, \treeEnd, \d);

    \draw[dotted, thick, gray] (0,\zetaLow,0) -- (\w,\zetaLow,0);
    \draw[dotted, thick, gray] (0,\zetaHigh,\d) -- (\w,\zetaHigh,\d);

    \filldraw (0,\zetaLow,0) circle (1.5pt) node[left, font=\small] {$\zeta$};
    \filldraw (\w,\zetaLow,0) circle (1.5pt) node[right, font=\small] {$\zeta_{\hq_k}=\zeta$};
    
    \filldraw (0,\zetaHigh,\d) circle (1.5pt) node[left, font=\small] {$\zeta_\hr$};
    \filldraw (\w,\zetaHigh,\d) circle (1.5pt) node[right, font=\small] {$\zeta_\hw=\zeta_\hr$};

    \draw[->, blue!40, thin, shorten >=3pt, shorten <=3pt] (0,\zetaLow,0) -- (0,\zetaHigh,\d);
    \draw[->, blue!40, thin, dashed, shorten >=3pt, shorten <=3pt] (\w,\zetaLow,0) -- (\w,\zetaHigh,\d);

    \filldraw (0,\kiH,0) circle (1.5pt) node[left] {$\kappa_i$};
    \filldraw (\w,\kkH,0) circle (1.5pt) node[right] {$\kappa_k$};
    \filldraw (0,\kiH,\d) circle (1.5pt) node[left] {$\kappa_\hr$};
    \filldraw (\w,\kkH,\d) circle (1.5pt) node[right] {$\kappa_\hw$};

    \draw[->, gray, thin] (0,\kiH,0) -- (\w,\kkH,0);
    \draw[->, gray, dashed, thin] (0,\kiH,\d) -- (\w,\kkH,\d);
    
    \draw[->, blue, shorten >=3pt] (0,\kiH,0) -- (0,\kiH,\d);
    \draw[->, blue, dashed, shorten >=3pt] (\w,\kkH,0) -- (\w,\kkH,\d);

    \draw[->, violet, thick, decorate, decoration={snake, amplitude=.4mm, segment length=2mm, post length=1mm}] 
        (0, \topTreeStart, \d) -- 
        node[above, sloped, font=\small, fill=white, inner sep=1pt] {$\T_{\hr, \hw}=\mathcal{Y}_{\geq \hr}$} 
        (\w, \topTreeEnd, \d);

    \draw[->, dotted, thick] (Qi_B) -- node[above, font=\small, fill=white, inner sep=1pt] {$\pi_{i,k}$} (Qk_B);
    
    \draw[->] (R_B) -- node[above, font=\small, fill=white, inner sep=1pt] {$\pi^+_{i,k}$} (W_B);

    \node[draw=black, thick, fill=gray!10, rounded corners, inner sep=8pt, anchor=north] 
        at (3.5, -2.5, 2.5) { 
        \begin{minipage}{10cm}
            \centering
            \textbf{\underline{Key Properties}}
            \begin{itemize}[label={\tiny$\bullet$}, itemsep=3pt, topsep=4pt, align=left]
                \item $\mathcal{Y}$ is the full normalization of $(\T_{\hq_i, \hq_k})^\frown \T_{\hq_k, \hw}$
                \item $\Y_{\geq \hr}$ is above $\k_{\hr}$.
                \item $\Y_{\geq \hr}$ is the $\pi_{\hq_i, \hr}$-minimal copy of $\T_{\hq_i, \hq_k}$.
            \end{itemize}
        \end{minipage}
    };

    \draw[gray!20, dotted] (Qi_T) -- (R_T) -- (W_T) -- (Qk_T) -- cycle;

    \end{tikzpicture}
    \caption{The $\hr$-to-$\hw$ iteration.}
    \label{fig:lifting_lemma_box_spaced}
\end{figure}

Notice now that if $\Y$ is the full normalization of $(\T_{i, k})^\frown \X$ then $\T_{\hq_i, \hr}\insegeq \Y$ and $\Y_{\geq \hr}$ is above $\k_{\hr}$ (in fact $\Y_{\geq \hr}$ is the minimal copy of $\T_{i, k}$ via $\pi_{\hq_i, \hr}$, see \cite[Definition 8.7]{blue2025nairian}). This finishes the proof of the lemma.
\end{proof}

We remind the reader that, following \cite{blue2025nairian}, if $\hw$ is a hod pair then $\hw=(\M^\hw, \Sigma^\hw)$, where $\M^\hw$ is a hod mouse and $\Sigma^\hw$ is an iteration stratgey with strong hull condensation.\footnote{See \cite[Definition 0.1]{MPSC} and  \cite[Theorem 1.4]{MPSC}.}

\begin{corollary}\label{cor: emb} 
Suppose $i\in C$, $k\in C$ or $k=\tau$, $i<k$, and $\hr\in \mathcal{F}_i$. Then letting $\hw=\pi^+_{i, k}(\hr)$\footnote{Notice that $\M^\hr|\k_\hr=\M^\hw|\k_\hw$.}, 
\[ \pi_{\hw, \hm_k}\circ \pi_{i, k}\rest (\M^\hr|\k_\hr)=\pi_{i, k}\circ \pi_{\hr, \hm_i}\rest (\M^\hr|\k_\hr). \]
\end{corollary}
\begin{proof} 
We have that $\pi_{i, k}^+(\pi_{\hr, \hm_i})=\pi_{\hw, \hm_k}$. Because $\pi_{i, k}^+\rest \k_\hr=id$, the claim follows.
\end{proof}

We now define an embedding 
\[ j: \M_\tau|\pi_{\hq_\tau, \hm_\tau}(\l)\rightarrow \M_\infty(\hq|\nu)|\pi_{\hq|\nu, \infty}(\l) \]
as follows. Fix $x\in \M_\tau|\pi_{\hq_\tau, \hm_\tau}(\l)$. We say $(\hr, \zeta, i)$ is \textit{good for} $x$ if the tuple $(\hr, \zeta, i)$ has the following properties: 
\vspace{0.3cm}
\begin{enumerate}[label=(\alph*), itemsep=0.3cm]
    \item[(p1)] (p1) $\hr\in \mathcal{F}_\tau$ and $x\in \rge(\pi_{\hr, \hm_\tau})$.
\item[(p2)] $\zeta<\tau$ and $\T_{\hq_\tau, \hr}$ is based on $\hq_\tau|\zeta$. 
\item[(p3)]  $\pi_{\hq_\tau, \hr}(\zeta)=\zeta$. 
\item[(p4)]  $\zeta$ is an inaccessible cardinal of $\Q_\tau$. 
\item[(p5)]  If $y\in \M^{\hr}$ is such that $\pi_{\hr, \hm_\tau}(y)=x$ then $y\in \M^{\hr}|\zeta$.
\item[(p6)] $i\in C$, $\zeta<\k_i$, and letting $\U=\downarrow(\T_{\hq_\tau, \hr},\hq_\tau|\zeta)$, $\U\in \Q_i[g_i\cap \Coll(\omega, <\zeta)]$.
\end{enumerate}
\vspace{0.3cm}

Fix now $(\hr, \zeta, i)$ that is good for $x$. Let  $\U^+=\uparrow (\U, \hq_i)$ (this makes sense because $\hq_i|\zeta=\hq_\tau|\zeta$). Let $\hr'$ be the last model of $\U^+$ (thus, $\hr'\in \mathcal{F}_i$), and set 
\[ j_{\hr, \zeta, i}(x)=\pi_{\hr'|\nu_{\hr'}, \infty}(y). \]

Below we will show that $j_{\hr, \zeta, i}(x)$ is independent of $(\hr, \zeta, i)$. To prove this, we will use the following observation, which is simply a restatement of \rcor{cor: emb}.
\vspace{0.3cm}
\begin{enumerate}[label=(2.\arabic*), itemsep=0.3cm]
    \item Suppose $k\in C$ or $k=\tau$, and $\hw\in \mathcal{F}_k$. Suppose $i\in k\cap C$ is such that for some $\Q_i$-inaccessible cardinal $\zeta<\k_i$, $\T_{\hq_k, \hw}$ is based on $\hq_k|\zeta$ and $\T_{\hq_k, \hw}\in \hq_k|\zeta[g_k\cap \Coll(\omega, <\zeta)]$. Let $\U=\downarrow(\T_{\hq_k, \hw}, \hq_k|\zeta)$ and $\hr$ be the last model of $\uparrow(\U, \hq_i|\zeta)$. Then
\item $\pi^+_{i, k}(\hr)=\hw$, 
\item  $\hw$ is an iterate of $\hr$,
\item  $\T_{\hr, \hw}$ is above $\k_\hr$, and
\item  $\T_{\hr, \hw}$ is based on $\hr|\nu_\hr$.
\end{enumerate}
\vspace{0.3cm}

Continuing with the notation of (2.1), we say $\hr$ is the $i$-preimage of $\hw$.

\begin{lemma}\label{lem: j is wd} 
Suppose $x\in \M_\tau|\pi_{\hq_\tau, \hm_\tau}(\l)$, and $(\hr, \zeta, i)$ and $(\hw, \a, k)$ are two tuples that are good for $x$. Then \[j_{\hr, \zeta, i}(x)=j_{\hw, \a, k}(x).\]
\end{lemma}
\begin{proof}  \rfig{fig:j_well_defined_combined} depicts the main points of the proof.
We have that
\vspace{0.3cm}
\begin{enumerate}[label=(3.\arabic*), itemsep=0.3cm]
\item $\hr\in \mathcal{F}_\tau$ and $\hw\in \mathcal{F}_\tau$.
\item $x\in \rge(\pi_{\hr, \hm_\tau})$ and $x\in \rge(\pi_{\hw, \hm_\tau})$.
\item $\zeta<\tau$ and $\a<\tau$.
\item $\T_{\hq_\tau, \hr}$ is based on $\hq_\tau|\zeta$, and $\T_{\hq_\tau, \hw}$ is based on $\hq_\tau|\a$.
\item $\pi_{\hq_\tau, \hr}(\zeta)=\zeta$ and $\pi_{\hq_\tau, \hw}(\a)=\a$.
\item $\zeta$ is an inaccessible cardinal of $\Q_\tau$, and $\a$ is an inaccessible cardinal of $\Q_\tau$.
\item If $y\in \M^{\hr}$ is such that $\pi_{\hr, \hm_\tau}(y)=x$ then $y\in \M^{\hr}|\zeta$, and if $z\in \M^{\hw}$ is such that $\pi_{\hw, \hm_\tau}(z)=x$ then $z\in \M^{\hw}|\a$.
\item $i\in C$, $\zeta<\k_i$, and letting $\U=\downarrow(\T_{\hq_\tau, \hr},\hq_\tau|\zeta)$, $\U\in \Q_i[g_i\cap \Coll(\omega, <\zeta)]$.
\item $k\in C$, $\a<\k_k$, and letting $\V=\downarrow(\T_{\hq_\tau, \hw},\hq_\tau|\a)$, $\V\in \Q_k[g_k\cap \Coll(\omega, <\a)]$.
\end{enumerate}
\vspace{0.3cm}

Let  $\hr'=(\pi_{i, \tau}^+)^{-1}(\hr)$ and $\hw'=(\pi_{k, \tau}^+)^{-1}(\hw)$. Notice that $\hr'$ is the $i$-preimage of $\hr$ and $\hw'$ is the $k$-preimage of $\hw$.

Applying (2.1) to $\hr$ and $\hw$ we get that
\vspace{0.3cm}
\begin{enumerate}[label=(4.\arabic*), itemsep=0.3cm]
\item $\hr$ is a complete iterate of $\hr'$,
\item $\T_{\hr', \hr}$ is above $\k_{\hr'}$,
\item $\T_{\hr', \hr}$ is based on $\hr'|\nu_{\hr'}$,
\item $\hw$ is a complete iterate of $\hw'$,
\item $\T_{\hw', \hw}$ is above $\k_{\hw'}$, and
\item $\T_{\hw', \hw}$ is based on $\hw'|\nu_{\hw'}$.
\end{enumerate}
\vspace{0.3cm}

Let $\b<\tau$ be such that $\T_{\hq_\tau, \hr}$ and $\T_{\hq_\tau, \hw}$ are based on $\hq_\tau|\b$, $\b$ is inaccessible in $\Q_\tau$ and $\pi_{\hq_\tau, \hr}(\b)=\pi_{\hq_\tau, \hw}(\b)=\b$. It follows that $\hr$ and $\hw$ have a common iterate $\hn\in \mathcal{F}_\tau$ such that $\T_{\hr, \hn}$ is based on $\hr|\b$ and $\T_{\hw, \hn}$ is based on $\hw|\b$ (this is because $\gen(\T_{\hq_\tau, \hr})\subseteq \b$ and $\gen(\T_{\hq_\tau, \hw})\subseteq \b$). It now follows that
\vspace{0.3cm}
\begin{enumerate}[label=(5.\arabic*), itemsep=0.3cm]
\item $\pi_{\hr, \hn}(y)=\pi_{\hw, \hn}(z)$.
\end{enumerate}
\vspace{0.3cm}
We can now find some $m\in [\max(i, k), \tau)\cap C$ such that $\b<\k_m$, and let \[\hr''=\pi_{i, m}^+(\hr') \ \text{and}\ \hw''=\pi_{k, m}^+(\hw').\] 
\begin{figure}[h]
\centering
\begin{tikzpicture}[scale=0.85, transform shape, >=stealth, font=\small]

    
    \node (N) at (0, 5.5) {$\hn$};

    \node[gray] (Rprime) at (-6.0, 3) {$\hr'$};
    \node[gray] (Rdoubleprime) at (-3.5, 3) {$\hr''$};
    \node (R) at (-1.5, 3) {$\hr$};
    \node[red, font=\scriptsize] at (-1.1, 3.0) {$y$}; 

    \node (W) at (1.5, 3) {$\hw$};
    \node[red, font=\scriptsize] at (1.1, 3.0) {$z$}; 
    \node[gray] (Wdoubleprime) at (3.5, 3) {$\hw''$};
    \node[gray] (Wprime) at (6.0, 3) {$\hw'$};

    \node (Qtau) at (0, 0) {$\hq_\tau$};
    \node (Qm_L) at (-3.5, 0) {$\hq_m$};
    \node (Qi) at (-6.0, 0) {$\hq_i$};
    \node (Qm_R) at (3.5, 0) {$\hq_m$};
    \node (Qk) at (6.0, 0) {$\hq_k$};

    \draw[->] (Qi) -- node[below, font=\scriptsize] {$\T_{i, m}$} (Qm_L);
    \draw[->] (Qm_L) -- node[below, font=\scriptsize] {$\T_{m, \tau}$} (Qtau);
    \draw[->] (Qk) -- node[below, font=\scriptsize] {$\T_{k, m}$} (Qm_R);
    \draw[->] (Qm_R) -- node[below, font=\scriptsize] {$\T_{m, \tau}$} (Qtau);

    \draw[->, gray] (Qi) -- (Rprime);
    \draw[->, gray] (Qm_L) -- (Rdoubleprime);
    \draw[->, gray] (Qk) -- (Wprime);
    \draw[->, gray] (Qm_R) -- (Wdoubleprime);
    \draw[->] (Qtau) -- (R);
    \draw[->] (Qtau) -- (W);

    \draw[->, gray] (Rprime) -- node[above, font=\scriptsize] {$\T_{\hr', \hr''}$} 
                                node[below, font=\tiny, yshift=-1pt] {above $\kappa_{\hr'}$} (Rdoubleprime);
    \draw[->, gray] (Wprime) -- node[above, font=\scriptsize] {$\T_{\hw', \hw''}$} 
                                node[below, font=\tiny, yshift=-1pt] {above $\kappa_{\hw'}$} (Wdoubleprime);
    
    \draw[->, gray, dotted] (Rdoubleprime) -- node[above, font=\scriptsize] {$\T_{\hr'', \hr}$} 
                                              node[below, font=\tiny, yshift=-1pt] {above $\kappa_{\hr''}$} (R);
    \draw[->, gray, dotted] (Wdoubleprime) -- node[above, font=\scriptsize] {$\T_{\hw'', \hw}$} 
                                              node[below, font=\tiny, yshift=-1pt] {above $\kappa_{\hw''}$} (W);

    \draw[->, dashed] (R) -- node[left, font=\scriptsize] {$\T_{\hr, \hn}$} (N);
    \draw[->, dashed] (W) -- node[right, font=\scriptsize] {$\T_{\hw, \hn}$} (N);
    
    \node[red, font=\scriptsize] at (0, 5.8) {$\pi_{\hr, \hn}(y)=\pi_{\hw, \hn}(z)$};

    \begin{scope}[yshift=-6.5cm]
    
        \node (Qm_Bottom) at (-3.5, 0) {$\hq_m$};
        \node (Rpp) at (-5.5, 2.5) {$\hr''$};
        \node (Wpp) at (-1.5, 2.5) {$\hw''$};
        \node (Np) at (-3.5, 5) {$\hn'$}; 

        \node (Qtau_Bottom) at (3.5, 0) {$\hq_\tau$};
        \node (R_Bottom) at (1.5, 2.5) {$\hr$};
        \node (W_Bottom) at (5.5, 2.5) {$\hw$};
        \node (N_Bottom) at (3.5, 5) {$\hn$};

        \draw[->] (Qm_Bottom) -- node[pos=0.3, left, font=\scriptsize] {$\T_{\hq_m, \hr''}$} (Rpp);
        \draw[->] (Qm_Bottom) -- node[pos=0.3, right, font=\scriptsize] {$\T_{\hq_m, \hw''}$} (Wpp);
        \draw[->, dashed] (Rpp) -- node[left, font=\scriptsize] {$\T_{\hr'', \hn'}$} (Np);
        \draw[->, dashed] (Wpp) -- node[right, font=\scriptsize] {$\T_{\hw'', \hn'}$} (Np);

        \draw[->] (Qtau_Bottom) -- node[pos=0.3, left, font=\scriptsize] {$\T_{\hq_\tau, \hr}$} (R_Bottom);
        \draw[->] (Qtau_Bottom) -- node[pos=0.3, right, font=\scriptsize] {$\T_{\hq_\tau, \hw}$} (W_Bottom);
        \draw[->, dashed] (R_Bottom) -- node[left, font=\scriptsize] {$\T_{\hr, \hn}$} (N_Bottom);
        \draw[->, dashed] (W_Bottom) -- node[right, font=\scriptsize] {$\T_{\hw, \hn}$} (N_Bottom);

        \draw[->] (Qm_Bottom) -- node[below, font=\scriptsize] {$\T_{m,\tau}$} (Qtau_Bottom);
        \draw[->, gray, dotted] (Rpp) to[bend right=35] 
            node[pos=0.15, above, font=\scriptsize] {$\T_{\hr'', \hr}$} (R_Bottom);
        \draw[->, gray, dotted] (Wpp) to[bend right=35] 
            node[pos=0.85, above, font=\scriptsize] {$\T_{\hw'', \hw}$} (W_Bottom);
        \draw[->] (Np) -- node[above, font=\scriptsize] {$\T_{\hn', \hn}$} 
                          node[below, font=\tiny, yshift=-1pt] {above $\kappa_{\hn'}$} (N_Bottom);

        
        \node[draw=black, thick, fill=gray!10, rounded corners, inner sep=6pt, anchor=north] 
            at (-3.6, -0.8) {
            \begin{minipage}{6cm}
                \centering
                \textbf{\underline{Key Facts: The $\hr$-side}}
                \begin{itemize}[label={\tiny$\bullet$}, itemsep=1pt, topsep=3pt, leftmargin=1em, font=\scriptsize]
                    \item $j_{r, \zeta, i}(x) = \pi_{\hr'|\nu_{\hr'}, \infty}(y)$
                    \item Since $\T_{\hr', \hr''}$ is above $\kappa_{\hr'}$: \\ 
                          $j_{r, \zeta, i}(x) = \pi_{\hr''|\nu_{\hr''}, \infty}(y)$
                    \item $\pi_{\hr''|\nu_{\hr''}, \infty}(y) = \pi_{\hn'|\nu_{\hn'}, \infty}(\pi_{\hr'', \hn'}(y))$
                    \item Thus, \[j_{r, \zeta, i}(x) =\pi_{\hn'|\nu_{\hn'}, \infty}(\pi_{\hr'', \hn'}(y))\]
                \end{itemize}
            \end{minipage}
        };

        \node[draw=black, thick, fill=gray!10, rounded corners, inner sep=6pt, anchor=north] 
            at (3.6, -0.8) {
            \begin{minipage}{6cm}
                \centering
                \textbf{\underline{Key Facts: The $\hw$-side}}
                \begin{itemize}[label={\tiny$\bullet$}, itemsep=1pt, topsep=3pt, leftmargin=1em, font=\scriptsize]
                    \item $j_{w, \alpha, k}(x) = \pi_{\hw'|\nu_{\hw'}, \infty}(z)$
                    \item Since $\T_{\hw', \hw''}$ is above $\kappa_{\hw'}$: \\ 
                          $j_{w, \alpha, k}(x) = \pi_{\hw''|\nu_{\hw''}, \infty}(z)$
                    \item $\pi_{\hw''|\nu_{\hw''}, \infty}(z) = \pi_{\hn'|\nu_{\hn'}, \infty}(\pi_{\hw'', \hn'}(z))$
                    \item Thus, \[j_{w, \alpha, k}(x)=\pi_{\hn'|\nu_{\hn'}, \infty}(\pi_{\hw'', \hn'}(z))\]
                \end{itemize}
            \end{minipage}
        };

        \node[draw=black, thick, fill=gray!10, rounded corners, inner sep=6pt, anchor=north] 
            at (0, -5.3) {
            \begin{minipage}{10.5cm}
                \centering
                \textbf{\underline{Conclusion: $j_{r, \zeta, i}(x) = j_{w, \alpha, k}(x)$}}
                \begin{align*}
                j_{\hr, \zeta, i}(x) &=  \pi_{\hn'|\nu_{\hn'}, \infty}(\pi_{\hr'', \hn'}(y)) \\
                &= \pi_{\hn|\nu_{\hn}, \infty}(\pi_{\hr'', \hn}(y))\ \text{(since $\T_{\hn', \hn}$ is above $\kappa_{\hn'}$)}\\
                &= \pi_{\hn|\nu_{\hn}, \infty}(\pi_{\hw'', \hn}(z))\ \text{(since $\pi_{\hr'', \hn}(y)=\pi_{\hw'', \hn}(z)$)}\\
                &=\pi_{\hn'|\nu_{\hn'}, \infty}(\pi_{\hw'', \hn'}(z))\  \text{(since $\T_{\hn', \hn}$ is above $\kappa_{\hn'}$)}\\
                &= j_{w, \alpha, k}(x)
                \end{align*}
            \end{minipage}
        };

    \end{scope}

\end{tikzpicture}
\caption{The diagrams show the reason behind the equality $j_{\hr, \zeta, i}(x)=j_{\hw, \alpha, k}(x)$}
\label{fig:j_well_defined_combined}
\end{figure}

We easily have that $\hn\in \rge(\pi^+_{m, \tau})$, and it also follows from (2) that $\hr''$ is the $m$-preimage of $\hr$ and $\hw''$ is the $m$-preimage of $\hw$.  Thus,
\vspace{0.3cm}
\begin{enumerate}[label=(6.\arabic*), itemsep=0.3cm]
\item $\hr''$ is a complete iterate of $\hr'$,
\item $\hw''$ is a complete iterate of $\hw$,
\item $\T_{\hr', \hr''}$ is above $\k_{\hr'}$,
\item $\T_{\hw', \hw''}$ is above $\k_{\hw'}$, 
\item  letting $\hn'=(\pi_{m, \tau}^+)^{-1}(\hn)$, $\hn'$ is a common iterate of both $\hr''$ and $\hw''$,
\item $\T_{\hr'', \hn'}$ is based on $\hr''|\b$ and $\T_{\hw'',\hn'}$ is based on $\hw''|\b$,
\item $\hn$ is a complete iterate of $\hn'$, and
\item $\T_{\hn', \hn}$ is above $\k_{\hn'}$.
\end{enumerate}
\vspace{0.3cm}
 It follows from (5) and (6.1)-(6.8) (especially (6.8)) that \[\pi_{\hr'', \hn'}(y)=\pi_{\hw'', \hn'}(z).\] It then follows from (4.3) and (4.6) that \[\pi_{\hr'|\nu_{\hr'}, \infty}(y)=\pi_{\hw'|\nu_{\hw'}, \infty}(z).\]
\end{proof}

Given $x\in \M_\tau|\pi_{\hq_\tau, \hm_\tau}(\l)$, we now set $j(x)=j_{\hr, \zeta, i}(x)$ where $(\hr, \zeta, i)$ is good for $x$. The next lemma finishes the proof of clause 2 of \rthm{thm: realizing as a nairian model}.

\begin{lemma}\label{lem: j is onto} $j$ is onto.
\end{lemma}
\begin{proof} 
Suppose \[y\in \M_\infty(\hq|\nu)|\pi_{\hq|\nu, \infty}(\l).\] We want to show that there is $x\in \M_\tau$ such that $j(x)=y$. Fix a complete iterate $\hr$ of $\hq$ such that $y\in \rge(\pi_{\hr|\nu_\hr, \infty})$, $\T_{\hq, \hr}$ is based on $\hq|\nu$ and $\lh(\T_{\hq, \hr})<\tau(=\omega_1^M)$. Because $\l$ is small, we can without loss of generality assume that $\gen(\T_{\hq, \hr})\subseteq \l_\hr$ (see \cite[Lemma 9.14]{blue2025nairian}). Let \[\gg=\sup(\pi_{\hr|\k_\hr, \infty}[\k_\hr]).\] It follows from our construction\footnote{See Clause 6g of \rnot{not: more not for reflection}.} that there is some successor $i<\tau$ and some $\zeta<\k_i$ such that
\vspace{0.3cm}
\begin{enumerate}[label=(7.\arabic*), itemsep=0.3cm]
\item $\zeta$ is a proper cutpoint of $\hq_i$ (see \rter{term: cutpoints etc}),
\item $\zeta$ is an inaccessible cardinal of $\Q_i$, and
\item $\pi_{\hq_i|\zeta, \infty}(\ts(\hq_i|\k_i))> \gg$ (see \rnot{not: more not for reflection}, recall that $\ts(W)$ is the least slw cardinal of $W$).
\end{enumerate}
\vspace{0.3cm}
It follows from \rcor{cor: technical corollary} that\footnote{Apply \rcor{cor: technical corollary} by setting $\hq=\hq_0$, $\hs=\hq_i$, $\hr=\hr$, $\zeta=\zeta$ and $\k=\k_{\hq_0}=\k_0$. Clauses 6 and 7 of \rcor{cor: technical corollary} are key.}
\vspace{0.3cm}
\begin{enumerate}[label=(8.\arabic*), itemsep=0.3cm]
\item  $\sup(\pi_{\hr|\nu_\hr, \infty}[\k_\hr])< \pi_{\hq_i|\zeta, \infty}(\ts(\hq_i|\k_i))$.
\end{enumerate}
\vspace{0.3cm}
Moreover, if $\hw'$ is a common iterate of $\hq_i$ and $\hr$ obtained via the least-extender-disagreement-coiteration then there is $\hw$ such that\footnote{Again we get (9.1)-(9.7) by applying \rcor{cor: technical corollary}.}
\vspace{0.3cm}
\begin{enumerate}[label=(9.\arabic*), itemsep=0.3cm]
\item $\T_{\hq_i, \hw'}$ is based on $\hq_i|\nu_{\hq_i}$ and $\T_{\hr, \hw'}$ is based on $\hr|\nu_\hr$,
\item $\hw$ is an iterate of $\hq_i$,
\item $\T_{\hq_i, \hw}$ is based on $\hq_i|\zeta$,
\item $\hw|\zeta_{\hw}=\hw'|\zeta_{\hw'}$,
\item $\T_{\hq_i, \hw}\insegeq \T_{\hq_i, \hw'}$,
\item $(\T_{\hq_i, \hw'})_{\geq \hw}$ is strictly above $\zeta_{\hw}$, and
\item $\sup(\pi_{\hr, \hw'}[\k_\hr])\leq \zeta_\hw$.
\end{enumerate}
\vspace{0.3cm}

Let $z=(\pi_{\hr|\nu_\hr, \infty})^{-1}(y)$. It follows from (9.7) that $\pi_{\hr, \hw'}(z)\in \M^{\hw}|\zeta_\hw$, and from (9.4), (9.5), and (9.6) that \[\pi_{\hr|\nu_\hr, \infty}(z)=\pi_{\hw|\nu_\hw, \infty}(\pi_{\hr, \hw'}(z)).\] Let $u=\pi_{\hr, \hw'}(z)$. We thus have that $\pi_{\hw|\nu_\hw, \infty}(u)=y$.

Recall that $\T_{\hq_i, \hw}$ is based on $\hq_i|\zeta$ (see (9.3)). Let then $\U=\downarrow(\T_{\hq_i, \hw}, \hq_i|\zeta)$. We can then find some $k\in C$ such that $\U\in \Q_k[g_k]$. Let then $\X$ be the full normalization of $(\T_{i, k})^\frown (\uparrow (\U, \hq_k))$. As in the proof of \rlem{lem: lifting iterations}, if $\hs$ is the last model of $\X$, then $\hs$ is a complete iterate of $\hw$ such that $\T_{\hw, \hs}$ is above $\k_\hw$. Moreover, $\hs\in \mathcal{F}_k$. It follows that $\pi_{\hs|\nu_\hs, \infty}(u)=y$. We now have that letting $x=\pi_{\hs, \hm_\tau}(u)$, we have that $j(x)=y$.
\end{proof}

This finishes the proof of Clause 1 and 2 of \rthm{thm: realizing as a nairian model}. Notice that a modification of \rlem{lem: j is onto} also gives that $j$ is elementary\footnote{This follows from (9.6). Because $\T_{\hw, \hw'}=(\T_{\hq_i, \hw'})_{\geq \hw}$ is strictly above $\zeta_\hw$, we have that $u$ satisfies the same formulas in $\M^\hw$ and $\M^{\hw'}$.}. Thus, we get that $j=id$. 

Clause 4 follows from the fact that if $\a<\k_\tau$, then 
\[ (\a^\omega)^{M[G]}=(\a^\omega)^M=(\a^\omega)^{\Q_\tau[g]}. \]
Indeed, if $X\in (\a^\omega)^M$ then using the proof of \rlem{lem: j is onto} we can find some $i<\tau$ and some $\hr\in \mathcal{F}_i$ such that $X\subseteq \rge(\pi_{\hr|\nu_\hr, \infty})$. Let now $z:\omega\rightarrow \M^\hr|\gg$ be a surjection where $\gg<\k_\hr$ is such that \[X\subseteq \pi_{\hr|\nu_\hr, \infty}[\gg],\] and let $w\subseteq \omega$ be such that \[X=\{ \pi_{\hr|\nu_\hr, \infty}(z(m)): m\in w\}.\] Finally, let $\hw=\pi_{i, \tau}^+(\hr)$. We then have that \[X=\{\pi_{\hw|\tau, \hm_\tau}(z(m)): m\in w\}\] and $(z, w)\in \Q_\tau[g]$. Hence, $X\in \Q_\tau[g]$. 
\end{proof}

The following is a useful fact used in the proof of \rlem{lem: j is onto}.

\begin{lemma}\label{lem: catching preimages} 
Suppose $\hq=(\Q, \Lambda)$ is a complete iterate of $\hp$, and $\k\in \tStr(\Q|\hd)$ is small relative to $\hd$. Suppose $\nu\in (\k, \hd)$ is a properly overlapped cardinal of $\Q$ and $\a<\pi_{\hq|\nu, \infty}(\k)$. Suppose further that $\hr$ and $\hs$ are two iterates of $\hq$ such that
\begin{enumerate}
    \item both $\T_{\hq, \hr}$ and $\T_{\hq, \hs}$ are based on $\hq|\nu$, 
    \item $\T_{\hq, \hs}$ is above $\k$ (but not necessarily strictly above),
    \item $\a_{\hr|\nu_\hr}$ is defined, and
    \item there is $\zeta<\l_\hs$ such that $\zeta$ is a properly overlapped inaccessible cardinal of $\M^\hs$ and $\sup(\pi_{\hr|\k_\hr, \infty}[\k_\hr])<\pi_{\hs|\zeta, \infty}(\ts(\hs|\zeta))$.
\end{enumerate}

Then there is a complete iterate $\hw$ of $\hs$ such that
\begin{enumerate}
    \item[(5)] $\T_{\hs, \hw}$ is based on $\hs|\zeta$,
    \item[(6)] $\a_{\hw|\nu_\hw}$ is defined, and 
    \item[(7)] $\pi_{\hw, \infty}(\a_{\hw|\nu_\hw})=\pi_{\hr, \infty}(\a_{\hr|\nu_\hr})$.
\end{enumerate}
\end{lemma}
\begin{proof} 
Let $\hw$ and $\hw'$ be as in the proof of \rlem{lem: j is onto} (they were produced using \rcor{cor: technical corollary}). As in that proof, we get that $\a_{\hw'}$\footnote{Denoted by $u$ in the proof of \rlem{lem: j is onto}.} is equal to $\a_\hw$. The additional information given by our current lemma is that $\pi_{\hw, \infty}(\a_{\hw|\nu_\hw})=\pi_{\hr, \infty}(\a_{\hr|\nu_\hr})$. This follows from the fact that $\T_{\hw, \hw'}$ is above $\zeta_\hw$, while $\a_{\hw|\nu_\hw}<\zeta_\hw$. We then have that $\pi_{\hw, \infty}(\a_{\hw|\nu_\hw})=\pi_{\hw', \infty}(\a_{\hw|\nu_\hw})=\pi_{\hr, \infty}(\a_{\hr|\nu_\hr})$.
\end{proof}

\section{Iteration sets and the uniqueness of realizations}\label{sec: it sets}

In this section, we define two important sets that will be useful in the calculations that follow. First we isolate the concept of \textit{iteration sets}.
The reader may wish to consult \rter{term: cutpoints etc} before reading the next definitions.

\begin{definition}\label{def: its}\normalfont Suppose $\nu<\hd_\infty$ is a proper cutpoint in $\mH$ relative to $\hd_\infty$ and is an inaccessible cardinal of $\mH$.  A set $Z\in \powerset_{\omega_1}(\mH|\nu)$ is an \textbf{iteration set at} $\nu$ if there is a hod pair $\hq$ such that 
\begin{enumerate}
    \item $\mH=\M_\infty(\hq)$, 

    \item $\nu_\hq$ is defined, and

    \item $Z(\hq, \nu)=\pi_{\hq, \infty}[\M^\hq|\nu_\hq]$.
\end{enumerate}
Let $\its(\nu)$ denote the set of iteration sets at $\nu$. If $\hq$ witnesses that $Z\in \its(\nu)$ then we say that $\hq$ is an $\its(\nu)$-certificate for $Z$. We will also write that $Z\in \its(\nu)$ as certified by $\hq$.
\end{definition}

The following is an easy corollary to the fact that $N$ is closed under ordinal definability (i.e., if $X, Y\in N$ and $Z\subseteq Y$ is $OD_X$ in $M$, then $Y\in N$).

\begin{proposition}\label{prop: its in n} Suppose $\nu<\hd_\infty$ is a proper cutpoint in $\mH$ relative to $\hd_\infty$ and is an inaccessible cardinal of $\mH$. Then $\its(\nu)\in N$.
\end{proposition}

\begin{notation}\label{not: its}\normalfont Suppose $\nu<\hd_\infty$ is a proper cutpoint in $\mH$ relative to $\hd_\infty$ and is an inaccessible cardinal of $\mH$. Suppose $Z\in \its(\nu)$. Let $\k=\sup(\tStr(\mH|\nu))$. Then let $\Q_Z$ be the transitive collapse of $Z$\footnote{Notice that $Z\prec \mH|\nu$.} and $\Q^Z$ be the transitive collapse of $\hull^{\mH|\nu}(Z\cup \k)$. Let $\tau_Z:\Q_Z\rightarrow \mH|\nu$ and $\tau^Z:\Q^Z\rightarrow \mH|\nu$ be the inverses of the transitive collapses, and let $\Lambda_Z$ and $\Lambda^Z$ be respectively the $\tau_Z$-pullback and $\tau^Z$-pullback of $\mH|\nu$. Set $\hq_Z=(\Q_Z, \Sigma_Z)$ and $\hq^Z=(\Q^Z, \Sigma^Z)$.
\end{notation}

The following is an easy lemma that follows from the results of \cite[Section 9.1]{blue2025nairian}. The reader may wish to consult \cite[Notation 9.3]{blue2025nairian} for the definition of notation $\hq^\k$ and point (3) in the proof of \cite[Theorem 10.13]{blue2025nairian}.

\begin{lemma}\label{lem: qz is qk} Suppose $\nu<\hd_\infty$ is a proper cutpoint in $\mH$ relative to $\hd_\infty$ and is an inaccessible cardinal of $\mH$. Suppose $Z\in \its(\nu)$ and $\k=\sup(\tStr(\mH|\nu))$. Then letting $\hq$ be any hod pair witnessing that $Z\in \its(\nu)$,
\begin{enumerate}
\item $\hq^Z=\hq^\k|\pi_{\hq, \infty}^\k(\nu_\hq)$,
\item $\hh|\nu$ is a complete iterate of $\hq^Z$, and
\item $\tau_Z=\pi_{\hq^Z, \hh|\nu}$.
\end{enumerate}
\end{lemma}

\begin{definition}\label{def: z its}\normalfont Suppose $\nu<\hd_\infty$ is a proper cutpoint in $\mH$ relative to $\hd_\infty$ and is an inaccessible cardinal of $\mH$. Suppose $Z\in \its(\nu)$ and $\k=\sup(\tStr(\mH|\nu))$. A set $Y\in \powerset_{\omega_1}(\mH|\k)$ is a $Z$-\textbf{iteration set} if $Z[Y]\in \its(\nu)$ where
\begin{center}
$Z[Y]=\{ f(a): a\in Y \wedge f\in Z\}$.
\end{center}
Let $\its(Z)$ denote the set of $Z$-iteration sets. 
\end{definition} 

The following lemma shows that the realizable maps are unique. The reader may benefit from reviewing \cite[Section 9.1]{blue2025nairian}. In particular, recall that for $\hq$ a complete iterate of $\hp$, $\hq^\nu$ is the least node $\hm$ of $\T_{\hq, \hh}$ (this just means the node with least index) such that $\hm|\nu=\hh|\nu$. We then have that \[\pi^\nu_{\hq, \infty}=\pi_{\hq, \hq^\nu}\] and \[\T^\nu_{\hq, \infty}=(\T_{\hq, \hh})_{\leq \hq^\nu}.\]

\begin{theorem}\label{thm: uniqueness of realizability witnesses}  Suppose $\nu<\hd_\infty$ is a proper cutpoint in $\mH$ relative to $\hd_\infty$ and is an inaccessible cardinal of $\mH$. Suppose $Z\in \its(\nu)$, and $\hq$ and $\hr$ are two $\its(\nu)$-certificates for $Z$. Let $\hs$ be a complete iterate of $\hq_Z$, and let $\hx$ be the last pair of $\uparrow(\T_{\hq_Z, \hs}, \hq)$ and $\hy$ be the last pair of $\uparrow(\T_{\hq_Z, \hs}, \hr)$. Then (see \rfig{fig:uniqueness_witnesses_parabola}) \begin{center}$\pi_{\hx, \infty}\rest \M^\hs=\pi_{\hy, \infty}\rest \M^\hs$.\end{center}
\end{theorem}
\begin{figure}[ht]
\centering
\begin{tikzpicture}[scale=0.85, transform shape, >=stealth, font=\small]

    \draw[thick] (0,0) coordinate (QZ_bottom) -- (0, 1.2) coordinate (QZ_top);
    \path (QZ_bottom) -- (QZ_top) coordinate[midway] (QZ_mid);
    \node[below] at (QZ_bottom) {$\hq_Z$};
    
    \draw[gray, thick] (QZ_top) parabola bend (QZ_top) (-2.5, 2.2) coordinate (Q_tip);
    \node[below] at (Q_tip) {$\hq$};
    
    \draw[gray, thick] (QZ_top) parabola bend (QZ_top) (2.5, 2.2) coordinate (R_tip);
    \node[below] at (R_tip) {$\hr$};

    \draw[thick] (0, 3.8) coordinate (S_bottom) -- (0, 5.0) coordinate (S_top);
    \path (S_bottom) -- (S_top) coordinate[midway] (S_mid);
    \node[right] at (S_bottom) {$\hs$};

    \draw[->, very thick, dotted, violet] (QZ_mid) to[out=45, in=-45] 
        node[midway, left, font=\scriptsize, xshift=-2pt] {$\T_{\hq_Z, \hs}$} 
        (S_mid);

    \draw[gray, thick] (S_top) parabola bend (S_top) (-2.5, 6.0) coordinate (HX_tip);
    \node[below left] at (HX_tip) {$\hx$};
    
    \draw[gray, thick] (S_top) parabola bend (S_top) (2.5, 6.0) coordinate (HY_tip);
    \node[below right] at (HY_tip) {$\hy$};

    \coordinate (Q_inner) at (-1.2, 1.45);
    \coordinate (HX_inner) at (-1.2, 5.25);
    \coordinate (R_inner) at (1.2, 1.45);
    \coordinate (HY_inner) at (1.2, 5.25);

    \draw[->, blue, dashed, thick] (Q_inner) -- 
        node[left, font=\tiny, xshift=-2pt] {$\T_{\hq, \hx} = \uparrow(\T_{\hq_Z, \hs}, \hq)$} 
        (HX_inner);

    \draw[->, blue, dashed, thick] (R_inner) -- 
        node[right, font=\tiny, xshift=2pt] {$\T_{\hr, \hy} = \uparrow(\T_{\hq_Z, \hs}, \hr)$} 
        (HY_inner);

    \node (HH) at (0, 8.5) {$\hh$};

    \draw[->, red, thick] (HX_tip) to[bend left=20] 
        node[midway, left, font=\scriptsize] {$\pi_{\hx, \infty}$} (HH);
    \draw[->, red, thick] (HY_tip) to[bend right=20] 
        node[midway, right, font=\scriptsize] {$\pi_{\hy, \infty}$} (HH);

    \node[draw=black, thick, fill=gray!10, rounded corners, inner sep=6pt, anchor=north west] 
        at (3.2, 2.2) {
        \begin{minipage}{5.5cm}
            \centering
            \textbf{\underline{Key Facts}}
            \begin{itemize}[label={\tiny$\bullet$}, itemsep=1pt, topsep=3pt, leftmargin=1em, font=\scriptsize]
                \item $\hq, \hr$ are certificates for $Z$
                \item $\hq_Z \unlhd \hq$ and $\hq_Z \unlhd \hr$
                \item $\hs \unlhd \hx$ and $\hs \unlhd \hy$
                \item $\T_{\hq_Z, \hs}$ lifts to $\hq$ and $\hr$
                \item \textbf{Result:}
                \[ \pi_{\hx, \infty}\rest \M^\hs = \pi_{\hy, \infty}\rest \M^\hs \]
            \end{itemize}
        \end{minipage}
    };

\end{tikzpicture}
\caption{Visualizing Theorem \ref{thm: uniqueness of realizability witnesses}. The base tree $\T_{\hq_Z, \hs}$ is lifted to both certificates.}
\label{fig:uniqueness_witnesses_parabola}
\end{figure}
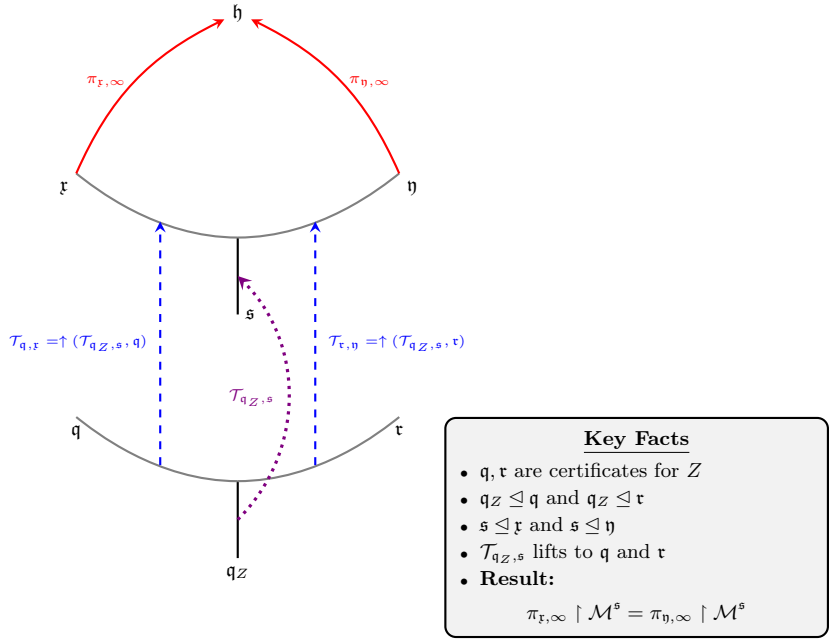
\begin{proof} 
The proof is a variation of arguments from \cite{blue2025nairian}. We have that \[\pi_{\hq|\nu_\hq, \hs}=\pi_{\hq_Z, \hs}=\pi_{\hr|\nu_\hr, \hs}\]
and \[\pi_{\hq, \infty}\rest \M^\hq|\nu_\hq=\pi_{\hr, \infty}\rest \M^\hr|\nu_\hr.\] We then let \[p=\pi_{\hq|\nu_\hq, \hs}=\pi_{\hr|\nu_\hr, \hs}\] and \[q=\pi_{\hq, \infty}\rest \M^\hq|\nu_\hq=\pi_{\hr, \infty}\rest \M^\hr|\nu_\hr.\]

Let $(\zeta_\a: \a\leq \iota)$ be the enumeration of $\tStr(\S)$, where $\hs=(\S, \Psi)$.  Let $\lambda_\a$ be the $\S$-successor of $\zeta_\a$. By induction on $\a\leq \iota$, we verify that\\
\begin{enumerate}[itemsep=0.3cm]
    \item[$(IH^0_\a)$]
    $\pi_{\hx, \infty}\rest \S|\zeta_\a=\pi_{\hy, \infty}\rest \S|\zeta_\a$, and
    \item[$(IH^1_\a)$] $\pi_{\hx, \infty}\rest \S|\l_\a=\pi_{\hy, \infty}\rest \S|\l_\a$.
\end{enumerate}
 We have that
 \vspace{0.3cm}
\begin{enumerate}[label=(1), itemsep=0.3cm]
    \item for every $\phi\in \tStr(\mH, \leq \k)$, letting $\psi$ be the successor of $\phi$ in $\mH$, $\M^{\hx^\phi}|\psi=\mH|\psi=\M^{\hy^\phi}|\psi$.
\end{enumerate}
\vspace{0.3cm}
 (1) follows from \cite[Lemma 9.4]{blue2025nairian}, and is a consequence of the fact that if $\phi$ is a small strong cardinal of $\mH|\hd_\infty$ then $\hx^\phi$ is on the main branch of $\T_{\hx, \infty}=\T_{\hx, \hh}$ and $\hy^\phi$ is on the main branch of $\T_{\hy, \infty}=\T_{\hy, \hh}$.
 
 For each $\a\leq \iota$ and $A\in \powerset(\zeta_\a)\cap S$, we have (e.g. see \cite[Corollary 9.10]{blue2025nairian})
 \begin{itemize}
 \item $f_\a\in \Q_Z$, 
 \item $f_A\in \Q_Z$, 
 \item $s_\a\in [\zeta_\a]^{<\omega}$, and 
 \item $s_A\in [\zeta_\a]^{<\omega}$
 \end{itemize}
 such that \[\zeta_\a=p(f_\a)(s_\a)\ \text{and}\ A=p(f_A)(s_A).\] It follows that
  \vspace{0.3cm}
\begin{enumerate}[label=(2.\arabic*), itemsep=0.3cm]
    \item for each $\a\leq \iota$ and $A\in \powerset(\zeta_\a)\cap S$, $\pi_{\hx, \infty}(\zeta_\a)$ and $\pi_{\hx, \infty}(A)$ are uniquely determined by $(f_\a, s_\a, \pi_{\hx, \infty}\rest \zeta_\a)$ and $(f_A, s_A, \pi_{\hx, \infty}\rest \zeta_\a)$ as \begin{center}$\pi_{\hx, \infty}(\zeta_\a)=q(f_\a)(\pi_{\hx, \infty}(s_\a))$ and $\pi_{\hx, \infty}(A)=q(f_A)(\pi_{\hx, \infty}(s_A))$,\end{center}
    \item and for each $\a\leq \iota$ and $A\in \powerset(\zeta_\a)\cap S$, $\pi_{\hy, \infty}(\zeta_\a)$ and $\pi_{\hy, \infty}(A)$ are uniquely determined by $(f_\a, s_\a, \pi_{\hy, \infty}\rest \zeta_\a)$ and $(f_A, s_A, \pi_{\hy, \infty}\rest \zeta_\a)$  as  \begin{center}$\pi_{\hy, \infty}(\zeta_\a)=q(f_\a)(\pi_{\hy, \infty}(s_\a))$ and $\pi_{\hy, \infty}(A)=q(f_A)(\pi_{\hy, \infty}(s_A))$.\end{center}
\end{enumerate}
\vspace{0.3cm}
It follows $(IH^0_\a)$ implies $(IH^1_\a)$. We now prove by induction that $(IH^0_\a)$ holds. The reader may find \rfig{fig:symmetric_agreement_final_v5} useful.
\begin{figure}[h]
\centering
\begin{tikzpicture}[scale=1.1, >=stealth, font=\small]

    \draw[thick] (0, 1.5) -- (0, 7.0);
    \node[below] at (0, 1.5) {$\mH$};

    \draw[dashed, gray] (-3.5, 2.5) -- (3.5, 2.5);
    \node[right, font=\scriptsize, gray] at (0.1, 2.3) {$\phi$};
    \filldraw (0, 2.5) circle (1pt);
    
    \node[left, font=\tiny] at (0, 6.2) {$\nu$};
    \filldraw (0, 6.2) circle (1pt);

    
    \draw[thick] (-3, 1.5) -- (-3, 4.5);
    \node[below] at (-3, 1.5) {$\hq^\phi$};
    
    \draw[thick] (-1.5, 1.5) -- (-1.5, 5.5);
    \node[below] at (-1.5, 1.5) {$\hx^\phi$};

    \node[left, font=\tiny] at (-3, 4.0) {$\nu_\hq^\phi$};
    \filldraw (-3, 4.0) circle (1pt);
    
    \node[left, font=\tiny] at (-1.5, 5.0) {$\nu_\hx^\phi$};
    \filldraw (-1.5, 5.0) circle (1pt);

    \draw[->, thick, violet] (-3, 3.2) -- (-1.5, 3.8)
        node[midway, above, sloped, font=\tiny] {$\T_{\hq^\phi, \hx^\phi}$};

    \draw[->, thick, violet] (-1.5, 4.2) -- (0, 4.8);

    \node (Q) at (-5.5, 0) {$\hq$};
    \node (X) at (-2.5, 0) {$\hx$};
    
    \draw[->, thick, violet] (Q) -- node[midway, below] {$\T_{\hq, \hx}$} (X);

    \draw[->, dashed, violet] (Q) .. controls (-5.5, 1.0) and (-3.5, 1.0) .. (-3.1, 1.8)
        node[midway, left, font=\tiny] {$\T_{\hq, \hq^\phi}$};

    \draw[->, dashed, violet] (X) .. controls (-2.5, 1.0) and (-2.0, 1.0) .. (-1.6, 1.8)
        node[midway, left, font=\tiny, xshift=2pt] {$\T_{\hx, \hx^\phi}$};

    
    \draw[thick] (3, 1.5) -- (3, 4.5);
    \node[below] at (3, 1.5) {$\hr^\phi$};
    
    \draw[thick] (1.5, 1.5) -- (1.5, 5.5);
    \node[below] at (1.5, 1.5) {$\hy^\phi$};

    \node[right, font=\tiny] at (3, 4.0) {$\nu_\hr^\phi$};
    \filldraw (3, 4.0) circle (1pt);
    
    \node[right, font=\tiny] at (1.5, 5.0) {$\nu_\hy^\phi$};
    \filldraw (1.5, 5.0) circle (1pt);

    \draw[->, thick, blue] (3, 3.2) -- (1.5, 3.8)
        node[midway, above, sloped, font=\tiny] {$\T_{\hr^\phi, \hy^\phi}$};

    \draw[->, thick, blue] (1.5, 4.2) -- (0, 4.8);

    \node (R) at (5.5, 0) {$\hr$};
    \node (Y) at (2.5, 0) {$\hy$};
    
    \draw[->, thick, blue] (R) -- node[midway, below] {$\T_{\hr, \hy}$} (Y);

    \draw[->, dashed, blue] (R) .. controls (5.5, 1.0) and (3.5, 1.0) .. (3.1, 1.8)
        node[midway, right, font=\tiny] {$\T_{\hr, \hr^\phi}$};

    \draw[->, dashed, blue] (Y) .. controls (2.5, 1.0) and (2.0, 1.0) .. (1.6, 1.8)
        node[midway, right, font=\tiny, xshift=-2pt] {$\T_{\hy, \hy^\phi}$};

    \node[draw=black, thick, fill=gray!10, rounded corners, inner sep=6pt, anchor=north] 
        at (0, -1.0) {
        \begin{minipage}{9cm}
            \centering
            \textbf{\underline{Key Facts}}
            \begin{itemize}[label={\tiny$\bullet$}, itemsep=1.2pt, topsep=3pt, leftmargin=1em, font=\scriptsize]
                \item $\hq, \hr$ are certificates for $Z$
                \item $\T_{\hq, \hx}=\uparrow(\T_{\hq_Z, \hs}, \hq)$ and $\T_{\hr, \hy}=\uparrow(\T_{\hq_Z, \hs}, \hr)$,
                \item $\hq^\phi|\nu_{\hq_\phi}=\hr^\phi|\nu_{\hr_\phi}$
                and $\hx^\phi|\nu_{\hx_\phi}=\hy^\phi|\nu_{\hy_\phi}$,
                \item $\T_{\hq^\phi,\hx^\phi}$ are $\T_{\hr^\phi, \hr^\phi}$ are above $\phi$ and are based on $\hq^\phi|\nu_{\hq_\phi}=\hr^\phi|\nu_{\hr_\phi}$,
                \item $\downarrow(\T_{\hq^\phi, \hx^\phi}, \hq^\phi|\nu_{\hq^\phi})=\downarrow(\T_{\hr^\phi, \hy^\phi}, \hr^\phi|\nu_{\hr^\phi})$,
                \item $\pi_{\hq, \hq^\phi}\rest \Q_Z=\pi_{\hr, \hr^\phi}\rest \Q_Z=_{def} m$,
                \item if $\hm$ is the least node of $\T_{\hq, \hx}$ such that \[\hm|\zeta_\a=\hx|\zeta_\a\] then $\T_{\hq^\phi, \hx^\phi}$ is the minimal $\pi_{\hm, \hm^\phi}$-copy of $\T_{\hm|\nu_\hm, \hs}$,
                \item if $\hn$ is the least node of $\T_{\hr, \hy}$ such that \[\hn|\zeta_\a=\hy|\zeta_\a\] then $\T_{\hr^\phi, \hy^\phi}$ is the minimal $\pi_{\hn, \hn^\phi}$-copy of $\T_{\hn|\nu_\hn, \hs}$.
            \end{itemize}
        \end{minipage}
    };

\end{tikzpicture}
\caption{The symmetric lifting structure. The trees $\T_{\hq, \hq^\phi}$ and $\T_{\hx, \hx^\phi}$ lift the certificates to the vertical models, where they agree with $\mH$ up to $\phi$.}
\label{fig:symmetric_agreement_final_v5}
\end{figure}

We have that $(IH^0_0)$ holds as \begin{center}
\begin{align*}
\pi_{\hx, \infty}\rest (\S|\zeta_0) &=\pi_{\hs, \infty}\rest (\S|\zeta_0)\\
&=\pi_{\hy, \infty}\rest (\S|\zeta_0)
\end{align*}
\end{center}
Clearly for $\a$ a limit ordinal, $(IH^0_\a)$ follows from $\forall \gg<\a(IH^0_\gg)$. It remains to show that $(IH^0_{\a+1})$ holds provided $(IH^1_\a)$ holds. This is just like $\a=0$ case. Indeed, let \[\phi=\pi_{\hx, \infty}(\zeta_\a)=\pi_{\hy, \infty}(\zeta_\a).\] We now have that
\vspace{0.3cm}
\begin{enumerate}[label=(3), itemsep=0.3cm]
    \item $\M^{\hx^\phi}|\pi^\phi_{\hx, \infty}(\nu_\hx)=\M^{\hy^\phi}|\pi^\phi_{\hy, \infty}(\nu_\hy)$ (this follows from $(IH^1_\a)$.
\end{enumerate}
\vspace{0.3cm}
To see (3), notice that \[\M^{\hx^\phi}|\pi^\phi_{\hx, \infty}(\nu_\hx)=Ult(\S, F_\hx)\] where $F_\hx$ is the long extender derived from $\pi_{\hx, \infty}\rest (\S|\l_\a)$ with space $\phi$, and similarly, \[\M^{\hy^\phi}|\pi^\phi_{\hy, \infty}(\nu_\hy)=Ult(\S, F_\hy)\] where $F_\hy$ is the long extender derived from $\pi_{\hy, \infty}\rest (\S|\l_\a)$ with space $\phi$ (see \rter{term: long extender}). But $(IH^1_\a)$ implies that $F_\hx=F_\hy$.

Next we have that
\vspace{0.3cm}
\begin{enumerate}[label=(4), itemsep=0.3cm]
    \item $\pi^\phi_{\hx, \infty}\rest \M^\hs=\pi^\phi_{\hy, \infty}\rest \M^\hs$.
\end{enumerate}
\vspace{0.3cm}
To see (4) notice that if $F=F_\hx=F_\hy$ then \[\pi^\phi_{\hx, \infty}=\pi^{\M^\hx}_F\ \text{and}\ \pi^\phi_{\hy, \infty}=\pi_F^{\M^\hy}\]
and \[\pi_F^{\M^\hx}\rest \M^\hs=\pi_F^{\M^\hy}\rest \M^\hs.\] 
Finally we have the following claim.
\begin{claim}\label{clm: hx is iterate of hq and hr}
  $\hq^\phi|\pi^\phi_{\hq, \infty}(\nu)=\hr^\phi|\pi^\phi_{\hr, \infty}(\nu)$, and $\hx^\phi|\pi^\phi_{\hx, \infty}(\nu_\hx)$ and $\hy^\phi|\pi^\phi_{\hy, \infty}(\nu_\hy)$ are complete iterates of $\hq^\phi|\pi^\phi_{\hq, \infty}(\nu)=\hr^\phi|\pi^\phi_{\hr, \infty}(\nu)$.
  \end{claim}
  \begin{proof}
Notice first that \begin{align*} \M^{\hq^\phi}|\pi^\phi_{\hq, \infty}(\nu)&=\chull^{\mH|\nu}(\phi\cup Z)\\ &=\M^{\hr^\phi}|\pi^\phi_{\hr, \infty}(\nu) \end{align*}
and if \[\sigma: \M^{\hq^\phi}|\pi^\phi_{\hq, \infty}(\nu)\rightarrow \mH|\nu\] is the inverse of the transitive collapse then \begin{align*}
\pi_{\hq^\phi, \hh}\rest (\M^{\hq^\phi}|\pi^\phi_{\hq, \infty}(\nu)) &= \sigma \\
&= \pi_{\hr^\phi, \hh}\rest (\M^{\hr^\phi}|\pi^\phi_{\hr, \infty}(\nu))
\end{align*}
and \begin{align*}
\hq^\phi|\pi^\phi_{\hq, \infty}(\nu) &= \sigma\text{-pullback of}\ \hh|\nu\\
&= \hr^\phi|\pi^\phi_{\hr, \infty}(\nu)
\end{align*}
Next, to finish the proof of the claim we need to argue that\\\\
(a) $\hx^\phi|\nu_{\hx^\phi}$ is a complete iterate of $\hq^\phi|\nu_{\hq^\phi}$, and\\
(b) $\hy^\phi|\nu_{\hy^\phi}$ is a complete iterate of $\hr^\phi|\nu_{\hr^\phi}$.\\\\
Because the argument for (a) is the same as the argument for (b), we give the argument for (a). Let $\hm$ be the least node of $\T_{\hq, \hx}$ such that \[\hm|\zeta_\a=\hx|\zeta_\a=\hs|\zeta_\a.\]
Notice that $\hm|\nu_\hm$ is the least node $\hm'$ of $\T_{\hq_Z, \hs}$ such that
\[\hm'|\zeta_\a=\hs|\zeta_\a.\]
Moreover,
\vspace{0.3cm}
\begin{enumerate}[label=(5.\arabic*), itemsep=0.3cm]
\item $(\T_{\hq_Z, \hs})_{\geq \hm|\zeta_\a}$ is an iteration tree based on $\hm|\nu_\hm$ and is strictly above $\zeta_\a$,
\item $(\T_{\hq, \hx})_{\geq \hm}$ is an iteration tree based on $\hm|\nu_\hm$ and is strictly above $\zeta_\a$,
\item $\T_{\hm, \hx}=\uparrow(\T_{\hm|\nu_\hm, \hs}, \hm)$
\item $\hm^\phi=\hq^\phi$ and $\pi_{\hq, \hq^\phi}=\pi_{\hm, \hm^\phi}\circ \pi_{\hq, \hm}$.
\end{enumerate}
\vspace{0.3cm}
\noindent
The full normalization applied to $(\T_{\hm, \hx})^\frown \T_{\hx, \hx^\phi}$
implies that
\vspace{0.3cm}
\begin{enumerate}[label=(6.\arabic*), itemsep=0.3cm]
\item $\hx^\phi$ is a complete iterate of $\hq^\phi$,
\item $\T_{\hq^\phi, \hx^\phi}$ is above $\phi$,
\item $\T_{\hq^\phi, \hx^\phi}$ is the minimal $\pi_{\hm, \hm^\phi}$-copy of $\T_{\hm, \hx}$.
\end{enumerate}
\vspace{0.3cm}
This finishes the proof of the \rcl{clm: hx is iterate of hq and hr}.
\end{proof}

Finally, applying (3), (4) and \rcl{clm: hx is iterate of hq and hr} we have that
\vspace{0.3cm}
\begin{enumerate}[label=(7), itemsep=0.3cm]
\item $\hx^\phi|\pi^\phi_{\hx, \infty}(\nu_\hx)=\hy^\phi|\pi^\phi_{\hy, \infty}(\nu_\hy)=_{def}\hw$.
\end{enumerate}
\vspace{0.3cm}
Set \[k=\pi^\phi_{\hx, \infty}\rest \M^\hs=\pi^\phi_{\hy, \infty}\rest \M^\hs.\] 

 It follows that if $\gg=k(\zeta_{\a+1})$ and $\b$ is such that $\hh|\b$ is a complete iterate of $\hw|\gg$, then 
 \begin{align*}
 \pi_{\hx, \infty}\rest \S|\zeta_{\a+1} &= \pi_{\hw|\gg, \hh|\b} \circ k \rest \S|\zeta_{\a+1} \\
 &= \pi_{\hy, \infty}\rest \S|\zeta_{\a+1}.
 \end{align*}
 This finishes the proof of \rthm{thm: uniqueness of realizability witnesses}.
\end{proof}

Motivated by \rthm{thm: uniqueness of realizability witnesses}, we make the following definition.

\begin{definition}\label{def: rel sets}\normalfont Suppose $\nu<\hd_\infty$ is a proper cutpoint in $\mH$ relative to $\hd_\infty$ and is an inaccessible cardinal of $\mH$. Suppose further that $Z\in \its(\nu)$. Let $\rel(Z)$ be the set of pairs $(\hr, \tau_\hr)$ such that $\hr$ is a complete iterate of $\hq_Z$ and whenever $\hq$ is an $\its(\nu)$-certificate for $Z$, letting $\hs$ be the last pair of $\uparrow(\T_{\hq_Z, \hr}, \hq)$, \[\tau_\hr=\pi_{\hs, \infty}\rest \M^\hr.\] To emphasize the dependence on $Z$, we will write $\tau^Z_\hr$.
\end{definition}

\begin{corollary}\label{cor: uniqueness of realizability witnesses} Suppose $\nu<\hd_\infty$ is a proper cutpoint in $\mH$ relative to $\hd_\infty$ and is an inaccessible cardinal of $\mH$. Suppose further that $Z\in \its(\nu)$. Then $\rel(Z)\in N$.
\end{corollary}

We now discuss the definability of $\its(\nu)$, $\its(Z)$, and $\rel(Z)$ in $N$. First recall the definition of $\infty$-Borel.

\begin{definition}\label{def: infty borel}\normalfont Suppose $\a$ is an ordinal and $A\subseteq \alpha^\omega$. Then $A$ is \textbf{$\infty$-Borel} if for some formula $\phi$ and some set of ordinals $S$, for all $X\in \a^\omega$,
\begin{center}
$X\in A\iff L[S, X]\models \phi[X]$,
\end{center}
and $S$ is an $\infty$-\textbf{Borel code} of $A$.
\end{definition}

\rcor{cor: def its and rel} shows that in fact we can define $\its(\nu)$ and $\rel(Z)$ in $N$ from  some $A\in \powerset(\nu)\cap \mH$. Applying the proof of \cite[Theorem 10.6]{blue2025nairian}, we get the following (see the outline of the proof after \rcor{cor: def its and rel}).

\begin{theorem}\label{thm: every set is infty borel} Suppose $\a<\varsigma$ and $A\in N\cap \powerset(\a^\omega)$. Then \[N\models ``A\ \text{is $\infty$-Borel''}.\] Moreover, assuming $A$ is ordinal definable from $X\in \varsigma^\omega$ in $M$, letting $\gg$ be the least Woodin cardinal of $\mH$ above $\sup(X)$, there is $S\in \mH|(\gg^+)^\mH$ such that the set $(S, X)$ is an $\infty$-Borel code of $A$. 
\end{theorem}

Because $\its(\nu)$ is ordinal definable and $(\its(Z), \rel(Z))$ is ordinal definable from $Z$ (in $M$), we have the following corollary. 

\begin{corollary}\label{cor: def its and rel} Suppose $\nu<\hd_\infty$ is a proper cutpoint in $\mH$ relative to $\hd_\infty$ and is an inaccessible cardinal of $\mH$. Let $\gg$ be the least Woodin cardinal of $\mH$ above $\nu$. Then there are $S, T\in \mH|(\gg^+)^\mH$ and formulas $\phi$ and $\psi$ such that for every $X, Y, Z, a, b\in (\mH|\nu)^\omega$ the following holds:
\begin{enumerate}
\item $Z\in \its(\nu)\iff L[S, Z]\models \phi[Z]$ (i.e. $S$ is an $\infty$-Borel code of $\its(\nu)$).
\item $Z\in \its(\nu)$, $Y\in \its(Z)$, $a\in \omega^\omega$ codes a complete iterate $\hr$ of $\hq_{Z[Y]}$, $b\in \omega^\omega$ codes a bijection $f: \omega\rightarrow \M^\hr$, and \begin{center} $X=\{\tau^{Z[Y]}_\hr(b(i)): i\in \omega\}$\end{center}
if and only if \begin{center}
$L[T, X, Y, Z, a, b]\models \psi[X, Y, Z, a, b]$.
\end{center}
\end{enumerate}
\end{corollary}

\begin{notation}\label{def: the borel codes at nu}\normalfont Suppose $\nu<\hd_\infty$ is a proper cutpoint in $\mH$ relative to $\hd_\infty$ and is an inaccessible cardinal of $\mH$. Let $\gg$ be the least Woodin cardinal of $\mH$ above $\nu$. We let $S_\nu, T_\nu\in \mH\cap \powerset(\gg)$ be the $\mH$-least sets witnessing \rcor{cor: def its and rel}.
\end{notation}

The following lemma will be used later on.

\begin{lemma}\label{lem: catching t and s} Suppose $\nu<\hd_\infty$ is a proper cutpoint in $\mH$ relative to $\hd_\infty$ and is an inaccessible cardinal of $\mH$. Let $\gg$ be the least Woodin cardinal of $\mH$ above $\nu$. Suppose $\hq$ is a complete iterate of $\hp$ such that $\nu_\hq$ is defined. Then \[(S_\nu, T_\nu)\in \rge(\pi_{\hq, \infty}).\] 
\end{lemma}
\begin{proof} Notice that $\gg_\hq$ is also defined. Let now $(S, T)$ be $\mH$-least pair satisfying \rcor{cor: def its and rel}. It then follows that $(S, T)$ is definable from $\nu$, and hence, $S_\hq$ and $T_\hq$ are defined.
\end{proof}

We give a very short outline of the proof of \rcor{cor: def its and rel}. The proof of  \cite[Theorem 10.6]{blue2025nairian} utilizes the fact that under $\ADR + V=L(\powerset(\bR))$, which holds in $M$, $V$ embeds into the derived model of $\mH$ at $\Theta$, and moreover, the embedding restricted to $\mH$ is in $\mH$. Thus, we have a formula $\phi$ such that for all $Z\in (\mH|\nu)^\omega$, \[Z\in \its(\nu)\ \text{if and only if}\ \mH[Z]\models \phi[Z].\] Letting $j: M\rightarrow W$ be the embedding of $M$ into the derived model of $\mH$, $\phi$ essentially expresses the statement \[W\models j[Z]\in \its(j(\nu)).\] The other key fact used in the proof of \cite[Theorem 10.6]{blue2025nairian} is that each such $Z$ is generic over $\mH$ by a poset of size $\gg$, where $\gg$ is the least Woodin cardinal of $\mH$ above $\nu$ (see \cite[Lemma 10.4]{blue2025nairian}). The $S$ itself is obtained as a Skolem Hull of a rank initial segment of $\mH$, and $S$ is such that there is an elementary embedding $m: S\rightarrow \mH|\zeta$ for some $\zeta$ much larger than $\Theta^M$.

\section{Reflection I}\label{sec: reflection i}

In this section we work towards establishing that the reflection property introduced in \rdef{def: kappa reflection} holds in Nairian Models. We first prove a basic form of it, which we can then generalize to obtain the version stated in \rdef{def: kappa reflection}. We start by describing the possible reflection points.

\begin{definition}\label{def: induced embeddings}\normalfont 
Suppose for some $\a\leq \b$, $\sigma': \mH|\a\rightarrow \mH|\b$ is a function.
Then $\sigma: N|\a\rightarrow N|\b$ is \textbf{induced} by $\sigma'$ if for every $X\in \powerset_{\omega_1}^b(\a)$, 
\[\sigma(X)=\sigma'[X]\] 
and for every $z\in N|\a$, $\sigma(z)=w$ if and only if whenever $(\phi, X)$ is such that:
\begin{itemize}
    \item $X\in \powerset_{\omega_1}^b(\a)$, 
    \item $\phi$ is a formula, and 
    \item $z$ is the unique $y$ such that $N|\a\models \phi[y, X]$,
\end{itemize} 
then $N|\b\models \phi[w, \sigma(X)]$.

If $\sigma'$ and $\sigma$ are as above then we write $\sigma=\ind(\sigma')$
\end{definition}

\begin{remark}\normalfont
If $\sigma\rest \powerset_{\omega_1}^b(\a)$ is elementary (i.e., for $X\in\powerset_{\omega_1}^b(\a)$, $N|\a\models \phi[X]$ if and only if $N|\b\models \phi[\sigma(X)]$), then $\sigma$ is well-defined and elementary.
\end{remark}

\begin{definition}\label{def: reflection tuple}\normalfont
Suppose $u=(\hq, \l, \nu)$ and \[M\models \rg_3(\hq, \l, \k, \nu).\]
Let \begin{center}$\sigma'_{u}: \mH|\pi_{\hq|\nu, \infty}(\l)\rightarrow \mH|\pi_{\hq,\infty}(\l)$\end{center} be the embedding given by $\sigma'_{u}(x)=y$ if and only if whenever $\hr$ is a complete iterate of $\hq$ such that 
\begin{itemize}
\item $\T_{\hq, \hr}$ is based on $\hq|\nu$, and
\item $x\in \rge(\pi_{\hr|\nu_\hr, \infty})$,
\end{itemize}
$y=\pi_{\hr, \infty}(x)$. Let 
\begin{center} $\sigma_{u}: N|\pi_{\hq|\nu, \infty}(\l)\rightarrow N|\pi_{\hq,\infty}(\l)$\end{center} be the embedding induced by $\sigma'_{u}$. We then let $\l_u=\pi_{\hq|\nu, \infty}(\l)$ and $\l^u=\pi_{\hq, \infty}(\l)$.
\end{definition}
\begin{figure}[ht]
\centering
\begin{tikzpicture}[scale=1.2, >=stealth, font=\small]

    \node (H_full) at (-2.5, 6.5) {$\mH$};
    \node (H_sub) at (-2.5, 4.5) {$\mH|\sup(\pi_{\hq|\nu, \infty}[\nu])$};
    \node at (-2.5, 5.5) {\rotatebox{90}{$\unlhd$}};

    \node (R_sub) at (4, 3.2) {$\hr|\nu_\hr$};
    \node (R_full) at (6.5, 3.2) {$\hr$};
    \node at (5.25, 3.2) {$\unlhd$};

    \node (Q_sub) at (0, 1) {$\hq|\nu$};
    \node (Q_full) at (2.5, 1) {$\hq$};
    \node at (1.25, 1) {$\unlhd$};

    \draw[->, thick, dashed, bend left=45] (H_sub) to 
        node[midway, left] {$\sigma'_u$} (H_full);

    \draw[->, thick, ForestGreen] (Q_sub) -- (R_sub)
        node[midway, sloped, above] {$\T_{\hq|\nu, \hr|\nu_\hr}$};

    \draw[->, thick] (Q_full) -- (R_full)
        node[midway, sloped, below, yshift=-2pt] {$\T_{\hq, \hr} = \uparrow(\T_{\hq|\nu, \hr|\nu_\hr})$};

    \draw[->, thick, violet] (Q_sub) .. controls (-2.5, 1.2) and (-3.5, 3.5) .. (H_sub)
        node[pos=0.8, left, xshift=-2pt] (pi_q_inf) {$\pi_{\hq|\nu, \infty}$};

    \draw[->, blue, thick] (R_sub) to[bend right=15] 
        node[midway, above, sloped] {$\pi_{\hr|\nu_\hr, \infty}$} (H_sub);

    \draw[->, red, thick] (R_full) to[bend right=30] 
        node (pi_r_inf) [midway, right, xshift=2pt] {$\pi_{\hr, \infty}$} (H_full);

    \node[blue] (x) at (-2.7, 3.8) {$x$};

    \node[red] (xr) at (3.2, 3.2) {$x_{\hr|\nu_\hr}$};

    \node[blue] (y) at (-1.8, 6.1) {$y = \sigma'_u(x)$};

    \draw[->, dotted, blue!70, thick] (x) to[bend left=15] 
        node[midway, below, sloped, font=\tiny] {pullback via $\pi_{\hr|\nu_\hr, \infty}$} (xr);
        
    \draw[->, dotted, red!70, thick] (xr) to[bend right=25] 
        node[midway, above, sloped, font=\tiny, yshift=1pt] {push forward via $\pi_{\hr, \infty}$} (y);

\end{tikzpicture}
\caption{The embedding $\sigma'_{u}$. The element $x$ is pulled back via $\pi_{\hr|\nu_\hr, \infty}$ to the iterate $\hr$ and then mapped forward via $\pi_{\hr, \infty}$ to its final value $y$.}
\label{fig:sigma_u_final}
\end{figure}
\begin{remark}\label{rem: also bounded by l}\normalfont Notice that because $\l$ is small, we can define $\sigma_u'$ by saying that $\sigma_u(x)=y$ if and only if whenever $\hr$ is a complete iterate of $\hq$ such that 
\begin{itemize}
\item $\T_{\hq, \hr}$ is based on $\hq|\nu$,
\item $\gen(\T_{\hq, \hr})\subseteq \l_\hq$, and
\item $x\in \rge(\pi_{\hr|\nu_\hr, \infty})$,
\end{itemize}
$y=\pi_{\hr, \infty}(x)$.
\end{remark}
Next we prove that $\sigma'_u$ is well-defined.\footnote{Notice that it is immediate from \rdef{def: reflection tuple} that $\sigma'_u(x)$ is independent of $\hr$.}
\begin{remark}\label{rem: the need for a woodin}\normalfont \rlem{lem: u is wd} does not require the Woodinness of $\nu$, and holds assuming just that $\nu$ is an inaccessible properly overlapped cardinal of $\M^\hq|\hd_\infty$.
\end{remark}
The definition of $\rg$ appears in \rdef{def: ref gen}.
\begin{lemma}\label{lem: u is wd} Suppose $u=(\hq, \l, \nu)$ and $\rg_3(\hq, \l, \k, \nu)$ holds. Then $\sigma'_u$ is well defined.
\end{lemma}
\begin{proof} Fix $x\in \mH|\l_u$. Also fix $\hr$ and $\hr'$ such that $x_{\hr|\nu_\hr}$ and $x_{\hr'|\nu_{\hr'}}$ are defined and both $\T_{\hq, \hr}$ and $\T_{\hq, \hr'}$ are both based on $\hq|\nu$. We want to see that
\[(a)\ \ \ \ \ \  \pi_{\hr, \infty}(x_{\hr|\nu_\hr})=\pi_{\hr', \infty}(x_{\hr'|\nu_{\hr'}}).\]
Using \cite[Theorem 11.1]{blue2025nairian} as we did in the proof of \rthm{thm: delta-gen it}, we can find $\hs$ and $\hs'$ such that
\vspace{0.3cm}
\begin{enumerate}[label=(1.\arabic*), itemsep=0.3cm]
\item $\hs$ is a complete iterate of $\hq$ such that $\T_{\hq, \hs}$ is based on $\hq|\nu$ and is above $\k$,
\item $\hs'$ is a complete iterate of $\hs$ such that $\T_{\hs, \hs'}$ is based on $\hs|\nu_\hs$ and is above $\k_\hs$,
\item for some $\zeta<\k_\hs$ that is properly overlapped inaccessible cardinal of $\M^\hs$, \[\sup(\pi_{\hr|\k_\hr, \infty}[\k_\hr])<\pi_{\hs|\zeta, \infty}(\ts(\hs|\zeta)),\] and
\item for some $\zeta'<\k_{\hs'}$ that is properly overlapped inaccessible cardinal of $\M^\hs$, \[\sup(\pi_{\hr'|\k_{\hr'}, \infty}[\k_{\hr'}])<\pi_{\hs'|\zeta', \infty}(\ts(\hs'|\zeta')).\]
\end{enumerate}
\vspace{0.3cm}
Applying \rlem{lem: catching preimages}, we get $\hw$ and $\hw'$ such that
\vspace{0.3cm}
\begin{enumerate}[label=(2.\arabic*), itemsep=0.3cm]
\item $\hw$ is a complete iterate of $\hs$ and $\hw'$ is a complete iterate of $\hs'$,
\item $\T_{\hs, \hw}$ is based on $\hs|\zeta$ and $\T_{\hs', \hw'}$ is based on $\hs'|\zeta'$,
\item $x_{\hw|\nu_\hw}$ and $x_{\hw'|\nu_{\hw'}}$ are defined,
\item $\pi_{\hw, \infty}(x_{\hw|\nu_\hw})=\pi_{\hr, \infty}(x_{\hr|\nu_\hr})$, and
\item $\pi_{\hw', \infty}(x_{\hw'|\nu_{\hw'}})=\pi_{\hr', \infty}(x_{\hr'|\nu_{\hr'}})$.
\end{enumerate}
\vspace{0.3cm}

\begin{figure}[ht]
\centering
\begin{tikzpicture}[node distance=2.5cm, auto, scale=1.2]
    \node (hs) at (0,0) {$\hs$};
    \node (hsprime) at (5,0) {$\hs'$};

    \node (hw) at (0,3) {$\hw$};
    \node (hwpp) at (3.5, 3) {$\hw''$}; 
    \node (hwp) at (6.5, 3) {$\hw'$};    

    \node (hwppp) at (5, 5.5) {$\hw'''$};

    \draw [arrow] (hs) -- node[left] {$\T_{\hs, \hw}$} (hw);

    \draw [arrow] (hsprime) -- (hwpp); 
    \draw [arrow] (hsprime) -- (hwp);
    
    \draw [arrow] (hw) -- node[above] {$\T_{\hw, \hw''}$} (hwpp);
    
    \draw [arrow] (hwpp) -- (hwppp);
    \draw [arrow] (hwp) -- (hwppp);

    \node[anchor=north, font=\footnotesize, text width=2.5cm, align=center] at (1.75, 3) {above $\kappa_{\hw}$};
    
\end{tikzpicture}
\caption{The relationship between $\hs, \hs', \hw, \hw', \hw''$ and $\hw'''$.}
\label{fig:lemma_13_5_diamond_final}
\end{figure}
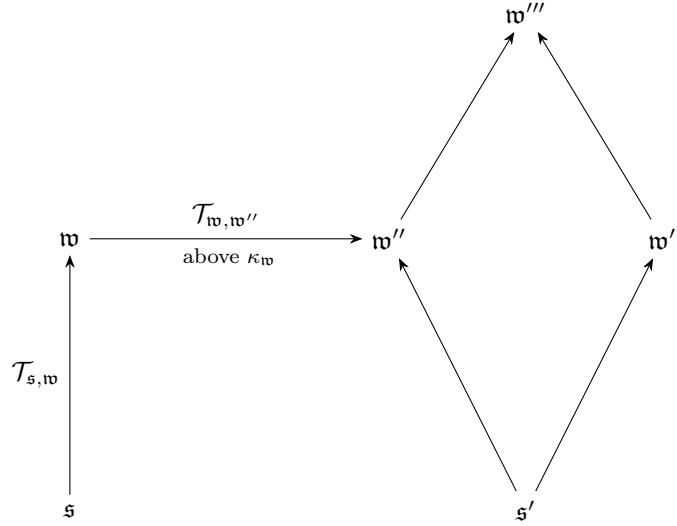
It follows from (2.4) and (2.5) that to show (a) it is enough to show
\[(b)\ \ \ \ \ \ \pi_{\hw, \infty}(x_{\hw|\nu_\hw})=\pi_{\hw', \infty}(x_{\hw'|\nu_{\hw'}}).\]
Let now $\hw''$ be the last model of $\uparrow(\U, \hs')$ where $\U=\downarrow(\T_{\hs, \hw}, \hs|\zeta)$. The \rfig{fig:lemma_13_5_diamond_final} might be useful to the reader. 

We have that
\vspace{0.3cm}
\begin{enumerate}[label=(3), itemsep=0.3cm]
\item $\hw''$ is a complete of $\hw$ such that $\T_{\hw, \hw''}$ is above $\k_\hw$ and is based on $\hw|\nu_\hw$.
\end{enumerate}
\vspace{0.3cm}
(3) follows easily from analyzing the full normalization of $(\T_{\hs, \hs'})^\frown \T_{\hs', \hw''}$ (we did such calculations in the proof of \rlem{lem: j is wd}). Because both $\T_{\hs', \hw'}$ and $\T_{\hs', \hw''}$ are based on $\hs'|\nu_{\hs'}$, we can find $\hw'''$ such that
\vspace{0.3cm}
\begin{enumerate}[label=(4.\arabic*), itemsep=0.3cm]
\item $\hw'''$ is a complete iterate of $\hw'$ and of $\hw''$,
\item $\T_{\hw', \hw'''}$ is based on $\hw'|\nu_{\hw'}$, and
\item $\T_{\hw'', \hw'''}$ is based on $\hw''|\nu_{\hw''}$.
\end{enumerate}
\vspace{0.3cm}
We then have that
\vspace{0.3cm}
\begin{enumerate}[label=(5.\arabic*), itemsep=0.3cm]
\item $\pi_{\hw, \hw'''}(x_{\hw|\nu_\hw})=x_{\hw'''|\nu_{\hw'''}}$, and
\item $\pi_{\hw', \hw'''}(x_{\hw'|\nu_{\hw'}})=x_{\hw'''|\nu_{\hw'''}}$.
\end{enumerate}
\vspace{0.3cm}
We now have that \begin{align*}
\pi_{\hw, \infty}(x_{\hw|\nu_\hw}) &= \pi_{\hw''', \infty}(x_{\hw'''|\nu_{\hw'''}})\\ &=
         \pi_{\hw', \infty}(x_{\hw'|\nu_\hw'}). \end{align*}
         Hence, (b) follows. 
\end{proof}

The following is the main theorem of this section. $\rel$ appeared in \rdef{def: rel sets}.

\begin{theorem}\label{thm: basic reflection} Suppose $u=(\hq, \l, \nu)$ and $\rg_3(\hq, \l, \k, \nu)$, and $X\in N|\l^u$ is definable from some countable $Y\subseteq \pi_{\hq, \infty}[\l]$ over $N|\l^u$. Then \begin{enumerate}
\item $\sigma_u$ is an  elementary embedding, 
\item $\sigma_u$ is definable over $N$ from $\rel(\pi_{\hq, \infty}[\M^\hq|\nu])$ (and hence, $\sigma_u\in N$),
\item $X\in \rge(\sigma_u)$.
\end{enumerate}
\end{theorem}
\begin{proof} Let $Y\subseteq \pi_{\hq, \infty}[\l]$ be such that $X$ is definable from $Y$ over $N|\l^u$. Then $Y\in \rge(\sigma_u)$ because if $Z=\pi_{\hq, \infty}^{-1}[Y]$, then $\sigma_u(\pi_{\hq|\nu, \infty}[Z])=Y$. It follows that $X\in \rge(\sigma_u)$. This shows Clause (3). Clause (2) follows immediately from the definition of $\sigma_u'$. Below we show Clause (1).

Let 
\begin{itemize}
\item $G\subseteq \Coll(\omega_1, \bR)$ be $M$-generic, 
\item $\vec{B}=(B_i: i<\omega_1^M)\in M[G]$ be an enumeration of $\Delta_{\hq, \nu}$, 
\item $\vec{a}=(a_i: i<\omega)\in M[G]$ be an enumeration of $\bR^M$, and 
\item $(\hq_i, \hq_i', \T_i, E_i: i<\omega_1)\in M[G]$ be a $\Delta_{\hq, \nu}$-genericity iteration relative to $(\vec{B}, \vec{a})$ (which exists because of \rthm{thm: delta-gen it}).
\end{itemize}
We set 
\begin{itemize}
\item for $i<j\leq \omega_1$, $\pi_{\hq_i, \hq_j}=\pi_{i, j}$,
\item for $i\leq \omega_1$, $\hm_i=\hm(\hq_i, \k_{\hq_i})$ (see \rnot{not: more not for reflection}),
\item for $i\leq\omega_1$, $N^l_i=N^l(\hq_i, \k_{\hq_i})$ (see \rnot{not: more not for reflection}).
\end{itemize}

We follow the proof of \rthm{thm: realizing as a nairian model}, and will use the notation introduced in \rnot{not: notation for the proof}. We let $C\subseteq \tau$ be as in the proof of \rthm{thm: realizing as a nairian model}. We will use the following observation (which easily follows from  \rdef{def: delta-gen it}, \rlem{lem: catching preimages}, and \rlem{lem: u is wd}).\\\\
(Obs) Suppose $Z\in \powerset_{\omega_1}^b(\l_u)$. Then there is $i_0\in C$ such that for every $i\in [i_0, \omega_1)$, there is \[\hr\in \mathcal{F}_i\ \text{and}\ Z'\in \powerset(\l_\hr)\cap \Q_i[g_i]\] such that \[\pi_{\hr|\nu_\hr, \infty}[Z']=Z\  \text{and}\ \sigma_u(Z)=\pi_{\hr, \infty}[Z'].\]\\

In fact, there is $i_0<\omega_1$, $\hr_0\in \mathcal{F}_{i_0}$, and $Z_0\in \powerset(\l_{\hr_0})\cap \Q_{i_0}[g_{i_0}]$ such that for every $i\in [i_0, \omega_1)$, setting \[\hr=\pi_{i_0, i}^+(\hr_0)\ \text{and}\ Z'=\pi_{i_0, i}^+[Z_0],\] $\sigma_u(Z)=\pi_{\hr, \infty}[Z']$.

\begin{claim} $\sigma_u$ is an elementary embedding.
\end{claim}
\begin{proof}
It is enough to show that if for some $\zeta<\l_u$, $Z\in (\mH|\zeta)^\omega$ and $\phi(x)$ is a formula, then \[N|\l_u\models \phi[Z]\ \text{implies that}\ N|\l^u\models \phi[\sigma_u(Z)].\] 

Fix then a pair $(Z, \phi)$ such that $N|\l_u\models \phi[Z]$. Applying (Obs), we can find some $i\in C$ and $\hr\in \mathcal{F}_i$ such that for some $Z'\in \powerset(\l_\hr)\cap \Q_i[g_i]$, \begin{center} $Z=\pi_{\hr|\nu_\hr, \infty}[Z']$ and $\sigma_u(Z)=\pi_{\hr, \infty}[Z']$. \end{center}
Let now  $z:\omega\rightarrow \M^\hr|\l_\hr$ be a surjection and $w\subseteq \omega$ be such that $(z, w)\in \Q_i[g_i]$ and \begin{center}$Z'=\{z(k): k\in w\}$.\end{center} It follows from \rthm{thm: realizing as a nairian model} that if \[Z''=\{ \pi_{\hr, \hm_i}(z(k)): k\in \omega\}=\pi_{\hr, \hm_i}[Z']\] then \begin{center}$Z''\in \Q_i[g_i]\cap N^l_i$ and $Z=\pi_{i, \tau}^+(Z'')$.\end{center}
$\pi^+_{i, k}$ appeared in \rlem{lem: lifting iterations}. It also follows from \rthm{thm: realizing as a nairian model} that since \begin{center}$\pi_{i, \tau}^+(N^l_i)=N|\pi_{\hr|\nu_\hr, \infty}(\k_\hr)$,\end{center} letting $\iota=\pi_{\hq_i, \hm_i}(\l_{\hq_i})$,
\begin{center}
$N^l_i|\iota\models \phi[Z'']$.
\end{center}
Working in $M[G]$, we can also find a $\Delta_{\hq_i, \hd_{\hq_i}}$-genericity iteration \[(\hs_m, \hs_m', \U_m, F_m: m<\tau)\] such that $\hs_0=\hq_i$. Let $\hs_\tau$ be the direct limit of \[(\hs_m, \pi_{\hs_m, \hs_k}: m<k<\tau),\] and let $h\subseteq \Coll(\omega, <\tau)$ be $\M^{\hs_\tau}$-generic such that \[g_i\subseteq h\ \text{and} \bR^{\M^{\hs_\tau}[h]}=\bR^M.\] We then have that $\pi_{\hq_i, \hs_\tau}$ can be lifted to \[\pi_{\hq_i, \hs_\tau}^+: \Q_i[g_i]\rightarrow \M^{\hs_\tau}[h].\] It then follows from \rthm{thm: realizing as a nairian model} that \[\pi^+_{\hq_i, \hs_\tau}(N^l_i)=N|\pi_{\hq_i, \infty}(\k_i),\] and since \[\pi_{\hq_i, \hs_\tau}^+(\iota)=\pi_{\hq, \infty}(\l)\] and \[\pi_{\hq_i, \hs_\tau}^+(Z'')=\sigma_u(Z),\] we have that \[N|\l^u\models \phi[\sigma(Z)].\]
\end{proof}
\end{proof}
We end this section with a sharper bound on the definability of $\sigma_u'$ (see \rcor{cor: sharper bound}). We will use the following notation in \rcor{cor: sharper bound}. 
\begin{terminology}\label{def: rest rel}\normalfont Suppose $\nu<\hd_\infty$ is a proper cutpoint in $\mH$ relative to $\hd_\infty$ and is an inaccessible cardinal of $\mH$. Suppose $U\in \its(\nu)$ and $\l\in U$ is a strong cardinal of $\mH|\nu$. We then let $\rel(U)\rest \l$ consist of pairs $(\hr, \tau_\hr)$ such that\footnote{Here, and elsewhere in the paper, we use the $x_\hw$ to denote the preimage of $x$ in $\M^\hw$, assuming the map is understood. So in the current case, $\l_\hr=(\tau'_\hr)^{-1}(\l)$.}
\begin{itemize}
\item for some $\tau'_\hr$, $(\hr, \tau'_\hr)\in \rel(U)$ and $\tau_\hr=\tau'_\hr\rest \M^\hr|\l_\hr$, and
\item $\gen(\T_{\hq_U, \hr})\subseteq \l_\hr$.
\end{itemize}
\end{terminology}
We remark that because $\tau_\hr'$ in the above definition is uniquely determined by $(U, \hr)$ (see \rdef{def: rel sets} and \rthm{thm: uniqueness of realizability witnesses}), $\tau_\hr$ is also uniquely determined by $(U, \hr)$. The following easy lemma will be used in the proof of \rcor{cor: sharper bound}.

\begin{lemma}\label{lem: characterizing rel rest}  Suppose $\nu<\hd_\infty$ is a proper cutpoint in $\mH$ relative to $\hd_\infty$ and is an inaccessible cardinal of $\mH$. Suppose
\begin{itemize}
\item $U\in \its(\nu)$ as certified by $\hq$, 
\item $\l\in U$ is a strong cardinal of $\mH|\nu$, and 
\item $(\hr, \tau_\hr)\in \rel(U)\rest \l$. 
\end{itemize}
There is then a complete $\l$-bounded iterate $\hs$ (so $\gen(\T_{\hq, \hs})\subseteq \l_\hs$) of $\hq$ such that
\begin{itemize}
\item $\T_{\hq, \hs}$ is based on $\hq|\nu$,
\item $\hr=\hs|\nu_\hs$, and
\item $\tau_\hr=\pi_{\hs, \infty}\rest \M^\hs|\l_\hs$.
\end{itemize}
\end{lemma}
To get $\hs$ as above, we simply let $\hs$ be the last model of $\uparrow(\T_{\hq|\nu, \hr}, \hq)$.
\begin{corollary}\label{cor: sharper bound} Suppose
\begin{itemize}
\item $u=(\hq, \nu, \l)$ is a reflection-generator, 
\item $\hr=\hq^{\l^u}$,
\item $\b$ is such that $\hh|\b$ is a complete iterate of $\hr|\nu_\hr$,
\item $Z=\pi_{\hr|\nu_\hr, \hh|\b}[\pi_{\hq, \hr}[\M^\hq|\nu]]$,
\item $\k=\min(\tStr(\hq|\nu)-(\l+1))$, and
\item $E\in \vec{E}^{\M^\hq}$ is the least with the property that $\lh(E)>\nu$ and $\cp(E)=\k$.
\end{itemize}
Then $Z\in \its(\b)$ as certified by $\hq_E$, and $\sigma_u$ is definable over $N|\b$ from $(\rel(Z)\rest \l^u, \l^u)$.
\end{corollary}
\begin{proof}
We first show that $\hq_E$ is an $\its(\b)$-certificate for $Z$. We thus need to show that \[Z=\pi_{\hq_E, \infty}[\M^{\hq_E}|\nu].\]
The equality follows because letting $\hs=\hq_E^{\l^u}$ and applying full normalization to $(\T_{\hq, \hq_E})^\frown \T_{\hq_E, \hs}$ we get that \\\\
\begin{enumerate}[label=(1.\arabic*), itemsep=0.3cm]
\item $\hs=\hr_F$ where $F=\pi_{\hq, \hr}(E)$,
\item $\hq_E|\nu=\hq|\nu$,
\item $\hs|\pi_{\hq_E, \hs}(\nu)=\hr|\nu_\hr$,\footnote{Notice that $\pi_{\hq_E, \hs}(\nu)\neq\nu_\hs$ as $\nu_\hs=\pi_{\hq, \hs}(\nu)=\pi_{\hq_E, \hs}(\nu_{\hq_E})$ and $\nu<\nu_{\hq_E}$.}
\item $\pi_{\hq, \hr}\rest \M^\hq|\nu=\pi_{\hq_E, \hs}\rest \M^{\hq_E}|\nu$.
\end{enumerate}
We then have that 
\begin{align*}
(*)\ \ \ \ \ \ \ \ \pi_{\hq_E, \infty}[\M^{\hq_E}|\nu] &=\pi_{\hs|\pi_{\hq_E, \hs}(\nu), \hh|\b}\circ \pi_{\hq_E, \hs}[\M^{\hq_E}|\nu]\\
&= \pi_{\hr|\nu_\hr, \hh|\b}\circ \pi_{\hq, \hr}[\M^\hq|\nu]\\
&= Z
\end{align*}
Next, we need to verify that $\sigma_u'$ is definable over $N|\b$ from $(\rel(Z)\rest \l^u, \l^u)$. We know that $\sigma_u'$ can be easily defined from $\rel(U)$ where \[U=\pi_{\hq, \infty}[\M^\hq|\nu].\] Applying \rrem{rem: also bounded by l}, we see that we can define $\sigma_u'$ from $\rel(U)\rest \l^u$. The following claim finishes the proof.
\begin{claim} $\rel(U)\rest \l^u=\rel(Z)\rest \l^u$.
\end{claim} 
\begin{proof} 
The reader may find \rfig{fig:sharper_bound_final} useful. To show that \[\rel(U)\rest \l^u\subseteq \rel(Z)\rest \l^u\]
it is enough to show that (see \rlem{lem: characterizing rel rest})\\\\
(a) if $\hr$ is a complete $\l^u$-bounded iterate of $\hq$ such that $\T_{\hq, \hr}$ is based on $\hq|\nu$, then \[(\hr|\nu, \pi_{\hr, \infty}\rest \M^\hr|\l_\hr)\in \rel(Z)\rest \l^u.\]\\
The proof of (a) is just like the argument showing that $\hq_E$ is an $\its(\b)$-certificate for $Z$. Indeed, let 
\begin{itemize}
\item $F=\pi_{\hq, \hr}(E)$, and
\item $\hs$ be the last model of $\uparrow(\T_{\hq|\nu, \hr|\nu_\hr}, \hq_E)$.
\end{itemize}
Notice that $\uparrow(\T_{\hq|\nu, \hr|\nu_\hr}, \hq_E)$ makes sense as $\hq_E|\nu=\hq|\nu$. Applying full normalization to $(\T_{\hq, \hq_E})^\frown \T_{\hq_E, \hs}$ we get that\\\\
(2) $\hr_F=\hs.$\\\\
Following the argument for $(*)$, we see that (keeping in mind that $\hr|\l_\hr=\hr_F|\l_{\hr_F}$)
\[\pi_{\hr, \infty}\rest \M^\hr|\l_\hr=\pi_{\hr_F, \infty}\rest \M^{\hr_F}|\l_{\hr_F}.\]
We thus have that (applying (2))\\\\
(3) $\pi_{\hr, \infty}\rest \M^\hr|\l_\hr=\pi_{\hs, \infty}\rest \M^{\hs}|\l_{\hs}$.\\\\
Since $\hs$ is a complete $\l^u$-bounded iterate of $\hq_E$ and $\hq_E$ is an $\its(\b)$-certificate for $Z$, we have that
\[(\hs|\nu_\hs, \pi_{\hs, \infty}\rest \M^\hs|\nu_\hs)\in \rel(Z)\]
and because $\hr_F|\nu_\hr=\hr|\nu_\hr$, we get that 
\[(\hr|\nu_\hr, \pi_{\hr, \infty}\rest \M^\hr|\l_\hr)\in \rel(Z)\rest \l^u.\]

Conversely, to show that \[\rel(Z)\rest \l^u\subseteq \rel(U)\rest \l^u\]
it is enough to show that (see \rlem{lem: characterizing rel rest})\\\\
(b) if $\hs$ is a complete $\l^u$-bounded iterate of $\hq_E$ such that $\T_{\hq_E, \hs}$ is based on $\hq_E|\nu$, then \[(\hs|\nu, \pi_{\hs, \infty}\rest \M^\hs|\l_\hs)\in \rel(U)\rest \l^u.\]\\
Notice that $\nu=\b_{\hq_E}$ (and this might seem confusing the reader). The key point here is that, just like in the proof of (a), if $\hr$ is the last model of $\uparrow(\T_{\hq_E|\nu, \hs|\b_\hs}, \hq)$, then letting $F=\pi_{\hq, \hr}(E)$, \[\hs=\hr_F.\] We then conclude, as in the proof of (a), that $\hr|\nu_\hr=\hs|\b_\hs$, \[\pi_{\hr, \infty}\rest \M^\hr|\l_\hr=\pi_{\hs, \infty}\rest \M^\hs|\l_\hs,\] and \[(\hs|\b_\hs, \pi_{\hs, \infty}\rest \M^\hs|\l_\hs)\in \rel(U)\rest \l^u.\]
\begin{figure}[ht]
\centering
\begin{tikzpicture}[node distance=3cm, auto, scale=1.2]
    \node (hq) at (-2.5,3) {$\hq$};
    \node (hqE) at (-2.5,6) {$\hq_E$};
    
    \node (hr) at (2.5,3) {$\hr$};
    \node (hs) at (2.5,6) {$\hs$};
    
    \node (hrF) at (5, 4.5) {$\hr_F = \hs$};
    \node (H) at (7.5, 4.5) {$\mH$};

    \draw [arrow] (hq) -- node[left] {$E$} (hqE);
    \draw [arrow] (hq) -- node[below] {$\T_{\hq, \hr}$} node[above, pos=0.5, font=\tiny] {$\l^u$-bounded} (hr);
    \draw [arrow] (hqE) -- node[above] {$\T_{\hq_E, \hs}$} node[below, pos=0.5, font=\tiny] {$\l^u$-bounded} (hs);
    
    \draw [arrow, dashed] (hr) -- node[below right, font=\scriptsize] {$F = \pi_{\hq, \hr}(E)$} (hrF);
    \draw [arrow] (hs) -- (hrF);
    \draw [arrow] (hrF) -- node[above] {$\pi_{\hs, \infty}$} (H);

    \node[draw=black, thick, fill=gray!10, rounded corners, inner sep=6pt, anchor=north] 
        at (2.5, 1.5) {
        \begin{minipage}{9cm}
            \centering
            \textbf{\underline{Key Facts}}
            \begin{itemize}[label={\tiny$\bullet$}, itemsep=1.2pt, topsep=3pt, leftmargin=1em, font=\scriptsize]
                \item $\nu_{\hr} = \b_{\hs}$,
                \item $\hr|\nu_{\hr} = \hs|\b_{\hs}$,
                \item $\pi_{\hr, \infty} \restriction \M^{\hr}|\l_{\hr} = \pi_{\hs, \infty} \restriction \M^{\hs}|\l_{\hs}$.
            \end{itemize}
        \end{minipage}
    };
\end{tikzpicture}
\caption{The relationship between $\hq, \hq_E, \hr, \hr_F$ and $\hs$.}
\label{fig:sharper_bound_final}
\end{figure}
\end{proof}
\end{proof}

\section{Reflection II: Closure under $\omega_1$-sequences in the $\pmax$ extension}

Corollary \ref{cor: tech cor ii} allows us to show that if $\hr$ and $\hs$ are as in the statement of \ref{cor: tech cor ii}, then $\hs$ can be used to generate $\sigma_{(\hr, \l_\hr, \nu_\hr)}$. Passing from $\hr$ to $\hs$ is useful because given any $\omega_1$-sequence $((\hr_\a, \l_{\hr_\a}, \nu_{\hr_\a}): \a<\omega_1)$ we can find a single $\hs$ satisfying the conditions of \ref{cor: tech cor ii}, and this allows us to show that $N$ is closed under $\omega_1$-sequences in the $\pmax$-extension of $M$ (see \rcor{cor: omega1 closure}) and under $<\Theta^N$-sequences in $M$ (see \rthm{thm: closure under theta sequences}). We start with the exact statement we need. The reader may wish to review the notation introduced in \rdef{def: rel sets}.

\begin{corollary}\label{cor: tech cor iii} Suppose $M\models \rg_3(\hq, \l, \k, \nu)$, and $\hr$ and $\hs$ are two complete iterates of $\hq$ such that 
\begin{itemize}
\item $\k<\l$,
\item $\k$ is the largest strong cardinal of $\M^\hq|\nu$,
\item $\T_{\hq, \hs}$ is above $\k$, 
\item $\gen(\T_{\hq, \hr})\subseteq \l_\hr$,
\item there is $\zeta<\k_\hs$ that is a proper cutpoint of $\M^\hs$, an inaccessible cardinal of $\M^\hs$, 
\begin{center}$\sup(\pi_{\hr|\k_\hr, \infty}[\k_\hr])<\pi_{\hs|\zeta, \infty}(\ts(\hs|\zeta))$,\end{center}
and $\gen(\T_{\hq, \hs})\subseteq \zeta$.
\end{itemize}
Let $Z=\pi_{\hs, \infty}[\M^\hs|\nu_\hs]$, and $u=(\hr, \l_\hr, \nu_\hr)$. Then there is a complete iterate $\hk$ of $\hs|\nu_\hs$ such that for some $\xi$, 
\begin{enumerate}
\item $\hk|\xi$ is a complete iterate of $\hr|\nu_\hr$, 
\item $\xi$ is a proper cutpoint of $\M^\hk$ and a Woodin cardinal of $\M^\hk$, 
\item if $\l'$ is the supremum of all the strong cardinals of $\hs|\k_\hs$, then \[\pi_{\hs|\nu_\hs, \hk}(\l')<\xi<\pi_{\hs|\nu_\hs, \hk}(\k_\hs),\]
\item for $\b<\l_u$, letting $\hm$ be a complete iterate of $\hk$ such that $\T_{\hk, \hm}$ is based on $\hk|\xi$ and $\b_{\hm|\xi_\hm}$ is defined, \begin{center}$\sigma_u(\b)=\tau^Z_{\hm}(\b_{\hm|\xi_\hm})$.\end{center}
\end{enumerate}
\end{corollary}
\begin{proof} The reader may find \rfig{fig:sigma_u_minimalist_final_boxed} useful.
Let $\hw$ be the result of the least-extender-disagreement coiteration of $\hs$ and $\hr$. Let $\hx$ be the least pair on $\T_{\hr, \hw}$ such that $\T_{\hr, \hx}$ is based on $\hr|\nu_\hr$. We then have that (see \rcor{cor: tech cor ii})
\vspace{0.3cm}
\begin{enumerate}[label=(1.\arabic*), itemsep=0.3cm]
\item $\gen(\T_{\hs, \hw})\subseteq \nu_\hw$ and $\gen(\T_{\hr, \hw})\subseteq \nu_\hw$
\item $\T_{\hs, \hw}$ is based on $\hs|\nu_\hs$,
\item $\hx$ is a complete iterate of $\hr$,
\item $\l_{\hw}=\l_{\hx}$ and $\hx|\nu_\hx=\hw|\nu_\hx$,
\item if $\hx'$ is the least node of $\T_{\hr, \hx}$ such that $\hx'|\k_{\hx'}=\hx|\k_\hx$, then $\hx'$ is on the main branch of $\T_{\hr, \hw}$,
\item if $E$ is the extender used on the main branch of $\T_{\hr, \hx}$ at $\hx'$ then \[\cp(E)=\k_{\hx'}\ \text{and}\ \lh(E)>\nu_\hx,\]
\item $(\T_{\hr, \hw})_{\geq \hx'_E}$ is above $\k_{\hx'_E}$, and 
\item $(\T_{\hr, \hw})_{\geq \hx'}$ is an iteration tree on $\hx'$, and so $\hw$ is a complete iterate of $\hx'$ and $\T_{\hx', \hw}$ is above $\k_{\hx'}$.
\end{enumerate}
\vspace{0.3cm}

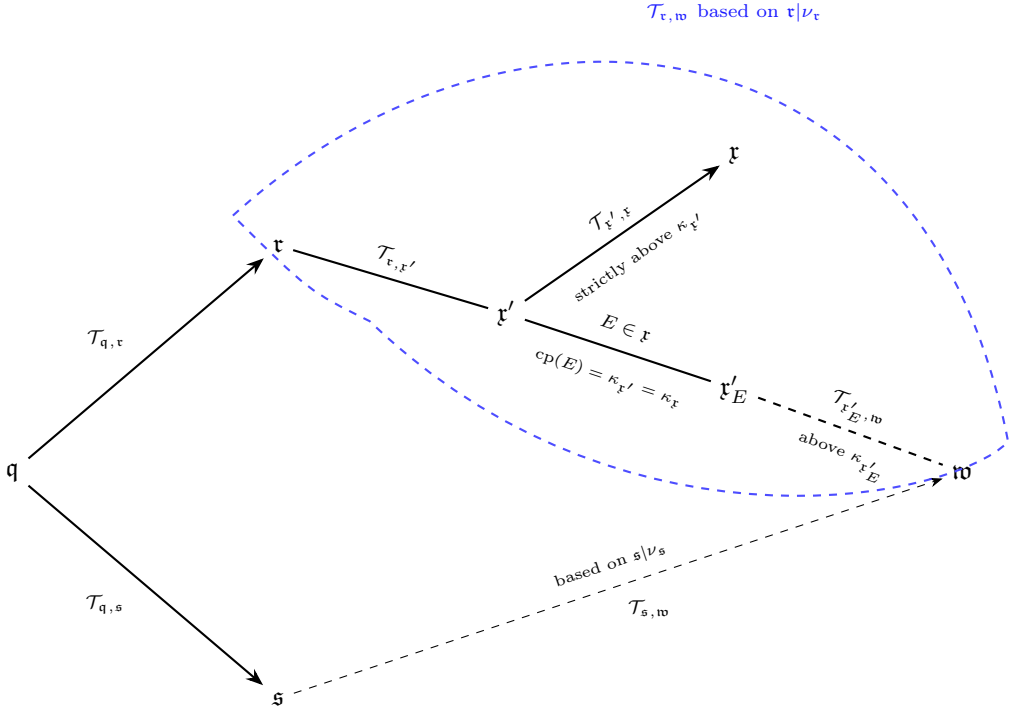
\usetikzlibrary{fit, shapes.geometric, calc}

\begin{figure}[ht]
    \centering
    \begin{tikzpicture}[node distance=2.5cm, auto, >=Stealth]
        
        \node (hq) at (0,0) {$\hq$};
        \node (hw) at (12.5, 0) {$\hw$};
        
        \node (hr) at (3.5, 3) {$\hr$};
        \node (hs) at (3.5, -3) {$\hs$};
        
        \node (hxp) at (6.5, 2.1) {$\hx'$};
        \node (hxe) at (9.5, 1.1) {$\hx'_E$};
        
        \node (hx) at (9.5, 4.2) {$\hx$};

        
        \path[->, thick] 
            (hq) edge node[above left, font=\scriptsize, xshift=-4pt] {$\T_{\hq,\hr}$} (hr)
            (hq) edge node[below left, font=\scriptsize, xshift=-4pt] {$\T_{\hq,\hs}$} (hs);
            
        \path[->, dashed] (hs) edge node[below right, font=\scriptsize, yshift=-2pt] {$\T_{\hs,\hw}$} 
            node[midway, sloped, above, font=\tiny, yshift=2pt] {based on $\hs|\nu_{\hs}$} (hw);

        \draw[thick] (hr) -- (hxp) node[midway, sloped, above, font=\scriptsize] {$\T_{\hr, \hx'}$};
        
        \draw[thick] (hxp) -- (hxe) 
            node[midway, sloped, above, font=\scriptsize, yshift=2pt] {$E \in \hx$}
            node[midway, sloped, below, font=\tiny, yshift=-4pt] {$\text{cp}(E)=\k_{\hx'}=\k_{\hx}$};
        
        \draw[thick, dashed] (hxe) -- (hw) 
            node[midway, sloped, above, font=\scriptsize] {$\T_{\hx'_E, \hw}$}
            node[midway, sloped, below, font=\tiny, yshift=-2pt] {above $\k_{\hx'_E}$};
        
        \draw[->, thick] (hxp) -- (hx) 
            node[midway, sloped, above, font=\scriptsize] {$\T_{\hx', \hx}$}
            node[midway, sloped, below, font=\tiny, yshift=-4pt] {strictly above $\k_{\hx'}$};

        \draw[blue!70, thick, dashed] 
            ($(hr) + (-0.6, 0.4)$) .. controls ($(hx) + (-3, 2.5)$) and ($(hx) + (2.5, 2)$) .. 
            ($(hw) + (0.6, 0.4)$) .. controls ($(hw) + (-0.2, -0.4)$) and ($(hs)!0.5!(hw) + (0, 0.3)$) ..
            ($(hr) + (1.2, -1.0)$) .. controls ($(hr) + (0.4, -0.6)$) .. 
            cycle;
        
        \node[text=blue!90, font=\scriptsize, anchor=south] at ($(hx.north) + (0, 1.4)$) {$\T_{\hr, \hw}$ based on $\hr|\nu_{\hr}$};

        \node[draw=black, thick, fill=gray!5, rounded corners, inner sep=8pt, anchor=north] 
        (keyfacts) at (6.25, -3.8) {
            \begin{minipage}{8.5cm}
                \centering
                \textbf{\underline{Key Facts}}
                \begin{itemize}[label={\tiny$\bullet$}, itemsep=3pt, topsep=4pt, leftmargin=1.5em, font=\small]
                    \item $\hx|\nu_{\hx} = \hw|\nu_{\hw}$
                    \item $\text{cp}(E) = \k_{\hx'}$ and $\text{lh}(E) > \nu_{\hx}$
                \end{itemize}
            \end{minipage}
        };

    \end{tikzpicture}
    \caption{The coiteration of $\hr$ and $\hs$.}
    \label{fig:sigma_u_minimalist_final_boxed}
\end{figure}
 Let $\xi=\nu_{\hx}$. Notice first that if $\b<\l_{u}$ then there is a complete iterate $\hy$ of $\hw$ such that $\T_{\hw, \hy}$ is based on $\hw|\xi$ and $\b_{\hy|\xi_\hy}$ is defined (here $\xi_\hy=\pi_{\hw, \hy}(\xi)$ and $\xi_\hy<\nu_\hy$). This is because $\hw|\xi=\hx|\xi$, and $\hx|\xi$ is a complete iterate of $\hr|\nu_\hr$.

 The following is our key claim.
\begin{claim}\label{clm: key clm catching ralpha} Suppose $\b<\l_{u}$ and $\hy$ is an iterate of $\hw$ such that $\T_{\hw, \hy}$ is based on $\hw|\xi$ and $\b_{\hy|\xi_\hy}$ is defined (here $\xi_\hy=\pi_{\hw, \hy}(\xi)$). Then \[\sigma_u(\b)=\pi_{\hy, \infty}(\b_{\hy|\xi_\hy}).\] 
\end{claim}
\begin{proof} The reader may find \rfig{fig:claim_11_5_relations} and \rfig{fig:claim_11_5_wide_angle_boxed} useful. Notice that because $\T_{\hr, \hx'}$ is based on $\hr|\nu_\hr$, \\\\
(2) if $\hm$ is a complete iterate of $\hx'$ such that 
\begin{itemize}
\item $\T_{\hx', \hm}$ is based on $\hx'|\nu_{\hx'}$, and 
\item $\b_{\hm|\nu_\hm}$ is defined,
\end{itemize}
then \[\sigma_u(\b)=\pi_{\hm, \infty}(\b_{\hm|\nu_\hm}).\]\\
We consider the full normalization $\T_{\hx', \hy}$ of $(\T_{\hx', \hw})^\frown \T_{\hw, \hy}$ (this makes sense because of (1.4) and (1.8)). Let $\hy'$ be the last model of $\uparrow(\T_{\hw|\xi, \hy|\xi_\hy}, \hx)$. We then have that\\
\begin{enumerate}[label=(3.\arabic*), itemsep=0.3cm]
\item $\T_{\hx', \hy'}$ is the longest initial segment of $\T_{\hx', \hy}$ that is based on $\hx'|\nu_{\hx'}$,
\item if $\hy''$ is the least node of $\T_{\hx', \hy'}$ such that $\hy''|\k_{\hy''}=\hy'|\k_{\hy'}$, then $\hy''$ is on the main branch of $\T_{\hx', \hy}$,
\item $\k_{\hy''}=\pi_{\hw, \hy}(\k_{\hx})$ and $\nu_{\hy'}=\pi_{\hw, \hy}(\xi)$ (recall, $\xi=\nu_{\hx})$),
\item if $F$ is the extender used on the main branch of $\T_{\hx', \hy}$ at $\hy''$, then \[\cp(F)=\k_{\hy''}\ \text{and}\ \lh(F)>\nu_{\hy'},\]
\item $\T_{\hy''_F, \hy}$ is above $\k_{\hy''_F}$, and 
\item $\hy$ is a complete iterate of $\hy''$ such that $\T_{\hy'', \hy}$ is above $\k_{\hy''}$.
\end{enumerate}
\begin{figure}[ht]
    \centering
    \begin{tikzpicture}[node distance=2.5cm, auto, >=Stealth]
        
        \node (hxp) at (0, 0) {$\hx'$};
        
        \node (hx) at (3, 2) {$\hx$};
        \node (hyp) at (8, 2) {$\hy'$};
        
        \node (hxe) at (3, -1) {$\hx'_E$};
        \node (hw) at (6, -1) {$\hw$};
        \node (hy) at (10, -1) {$\hy$};

        
        \draw[->, thick] (hxp) -- (hx) 
            node[midway, sloped, above, font=\scriptsize] {$\T_{\hx', \hx}$};
            
        \draw[thick] (hxp) -- (hxe) 
            node[midway, sloped, above, font=\scriptsize] {$E$}
            node[midway, sloped, below, font=\tiny, yshift=-2pt] {$\text{cp}(E)=\k_{\hx'}$};
            
        \draw[thick, dashed] (hxe) -- (hw) 
            node[midway, sloped, above, font=\scriptsize] {$\T_{\hx'_E, \hw}$}
            node[midway, sloped, below, font=\tiny] {above $\k_{\hx'_E}$};
            
        \path[->, dashed, thick] (hw) edge 
            node[midway, sloped, above, font=\scriptsize] {$\T_{\hw, \hy}$} 
            node[midway, sloped, below, font=\tiny, yshift=-2pt] {based on $\hw|\nu_{\hx}$} (hy);

        \path[->, dashed, thick] (hx) edge 
            node[midway, sloped, above, font=\scriptsize] {lift of $\T_{\hw, \hy}$} 
            node[midway, sloped, below, font=\tiny] {(based on $\hx|\nu_\hx$)} (hyp);

    \end{tikzpicture}
    \caption{The relationship between the iterates in Claim 11.5. The tree $\T_{\hw, \hy}$ is lifted to $\hx$ (because $\hx|\nu_\hx = \hw|\nu_\hx$) to produce $\hy'$.}
    \label{fig:claim_11_5_relations}
\end{figure}
\begin{figure}[ht]
    \centering
    \begin{tikzpicture}[auto, >=Stealth]
        
        \node (hr) at (-2, 0) {$\hr$};
        \node (hxp) at (0, 0) {$\hx'$};
        
        \node (hx) at (3, 2.5) {$\hx$};
        \node (hypp) at (7, 2.5) {$\hy''$};
        \node (hyp) at (12, 2.5) {$\hy'$};
        
        \node (hxe) at (3, -2.5) {$\hx'_E$};
        \node (hw) at (7, -2.5) {$\hw$};
        \node (hy) at (12, -2.5) {$\hy$};
        
        \node (hyppF) at (9.5, 0) {$\hy''_F$};

        
        \draw[->, thick] (hr) -- (hxp) 
            node[midway, sloped, above, font=\scriptsize] {$\T_{\hr, \hx'}$}
            node[midway, sloped, below, font=\tiny, yshift=-1pt] {based on $\hr|\nu_{\hr}$};

        \draw[->, thick] (hxp) -- (hx) 
            node[midway, sloped, above, font=\scriptsize] {$\T_{\hx', \hx}$}
            node[midway, sloped, below, font=\tiny, yshift=-1pt] {strictly above $\k_{\hx'}$};
            
        \draw[thick] (hxp) -- (hxe) 
            node[midway, sloped, above, font=\scriptsize] {$E$}
            node[midway, sloped, below, font=\tiny, yshift=-1pt] {cp$(E)=\k_{\hx'}$};
            
        \draw[thick, dashed] (hxe) -- (hw) 
            node[midway, sloped, above, font=\scriptsize] {$\T_{\hx'_E, \hw}$}
            node[midway, sloped, below, font=\tiny] {above $\k_{\hx'_E}$};
            
        \path[->, dashed, thick] (hw) edge 
            node[midway, sloped, above, font=\scriptsize] {$\T_{\hw, \hy}$} 
            node[midway, sloped, below, font=\tiny, yshift=-1pt] {based on $\hw|\nu_{\hx}$} (hy);

        \draw[->, dashed, thick] (hx) -- (hypp) 
            node[midway, sloped, above, font=\scriptsize] {lift of $\T_{\hw, \hy}$}
            node[midway, sloped, below, font=\tiny, yshift=-1pt] {based on $\hx|\nu_{\hx}$};
            
        \draw[->, dashed, thick] (hypp) -- (hyp)
            node[midway, sloped, above, font=\scriptsize] {$\T_{\hy'', \hy'}$}
            node[midway, sloped, below, font=\tiny, yshift=-1pt] {strictly above $\k_{\hy''}$};

        \draw[thick] (hypp) -- (hyppF) 
            node[midway, sloped, above, font=\scriptsize] {$F \in \hy'$}
            node[midway, sloped, below, font=\tiny, yshift=-1pt] {cp$(F)=\k_{\hy''}$};

        \draw[->, dashed, thick] (hyppF) -- (hy) 
            node[midway, sloped, above, font=\scriptsize] {$\T_{\hy''_F, \hy}$}
            node[midway, sloped, below, font=\tiny, yshift=-1pt] {above $\k_{\hy''_F}$};

        \node[draw=black, thick, fill=gray!10, rounded corners, inner sep=6pt, anchor=north] 
            at (5, -3.5) {
            \begin{minipage}{10cm}
                \centering
                \textbf{\underline{Key Facts}}
                \begin{itemize}[label={\tiny$\bullet$}, itemsep=1.2pt, topsep=3pt, leftmargin=1em, font=\scriptsize]
                    \item Since $\T_{\hy'', \hy'}$ is strictly above $\k_{\hy''}$, we have 
                    \[\pi_{\hy'', \infty}(\b_{\hy|\xi_\hy}) = \pi_{\hy, \infty}(\b_{\hy|\xi_\hy})\]
                    \item Since $\T_{\hr, \hy''}$ is based on $\hr|\nu_\hr$, we have 
                    \[\sigma_u(\b_{\hy|\xi_\hy}) = \pi_{\hy'', \infty}(\b_{\hr|\xi_\hr}).\]
                    \item Hence, \[\sigma_u(\b_{\hy|\xi_\hy}) =\pi_{\hy, \infty}(\b_{\hy|\xi_\hy})\]
                \end{itemize}
            \end{minipage}
        };

    \end{tikzpicture}

    \caption{The relationship between various hod pairs.}
    \label{fig:claim_11_5_wide_angle_boxed}
\end{figure}
Because $\b_{\hy|\xi_\hy}<\l_\hy=\l_{\hy''}<\k_{\hy''}$ (here $\xi_\hy=\pi_{\hw, \hy}(\xi)$) and $\hy''|\l_{\hy''}=\hy|\l_{\hy}$, (3.1)-(3.6) imply that\\\\
(3) $\pi_{\hy''_F, \infty}(\b_{\hy|\xi_\hy})=\pi_{\hy'', \infty}(\b_{\hy|\xi_\hy})=\pi_{\hy, \infty}(\b_{\hy|\xi_\hy})$.\\\\
Moreover, $\hy''$ is an iterate of $\hr$ such that $\T_{\hr, \hy''}$ is based on $\hr|\nu_\hr$. (2) then implies that \[\sigma_{u}(\b)=\pi_{\hy'', \infty}(\b_{\hy|\xi_\hy}),\] and so (4) implies the claim.
\end{proof}
Set now $\hk=\hw|\nu_\hw$. Recall that we set $\xi=\nu_\hx$. The next claim finishes the proof.
\begin{claim} $(\hk, \xi)$ satisfy clauses (1)-(4) of \rcor{cor: tech cor iii}. 
\end{claim}
\begin{proof} We verify all the clauses.\\\\
\textbf{Clause 1.} This is a consequence of (1.4). We have that $\hx|\xi$ is a complete iterate of $\hr|\nu_\hr$, and $\hx|\xi=\hw|\xi$.\\
\textbf{Clause 2.} This is a consequence of (1.6) and the fact that $\nu_\hx$ is a Woodin cardinal of $\M^\hx$. (1.6) implies that \[\powerset(\xi)\cap \M^\hw\subset \M^\hx\]\\
\textbf{Clause 3.} This is a consequence of (1.2), (1.5) and (1.6). We have that \[\hx'|\k_{\hx'}=\hw|\k_{\hx'},\]
which implies that \[\pi_{\hs, \hw}(\l')<\k_{\hx'}.\]
Since $\k_{\hx'}<\nu_\hx$, we have that 
\[\pi_{\hs, \hw}(\l')<\xi.\]
Applying (1.2) we get that 
\[\pi_{\hs, \hw}(\l', \k_\hs)=\pi_{\hs|\nu_\hs, \hk}(\l', \k_\hs),\]
and so \[\pi_{\hs|\nu_\hs, \hk}(\l')<\xi.\]
Notice next that (1.6) implies that $\xi<\k_\hw$, and since $\k_\hw=\pi_{\hs|\nu_\hs, \hk}(\k_\hs)$, we have that
\[\pi_{\hs|\nu_\hs, \hk}(\l')<\xi<\pi_{\hs|\nu_\hs, \hk}(\k_\hs).\]\\
\textbf{Clause 4.}
Let now $Z=\pi_{\hs, \infty}[\M^\hs|\nu_\hs]$. It follows from \rcl{clm: key clm catching ralpha} that\\\\
(4) for each $\b<\l_{u}$, letting $\xi=\nu_{\hx}$ and $\hy$ be any complete iterate of $\hw$ such that $\T_{\hw, \hy}$ is based on $\hw|\xi$ and $\b_{\hy|\xi_{\hy}}$ is defined where $\xi_{\hy}=\pi_{\hw, \hy}(\xi)$, \begin{center} $\sigma_u(\b)=\tau_{\hy|\nu_{\hy}}^Z(\b_{\hy|\xi_\hy})$.\end{center}
This is because
\begin{enumerate}[label=(5.\arabic*), itemsep=0.3cm]
\item $\hs$ is an $\its(\pi_{\hq, \infty}(\nu))$-certificate for $Z$,
\item $\xi=\nu_\hx<\k_\hw<\nu_\hw$ (see (1.6)), and
\item $\T_{\hs, \hy}$ is based on $\hs|\nu_\hs$.
\end{enumerate}
So (see \rthm{thm: uniqueness of realizability witnesses} and \rdef{def: rel sets})
\[\tau_{\hy|\nu_{\hy}}^Z=\pi_{\hy, \infty}\rest \M^\hy|\nu_\hy.\]
\end{proof}
\end{proof}

\begin{notation}\label{not: sigmazkxi}\normalfont Suppose $\nu<\hd_\infty$ is a proper cutpoint in $\mH$ relative to $\hd_\infty$ and is a Woodin cardinal of $\mH$ that is not a limit of Woodin cardinals of $\mH$. Suppose 
\begin{itemize}
\item $Z\in \its(\nu)$, 
\item $\hk$ is a complete iterate of $\hq_Z$, and
\item $\xi$ is an inaccessible proper cutpoint cardinal of $\M^\hk$.
\end{itemize}
Let $v=(Z, \hk, \xi)$. We then let \[\sigma_v: \mH|\pi_{\hk|\xi+1, \infty}(\xi)\rightarrow \mH|\gamma\] be the embedding given by \[\sigma_v(x)=\tau^Z_{\hr}(x_\hr),\] where 
\begin{itemize}
\item $\hr$ is a complete iterate of $\hk$ such that $\T_{\hk, \hr}$ is based on $\hk|\xi$, 
\item $x = \pi_{\hr, \infty}(x_\hr)$, and 
\item $\gamma=\sup(\sigma_v[\pi_{\hk|\xi+1, \infty}(\xi)])$.
\end{itemize}
\end{notation}

\begin{remark}\normalfont Given the results of \rsec{sec: it sets} (in particular, \rthm{thm: uniqueness of realizability witnesses}) and \rsec{sec: reflection i} (in particular, \rlem{lem: u is wd}), it is clear that $\sigma_v$ is well-defined. 
\end{remark}

\begin{terminology}\label{term: embeddings generated below}\normalfont Suppose $M\models \rg_3(\hq, \l, \k, \nu)$, and suppose $\hs$ is a complete iterate of $\hq$ such that $\T_{\hq, \hs}$ is above $\k$. Let \[\sigma: N|\a\rightarrow N|\pi_{\hq, \infty}(\l)\] be an elementary embedding. We say that $\sigma$ is \textbf{generated below} $\hs|\nu_\hs$ if there are
\begin{itemize}
\item a complete iterate $\hr$ of $\hq$, and
\item an inaccessible cardinal $\zeta<\k_\hs$ of $\M^\hs$ that is a proper cutpoint in $\M^\hs$,
\end{itemize}
such that \[\sup(\pi_{\hr|\k_\hr, \infty}[\k_\hr])<\pi_{\hs|\zeta, \infty}(\ts(\hs|\zeta)),\]
$\gen(\T_{\hq, \hr})\subseteq \l_\hr$, $\gen(\T_{\hq, \hs})\subseteq \zeta$, and $\sigma=\sigma_{(\hr, \l_\hr, \nu_\hr)}$.
\end{terminology}

We now have the following easy corollary. 

\begin{corollary}\label{cor: genrated by z} Suppose 
\begin{itemize}
\item $M\models \rg_3(\hq, \l, \k, \nu)$, 
\item $\hs$ is a complete iterate of $\hq$ such that $\T_{\hq, \hs}$ is above $\k$, and 
\item $\sigma: N|\a\rightarrow N|\pi_{\hq, \infty}(\l)$ is generated below $\hs|\nu_\hs$.
\end{itemize}
Letting $Z=\pi_{\hq, \infty}[\M^\hq|\nu]$, there is a complete iterate $\hk$ of $\hq_Z$ and an inaccessible proper cutpoint of $\M^\hk$ such that, letting\footnote{$\ind$ is defined in \rdef{def: induced embeddings}.} $v=(Z, \hk, \xi)$, \[\sigma=\ind(\sigma_{v}).\]
\end{corollary}

\begin{corollary}\label{cor: uniqueness} Suppose $M\models \rg_3(\hq, \l, \k, \nu)$, and suppose $\hs$ is a complete iterate of $\hq$ such that $\T_{\hq, \hs}$ is above $\k$. Suppose \[\sigma: N|\a\rightarrow N|\pi_{\hq, \infty}(\l)\ \text{and}\ \sigma': N|\a\rightarrow N|\pi_{\hq, \infty}(\l)\] are generated below $\hs|\nu_\hs$. Then $\sigma=\sigma'$.
\end{corollary}
\begin{proof} Let \(\hr\) and \(\hr'\) witness that \(\sigma\) and \(\sigma'\) are generated below \(\hs|\nu_\hs\). Let \((\hk, \xi)\) and \((\hk', \xi')\) be as in \rcor{cor: tech cor iii} applied to \((\hr, \hs)\) and \((\hr', \hs)\). Let \(\hx\) be a complete iterate of \(\hk\) and \(\hk'\) produced via the least-extender-disagreement coiteration. It follows from \cite[Lemma 9.12]{blue2025nairian} that \[\pi_{\hk, \hx}(\xi)=\pi_{\hk', \hx}(\xi')=_{\text{def}}\phi\] and \(\hx|\phi\) is a common complete iterate of both \(\hk|\xi\) and \(\hk'|\xi'\). 

Let $Z=\pi_{\hs, \infty}[\M^\hs|\nu_\hs]$. Let now $v=(Z, \hk, \xi)$, $v'=(Z, \hk', \xi')$ and $x=(Z, \hx, \phi)$. 

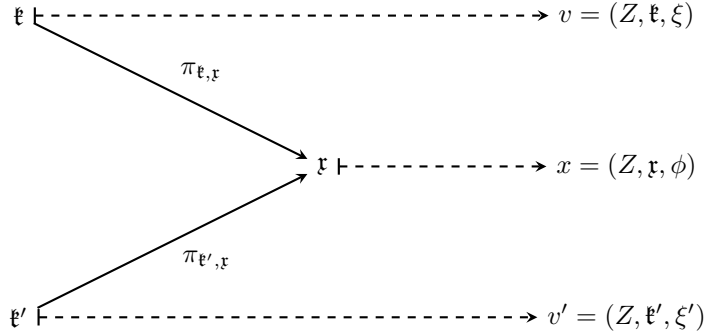
\begin{figure}[htbp]
    \centering
    \begin{tikzpicture}[>=stealth, thick]
        \node (hk) at (0, 2) {$\hk$};
        \node (hkp) at (0, -2) {$\hk'$};
        \node (hx) at (4, 0) {$\hx$};
        
        \node (v) at (8, 2) {$v=(Z, \hk, \xi)$};
        \node (vp) at (8, -2) {$v'=(Z, \hk', \xi')$};
        \node (x) at (8, 0) {$x=(Z, \hx, \phi)$};

        \draw[->] (hk) -- node[above right] {$\pi_{\hk, \hx}$} (hx);
        \draw[->] (hkp) -- node[below right] {$\pi_{\hk', \hx}$} (hx);
        
        \draw[|->, dashed] (hk) -- (v);
        \draw[|->, dashed] (hkp) -- (vp);
        \draw[|->, dashed] (hx) -- (x);

        \node[draw=black, thick, fill=gray!10, rounded corners, inner sep=6pt, anchor=north] 
        at (4, -3.5) {
        \begin{minipage}{9cm}
            \centering
            \textbf{\underline{Uniqueness via Coiteration}}
            \begin{itemize}[label={\tiny$\bullet$}, itemsep=1.2pt, topsep=3pt, leftmargin=1em, font=\scriptsize]
                \item $\hx$ is the common iterate of $\hk$ and $\hk'$.
                \item $\pi_{\hk, \hx}(\xi) = \pi_{\hk', \hx}(\xi') = \phi$.
                \item $\sigma = \ind(\sigma_v) = \ind(\sigma_x) = \ind(\sigma_{v'}) = \sigma'$.
            \end{itemize}
        \end{minipage}
        };
    \end{tikzpicture}
    \caption{Least-extender-disagreement coiteration aligning the extraction tuples.}
    \label{fig:uniqueness_coiteration}
\end{figure}

Applying \rcor{cor: genrated by z} we get that 
\begin{align*}
\sigma&=\ind(\sigma_{v})\\
&=\ind(\sigma_{x})\\
&=\ind(\sigma_{v'})\\
&=\sigma'.
\end{align*}
\end{proof}

\begin{theorem}\label{thm: embedding is in n} Suppose $u=(\hq, \nu, \l)$, $\rg_3(\hq, \l, \k, \nu)$ holds, and $G\subseteq \pmax$ is $M$-generic. Suppose that $(\hr_\a: \a<\omega_1)\in M[G]$ is a sequence of complete iterates of $\hq$ such that for every $\a<\omega_1$, \[\gen(\T_{\hq, \hr_\a})\subseteq \l_{\hr_\a}.\] 
For $\a<\omega_1$, let 
\begin{enumerate}
    \item $\l_\a=\l_{\hr_\a}$, 
    \item $\nu_\a=\nu_{\hr_\a}$,
    \item $u_\a=(\hr_\a,\l_\a, \nu_\a)$, and 
    \item $\sigma_\a=\sigma_{u_\a}: N|\l_{u_\a}\rightarrow N|\l^{u_\a}$.
\end{enumerate}
Then \[(\sigma_\a: \a<\omega_1)\in N[G].\] 
\end{theorem}
\begin{proof} Set $\k_\a=\k_{\hr_\a}$. First notice that \[(\l_{u_\a}: \a<\omega_1)\in N[G],\] and for every $\a<\omega_1$, $\l^{u_\a}=\pi_{\hq, \infty}(\l)=\l_\infty$. The first claim follows because \[\cf^{M[G]}(\Theta^N)=\omega_2^{M[G]},\] and so if 
\begin{itemize}
    \item $\gg<\Theta^N$, 
    \item $(x_\a: \a<\omega_1)\subseteq \bR$, and 
    \item $f:\bR\rightarrow \gg$
\end{itemize}
are such that for every $\a<\omega_1$, \[\l_{u_\a}<\gg, \quad f\in N, \quad \text{and} \quad f(x_\a)=\l_{u_\a},\] and \[(x_\a: \a<\omega_1)\in M[G],\] then \[(x_\a: \a<\omega_1)\in L(\bR)[G]\subseteq N[G],\] and so \[(\l_{u_\a}: \a<\omega_1)\in N[G].\]

Working in $M[G]$ and using \cite[Theorem 11.1]{blue2025nairian}, we can find
\begin{itemize}
    \item a complete iterate $\hs$ of $\hq$ and 
    \item $\zeta<\k_\hs$
\end{itemize}
such that $\T_{\hq, \hs}$ is above $\k$, and $\zeta$ is an inaccessible proper cutpoint cardinal of $\M^\hs$ with the property that \[\gen(\T_{\hq, \hs})\subseteq \zeta,\] and for every $\a<\omega_1$, \[\sup(\pi_{\hr_\a|\k_\a, \infty}[\k_\a])<\pi_{\hs|\zeta, \infty}(\ts(\hs|\zeta)).\] 

Applying \rcor{cor: tech cor iii} in $M[G]$ and setting \[Z=\pi_{\hs, \infty}[\M^\hs|\nu_\hs],\] for each $\a<\omega_1$, we can find
\begin{itemize}
    \item a complete iterate $\hw_\a$ of $\hs|\nu_\hs$, and
    \item $\xi_\a$ that is an inaccessible proper cutpoint of $\M^{\hw_\a}$
\end{itemize}
such that letting $v_\a=(Z, \hw_\a, \xi_\a)$, \[\sigma_\a=\ind(\sigma_{v_\a}).\] 

Because for each $\a<\omega_1$, $v_\a\in M$ is coded by a real $y_\a \in \bR$, the sequence of coding reals $(y_\a: \a<\omega_1) \in M[G]$ falls into $L(\bR)[G] \subseteq N[G]$. Thus, we again have that \[(v_\a: \a<\omega_1)\in N[G],\] and so \[(\sigma_\a: \a<\omega_1)\in N[G].\] 
\end{proof}

\begin{corollary}\label{cor: omega1 closure} 
Suppose $G\subseteq \pmax$ is $M$-generic. Then in $M[G]$, for any $\l<\varsigma_\infty$, 
\[ {}^{\omega_1}(N|\l[G]) \subseteq N[G]. \]
\end{corollary}

\begin{proof} 
Because $\varsigma_\infty$ is a limit of strong cardinals of $\mH|\varsigma_\infty$, it is enough to prove the claim assuming $\l$ is a strong cardinal of $\mH|\varsigma_\infty$. So suppose $\l$ is a ${<}\varsigma_\infty$-strong cardinal of $\mH$ and $f: \omega_1\rightarrow N|\l[G]$ is a function in $M[G]$. Because $M[G]\models \DC_{\aleph_2}$, we can find $h: \omega_1\rightarrow N|\l$ such that for every $\a<\omega_1$, $h(\a)$ is a $\pmax$-name such that $h(\a)_G=f(\a)$. It suffices to show that $h\in N[G]$. 

Let $\k$ be the least ${<}\varsigma_\infty$-strong cardinal of $\mH$ such that $\l<\k$, and let $\nu<\varsigma_\infty$ be a Woodin cardinal of $\mH$ such that for some complete iterate $\hq$ of $\hp$, $(\l_\hq, \k_\hq, \nu_\hq)$ are defined and 
\[ M\models \rg_3(\hq, \l_\hq, \k_\hq, \nu_\hq). \]
Using \rthm{thm: basic reflection} and the fact that $M[G]\models \DC_{\aleph_2}$, we can find a sequence $(\hr_\a: \a<\omega_1)\in M[G]$ of complete iterates of \(\hq\) such that, letting \(u_\a=(\hr_\a, \l_{\hr_\a}, \nu_{\hr_\a})\), for every \(\a<\omega_1\), \(h(\a)\in \rge(\sigma_{u_\a})\). 

It follows from \rthm{thm: embedding is in n} that \((\sigma_{u_\a}: \a<\omega_1)\in N[G]\). To finish the proof, we make the following observations:

\vspace{0.3cm}
\begin{enumerate}[label=(1.\arabic*), itemsep=0.3cm]
    \setcounter{enumi}{0}
    \item Letting $\gamma=\sup_{\a<\omega_1}\l_{u_\a}$, we have $\gamma<\Theta^N$ and \(((\sigma_{u_\a})^{-1}(h(\a)): \a<\omega_1)\subseteq N|\gamma\).
    \item We have a surjection $k: \bR\rightarrow N|\gamma$ with $k\in N$.
    \item Since $\powerset(\omega_1)\cap M[G]=\powerset(\omega_1)\cap L(\bR)\subseteq N[G]$, we obtain that \(((\sigma_{u_\a})^{-1}(h(\a)): \a<\omega_1)\in N[G]\).
\end{enumerate}
\vspace{0.3cm}

Because both the embeddings and the preimages of the names are in $N[G]$, we conclude that the sequence itself, $h$, is in $N[G]$.
\end{proof}

\begin{corollary}\label{cor: omega1 closure2} In $M$, for any $\b<\varsigma_\infty$, $$^{\omega_1}N|\b\subseteq N.$$
\end{corollary}
\begin{proof} Let $f: \omega_1\rightarrow N|\b$ with $f\in M$. Then whenever $G\subseteq \pmax$ is $M$-generic, $f\in N[G]$. It follows that $f\in N$.
\end{proof}

\begin{theorem}\label{thm: closure under theta sequences} Suppose $\zeta<\Theta^N$ and $\l<\varsigma_\infty$. Then $(N|\l)^\zeta\subseteq N$. 
\end{theorem}
\begin{proof} Again, it is enough to assume that $\l$ is a strong cardinal of $\mH|\varsigma_\infty$. Let now 
\begin{itemize}
\item $\l$ be a strong cardinal of $\M^\hp|\varsigma_\hp$,
\item $\k=\min(\tStr(\M^\hp|\varsigma_\hp)-(\l+1))$, and
\item $\nu\in \tW(\M^\hp|\varsigma_\hp)$ such that $\tW(\M^\hp)\cap (\k, \nu)\not =\emptyset$ and $\nu$ is not a limit of Woodin cardinals of $\M^\hp$.
\end{itemize}
We want to see that 

\vspace{0.3cm}
\begin{enumerate}[label=(\alph*), itemsep=0.3cm]
    \item[(a)] if \(f:\zeta\rightarrow N|\l_\infty\), then $f\in N$.
\end{enumerate}
\vspace{0.3cm}

Working in $M$, for each $\a<\zeta$, let $\l_\a$ be the least such that for some complete iterate $\hr$ of $\hp$, if $u=(\hr, \l_\hr, \nu_\hr)$, then \[f(\a)\in \rge(\sigma_u)\ \text{and}\ \l_u=\l_\a.\] 

Because \[M\models \cf(\Theta^N)>\zeta,\] we must have that \[\sup_{\a<\zeta}\l_\a<\Theta^N.\] Therefore, by the Coding Lemma, \[(\l_\a: \a<\zeta)\in N.\] 

Let $G\subseteq \pmax$ be $M$-generic. Because $M[G]\models \DC_{\omega_2}$ and because \[M[G]\models \card{\zeta}=\omega_2,\] in $M[G]$, applying \rthm{thm: basic reflection}, we can find a sequence \[(u_\a=(\hr_\a, \l_{\hr_\a}, \nu_{\hr_\a}): \a<\zeta)\] such that for every $\a<\zeta$, 

\vspace{0.3cm}
\begin{enumerate}[label=(1.\arabic*), itemsep=0.3cm]
    \item[(1.1)] $\l_\a=\l_{u_\a}$,
    \item[(1.2)] $\gen(\T_{\hp, \hr_\a})\subseteq \l_{\hr_\a}$, and
    \item[(1.3)] \(f(\a)\in \rge(\sigma_{u_\a}).\)
\end{enumerate}
\vspace{0.3cm}

We have that

\vspace{0.3cm}
\begin{enumerate}[label=(2), itemsep=0.3cm]
    \item[(2)] $\sigma_{u_\a}\in N$.
\end{enumerate}
\vspace{0.3cm}

Now let $\hs$ be a complete iterate of $\hp$ such that $\T_{\hp, \hs}$ is above $\k$ and for some $\M^\hs$-inaccessible $\xi\in (\l_\hs, \k_\hs)$ that is an inaccessible proper cutpoint of \(\M^\hs\), 

\vspace{0.3cm}
\begin{enumerate}[label=(3.\arabic*), itemsep=0.3cm]
    \item[(3.1)] for every \(\a<\zeta\), \[\sup(\pi_{\hr_\a|\k_{\hr_\a}, \infty}[\k_{\hr_\a}])<\pi_{\hs|\xi, \infty}(\ts(\hs|\xi)),\]
    \item[(3.2)] $\gen(\T_{\hp, \hs})\subseteq \xi$.
\end{enumerate}
\vspace{0.3cm}
\begin{figure}[htbp]
    \centering
    \begin{tikzpicture}[>=stealth, thick]
        \node (hp) at (0, 0) {$\hp$};
        \node (hra) at (3, 2) {$\hr_\a$};
        \node (hs) at (3, -2) {$\hs$};
        \node (Nla) at (7, 2) {$N|\l_\a$};
        \node (Nlinf) at (9, 0) {$N|\l_\infty$};

        \draw[->] (hp) -- node[above left] {$\gen \subseteq \l_{\hr_\a}$} (hra);
        \draw[->] (hp) -- node[below left] {$\gen \subseteq \xi$} (hs);
        
        \draw[|->, dashed] (hra) -- node[above] {extracts via $u_\a$} (Nla);
        
        \draw[->] (Nla) -- node[above right] {$\sigma_{u_\a}$} (Nlinf);
        
        \node[draw=black, thick, fill=gray!10, rounded corners, inner sep=6pt, anchor=north] 
        at (4.5, -3.5) {
        \begin{minipage}{9cm}
            \centering
            \textbf{\underline{Key Properties of the Reflection}}
            \begin{itemize}[label={\tiny$\bullet$}, itemsep=1.2pt, topsep=3pt, leftmargin=1em, font=\scriptsize]
                \item Domain bound: $\zeta < \Theta^N \implies \sup_{\a<\zeta} \l_\a < \Theta^N$.
                \item Disjointness: $\xi$ is a proper cutpoint.
                \item Unique decoding: $\sigma_{u_\a}$ uniquely defined from $(Z, \hk, \xi)$.
                \item $f \in N$ since $(\sigma_{u_\a}: \a < \zeta) \in N$ and $f(\a) \in \rge(\sigma_{u_\a})$.
            \end{itemize}
        \end{minipage}
        };
    \end{tikzpicture}
    \caption{Reflection of $f(\a)$ via the embeddings $\sigma_{u_\a}$ extracted from $\hr_\a$ using the tuple $u_\a$.}
    \label{fig:closure_under_theta}
\end{figure}

Set now \[Z=\pi_{\hs, \infty}[\M^\hs|\nu].\] It follows that for each \(\a<\zeta\), \(\sigma_{u_\a}\) is generated below \(\hs|\nu_\hs\). It then follows that (see \rcor{cor: uniqueness}) \[\sigma_{u_\a}: N|\l_\a\rightarrow N|\l_\infty\] is the unique embedding of the form \(\sigma_{(Z, \hk, \xi)}\), where \((Z, \hk, \xi)\) witness \rcor{cor: tech cor iii} applied to \((\hp, \hr_\a, \hs, \l, \k, \nu)\). Hence, \[(\sigma_{u_\a}: \a<\zeta)\in N\] as it can be defined in \(N\) from \(Z\), \(\rel(Z)\), and \((\l_\a: \a<\zeta)\). Since \(f(\a)\in \rge(\sigma_{u_\a})\) for all \(\a < \zeta\), $f$ can be recovered in $N$, thus $f \in N$.\footnote{Here, by the Coding Lemma, the set $(\sigma_{u_\alpha}^{-1}(f(\a)): \a<\zeta)\in N$.}
\end{proof}

\section{Reflection III: useful lemmas}

In the next few sections, we will establish the reflection principle for critical points strictly greater than $\Theta^N$. Recall that our primary objective is to construct sufficiently closed hulls of a fixed initial segment $N|\l$.
We proceed toward establishing $\rref(X, \nN_\a, \nN_{\a+1})$ within $N$. Our aim is to implement the canonical construction of a continuous chain of elementary submodels of $N|\l$ in the setting of Nairian Models. The primary obstacle is the absence of Skolem functions, which necessitates a more delicate treatment of the successor steps in this construction. The subsequent lemmas are designed to perform the successor step of the continuous chain construction in a constructive manner that supports iteration.

\begin{lemma}\label{lem: pullback consistency for sigma} 
Suppose \[M\models \rg_3(\hq, \l, \k, \nu)\] and $\eta<\l$ is a strong cardinal of $\M^\hq|\l$. Let 
\begin{itemize}
    \item $u=(\hq, \l, \nu)$,
    \item $\xi\in [\k, \nu)$ be an inaccessible proper cutpoint cardinal of $\M^\hq|\nu$,
    \item $\a=\pi_{\hq|\nu, \infty}(\eta)$ and $\a'=\eta_\infty=\pi_{\hq, \infty}(\eta)$,
    \item $\hs=\hq^\a$ and $\hs'=\hq^{\a'}$, 
    \item $j=\pi_{\hq|\nu, \infty}^\a$ and $j'=\pi_{\hq, \infty}^{\a'}$.
\end{itemize}
Then 
\begin{enumerate}
    \item $\sigma_u(\M^\hs|j(\nu))=\M^{\hs'}|j'(\nu)$, 
    \item $\sigma_u(\hs|j(\xi))=\hs'|j'(\xi)$, and 
    \item $\sigma_u(j\rest \M^\hq|\xi)=j'\rest \M^\hq|\xi$ (see \rfig{fig: sigma_precise_gap}).
\end{enumerate}

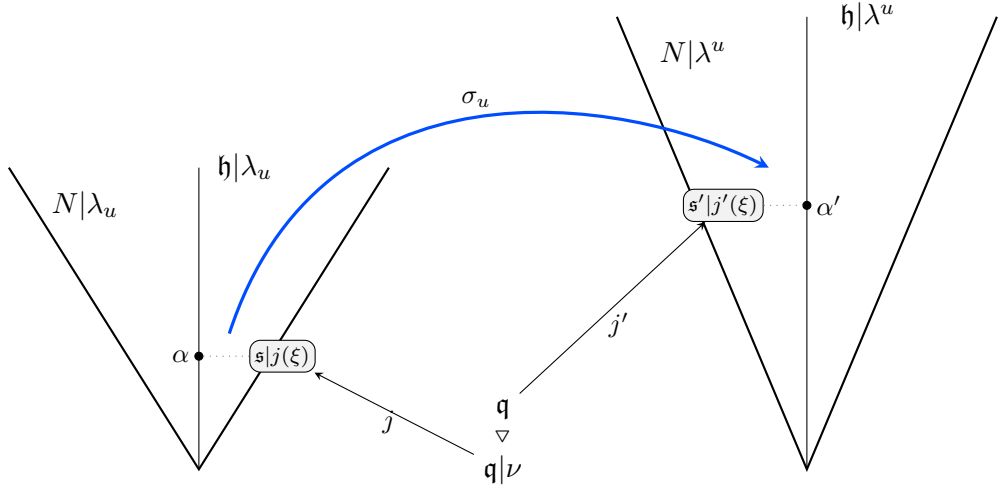
\begin{figure}[ht]
\centering
\begin{tikzpicture}[>=stealth]

    \draw[thick] (-2.5, 4) -- (0, 0) -- (2.5, 4);
    \node at (-1.5, 3.5) {\large $N|\lambda_u$};
    
    \draw[thin] (0, 0) -- (0, 4);
    \node at (0.6, 4) {\large $\hh|\lambda_u$};
    
    \coordinate (alpha) at (0, 1.5);
    \filldraw (alpha) circle (1.5pt) node[left] {$\alpha$}; 
    
    \node (HS_left) at (1.1, 1.5) [draw, fill=gray!10, rounded corners, inner sep=2pt] {\footnotesize $\hs|j(\xi)$};
    \draw[dotted, gray] (HS_left) -- (alpha);

    \draw[thick] (5.5, 6) -- (8, 0) -- (10.5, 6);
    \node at (6.5, 5.5) {\large $N|\lambda^u$};
    
    \draw[thin] (8, 0) -- (8, 6);
    \node at (8.8, 6) {\large $\hh|\lambda^u$};
    
    \coordinate (alpha_prime) at (8, 3.5);
    \filldraw (alpha_prime) circle (1.5pt) node[right] {$\alpha'$};
    
    \node (HS_right) at (6.9, 3.5) [draw, fill=gray!10, rounded corners, inner sep=2pt] {\footnotesize $\hs'|j'(\xi)$};
    \draw[dotted, gray] (HS_right) -- (alpha_prime);

    \node (Q) at (4, 0.8) {\large $\hq$};
    \node (Qnu) at (4, 0) {\large $\hq|\nu$};
    \path (Qnu) -- node[sloped] {$\inseg$} (Q);

    \draw[->] (Qnu) -- node[left, pos=0.4, xshift=-2pt] {$j$} (HS_left);
    \draw[->] (Q) -- node[right, pos=0.4, xshift=2pt] {$j'$} (HS_right);

    \draw[->, very thick, color=blue!70!cyan] 
        (0.4, 1.8) 
        .. controls (1.6, 5.5) and (5.5, 5.0) .. 
        node[midway, above, text=black, yshift=0pt] {$\sigma_u$} 
        (7.5, 4.0);

\end{tikzpicture}
\caption{The reflection map $\sigma_u$ passing precisely through the gap of $N|\lambda_u$.}
\label{fig: sigma_precise_gap}
\end{figure}
\end{lemma}

\begin{proof}
We start by letting 
 \begin{itemize}
 \item $\hq=(\Q, \Lambda)$,
 \item $\zeta=(\eta^+)^{\Q}$,
 \item $Z=j[\Q|\zeta]$ and $Z'=j'[\Q|\zeta]$,
  \item $i=j\rest (\Q|\zeta)$ be the inverse of the transitive collapse of $Z$, 
 \item  $i'=j'\rest (\Q|\zeta)$ be the inverse of the transitive collapse of $Z'$,
 \item $G=\{ (a, A): A\in \Q|\zeta \wedge a\in [\a]^{<\omega} \wedge a\in i(A)\}$, and
 \item $G'=\{ (a, A): A\in \Q|\zeta \wedge a\in [\a']^{<\omega} \wedge a\in i'(A)\}$.
\end{itemize}
 Notice that
 \vspace{0.3cm}
 \begin{enumerate}[label=(\arabic*), itemsep=0.3cm]
    \item[(1.1)] $\sigma_u(\a)=\a'$,
    \item[(1.2)] $j\rest \Q|\zeta=\pi_{\hq|\nu, \infty}\rest \Q|\zeta$ (see \cite[Chapter 9.1]{blue2025nairian}),
    \item[(1.3)] $j'\rest \Q|\zeta=\pi_{\hq, \infty}\rest \Q|\zeta$ (see \cite[Chapter 9.1]{blue2025nairian}),
    \item[(1.4)] for each $z$, $z\in Z$ if and only if $j'(z_{\hq|\nu})\in Z'$,
    \item[(1.5)] $Z'=\sigma_u[Z]$ (see \rdef{def: reflection tuple}),
    \item[(1.6)] $\sigma_u(Z)=Z'$ (this follows from (1.5) and the fact that $Z$ is countable), and
    \item[(1.7)] $\sigma_u(i)=i'$ and $\sigma_u(G)=G'$.
 \end{enumerate}
 \vspace{0.3cm}
Because $\hs=Ult(\hq, G)$, $\hs'=Ult(\hq, G')$, $j=\pi_G^{\M^\hq}$, and $j'=\pi_{G'}^{\M^\hq}$, we have that
\vspace{0.3cm}
\begin{enumerate}[label=(\arabic*), itemsep=0.3cm]
    \setcounter{enumi}{1}
    \item[(2)] for every $\iota\in (\eta, \omega_1^M)$, 
    \[\sigma_u(\M^\hs|j(\iota))=\M^{\hs'}|j'(\iota)\ \text{and}\ \sigma_u(j\rest \M^\hq|\iota)=j'\rest \M^\hq|\iota.\]
\end{enumerate}
\vspace{0.3cm}
Notice that (2) verifies clauses $\textit{(1)}$ and $\textit{(3)}$ of \rlem{lem: pullback consistency for sigma}.
The following claim finishes the proof of \rlem{lem: pullback consistency for sigma} by verifying the remaining clause $\textit{(2)}$.

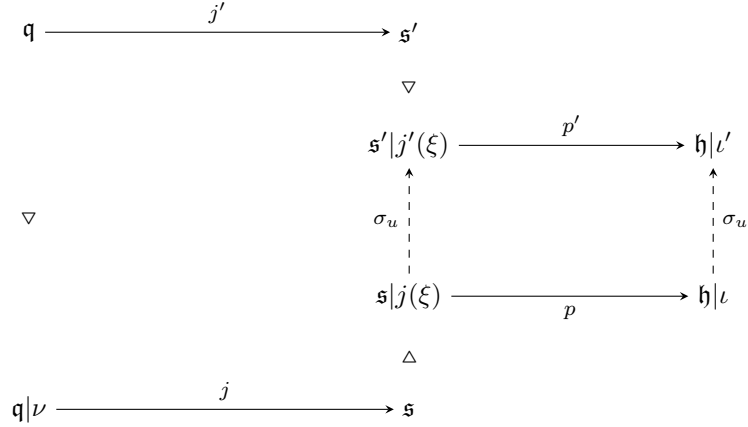
\begin{figure}[ht]
\centering
\begin{tikzpicture}[>=stealth, node distance=2.5cm]
    \node (Q) at (0, 5) {$\hq$};
    \node (Qnu) at (0, 0) {$\hq|\nu$};

    \node (Sprime_seg) at (5, 3.5) {$\hs'|j'(\xi)$};
    \node (S_seg) at (5, 1.5) {$\hs|j(\xi)$};
    
    \node (Sprime) at (5, 5) {$\hs'$};
    \node (S) at (5, 0) {$\hs$};

    \node (Hiota_prime) at (9, 3.5) {$\hh|\iota'$};
    \node (Hiota) at (9, 1.5) {$\hh|\iota$};

    \path (Qnu) -- node[sloped] {$\inseg$} (Q);
    \path (S_seg) -- node[sloped] {$\inseg$} (S);
    \path (Sprime_seg) -- node[sloped] {$\inseg$} (Sprime);

    \draw[->] (Qnu) -- node[above] {\footnotesize $j$} (S);
    \draw[->] (Q) -- node[above] {\footnotesize $j'$} (Sprime);

    \draw[->] (S_seg) -- node[below] {\footnotesize $p$} (Hiota);
    \draw[->] (Sprime_seg) -- node[above] {\footnotesize $p'$} (Hiota_prime);

    \draw[->, dashed] (S_seg) -- node[left] {\footnotesize $\sigma_u$} (Sprime_seg);
    
    \draw[->, dashed] (Hiota) -- node[right] {\footnotesize $\sigma_u$} (Hiota_prime);

    \node[draw=black, thick, fill=gray!10, rounded corners, inner sep=6pt, anchor=north] 
        at (4.5, -1.5) {
        \begin{minipage}{7cm}
            \centering
            \textbf{\underline{Key Facts}}
            \begin{itemize}[label={\tiny$\bullet$}, itemsep=1.2pt, topsep=3pt, leftmargin=1em, font=\scriptsize]
                \item $\sigma_u(j \restriction \M^\hq|\iota) = j' \restriction \M^\hq|\iota$ \text{ for } $\iota \in (\eta, \omega_1^M)$,
                \item $\sigma_u(\hs|j(\xi)) = \hs'|j'(\xi)$,
                \item $\sigma_u(p) = p'$.
            \end{itemize}
        \end{minipage}
    };
    \end{tikzpicture}
\caption{The action of $\sigma_u$.}
\label{fig: sigma_action}
\end{figure}

\begin{claim}\label{subcl: in the range 1} $\sigma_u(\Sigma^{\hs|j(\xi)})=\Sigma^{\hs'|j'(\xi)}$.
\end{claim}
\begin{proof} The reader may benefit from \rfig{fig: sigma_action}. Let $\iota$ and $\iota'$ be such that $\hh|\iota$ is the complete iterate of $\hs|j(\xi)$ and  $\hh|\iota'$ is the complete iterate of $\hs'|j'(\xi)$. Let \begin{center}$p=\pi_{\hs|j(\xi), \hh|\iota}\rest M^\hs|j(\xi)$ and $p'=\pi_{\hs'|j'(\xi), \hh|\iota'}\rest \M^{\hs'}|j'(\xi)$.\end{center}
Because $\hs|j(\xi)$ is the $p$-pullback of $\hh|\iota$ and $\hs'|j'(\xi)$ is the $p'$-pullback of $\hh|\iota'$, it is enough to show that $\sigma_u(p)=p'$.

Let $E\in \vec{E}^{\Q}$ be the extender with the least index such that $\cp(E)>\eta$ and $\lh(E)>\xi$. Let 
\begin{itemize}
\item $\hr=Ult(\hq, E)$, 
\item $F=j(E)$, 
\item $F'=j'(E)$, 
\item $\hx=Ult(\hs, F)$, and 
\item $\hx'=Ult(\hs', F')$.
\end{itemize}
Applying (2) we get that \[\sigma_u(F)=F'.\] It follows from the full normalization that
\vspace{0.3cm}
\begin{enumerate}[label=(\arabic*), itemsep=0.3cm]
    \setcounter{enumi}{2}
    \item[(3)] $\hr^{\a}=\hx$ and $\hr^{\a'}=\hx'$.
\end{enumerate}
\vspace{0.3cm}
Moreover, we have that
\vspace{0.3cm}
\begin{enumerate}[label=(\arabic*), itemsep=0.3cm]
    \item[(4.1)] $\hs|j(\xi)=\hx|j(\xi)$ and $\hs'|j'(\xi)=\hx'|j'(\xi)$,
    \item[(4.2)] $p=\pi_{\hx|j(\xi), \hh|\iota}$ and $p'=\pi_{\hx'|j'(\xi), \hh|\iota'}$,
    \item[(4.3)] $\pi^\a_{\hr|\nu_\hr, \infty}\rest \M^\hr|\xi=j\rest \M^\hq|\xi$, and
    \item[(4.4)] $\pi^{\a'}_{\hr, \infty}\rest \M^\hr|\xi=j'\rest \M^\hq|\xi$.
\end{enumerate}
\vspace{0.3cm}
Suppose now that $x\in \M^{\hx}|j(\xi)$. We have that for some $f\in \M^\hr|\xi$ and $s\in [\a]^{<\omega}$, \[x=\pi^\a_{\hr|\nu_\hr, \infty}(f)(a).\] Then applying (4.2) and the results of \cite[Chapter 9.1]{blue2025nairian}, we get that
\vspace{0.3cm}
\begin{enumerate}[label=(\arabic*), itemsep=0.3cm]
    \setcounter{enumi}{4}
    \item[(5)] $p(x)=\pi_{\hr|\nu_\hr, \infty}(f)(a)$. 
\end{enumerate}
\vspace{0.3cm}
Applying (4.2) and (4.3) we get that 
\begin{align*}
\sigma_u(p(x)) &= \sigma_u(\pi_{\hr|\nu_\hr, \infty}(f)(a))\\
& =  \sigma_u(\pi_{\hr|\nu_\hr, \infty}(f))(\sigma_u(a)) \\
& = \pi_{\hr, \infty}(f)(\sigma_u(a)) \\
& = p'(j'(f))(\sigma_u(a)))\ \text{(see \rdef{def: reflection tuple})}\\
& = p'(j'(f)(\sigma_u(a)))\ \text{(because $\cp(p')>\a'$)}\\
& = p'(\sigma_u(j(f))(\sigma_u(a))) \\
& = p'(\sigma_u(j(f)(a)))\\
& = p'(\sigma_u(x)).
\end{align*}
Thus, $\sigma_u(p)=p'$.
\end{proof}
\end{proof}

We will also need the following lemma.

\begin{lemma}\label{lem: pullback consistency for sigma1} 
Suppose \[M\models \rg_3(\hq, \l, \k, \nu)\] and $\l$ is a strong cardinal of $\M^\hq|\l$. Let $\hs$ be a complete $\l$-bounded iterate of $\hq$ and let $u=(\hq, \l, \nu)$ and $v=(\hs, \l_\hs, \nu_\hs)$. Then $\rge(\sigma_u)\subseteq \rge(\sigma_v)$.
\end{lemma}
\begin{proof} We have that \[\sigma_u: N|\pi_{\hq|\nu, \infty}(\l)\rightarrow N|\l_\infty\] and 
\[\sigma_v: N|\pi_{\hs|\nu_\hs, \infty}(\l_\hs)\rightarrow N|\l_\infty.\]
It is then enough to show that \[\rge(\sigma_u\rest \mH|\pi_{\hq|\nu, \infty}(\l))\subseteq \rge(\sigma_v\rest \mH|\pi_{\hs|\nu_\hs, \infty}(\l_\hs)).\]
Fix then $x\in \mH|\pi_{\hq|\nu, \infty}(\l)$. Let $\hr$ be a complete iterate of $\hq$ such that $\T_{\hq, \hr}$ is based on $\hq|\nu_\hq$ and for some $y\in \M^\hr|\l_\hr$, $\pi_{\hr|\l_\hr, \infty}(y)=x$. Let $\hw$ be the result of the least-extender-disagreement comparison of $\hr$ and $\hs$. It follows from \rcor{cor: comparing tau bounded iterates} that $\T_{\hs, \hw}$ is based on $\hs|\nu_\hs$. Let then $z=\pi_{\hw|\nu_\hw, \infty}(\pi_{\hr, \hw}(y))$. We then have that
\begin{align*}
\sigma_u(x) & =\pi_{\hr, \infty}(y)\\
&= \pi_{\hw, \infty}(\pi_{\hr, \hw}(y))\\
&= \sigma_v(\pi_{\hw|\nu_\hw, \infty}(\pi_{\hr, \hw}(y)))\\
&= \sigma_v(z)
\end{align*}
It follows that $\rge(\sigma_u)\subseteq \rge(\sigma_v)$. 
\end{proof}

Lemma \ref{lem: defining smaller hulls} is our first step towards showing that we can indeed build a continuous sequence of elementary chains. The key point here is that $\sigma_0^-\in \rge(\sigma_1^-)$, which goes beyond \rlem{lem: pullback consistency for sigma}. The cost is that we must have three Woodin cardinals. 

\begin{lemma}\label{lem: defining smaller hulls}
Suppose
\begin{itemize}
    \item $\eta<\l<\k<\varsigma$ are ${<}\varsigma$-strong cardinals of $\hh$,
    \item $\nu^0<\gg<\nu^1\in (\k,\varsigma_\infty)$ are three Woodin cardinals of $\hh$ such that $\hh$ has a Woodin cardinal in the interval $(\k, \nu^0)$ and $\gg$ is the least Woodin cardinal of $\hh$ above $\nu^0$,
    \item $\hr=(\R, \Psi)$ is a complete iterate of $\hp$ such that (see \rfig{fig:model_timelines})
    \[\{\eta, \l, \k, \nu^0, \gg, \nu^1\}\subseteq \rge(\pi_{\hr, \infty}),\]
    \item $\hq_0$ and $\hq_1$ are two $\eta$-bounded\footnote{See \cite[Definition 9.1]{blue2025nairian}.} complete iterates of $\hr$ such that 
    \[\pi_{\hq_0, \infty}[\eta_{\hq_0}]\subseteq \pi_{\hq_1, \infty}[\eta_{\hq_1}].\]
\end{itemize}
For $i\in 2$, let
\begin{itemize}
    \item $u_i=(\hq_i, \l_{\hq_i}, \nu^i_{\hq_i})$, 
    \item $u_i^-=(\hq_i, \eta_{\hq_i}, \nu^i_{\hq_i})$, 
    \item $\sigma_i=\sigma_{u_i}$, and 
    \item $\sigma_i^-=\sigma_{u_i^-}$. 
\end{itemize}
Then (see \rfig{fig:nested_hulls_with_box}) \[\rge(\sigma_0^-)\subseteq \rge(\sigma_1^-)\] and \[\sigma_0^-\in \rge(\sigma_1^-).\]

\begin{figure}[ht]
    \centering
    \begin{tikzpicture}[>=Stealth, auto, node distance=2cm, font=\small]
        
        \node (HR) at (0, 0) {$\hr$};

        \tikzstyle{modelline} = [thick, -]
        \tikzstyle{tick} = [thin, -]

        \draw[modelline] (-4, 2.0) -- (-4, 5.1);
        \node[below] at (-4, 2.0) {$\hq_0$};
        
        \draw[tick] (-4.1, 2.5) -- (-3.9, 2.5) node[left=2pt, font=\scriptsize] {$\eta_{\hq_0}$};
        \draw[tick] (-4.1, 3.1) -- (-3.9, 3.1) node[left=2pt, font=\scriptsize] {$\l_{\hq_0}$};
        \draw[tick] (-4.1, 3.7) -- (-3.9, 3.7) node[left=2pt, font=\scriptsize] {$\nu^0_{\hq_0}$};
        \draw[tick] (-4.1, 4.3) -- (-3.9, 4.3) node[left=2pt, font=\scriptsize] {$\gg_{\hq_0}$};
        \draw[tick] (-4.1, 4.9) -- (-3.9, 4.9) node[left=2pt, font=\scriptsize] {$\nu^1_{\hq_0}$};

        \draw[modelline] (4, 2.5) -- (4, 6.6);
        \node[below] at (4, 2.5) {$\hq_1$};

        \draw[tick] (3.9, 3.2) -- (4.1, 3.2) node[right=2pt, font=\scriptsize] {$\eta_{\hq_1}$};
        \draw[tick] (3.9, 4.0) -- (4.1, 4.0) node[right=2pt, font=\scriptsize] {$\l_{\hq_1}$};
        \draw[tick] (3.9, 4.8) -- (4.1, 4.8) node[right=2pt, font=\scriptsize] {$\nu^0_{\hq_1}$};
        \draw[tick] (3.9, 5.6) -- (4.1, 5.6) node[right=2pt, font=\scriptsize] {$\gg_{\hq_1}$};
        \draw[tick] (3.9, 6.4) -- (4.1, 6.4) node[right=2pt, font=\scriptsize] {$\nu^1_{\hq_1}$};

        \draw[modelline] (0, 3.5) -- (0, 8.1);
        \node[below] at (0, 3.5) {$\hh$};

        \draw[tick] (-0.1, 4.3) -- (0.1, 4.3) node[right=2pt, font=\scriptsize] {$\eta_{\hh}$};
        \draw[tick] (-0.1, 5.2) -- (0.1, 5.2) node[right=2pt, font=\scriptsize] {$\l_{\hh}$};
        \draw[tick] (-0.1, 6.1) -- (0.1, 6.1) node[right=2pt, font=\scriptsize] {$\nu^0_{\hh}$};
        \draw[tick] (-0.1, 7.0) -- (0.1, 7.0) node[right=2pt, font=\scriptsize] {$\gg_{\hh}$};
        \draw[tick] (-0.1, 7.9) -- (0.1, 7.9) node[right=2pt, font=\scriptsize] {$\nu^1_{\hh}$};

        \draw[->, thick] (HR) to[bend left=15] (-3.5, 1.8);
        \draw[->, thick] (HR) to[bend right=15] (3.5, 2.3);

        \draw[->, dashed, thick] (-3.5, 2.2) to[bend right=20] node[below, font=\scriptsize] {$\pi_{\hq_0, \hh}$} (-0.5, 3.7);
        
        \draw[->, dashed, thick] (3.5, 2.7) to[bend left=20] node[below, font=\scriptsize] {$\pi_{\hq_1, \hh}$} (0.5, 3.7);

        \node[draw=black, dashed, rounded corners, fill=gray!5, inner sep=6pt] at (0, -1.5) {
            $\pi_{\hq_0, \infty}[\eta_{\hq_0}] \subseteq \pi_{\hq_1, \infty}[\eta_{\hq_1}]$
        };

    \end{tikzpicture}
    \caption{The relationship between $\hr$, $\hq_0$, $\hq_1$ and $\hh$.}
    \label{fig:model_timelines}
\end{figure}
\end{lemma}

\begin{proof}
We set 
\begin{itemize}
    \item $\hw=\hr^\eta$, 
    \item $\hy=\hq_1$ and $\hx=\hq_0$,
    \item $\zeta=\pi_{\hr|\nu^1_{\hr}, \infty}(\eta_{\hr})$, 
    \item $\hs=\hy^{\zeta}$, 
    \item $\a=(\gg^+)^\hh$.
\end{itemize}

\begin{figure}[ht]
    \centering
    \begin{tikzpicture}[>=Stealth, auto, node distance=2cm]
        
        
        \node (L0) at (-5, 0) {$N|\l_{u_0}$};
        \node (T0) at (-2, 0) {$N|\eta_{u_0}$};
        \path (L0) -- node[font=\large] {$\unrhd$} (T0);

        \node (L1) at (5, 0) {$N|\l_{u_1}$};
        \node (T1) at (2, 0) {$N|\eta_{u_1}$};
        \path (T1) -- node[font=\large] {$\unlhd$} (L1);

        
        \node (L) at (0, 5) {$N|\l$};
        \node (T) at (0, 2.5) {$N|\eta$};
        \path (T) -- node[font=\large, rotate=90] {$\unlhd$} (L);


        \draw[->, thick] (L0) to[bend left=10] node[left, pos=0.7] {$\sigma_0$} (L);
        \draw[->, thick] (L1) to[bend right=10] node[right, pos=0.7] {$\sigma_1$} (L);

        \draw[->, thick, blue] (T0) -- node[above, xshift=-2pt] {$\sigma_0^-$} (T);
        \draw[->, thick, blue] (T1) -- node[above, xshift=2pt] {$\sigma_1^-$} (T);

        \draw[->, dashed, gray, thick] (T0) -- node[below, font=\scriptsize] {$(\sigma_1^-)^{-1}\circ \sigma_0^-$} (T1);

        \node[draw=black, thick, fill=gray!10, rounded corners, inner sep=6pt] 
            at (0, -2.8) {
            \begin{minipage}{10cm}
                \centering
                \textbf{\underline{Key Properties}}
                \begin{itemize}[label={\tiny$\bullet$}, itemsep=1.2pt, topsep=3pt, leftmargin=1em, font=\scriptsize]
                    \item $\sigma_0^- = \sigma_0 \restriction N|\eta_{u_0}$,
                    \item $\sigma_1^- = \sigma_1 \restriction N|\eta_{u_1}$,
                    \item \textbf{Goal:} $\rge(\sigma_0^-) \subseteq \rge(\sigma_1^-)$ and $\sigma_0^- \in \rge(\sigma_1^-)$,
                    \item \textbf{Key Point:} $\gg$ is needed for an application of \rcor{cor: def its and rel}.
                \end{itemize}
            \end{minipage}
        };

    \end{tikzpicture}
    \caption{The $\sigma$ embeddings.}
    \label{fig:nested_hulls_with_box}
\end{figure}

\begin{figure}[p]
    \centering
    
    \begin{minipage}{\textwidth}
        \centering
        \textbf{(a) $\hx$ to $\hw$ to $\hh$}
        \vspace{0.3cm}
        
        \begin{tikzpicture}[>=Stealth, auto, node distance=2cm, font=\small]
            \tikzstyle{modelline} = [thick, -]
            \tikzstyle{tick} = [thin, -]
            \tikzstyle{maparrow} = [->, gray, densely dashed, thin]
            \tikzstyle{iterarrow} = [->, black, dashed, thick]
            \tikzstyle{agree} = [dashed, gray, thin]

            \draw[modelline] (-4, 0) -- (-4, 4.5);
            \node[below] at (-4, 0) {$\hx$};
            \draw[tick] (-4.1, 1.2) -- (-3.9, 1.2) node[left=2pt, font=\scriptsize] {$\eta_{\hx}$};
            \draw[tick] (-4.1, 2.5) -- (-3.9, 2.5) node[left=2pt, font=\scriptsize] {$\nu^0_{\hx}$};
            \draw[tick] (-4.1, 3.5) -- (-3.9, 3.5) node[left=2pt, font=\scriptsize] {$\gg_{\hx}$};
            \draw[tick] (-4.1, 4.2) -- (-3.9, 4.2) node[left=2pt, font=\scriptsize] {$\a_{\hx}=\gg_\hx^+$};

            \draw[modelline] (0, 0) -- (0, 6.0);
            \node[below] at (0, 0) {$\hw = \hx^\eta$};
            \draw[tick] (-0.1, 1.5) -- (0.1, 1.5) node[below right=-2pt, font=\scriptsize] {$\eta$};
            \draw[tick] (-0.1, 2.8) -- (0.1, 2.8) node[below right=-2pt, font=\scriptsize] {$\nu^0_\hw$};
            \draw[tick] (-0.1, 3.8) -- (0.1, 3.8) node[below right=-2pt, font=\scriptsize] {$\gg_\hw$};
            \draw[tick] (-0.1, 4.8) -- (0.1, 4.8) node[below right=-2pt, font=\scriptsize] {$\a_\hw$};

            \draw[modelline] (4, 0) -- (4, 7.5);
            \node[below] at (4, 0) {$\hh$};
            \draw[tick] (3.9, 1.5) -- (4.1, 1.5) node[right=2pt, font=\scriptsize] {$\eta$};
            \draw[tick] (3.9, 3.5) -- (4.1, 3.5) node[right=2pt, font=\scriptsize] {$\xi$};
            \draw[tick] (3.9, 6.5) -- (4.1, 6.5) node[right=2pt, font=\scriptsize] {$\b$};

            \draw[agree] (0.1, 1.5) -- (3.9, 1.5);

            \draw[->, thick] (-3.5, 0.5) -- node[above, font=\scriptsize] {$\pi_{\hx, \hw}$} (-0.5, 0.5);
            \draw[->, thick] (0.5, 0.5) -- node[above, font=\scriptsize] {$\pi_{\hw, \hh}$} (3.5, 0.5);

            \draw[maparrow] (-3.9, 2.5) to[out=10, in=170] (-0.1, 2.8);
            \draw[maparrow] (-3.9, 3.5) to[out=10, in=170] (-0.1, 3.8);
            \draw[maparrow] (-3.9, 4.2) to[out=10, in=170] (-0.1, 4.8);
            \draw[maparrow] (0.1, 2.8) to[out=10, in=170] (3.9, 3.5);

            \draw[iterarrow] (0.2, 5.0) to[out=30, in=150] node[above, font=\scriptsize, align=center, pos=0.3, yshift=20pt] {iteration of\\$\hw|\a_\hw$ to $\hh|\b$} (3.8, 6.5);
        \end{tikzpicture}
    \end{minipage}

    \vspace{0.8cm}
    \hrule
    \vspace{0.8cm}

    \begin{minipage}{\textwidth}
        \centering
        \textbf{(b) $\hy$ to $\hw$ to $\hh$}
        \vspace{0.3cm}
        
        \begin{tikzpicture}[>=Stealth, auto, node distance=2cm, font=\small]
            \tikzstyle{modelline} = [thick, -]
            \tikzstyle{tick} = [thin, -]
            \tikzstyle{maparrow} = [->, gray, densely dashed, thin]
            \tikzstyle{iterarrow} = [->, black, dashed, thick]
            \tikzstyle{agree} = [dashed, gray, thin]

            \draw[modelline] (-4, 0) -- (-4, 4.5);
            \node[below] at (-4, 0) {$\hy$};
            \draw[tick] (-4.1, 1.2) -- (-3.9, 1.2) node[left=2pt, font=\scriptsize] {$\eta_{\hy}$};
            \draw[tick] (-4.1, 2.5) -- (-3.9, 2.5) node[left=2pt, font=\scriptsize] {$\nu^0_{\hy}$};
            \draw[tick] (-4.1, 3.5) -- (-3.9, 3.5) node[left=2pt, font=\scriptsize] {$\gg_{\hy}$};
            \draw[tick] (-4.1, 4.2) -- (-3.9, 4.2) node[left=2pt, font=\scriptsize] {$\a_{\hy}=\gg_\hy^+$};

            \draw[modelline] (0, 0) -- (0, 6.0);
            \node[below] at (0, 0) {$\hw = \hy^\eta$};
            \draw[tick] (-0.1, 1.5) -- (0.1, 1.5) node[below right=-2pt, font=\scriptsize] {$\eta$};
            \draw[tick] (-0.1, 2.8) -- (0.1, 2.8) node[below right=-2pt, font=\scriptsize] {$\nu^0_\hw$};
            \draw[tick] (-0.1, 3.8) -- (0.1, 3.8) node[below right=-2pt, font=\scriptsize] {$\gg_\hw$};
            \draw[tick] (-0.1, 4.8) -- (0.1, 4.8) node[below right=-2pt, font=\scriptsize] {$\a_\hw$};

            \draw[modelline] (4, 0) -- (4, 7.5);
            \node[below] at (4, 0) {$\hh$};
            \draw[tick] (3.9, 1.5) -- (4.1, 1.5) node[right=2pt, font=\scriptsize] {$\eta$};
            \draw[tick] (3.9, 3.5) -- (4.1, 3.5) node[right=2pt, font=\scriptsize] {$\xi$};
            \draw[tick] (3.9, 6.5) -- (4.1, 6.5) node[right=2pt, font=\scriptsize] {$\b$};

            \draw[agree] (0.1, 1.5) -- (3.9, 1.5);

            \draw[->, thick] (-3.5, 0.5) -- node[above, font=\scriptsize] {$\pi_{\hy, \hw}$} (-0.5, 0.5);
            \draw[->, thick] (0.5, 0.5) -- node[above, font=\scriptsize] {$\pi_{\hw, \hh}$} (3.5, 0.5);

            \draw[maparrow] (-3.9, 2.5) to[out=10, in=170] (-0.1, 2.8);
            \draw[maparrow] (-3.9, 3.5) to[out=10, in=170] (-0.1, 3.8);
            \draw[maparrow] (-3.9, 4.2) to[out=10, in=170] (-0.1, 4.8);
            \draw[maparrow] (0.1, 2.8) to[out=10, in=170] (3.9, 3.5);

            \draw[iterarrow] (0.2, 5.0) to[out=30, in=150] node[above, font=\scriptsize, align=center, pos=0.3, yshift=20pt] {iteration of\\$\hw|\a_\hw$ to $\hh|\b$} (3.8, 6.5);
        \end{tikzpicture}
    \end{minipage}

    \caption{The relationship between $\hx$, $\hy$, $\hw$, and $\hh$, showing the agreement at $\eta$.}
    \label{fig:separated_ordinal_flows_v4}
\end{figure}

Below we will write $\a_\hw$ and $\a_\hs$ for the preimage of $\a$ in those pairs, and will use this notation for all other relevant ordinals and objects. Notice that because $\hy$ and $\hx$ are $\eta$-bounded iterates of $\hr$ (see \cite[Lemma 9.6]{blue2025nairian}),
\vspace{0.3cm}
\begin{enumerate}[label=(\arabic*), itemsep=0.3cm]
    \item[(1.1)] $\hw=\hy^\eta=\hx^\eta$,
    \item[(1.2)] $\M^{\hy}=\{\pi_{\hr, \hy}(f)(a): f\in \P \wedge a\in [\eta_{\hy}]^{<\omega}\}$,
    \item[(1.3)] $\M^{\hx}=\{\pi_{\hr, \hx}(f)(a): f\in \P \wedge a\in [\eta_{\hx}]^{<\omega}\}$,
    \item[(1.4)] letting $k=(\pi_{\hy, \infty})^{-1}\circ \pi_{\hx, \infty}$, \[k: \M^{\hx}\rightarrow \M^\hy\] is an elementary embedding and \[\pi_{\hx, \infty}=\pi_{\hy, \infty}\circ k,\] and
    \item[(1.5)] $\pi^\eta_{\hx, \infty}\rest (\M^\hx|\nu^0_\hx)=\pi^\eta_{\hy, \infty}\circ k \rest (\M^\hx|\nu^0_\hx)$.
\end{enumerate}
\vspace{0.3cm}

\begin{figure}[ht]
    \centering
    \begin{tikzpicture}[>=Stealth, auto, node distance=3.5cm, font=\small]

        \node (MX) at (0, 0) {$\mathcal{M}^\mathcal{X}$};
        \node (MY) at (5, 0) {$\mathcal{M}^\mathcal{Y}$};
        \node (H) at (2.5, 3) {$\mathcal{H}$};

        \draw[->, thick] (MX) -- node[below] {$k$} (MY);
        \draw[->, thick] (MX) to[out=60, in=210] node[above left] {$\pi_{\hx, \infty}$} (H);
        \draw[->, thick] (MY) to[out=120, in=330] node[above right] {$\pi_{\hy, \infty}$} (H);

        \node[draw=black, thick, fill=gray!10, rounded corners, inner sep=5pt] 
            at (2.5, -3.2) {
            \begin{minipage}{9.5cm}
                \footnotesize
                \textbf{\underline{Claim:}} $\rge(\pi_{\hx, \infty}) \subseteq \rge(\pi_{\hy, \infty})$.
                
                \vspace{0.1cm}
                \textit{Proof:} 
                Recall that $\hx=\hq_0$ and $\hy=\hq_1$ are $\eta$-bounded iterates of $\hr$. 
                Any $z \in \rge(\pi_{\hx, \infty})$ can be written as:
                \[ z = \pi_{\hr, \infty}(f)(s) \]
                for some $f \in \M^\hr$ and a finite $s\subseteq \pi_{\hx, \infty}[\eta_{\hx}]$.
                
                \vspace{0.1cm}
                By hypothesis, $\pi_{\hq_0, \infty}[\eta_{\hq_0}] \subseteq \pi_{\hq_1, \infty}[\eta_{\hq_1}]$. 
                Thus, there exists a finite $t \subseteq \eta_{\hy}$ such that $s= \pi_{\hy, \infty}(t)$. Substituting this back:
                \[ z = \pi_{\hr, \infty}(f)(\pi_{\hy, \infty}(t)) = \pi_{\hy, \infty}(\pi_{\hr, \hy}(f)(t)). \]
                Therefore, $z \in \rge(\pi_{\hy, \infty})$. 
            \end{minipage}
        };

    \end{tikzpicture}
    \caption{$\rge(\pi_{\hx, \infty}) \subseteq \rge(\pi_{\hy, \infty})$}
    \label{fig:mice_commutes_proof}
\end{figure}

It follows from \rlem{lem: pullback consistency for sigma} that
\vspace{0.3cm}
\begin{enumerate}[label=(\arabic*), itemsep=0.3cm]
    \setcounter{enumi}{1}
    \item[(2.1)] $\sigma_1(\hs|\a_\hs)=\hw|\a_\hw$, 
    \item[(2.2)] $\sigma_1(\hs|\nu^0_\hs)=\hw|\nu^0_\hw$, and 
    \item[(2.3)] $\pi^\eta_{\hy, \infty}\rest \M^\hy|\a_\hy=\sigma_1(\pi^\zeta_{\hy|\nu^1_\hy, \infty}\rest \M^\hy|\a_\hy)$.
\end{enumerate}
\vspace{0.3cm}

Let $\b$ be such that $\hh|\b$ is a complete iterate of $\hw|\a_\hw$, and set 
\[Y= \pi_{\hx, \infty}^\eta[\M^\hx|\a_\hx].\] 
Let 
\[Z=\pi_{\hw|\a_\hw, \hh|\b}[Y].\]
We now have that
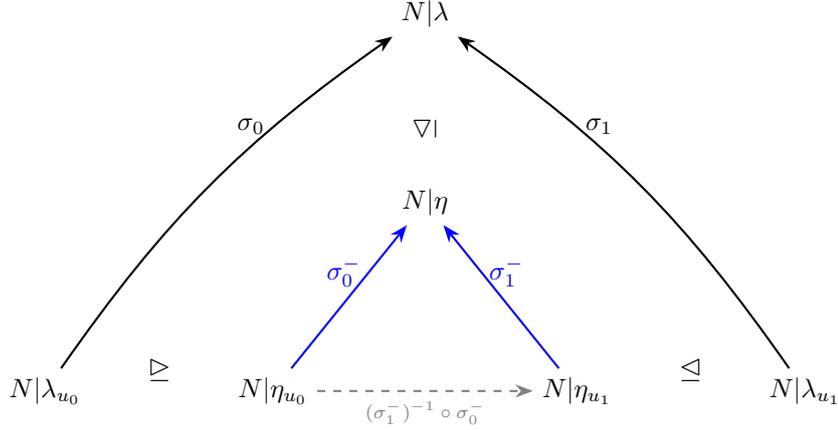
\vspace{0.3cm}
\begin{enumerate}[label=(\arabic*), itemsep=0.3cm]
    \setcounter{enumi}{2}
    \item[(3)] $Y\subseteq \rge(\sigma_1)$.
\end{enumerate}
\vspace{0.3cm}
Indeed, we have that for $x\in Y$, letting $y\in \M^\hx|\a_\hx$ be such that $x=\pi_{\hx, \infty}^\eta(y)$, 
\begin{align*}
x&=\pi^\eta_{\hx, \infty}(y)\\ &=\pi^\eta_{\hy, \infty}(k(y))\\
&=\sigma_1(\pi^\zeta_{\hy|\nu^1_\hy, \infty}(k(y)) \quad \text{(see (2.3))}.
\end{align*}
Applying (2.1) and (3) we get that
\vspace{0.3cm}
\begin{enumerate}[label=(\arabic*), itemsep=0.3cm]
    \setcounter{enumi}{3}
    \item[(4)] $Z\subseteq \rge(\sigma_1)$.
\end{enumerate}
\vspace{0.3cm}

We now have the following claim. The reader may wish to review \rnot{def: the borel codes at nu}. Set \[\xi=\pi_{\hw|\a_\hw, \hh|\b}(\nu^0_\hw).\]
\begin{claim}\label{clm: z is rge sigma1}
$Z\cap \xi\in \its(\xi)$, $(S_{\xi}, T_{\xi})\in \rge(\sigma_1)$ and \[\rel(Z)\in \rge(\sigma_1).\]
\end{claim}
\begin{proof}
The proof that $Z\cap \xi\in \its(\xi)$ is just like the proof of \rcor{cor: sharper bound}. Because of this we give a quick outline. Let $E\in \vec{E}^{\M^\hx}$ be the extender with the least index such that $\lh(E)>\a$ and $\cp(E)>\eta$. We then have that \[Z=\pi_{\hx_E, \infty}[\M^{\hx_E}|\a_\hx].\]
It now follows that $Z\cap \xi\in \its(\xi)$ as certified by $\hx$, and it also follows from \rlem{lem: catching t and s} that \[(S_{\xi}, T_{\xi})\in Z.\] Since $Z\subseteq \rge(\sigma_1)$ (see (4)) we get that \[(S_{\xi}, T_{\xi})\in \rge(\sigma_1).\] Finally, \rcor{cor: def its and rel} (specifically clause $\textit{(2)}$) implies that \[\rel(Z)\in \rge(\sigma_1).\]
\end{proof}

Because $\sigma_0^-$ is definable from $\rel(Z)$ over $N|\l^u$, we have that there is \[\sigma': N|\eta_{u_0}\rightarrow N|\zeta\] such that $\sigma'\in N|\l_{u_1}$ and \[\sigma_1(\sigma')=\sigma_0^-.\] But because\footnote{This follows because $\sup(\sigma_0^-[\eta_{u_0]})<\eta_\infty$ (see \cite[Theorem 10.20]{blue2025nairian}) and because of \cite[Theorem 10.6]{blue2025nairian}.} \[\sigma_0^-\in N|\eta_\infty,\] we have that \[\sigma'\in \dom(\sigma_1^-).\] Thus, $\sigma_1^-(\sigma')=\sigma_0^-$.
\end{proof}

We will also need the following corollary which follows from the proof of \rcl{clm: z is rge sigma1}.

\begin{corollary}\label{cor: defining smaller hulls}
Suppose
\begin{itemize}
    \item $\eta<\l<\k<\varsigma$ are ${<}\varsigma$-strong cardinals of $\hh$,
    \item $\nu^0<\gg<\nu^1\in (\k,\varsigma_\infty)$ are three Woodin cardinals of $\hh$ such that $\hh$ has a Woodin cardinal in the interval $(\k, \nu^0)$ and $\gg$ is the least Woodin cardinal of $\hh$ above $\nu^0$,
    \item $\hr=(\R, \Psi)$ is a complete iterate of $\hp$ such that (see \rfig{fig:model_timelines})
    \[\{\eta, \l, \k, \nu^0, \gg, \nu^1\}\subseteq \rge(\pi_{\hr, \infty}),\]
    \item $\hq_0$ and $\hq_1$ are two $\eta$-bounded\footnote{See \cite[Definition 9.1]{blue2025nairian}.} complete iterates of $\hr$ such that 
    \[\pi_{\hq_0, \infty}[\eta_{\hq_0}]\subseteq \pi_{\hq_1, \infty}[\eta_{\hq_1}].\]
\end{itemize}
For $i\in 2$, let
\begin{itemize}
    \item $u_i=(\hq_i, \l_{\hq_i}, \nu^i_{\hq_i})$, 
    \item $u_i^-=(\hq_i, \eta_{\hq_i}, \nu^i_{\hq_i})$, 
    \item $\sigma_i=\sigma_{u_i}$, 
    \item $\sigma_i^-=\sigma_{u_i^-}$,
    \item $\hw=\hr^\eta$,
    \item $\a=(\gg^+)^\mH$,
    \item $Y= \pi_{\hr, \hw}[\M^\hr|\a_\hr]$, 
    \item $\b$ is such that $\hh|\b$ is a complete iterate of $\hw|\a_\hw$,
    \item $Z=\pi_{\hw|\a_\hw, \hh|\b}[Y]$, and
    \item $\xi=\pi_{\hw|\a, \hh|\b}(\nu^0_\hw)$.
\end{itemize}
Then \begin{enumerate}
 \item $\rge(\sigma_0^-)\subseteq \rge(\sigma_1^-)$, 
 \item $\sigma_0^-\in \rge(\sigma_1^-)$,
 \item $\its(\xi)\in \rge(\sigma_1)$, and 
 \item $\its(Z)\in \rge(\sigma_1)$.
 \end{enumerate}
\end{corollary}

\section{Reflection IV: Hulls above $\Theta^N$, a directed system}

Suppose 
\begin{itemize}
\item $\eta<\l<\k$ are strong cardinals of $\mH|\varsigma$, 
\item $\nu\in (\k, \varsigma)$ is a Woodin cardinal of $\mH$ such that $\nu$ is a proper cutpoint of $\mH|\varsigma$ and there is a Woodin cardinal of $\mH$ in the interval $(\k, \nu)$, and
\item $Z\in \its(\nu)$.
\end{itemize}
Set $x=(\eta, \l, \k, \nu, Z)$ and set \[\its(x)=\its(Z)\cap \powerset(\eta).\] In the sequel, we may write $\eta^x$ for $\eta$, etc. Thus, $x=(\eta^x, \l^x, \k^x, \nu^x, Z^x)$. We call such an $x$ a \textbf{directed-system generator} or just a \textbf{ds-generator}.

Given $Y, Y'\in \its(x)$, we write \[Y=_x Y'\] if $Z[Y]=Z[Y']$. Let $[Y]=[Y]_{=_x}$ be the equivalence class of $Y$. We remark that for every $Y\in \its(x)$, \[Y=_x Z[Y]\cap \eta.\]

Let $\mathcal{I}_x$ be the set of $=_x$-equivalence classes. Given $p\in \mathcal{I}_x$, we let \[Z_p=Z[Y]\] where $p=[Y]_{=_x}$.  Also for $p=[Y]\in\mathcal{I}_x$, letting $\hs$ be an $\its(\nu)$-certificate for $Z_p$ \[W_p=N|\pi_{\hs|\nu_\hs, \infty}(\l)\] and \[\sigma_p=\sigma_u: W_p\rightarrow N|\l\] where $u=(\hs, \l_\hs, \nu_\hs)$. To stress the dependence on $x$, we will write $W_{p,x}$ instead of $W_p$ and $\sigma_{p, x}$ instead of $\sigma_p$.

Suppose now that $p, q\in \mathcal{I}_x$. We write \[p\leq_x q\] if and only if there is an $\its(\nu)$-certificate $\hs$ of $Z_p$ such that some complete $\eta$-bounded iterate $\hs'$ of $\hs$ is an $\its(\nu)$-certificate for $Z_q$. Applying \rlem{lem: pullback consistency for sigma1} we get that if \[p\leq_x q\] then \[\rge(\sigma_p)\subseteq \rge(\sigma_q).\] Assuming now that $p\leq_x q$, we let \[\sigma_{p, q}=\sigma_q^{-1}\circ \sigma_p.\] 



We also have the following lemma.

\begin{lemma}\label{lem:Z_directed}
    $\leq_x$ is countably directed.
\end{lemma}
\begin{proof} 
    We prove that if 
    \[
        (p_i: i<\omega)\subseteq \mathcal{I}_x,
    \] 
    then there is $p\in \mathcal{I}_x$ such that for every $i<\omega$, $p_i\leq_x p$. For each $i<\omega$, let $\hq_i$ be an $\its(\nu)$-certificate for $Z_{p_i}$. We can then find for each $i<\omega$, a complete $\eta$-bounded iterate $\hs_i$ of $\hq_i$ such that for every $i<j<\omega$, 
    \[
        \pi_{\hs_i, \infty}[\M^{\hs_i}|\nu_{\hs_i}]=\pi_{\hs_j, \infty}[\M^{\hs_j}|\nu_{\hs_j}].
    \] 
    To find such a sequence we inductively construct a sequence $(\hs_i^k: k<\omega, i<\omega)$ such that
    
    \vspace{0.3cm}
    \begin{enumerate}[label=(1.\arabic*), itemsep=0.3cm]
        \item for each $k<\omega$, letting \[X_k=\bigcup_{i<\omega}\pi_{\hs_i^k, \infty}[\M^{\hs^k_i}|\eta_{\hs^k_i}],\] for each $i<\omega$, \[X_k\subseteq \pi_{\hs_i^{k+1}, \infty}[\M^{\hs^{k+1}_i}|\eta_{\hs^{k+1}_i}],\]
        \item for each $i<\omega$, $\hs^0_i=\hq_i$, and
        \item for each $i<\omega$ and $k<\omega$, $\hs_i^{k+1}$ is an $\eta$-bounded complete iterate of $\hs_i^k$.
    \end{enumerate}
    \vspace{0.3cm}
    
    We then let for $i<\omega$, $\hs_i$ be the direct limit of $(\hs_i^k: k<\omega)$. 
    
    Set now 
    \[
        Y=\pi_{\hs_0, \infty}[\M^{\hs_0}|\eta_{\hs_0}].
    \] 
    It follows that if $p=[Y]$, then for every $i<\omega$, $p_i\leq_x p$.
\end{proof}

Let now $\mathcal{F}_x$ be the directed system consisting of models $W_p$ and embeddings $\sigma_{p, q}$ for $p \leq_x q$. Notice first that the following lemma holds.

\begin{lemma}\label{lem: quasi transitive}
    Suppose $q \leq_x r \leq_x s$ and $q\leq_x s$. Then \[\sigma_{q, s}=\sigma_{r, s}\circ \sigma_{q, r}.\]
\end{lemma}
\begin{proof} 
    By \rthm{thm: uniqueness of realizability witnesses}, the maps $\sigma_{q,s}$, $\sigma_{r,s}$, and $\sigma_{q,r}$ are independent of the specific certificates used to define them. Let $\hq, \hr, \hs$ be such that
    
    \vspace{0.3cm}
    \begin{enumerate}[label=(1.\arabic*), itemsep=0.3cm]
        \item $\hq$ is an $\its(\nu)$-certificate of $Z_q$ and some $\eta$-bounded complete iterate of $\hq$ is an $\its(\nu)$-certificate of $Z_r$,
        \item $\hr$ is an $\its(\nu)$-certificate of $Z_r$ and some $\eta$-bounded complete iterate of $\hr$ is an $\its(\nu)$-certificate of $Z_s$,
        \item $\hs$ is an $\its(\nu)$-certificate for $Z_s$.
    \end{enumerate}
    \vspace{0.3cm}
    
    Let $\hq'$ be an $\eta$-bounded complete iterate of $\hq$ witnessing (1.1). We can also assume that $\hs$ is a complete $\eta$-bounded iterate of some $\hq''$ that is an $\its(\nu)$-certificate for $Z_q$.
    
    Let $u=(\hq, \l_\hq, \nu_\hq)$, $v=(\hr, \l_\hr, \nu_\hr)$, and $w=(\hs, \l_\hs, \nu_\hs)$. We then have that
    
    \vspace{0.3cm}
    \begin{enumerate}[label=(2.\arabic*), itemsep=0.3cm]
        \item $\sigma_{q, r}=\sigma_v^{-1}\circ \sigma_{u}$,
        \item $\sigma_{r, s}=\sigma_w^{-1}\circ \sigma_v$, and
        \item $\sigma_{q, s}=\sigma_w^{-1}\circ \sigma_{u}$.
    \end{enumerate}
    \vspace{0.3cm}
    
    All three equalities follow from \rlem{lem: pullback consistency for sigma1} and \rthm{thm: uniqueness of realizability witnesses}. For example, regarding (2.1), if we let $u'=(\hq', \l_{\hq'}, \nu_{\hq'})$, then by definition $\sigma_{q, r} = \sigma_{u'}^{-1} \circ \sigma_u$. However, since $\hq'$ and $\hr$ both certify $Z_r$, \rthm{thm: uniqueness of realizability witnesses} implies $\sigma_{u'} = \sigma_v$, yielding the equation. The reason we need the condition $q\leq_x s$ is so that we can conclude (2.3), which can be done in exactly the same way that we concluded (2.1).
    
    Substituting (2.1) and (2.2) into the composition, we get:
    \[
        \sigma_{r, s}\circ \sigma_{q, r} = (\sigma_w^{-1}\circ \sigma_v) \circ (\sigma_v^{-1}\circ \sigma_{u}) = \sigma_w^{-1} \circ (\sigma_v \circ \sigma_v^{-1}) \circ \sigma_u = \sigma_w^{-1} \circ \sigma_u.
    \]
    By (2.3), this is exactly $\sigma_{q, s}$.
\end{proof}

Even though $\leq_x$ is not transitive, we can define the direct limit $W_{x, \infty}$ and the embedding \[\sigma_{x, \infty}: W_{x, \infty}\rightarrow N|\l\] as follows. Consider pairs $(q, y)$ such that $y\in W_q$. We define an equivalence relation $\sim_x$ by \[(q, y) \sim_x (q', y')\] if and only if there exists $s \in \mathcal{I}_x$ such that $q, q' \leq_x s$ and \[\sigma_{q, s}(y) = \sigma_{q', s}(y').\]

Let $\mathcal{D}$ be the set of $\sim_x$-equivalence classes of pairs $(q, y)$, which we denote by $[q, y]$. Given \[A=[q, y], B=[q', y']\in \mathcal{D},\] we write $A\in^* B$ if and only if whenever $r\in \mathcal{I}_x$ is such that $q, q'\leq_x r$, \[\sigma_{q, r}(y)\in \sigma_{q', r}(y').\] 

\begin{lemma}\label{lem: in* is well defined}
    The relation $\in^*$ is well-defined. That is, if $(q, y) \sim_x (u, w)$ and $(q', y') \sim_x (u', w')$, then
    \[
        [q, y] \in^* [q', y'] \iff [u, w] \in^* [u', w'].
    \]
\end{lemma}
\begin{proof}
    Assume $[q, y] \in^* [q', y']$. We show that $[u, w] \in^* [u', w']$. Let $k \in \mathcal{I}_x$ be such that $u, u' \leq_x k$. We must show that $\sigma_{u, k}(w) \in \sigma_{u', k}(w')$.

    Let $s, s' \in \mathcal{I}_x$ be witnesses for the equivalence of the representatives, meaning
    \vspace{0.3cm}
    \begin{enumerate}[label=(1.\arabic*), itemsep=0.3cm]
        \item $q, u \leq_x s$ and $\sigma_{q, s}(y) = \sigma_{u, s}(w)$,
        \item $q', u' \leq_x s'$ and $\sigma_{q', s'}(y') = \sigma_{u', s'}(w')$.
    \end{enumerate}
    \vspace{0.3cm}

    Using \rlem{lem:Z_directed}, let $t \in \mathcal{I}_x$ be a common upper bound such that
    \[
        q, q', u, u', s, s', k \leq_x t.
    \]
    (Note that since $\leq_x$ is not transitive, we explicitly require $t$ to extend the base conditions $q, u, q', u'$ as well as the intermediate witnesses $s, s', k$.)

    Applying \rlem{lem: quasi transitive} we get
    \vspace{0.3cm}
    \begin{enumerate}[label=(2.\arabic*), itemsep=0.3cm]
        \item $\sigma_{u, t} = \sigma_{s, t} \circ \sigma_{u, s}$ and $\sigma_{q, t} = \sigma_{s, t} \circ \sigma_{q, s}$,
        \item $\sigma_{u', t} = \sigma_{s', t} \circ \sigma_{u', s'}$ and $\sigma_{q', t} = \sigma_{s', t} \circ \sigma_{q', s'}$.
    \end{enumerate}
    \vspace{0.3cm}

    Applying (1.1) and (1.2), we get:
    \[
        \sigma_{u, t}(w) = \sigma_{s, t}(\sigma_{u, s}(w)) = \sigma_{s, t}(\sigma_{q, s}(y)) = \sigma_{q, t}(y),
    \]
    and similarly $\sigma_{u', t}(w') = \sigma_{q', t}(y')$.

    Since $[q, y] \in^* [q', y']$ and $q, q' \leq_x t$, we know by definition that
    \[
        \sigma_{q, t}(y) \in \sigma_{q', t}(y').
    \]
    Substituting the values derived above, we get
    \[
        \sigma_{u, t}(w) \in \sigma_{u', t}(w').
    \]

    Finally, we pull this back to $k$. Since $u, u', k \leq_x t$, \rlem{lem: quasi transitive} implies $\sigma_{u, t} = \sigma_{k, t} \circ \sigma_{u, k}$ and $\sigma_{u', t} = \sigma_{k, t} \circ \sigma_{u', k}$. Thus,
    \[
        \sigma_{k, t}(\sigma_{u, k}(w)) \in \sigma_{k, t}(\sigma_{u', k}(w')).
    \]
    By the elementarity of $\sigma_{k, t}$, it follows that $\sigma_{u, k}(w) \in \sigma_{u', k}(w')$.
\end{proof}

We also have the following. 

\begin{lemma}\label{lem: in* is independent strong}
    $A\in^* B$ if and only if there exist pairs $(q, y)$ and $(q', y')$, and $r \in \mathcal{I}_x$ such that $[q, y]=A$, $[q', y']=B$, $q, q'\leq_x r$, and
    \[
        \sigma_{q, r}(y)\in \sigma_{q', r}(y').
    \]
\end{lemma}
\begin{proof}
    The forward direction is immediate from the definition of $\in^*$. We prove the backward direction. Suppose there exist $(q, y)$, $(q', y')$, and $r$ witnessing the right-hand side. That is,
    \vspace{0.3cm}
    \begin{enumerate}[label=(\arabic*), itemsep=0.3cm]
        \item[(1)] $[q, y]=A, [q', y']=B, r \in \mathcal{I}_x, q, q' \leq_x r, \text{ and } \sigma_{q, r}(y) \in \sigma_{q', r}(y').$
    \end{enumerate}
    \vspace{0.3cm}

    We first claim that $[q, y] \in^* [q', y']$. To see this, we must show that for \textit{any} $k \in \mathcal{I}_x$ such that $q, q' \leq_x k$,
    \[
        \sigma_{q, k}(y) \in \sigma_{q', k}(y').
    \]
    Let $k$ be such a condition. By \rlem{lem:Z_directed}, let $t \in \mathcal{I}_x$ be a common upper bound for $q$, $q'$, $r$, and $k$. By \rlem{lem: quasi transitive}, we have the following commutative relations:
    \vspace{0.3cm}
    \begin{enumerate}[label=(\arabic*), itemsep=0.3cm]
        \setcounter{enumi}{1}
        \item[(2.1)] $\sigma_{q, t} = \sigma_{r, t} \circ \sigma_{q, r} = \sigma_{k, t} \circ \sigma_{q, k}$,
        \item[(2.2)] $\sigma_{q', t} = \sigma_{r, t} \circ \sigma_{q', r} = \sigma_{k, t} \circ \sigma_{q', k}$.
    \end{enumerate}
    \vspace{0.3cm}

    Applying the elementary embedding $\sigma_{r, t}$ to the inclusion in (1), and using the first equalities in (2.1) and (2.2), we obtain:
    \[
        \sigma_{q, t}(y) \in \sigma_{q', t}(y').
    \]
    Substituting the second equalities from (2.1) and (2.2), this becomes:
    \[
        \sigma_{k, t}(\sigma_{q, k}(y)) \in \sigma_{k, t}(\sigma_{q', k}(y')).
    \]
    By the elementarity of $\sigma_{k, t}$, we conclude that $\sigma_{q, k}(y) \in \sigma_{q', k}(y')$. Thus, $[q, y] \in^* [q', y']$.

    Finally, since $[q, y] = A$ and $[q', y'] = B$, it follows from \rlem{lem: in* is well defined} that $A \in^* B$.
\end{proof}

Because $\leq_x$ is countably directed we have that $\in^*$ is well-founded. We then let $W_{x, \infty}$ be the transitive collapse of $(\mathcal{D}, \in^*)$. For $q\in \mathcal{I}_x$, we let \[\sigma_{q, \infty}: W_q\rightarrow W_{x, \infty}\] be given by \[\sigma_{q, \infty}(y)=[q, y].\] Applying \rlem{lem: quasi transitive}, we get that:

\begin{lemma}\label{lem: sigmaqinfty} 
    $\sigma_{q, \infty}$ is an elementary embedding.
\end{lemma}

Finally we set \[\sigma_{x, \infty}:W_{x, \infty}\rightarrow N|\l\] to be the unique embedding with the property that \[\sigma_{x, \infty}([q, y])=\sigma_q(y).\] Again, \rlem{lem: quasi transitive} implies that $\sigma_{x, \infty}$ is well-defined and elementary. To stress the dependence on $x$, we will write $\sigma^x_{p, q}$ for $\sigma_{p, q}$ and $\sigma^x_{q, \infty}$ for $\sigma_{q, \infty}$.

\section{Reflection V: reflection of the directed system}

Notice that $\sigma_{x, \infty}: W_x\rightarrow N|\l$ is a hull of $N|\l$, and it could be used to show that $\rref$ holds in $N$ if we knew that $W_x$ is closed under $\eta$-sequences. In order to achieve this closure, we will need to build a chain \[(\sigma_{x_\alpha, \infty}, W_{x_{\alpha}}: \a<\eta)\] such that for $\a<\b$, \[\sigma_{x_\alpha, \infty} \in \rge(\sigma_{x_\b, \infty}).\] The goal of this section is to make sure that this can be done, and we use the ideas behind \rcor{cor: defining smaller hulls} to do it. In particular, we will show that for many $x=(\eta^x, \l^x, \k^x, \nu^x, Z^x)$ and $y=(\eta^y, \l^y, \k^y, \nu^y, Z^y)$ we have that 
 \[\sigma_{x, \infty} \in \rge(\sigma_{y, \infty}).\]
 We will work with $x$ and $y$ such that $\eta^x=\eta^y$ and $\l^x=\l^y$. 

\begin{definition}\label{def: mesh}\normalfont Suppose \[x=(\eta^x, \l^x, \k^x, \nu^x, Z^x)\] and \[y=(\eta^y, \l^y, \k^y, \nu^y, Z^y)\] are two ds-generators. We say that $x$ \textbf{projects} to $y$ or $y$ is a projection of $x$ (see \rfig{fig:projection_structure}) if
\begin{enumerate}
\item $\eta^x=\eta^y$ and $\lambda^x=\lambda^y$,
\item $\k^y<\nu^y<\kappa^x$, 
\item $Z^y\cap \k^y=Z^x\cap \k^y$,
\item $\k^y \in Z^x$,
\item letting 
\begin{itemize}
\item $\R=\chull^{\mH|\nu^x}(\k^y\cup Z^x)$, 
\item $i: \R\rightarrow \mH|\nu^x$ be the inverse of the transitive collapse, and
\item $\hr$ be the $i$-pullback of $\hh|\nu^x$,
\end{itemize}
 $\hh|\nu^y$ is a complete iterate of $\hr$ and 
 \[Z^y=\pi_{\hr, \hh|\nu^y}[i^{-1}[Z^x]].\]
\end{enumerate}
\end{definition}
\begin{figure}[ht]
\centering
\begin{tikzpicture}[scale=1.2, >=stealth]


    \coordinate (hr_base) at (0,0);
    \coordinate (hr_top) at (0,1.8); 
    
    \coordinate (hh_base) at (4,0);
    \coordinate (hh_top) at (4,5.0);

    \draw[thick] (hr_base) -- (hr_top);
    \draw[thick] (hh_base) -- (hh_top);

    \node[below] at (hr_base) {$\hr$};
    \node[below] at (hh_base) {$\hh$};

    \draw[dashed] (0, 1.0) -- (4, 1.0);
    \node[left] at (0, 1.0) {$\kappa^y$};

    \draw (3.9, 2.8) -- (4.1, 2.8) node[right] {$\nu^y$};
    \draw (3.9, 3.8) -- (4.1, 3.8) node[right] {$\kappa^x$};
    \draw (3.9, 4.8) -- (4.1, 4.8) node[right] {$\nu^x$};

    \draw[->, dashed] (3.8, 4.5) -- node[above, sloped, font=\footnotesize] {$i^{-1}$} (0.2, 1.9);

    \draw[->] (0.2, 1.6) -- node[below, sloped, font=\footnotesize] {$\pi_{\hr, \hh|\nu^y}$} (3.8, 2.6);

    \draw[gray!50, thick] (5.2, -0.5) -- (5.2, 5.5);

    
    \begin{scope}[shift={(6.5,0)}]
        
        \node[draw, rectangle, font=\small, inner sep=3pt] at (0, 4.5) {$Z^x$-to-$Z^y$};

        \node (Zx) at (1.5, 4.0) {$Z^x$};
        
        \node (R_node) at (-0.8, 1.2) {$\mathcal{R}$};
        
        \node (Zy) at (1.5, 2.8) {$Z^y$};

        \draw[->, dashed] (Zx) -- node[above left, font=\footnotesize] {$i^{-1}$} (R_node);
        
        \draw[->] (R_node) -- node[below right, font=\footnotesize] {$\pi_{\hr, \hh|\nu^y}$} (Zy);

    \end{scope}

\end{tikzpicture}
\caption{$Z^x$ first collapses to $R$ and then gets pushed up by $\pi_{\hr, \hh|\nu^y}$  producing $Z^y$.}
\label{fig:projection_structure}
\end{figure}
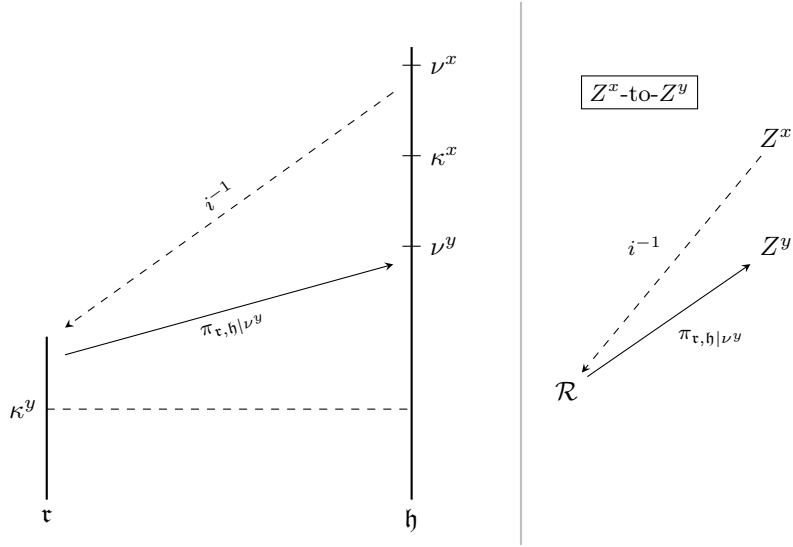

The proof of the following lemma is just like the proof of \rcl{clm: z is rge sigma1}. Below $\nu^x_\hq$ is just $(\nu^x)_\hq$, the preimage of $\nu^x$ in $\hq$.

\begin{lemma}\label{lem: certificates for generators} 
Suppose $x$ and $y$ are two ds-generators such that $x$ projects to $y$. Fix an $\its(\nu^x)$-certificate $\hq$ for $Z^x$ and let $E\in \vec{E}^{\M^\hq}$ be the extender with the least index such that $\cp(E)\in (\k^y_\hq, \nu^x_\hq)$ and $\lh(E)>\nu^x_\hq$. Then $\hq_E$ is an $\its(\nu^y)$-certificate for $Z^y$ and $\nu^x_\hq=\nu^y_{\hq_E}$ (see \rfig{fig:lemma_certificates}).
\end{lemma}
\begin{proof}
Let $\hr=\hq^{\k^y}$. Let also $\tau=\pi_{\hq, \hr}(\nu^x_\hq)$ and $G=\pi_{\hq, \hr}(E)$. Note that because $\hq$ is an $\its(\nu^x)$-certificate, $\nu^x_\hq$ is inaccessible in $\M^\hq$, and by the choice of $E$, $\lh(E)>\nu^x_\hq$.

We verify the properties required for $\hq_E$ to be a certificate. We have the following facts::
\vspace{0.3cm}
\begin{enumerate}[label=(1.\arabic*), itemsep=0.3cm]
    \item $(\hq_E)^{\k^y}=\hr_G$,
    \item $\cp(G)>\k^y$,
    \item $\hr_G|\tau=\hr|\tau$,
    \item $\pi_{\hq_E, \infty}^{\k^y}\rest (\M^{\hq_E}|\nu^x_\hq+1)=\pi_{\hq, \infty}^{\k^y}\rest (\M^\hq|\nu^x_\hq+1)$, and
    \item $\tau=\pi_{\hq_E, \infty}^{\k^y}(\nu^x_\hq)$.
\end{enumerate}
\vspace{0.3cm}
Item (1.4) implies immediately that $\nu^y_{\hq_E}=\nu^x_\hq$. We then have that:
\begin{align*}
\pi_{\hq_E, \infty}[\M^{\hq_E}|\nu^y_{\hq_E}] &= \pi_{\hr_G|\tau, \hh|\nu^y}\circ \pi_{\hq_E, \infty}^{\k^y} [\M^{\hq_E}|\nu^x_\hq] \\
&= \pi_{\hr|\tau, \hh|\nu^y} \circ  \pi_{\hq, \infty}^{\k^y}[\M^\hq|\nu^x_\hq] \quad \text{(by (1.3) and (1.4))}\\
&=\pi_{\hr|\tau, \infty}[U]\ \ \ \text{where $U=\pi_{\hq, \hr}[\M^\hq|\nu^x_\hq]$}\\
&=Z^y
\end{align*}
The last equality holds because $\hq$ is a certificate for $Z^x$, implying that \begin{align*}
Z^x &=\pi_{\hq, \infty}[\M^\hq|\nu^x_\hq]\\
& =\pi_{\hr, \hh}\circ \pi_{\hq, \hr}[\M^\hq|\nu^x_\hq]\\
& =\pi_{\hr, \hh}[U]
\end{align*}
so because $x$ projects to $y$, we have that $Z^y=\pi_{\hr|\tau, \hh|\nu^y}[U]$.
\end{proof}

\begin{figure}[ht]
\centering
\begin{tikzpicture}[scale=1.2, >=stealth]


    \coordinate (hq_base) at (0, 0);
    \coordinate (hq_top) at (0, 3.5);
    \draw[thick] (hq_base) -- (hq_top);
    \node[below] at (hq_base) {$\hq$};

    \draw (-0.1, 1.0) -- (0.1, 1.0);
    \node[right, font=\footnotesize] at (0.1, 1.0) {$\kappa^y_{\hq}$};

    \fill (0, 2.0) circle (1.5pt);
    \node[right] at (0.1, 2.0) {$\nu^x_{\hq}$};
    \coordinate (nu_src_x) at (0, 2.0);

    \draw[thick, blue!80!black] (-0.3, 1.2) -- (-0.5, 1.2) -- (-0.5, 2.2) -- (-0.3, 2.2);
    \node[left, blue!80!black] at (-0.5, 1.7) {$E$};

    \coordinate (hqE_base) at (3.5, 0);
    \coordinate (hqE_top) at (3.5, 3.5);
    \draw[thick] (hqE_base) -- (hqE_top);
    \node[below] at (hqE_base) {$\hq_E$};

    \fill (3.5, 2.0) circle (1.5pt);
    \node[right] at (3.5, 2.0) {$\nu^x_{\hq} = \nu^y_{\hq_E}$}; 
    \coordinate (nu_src_y) at (3.5, 2.0);

    \draw[dotted, gray] (0, 2.0) -- (3.5, 2.0);

    
    \coordinate (hh_base) at (8.0, 0);
    \coordinate (hh_top) at (8.0, 5.5);
    \draw[thick] (hh_base) -- (hh_top);
    \node[below] at (hh_base) {$\hh$};

    \coordinate (nux) at (8.0, 4.5);
    \coordinate (nuy) at (8.0, 3.0);

    \draw (7.9, 4.5) -- (8.1, 4.5) node[right] {$\nu^x$};
    \draw (7.9, 3.0) -- (8.1, 3.0) node[right] {$\nu^y$};


    \draw[->, thick] (0.8, 1.5) -- node[above, font=\footnotesize] {$\pi_E$} (2.7, 1.5);

    \draw[->] (nu_src_x) .. controls (1.5, 4.8) and (6.0, 4.8) .. node[above, sloped, near end, font=\footnotesize] {$\pi_{\hq, \hh}$} (nux);

    \draw[->] (nu_src_y) .. controls (5.0, 2.2) and (6.5, 2.8) .. node[above, sloped, font=\footnotesize] {$\pi_{\hq_E, \hh}$} (nuy);

\end{tikzpicture}
\caption{$\nu^x_\hq=\nu^y_{\hq_E}$ and $Z^y=\pi_{\hq_E, \infty}[\M^{\hq_E}|\nu^x_\hq]$.}
\label{fig:lemma_certificates}
\end{figure}

Suppose then $x$ and $y$ are two ds-generators such that $x$ projects to $y$. We say $(\hq, E)$ is an \textbf{$(x, y)$-certificate} if $\hq$ is an $\its(\nu^x)$-certificate for $Z^x$ and $E\in \vec{E}^{\M^\hq}$ is the extender with the least index such that $\cp(E)\in (\k^y_\hq, \nu^x_\hq)$ and $\lh(E)>\nu^x_\hq$. \rlem{lem: certificates for generators} shows a bit more.
\begin{lemma}\label{lem: certificates for generators1} 
Suppose $x$ and $y$ are two ds-generators such that $x$ projects to $y$. Suppose $Y\in \its(x)$ and $\hq$ is an $\its(\nu^x)$-certificate for $Z^x[Y]$. Let $E\in \vec{E}^{\M^\hq}$ be the extender with the least index such that $\cp(E)\in (\k^y_\hq, \nu^x_\hq)$ and $\lh(E)>\nu^x_\hq$. Let \[x(Y)=(\eta^x, \l^x, \k^x, \nu^x, Z^x[Y])\] and 
\[y(Y)=(\eta^y, \l^y, \k^y, \nu^y, Z^y[Y]).\] Then $(\hq, E)$ is a $(x(Y), y(Y))$-certificate and $\nu^y_{\hq_E}=\nu^x_\hq$.
\end{lemma}

\begin{definition}\label{def: it space}\normalfont
Suppose $x$ and $y$ are two ds-generators such that $x$ projects to $y$. We say $D$ is the $(x, y)$-iteration space if it consists of $Y\in \its(x)$ for which there are
\begin{itemize}
\item an $(x, y)$-certificate $(\hq, E)$, 
\item a complete $\eta$-bounded iterate $\hx$ of $\hq$ and
\item a complete $\eta$-bounded iterate $\hy$ of $\hq_E$
\end{itemize}
such that\footnote{One can show that $\hy=\hx_F$ where $F=\pi_{\hq, \hx}(E)$. This is like the proof of embedding normalization for two extenders that can be found in \cite{SteelCom}. Since we do not really need this fact, we will not prove it.}
\[Y = \pi_{\hx, \infty}[\M^{\hx}|\eta_{\hx}] = \pi_{\hy, \infty}[\M^{\hy}|\eta_{\hy}].\] 
We then let $D_x = \{ [Y]_{=_x} : Y \in D \}$ and $D_y = \{ p|y : p \in D_x \}$, where for $p=[Y]\in D_x$, $p|y=[Y]_{=_y}$.
\end{definition}

\begin{lemma}\label{lem: rest is dense}
Suppose $x$ and $y$ are two ds-generators such that $x$ projects to $y$ and $D$ is the $(x, y)$-iteration space. Then the following holds:
\begin{enumerate}
    \item $D\in N$.
    \item $D_x\subseteq \mathcal{I}_x\cap \its(Z^y)$ (and hence, for every $p\in D_x$, $p|y$ is defined).
    \item For $Y, Y'\in D$, $Y=_xY'$ if and only if $Y=_y Y'$.
    \item If $p=[Y]_x\in D$, then \[W_{p, x}=W_{p|y, y}\] and \[\sigma_{p, x}=\sigma_{p|y, y}.\]
    \item $D_y$ is dense in $\mathcal{I}_y$. Moreover, for any $p, p' \in D_x$,
    if $p \leq_x p'$ and $p|y \leq_y p'|y$, then \[ \sigma^x_{p, p'} = \sigma^y_{p|y, p'|y}. \]
    \item For every $p, p'\in D_x$ there is $q\in D_x$ such that
    \begin{enumerate}
    \item $p\leq_x q$ and $p'\leq_x q$, and
    \item $p|y\leq_y q|y$ and $p'|y\leq_y q|y$.
    \end{enumerate}
    \item $W_{x, \infty}=W_{y, \infty}$ and $\sigma_{x, \infty}=\sigma_{y, \infty}$.
\end{enumerate}
\end{lemma}
\begin{proof}
We set $\eta=\eta^x=\eta^y$ and $\l=\l^x=\l^y$. 
$D\in N$ as it is ordinal definable from $(x, y)$ (in fact, just from $x$ as $y$ is ordinal definable from $x$).\\
 
\textbf{Clause 2:} For clause $(2)$, it is clear that $D_x\subseteq \mathcal{I}_x$. Fix then $Y \in D_x$ and let $(\hq, E)$ be a $(x, y)$-certificate such that there is an $\eta$-bounded complete iterate $\hr$ of $\hq$ and an $\eta$-bounded complete iterate $\hs$ of $\hq_E$ such that
\[Y = \pi_{\hr, \infty}[\M^{\hr}|\eta_{\hr}] = \pi_{\hs, \infty}[\M^{\hs}|\eta_{\hs}].\] 
Notice now that because $\hs$ is an $\eta$-bounded complete iterate of $\hq_E$, we have that 
\[\pi_{\hs, \infty}[\M^\hs|\nu^y_\hs]=Z^y[Y].\]
This is because 
\begin{align*}
\pi_{\hs, \infty}[\M^\hs|\nu^y_\hs] & =(\pi_{\hq_E, \infty}[\M^{\hq_E}|\nu^y_{\hq_E}])[Y] \\
&= Z^y[Y].
\end{align*}
Thus,  $Y \in \its(Z^y)$, and consequently $D_y \subseteq \mathcal{I}_y$.\\

\textbf{Clause 3:} We now verify clause $(3)$. Letting $(\hr, i)$ be as in \rdef{def: mesh}, we have that
\[ Z^y[Y]=\pi_{\hr, \hh|\nu^y}[i^{-1}[Z^x[Y]]] \quad \text{and} \quad Z^y[Y']=\pi_{\hr, \hh|\nu^y}[i^{-1}[Z^x[Y']]]. \]
Therefore we have that \[Z^x[Y]=Z^x[Y']\ \text{if and only if}\ Z^y[Y]=Z^y[Y'].\]

\textbf{Clause 4:} For clause $(4)$, suppose $p=[Y]_x \in D_x$ with $Y\in D$, and let $(\hq, E)$ be an $(x, y)$-certificate witnessing that $Y\in D$. Let $\hx$ be a complete $\eta$-bounded iterate of $\hq$ and $\hy$ be a complete $\eta$-bounded iterate of $\hq_E$ such that \[Y = \pi_{\hx, \infty}[\M^{\hx}|\eta_{\hx}] = \pi_{\hy, \infty}[\M^{\hy}|\eta_{\hy}].\] We then have that 
\[W_{p, x} =N|\pi_{\hx|\nu^x_\hx, \infty}(\l_\hx),\]
and 
\[W_{p|y, y}= N|\pi_{\hy|\nu^y_\hy, \infty}(\l_\hy).\]
Let $F=\pi_{\hq, \hx}(E)$. It then follows from \rlem{lem: certificates for generators1} that $(\hx, F)$ is a $(x(Y), y(Y))$-certificate, and so 
\[Z^y[Y]=\pi_{\hx_F, \infty}[\M^\hx|\nu^x_\hx]\] 
and $\nu^x_{\hx}=\nu^y_{\hx_F}$.
Therefore, 
\begin{align*}
W_{p|y, y} &= N|\pi_{\hx_F|\nu^x_\hx, \infty}(\l_\hx) \\
&= N|\pi_{\hx|\nu^x_\hx, \infty}(\l_\hx)\\
&= W_{p, x}.
\end{align*}
Next we need to show that $\sigma_{p, x}=\sigma_{p|y, y}$. We have that $\sigma_{p, x}=\sigma_u$ where $u=(\hx, \l_\hx, \nu^x_\hx)$ and $\sigma_{p|y, y}=\sigma_v$ where $v=(\hx_F, \l_\hx, \nu^x_\hx)$ (because as shown above $\hx_F$ is an $\its(\nu^y)$-certificate for $Z^y[Y]$ and $\nu^x_\hx=\nu^y_{\hx_F}$). But it follows from \rlem{cor: sharper bound} that $\sigma_u=\sigma_v$ as \[\rel(Z^x[Y])\rest \l=\rel(Z^y[Y])\rest \l.\] Thus, we get that \[\sigma_{p, x}=\sigma_{p|y, y}.\]

\textbf{Clause 5:} To see clause $(5)$, we first show density. Fix an $(x, y)$-certificate $(\hq, E)$. Fix any $r \in \mathcal{I}_y$ and let $\hr$ be an $\its(\nu^y)$-certificate for $Z^r$. We have that $\hq_E$ is an $\its(\nu^y)$-certificate for $Z^y$. Just like in the proof of \rlem{lem:Z_directed}, we can find $\hs$, $\hx$, and $\hy$ such that 
\vspace{0.3cm}
\begin{enumerate}[label=(2.\arabic*), itemsep=0.3cm]
    \item $\hs$ is a complete $\eta$-bounded iterate of $\hq$,
    \item $\hx$ is a complete $\eta$-bounded iterate of $\hq_E$, 
    \item $\hy$ is a complete $\eta$-bounded iterate of $\hr$,
   and 
    \item $\pi_{\hs, \infty}[\M^{\hs}|\eta_{\hs}]=\pi_{\hx, \infty}[\M^{\hx}|\eta_{\hx}] = \pi_{\hy, \infty}[\M^{\hy}|\eta_{\hy}]$.
\end{enumerate}
\vspace{0.3cm}
 Let $Y=\pi_{\hx, \infty}[\M^{\hx}|\eta_{\hx}]$. (2.1) and (2.2) imply that $Y \in D$. 

Set $s_x=[Y]_x$ and $s_y=[Y]_y$. By construction, $s_y \in D_y$. Since $\hy$ is an iterate of $\hr$, we have $r \leq_y s_y$. Thus $D_y$ is dense in $\mathcal{I}_y$.

For the second part of clause $(5)$, suppose $p, p' \in D_x$,  $p \leq_x p'$, and  $p|y \leq_y p'|y$. By clause $(4)$, 
\vspace{0.3cm}
\begin{enumerate}[label=(3.\arabic*), itemsep=0.3cm]
    \item $W_{p, x}=W_{p|y, y}$ and $W_{p', x}=W_{p'|y, y}$,
    \item $\sigma_{p, x}=\sigma_{p|y, y}$, and
    \item $\sigma_{p', x}=\sigma_{p'|y, y}$.
\end{enumerate}
\vspace{0.3cm}  Thus 
\begin{align*}
\sigma^x_{p, p'} & = \sigma_{p', x}^{-1}\circ \sigma_{p, x}\\
&= \sigma_{p'|y, y}^{-1}\circ \sigma_{p|y, y}\\
&= \sigma^y_{p|y, p'|y}.
\end{align*}

\textbf{Clause 6:} For clause $(6)$, fix an $(x, y)$-certificate $(\hq, E)$, and let $p, p'\in D$. Repeating the argument given in the proof of clause (5), we can find $Y\in D$ such that for some $\eta$-bounded complete iterate $\hx$ of $\hq$ and for some $\eta$-bounded complete iterate $\hy$ of $\hq_E$, 
\vspace{0.3cm}
\begin{enumerate}[label=(4.\arabic*), itemsep=0.3cm]
    \item $Y=\pi_{\hx, \infty}[\M^\hx|\eta_\hx]=\pi_{\hy, \infty}[\M^\hy|\eta_\hy]$,
    \item $p\leq_x [Y]_x$ and $p'\leq_x [Y]_x$, and
    \item $p|y \leq_y [Y]_y$ and $p'|y \leq_y [Y]_y$.
\end{enumerate}
\vspace{0.3cm} 
Thus, $q=[Y]_x$ is as desired.\\

\textbf{Clause 7:} 
 We define $j: W_{x, \infty} \rightarrow W_{y, \infty}$. Given $a\in W_{x, \infty}$, we fix $p_0\in \mathcal{I}_x$ such that $\sigma^x_{p_0, \infty}(a_0)=a$ for some $a_0\in W_{p_0, x}$. We can now find $p\in D$ such that $p_0\leq p$. We then set \[j(a)=\sigma_{p|y, y}(\sigma^x_{p_0, p}(a_0)).\] 
 We now show that $j$ is well-defined. 

 \begin{claim}\label{cl: j is well-defined} Suppose $(p_0, a_0)$ and $(p_1, a_1)$ are such that for $i\in 2$, $p_i\in \mathcal{I}_x$, $a_i\in W_{p_i, x}$, and $\sigma^x_{p_i, \infty}(a_i)=a$. Let $q_0, q_1\in D$ be such that for $i\in 2$, $p_i\leq_x q_i$. Then \[\sigma_{q_0|y, y}(\sigma^x_{p_0, q_0}(a_0))=\sigma_{q_1|y, y}(\sigma^x_{p_1, q_1}(a_1)).\]
 \end{claim}
 \begin{proof}
 We can find $q\in D_x$ such that for $i\in 2$, $q_i\leq_x q$, $q_i|y\leq_y q|y$, $p_i\leq_x q$ and $p_i|y\leq_y q|y$. We thus have that
 \begin{align*}
 \sigma_{q_0|y, y}(\sigma^x_{p_0, q_0}(a_0)) &= \sigma_{q|y, y}(\sigma^y_{q_0|y, q|y}(\sigma^x_{p_0, q_0}(a_0)))\\
 &=\sigma_{q, x}(\sigma^x_{q_0, q}(\sigma^x_{p_0, q_0}(a_0)))\\
 &=\sigma_{q, x}(\sigma^x_{p_0, q}(a_0))\\
 &=\sigma_{p_0, x}(a_0)\\
 & =\sigma_{p_1, x}(a_1)\\
 &=\sigma_{q, x}(\sigma^x_{q_1, q}(\sigma^x_{p_1, q_1}(a_1)))\\
 &=\sigma_{q|y, y}(\sigma^y_{q_1|y, q|y}(\sigma^x_{p_1, q_1}(a_1)))\\
 &=\sigma_{q_1|y, y}(\sigma^x_{p_1, q_1}(a_1))
 \end{align*}
 \end{proof}
We now show that $j$ is onto. Indeed, fix $b\in W_{y, \infty}$ and let $p_0\in D$ be such that for some $b_0\in W_{p_0|y, y}$, $\sigma^y_{p_0|y, \infty}(b_0)=b$. Let then $a=\sigma^x_{p_0, \infty}(b_0)$. We then have that $j(a)=b$. We thus have that 
\[ W_{x, \infty} =  W_{y, \infty}.\]
A similar argument shows that $\sigma_{x, \infty}=\sigma_{y, \infty}$.
\end{proof}

We finish this section by adopting \rcor{cor: defining smaller hulls} to our current context (see \rlem{lem: agreement of hulls}). 

\begin{definition}\label{def: projection}\normalfont Suppose $x$ is a ds-generator and $\xi\in (\l^x, \k^x)$ is a strong cardinal of $\mH|\varsigma$. We then let  $x\downarrow \xi$ be the unique ds-generator $y$ such that $\eta^y=\eta^x$, $\l^y=\l^x$, $\k^y=\xi$, and $x$ projects to $y$. 
\end{definition}

\begin{lemma}\label{lem: agreement of hulls} Suppose 
\begin{itemize}
\item $\nu^0<\gg<\nu^1$ are proper cutpoint Woodin cardinals of $\mH$ relative to $\varsigma$, 
\item $\eta<\l<\k<\k'<\nu^0$ are a $<\varsigma$-strong cardinal of $\mH$, 
\item $Z\in \its(\nu^1)$ is such that \[\{\eta, \l, \k, \k', \nu^0, \gg\}\subseteq Z,\] 
\item for $i\in 2$, $x_i=(\eta, \l, \k, \nu^i, Z|\k)$ is a ds-generator. 
\end{itemize}
Then 
\begin{enumerate}
\item $x_0\downarrow \k\in \rge(\sigma_{x_1})$,
\item $\mathcal{F}_{x_0\downarrow \k}\in \rge(\sigma_{x_1})$,
\item $W_{x_0}\in W_{x_1}$, and
\item $\sigma_{x_0, \infty}\in \rge(\sigma_{x_1, \infty})$.
\end{enumerate}
\end{lemma}

\section{Reflection VI: Closed hulls above $\Theta^N$}

The goal of this section is to merge the concepts defined in the previous sections to prove the following theorem. 

\begin{theorem}\label{thm: closed hulls above theta} Suppose $\zeta<\eta<\l$ are strong cardinals of $\mH|\varsigma$ and $X\in N|\l$. There is then $W\in N|(\eta^+)^N$ and $\sigma: W\rightarrow N|\l$ such that 
\begin{enumerate}
\item $\sigma\in N$, 
\item $X\in \rge(\sigma)$, and 
\item in $N$, $W^{\zeta^\omega}\subseteq W$.
\end{enumerate}
\end{theorem}
\begin{proof} Let $\k'<\k$ be two strong cardinals of $\mH|\varsigma$ such that $\l<\k'$. Let $\nu\in (\k, \varsigma)$ be a proper cutpoint of $\mH|\varsigma$ such that $\nu$ is a limit of Woodin cardinals and \[\mH\models \cf(\nu)=\eta.\] Let $(\nu_\a: \a<\eta)$ be a sequence of Woodin cardinals of $\mH$ belonging to the interval $(\k, \nu)$ such that 
\begin{itemize}
\item the interval $(\k, \nu_0)$ contains a Woodin cardinal of $\mH$, 
\item for all $\a<\eta$, the interval $(\nu_\a, \nu_{\a+1})$ contains a Woodin cardinal of $\mH$,
\item and $\nu=\sup_{\a<\eta}\nu_\a$. 
\end{itemize}
Let $\nu'=(\nu^+)^\mH$, and let $(\hq, Z)$ be such that
\begin{itemize}
\item $\{\zeta,\eta, \l,\k',\k\}\subseteq \rge(\pi_{\hq, \infty})$,
\item $(\nu_\a: \a<\eta)\in \rge(\pi_{\hq, \infty})$, 
\item $Z=\pi_{\hq, \infty}[\M^\hq|\nu'_\hq]$, and
\item $X$ is definable in $N|\l$ from $Z\cap \l$ and a real.
\end{itemize}
We will now define a sequence $(W_\a, \sigma_\a, \sigma_{\a, \b}: \a<\b<\eta)\in N$ such that

\vspace{0.3cm}
\begin{enumerate}[label=(1.\arabic*), itemsep=0.3cm]
    \item[(1.1)] for every $\a<\eta$, $W_\a\in N|(\eta^+)^N$,
    \item[(1.2)] for every $\a<\eta$, $\sigma_\a: W_\a\rightarrow N|\l$ is an elementary embedding,
    \item[(1.3)] for every $\a<\b<\eta$, \[\sigma_{\a, \b}: W_\a\rightarrow W_\b\] is an elementary embedding and $(W_\a, \sigma_{\a, \b})\in W_\b$, 
    \item[(1.4)] for $\a<\b<\gamma<\eta$, \[\sigma_{\a, \gamma}=\sigma_{\b, \gamma}\circ \sigma_{\a, \b},\]
    \item[(1.5)] for every $\a<\eta$, $N|\eta=W_\a|\eta$ and $\sigma_\a\rest \eta=id$, and
    \item[(1.6)] $Z\cap \l\in \rge(\sigma_0)$.
\end{enumerate}
\vspace{0.3cm}

Assuming we have such a sequence, the following claim finishes the proof of \rthm{thm: closed hulls above theta}. Let $W$ be the direct limit of $(W_\a, \sigma_{\a, \b}: \a<\b<\eta)$ and let $\sigma: W\rightarrow N|\l$ be given by $\sigma(x)=\sigma_\a(x')$ where $(\a, x')$ is such that letting \[\sigma_{\a, \infty} : W_\a\rightarrow W\] be the direct limit embedding, $\sigma_{\a, \infty}(x')=x$. 

\begin{claim}\label{cl: finishing proof} Suppose \[(W_\a, \sigma_\a, \sigma_{\a, \b}: \a<\b<\eta)\in N\]
satisfies (1.1)-(1.6). Let $W$ be the direct limit of $(W_\a, \sigma_{\a, \b}: \a<\b<\eta)$. Then the following holds:
\begin{enumerate}
\item $W\in N|(\eta^+)^N$. 
\item In $N$, $W^{\zeta^\omega}\subseteq W$. 
\item $\sigma$ is elementary and $X\in \rge(\sigma)$. 
\end{enumerate}
\end{claim}
\begin{proof}
(1.6) implies that $X\in \rge(\sigma)$. Clauses (1) and (3) are standard; we verify clause (2). Let $f: \zeta^\omega\rightarrow W$ with $f\in N$. Then since $\eta$ is a strongly regular cardinal in $N$ (see \cite[Theorem 10.20]{blue2025nairian}), we get that

\vspace{0.3cm}
\begin{enumerate}[label=(2), itemsep=0.3cm]
    \item[(2)] for some $\a<\eta$, \[\rge(f)\subseteq \rge(\sigma_{\a, \infty}).\]
\end{enumerate}
\vspace{0.3cm}

Because \[W_\a\in W_{\a+1}|(\eta^+)^{W_{\a+1}},\] we can fix a surjection \[F:\eta^\omega\rightarrow W_\a\] with $F\in W_{\a+1}$. Again because $\eta$ is a strongly regular cardinal in $N$, we have $\b<\eta$ such that 

\vspace{0.3cm}
\begin{enumerate}[label=(3), itemsep=0.3cm]
    \item[(3)] for every $Y\in \zeta^\omega$ there is $Y'\in \b^\omega$ such that \[F(Y')=\sigma_{\a, \infty}^{-1}(f(Y)).\]
\end{enumerate}
\vspace{0.3cm}

Let then $A$ consist of pairs $(Y', Y)$ such that $Y'\in \b^\omega$, $Y\in \zeta^\omega$, and\[F(Y')=\sigma_{\a, \infty}^{-1}(f(Y)).\]
Because $W_{\a+1}|\eta=N|\eta$ and $A\subseteq \b^\omega\times \zeta^\omega$, we have 

\vspace{0.3cm}
\begin{enumerate}[label=(4.\arabic*), itemsep=0.3cm]
    \item[(4.1)] $A\in N|\eta$,
    \item[(4.2)] and hence, $A\in W_{\a+1}$.
\end{enumerate}
\vspace{0.3cm} 

Therefore, since for $Y\in \zeta^\omega$, $\sigma_{\a, \infty}^{-1}(f(Y))=F(Y')$ where $Y'\in \b^\omega$ is any set such that $(Y', Y)\in A$,

\vspace{0.3cm}
\begin{enumerate}[label=(5), itemsep=0.3cm]
    \item[(5)] $(\sigma_{\a, \infty}^{-1}(f(Y)): Y\in \zeta^\omega)\in W_{\a+1}$.
\end{enumerate}
\vspace{0.3cm} 

Notice that (1.1) and (1.5) imply that \[\sigma_{\a+1, \infty}\rest W_\a=id,\] and since $\sigma_{\a, \a+1}\in W_{\a+1}$ and $\sigma_{\a, \infty}=\sigma_{\a+1, \infty}\circ \sigma_{\a, \a+1}$, we get that

\vspace{0.3cm}
\begin{enumerate}[label=(6), itemsep=0.3cm]
    \item[(6)] $\sigma_{\a+1, \infty}(\sigma_{\a, \a+1})=\sigma_{\a, \infty}.$
\end{enumerate}
\vspace{0.3cm} 

Applying (5) and (6), we get that $f\in W$. 
\end{proof}

\begin{figure}[htbp]
    \centering
    \begin{tikzpicture}[>=stealth, thick]
        \node (W0) at (0, 0) {$W_0$};
        \node (Wa) at (3, 2) {$W_\a$};
        \node (Wb) at (6, 0) {$W_\b$};
        \node (Winf) at (9, 0) {$W$};
        \node (Nl) at (12, 0) {$N|\l$};

        \draw[->] (W0) -- node[above left] {$\sigma_{0, \a}$} (Wa);
        \draw[->] (W0) -- node[below] {$\sigma_{0, \b}$} (Wb);
        \draw[->] (Wa) -- node[above right] {$\sigma_{\a, \b}$} (Wb);
        
        \draw[->, dashed] (Wb) -- node[above] {$\sigma_{\b, \infty}$} (Winf);
        \draw[->, dashed] (Wa) to[bend left=15] node[above] {$\sigma_{\a, \infty}$} (Winf);
        \draw[->, dashed] (W0) to[bend right=25] node[below] {$\sigma_{0, \infty}$} (Winf);
        
        \draw[->, thick, blue] (Winf) -- node[above] {$\sigma$} (Nl);
        \draw[->, thick, blue, bend left=35] (Wa) to node[above] {$\sigma_\a$} (Nl);

        \node[draw=black, thick, fill=gray!10, rounded corners, inner sep=6pt, anchor=north] 
        at (6, -2.5) {
        \begin{minipage}{9cm}
            \centering
            \textbf{\underline{Key Properties of the Hulls}}
            \begin{itemize}[label={\tiny$\bullet$}, itemsep=1.2pt, topsep=3pt, leftmargin=1em, font=\scriptsize]
                \item $W = \varinjlim (W_\a, \sigma_{\a, \b})$ with $W \in N|(\eta^+)^N$.
                \item $\sigma_\a = \sigma \circ \sigma_{\a, \infty}$ for all $\a < \eta$.
                \item $\sigma_{\a, \gamma} = \sigma_{\b, \gamma} \circ \sigma_{\a, \b}$ for $\a < \b < \gamma < \eta$.
                \item $Z \cap \l \in \rge(\sigma_0) \implies X \in \rge(\sigma)$.
            \end{itemize}
        \end{minipage}
        };
    \end{tikzpicture}
    \caption{Direct limit of the closed hulls $(W_\a : \a < \eta)$ and the final embedding $\sigma$.}
    \label{fig:closed_hulls_limit}
\end{figure}

We now work towards constructing the sequence \[(W_\a, \sigma_\a, \sigma_{\a, \b}: \a<\b<\eta).\]
Fix now $\a<\eta$. We say $Y\in \powerset_{\omega_1}(\mH|\eta)$ is an $(\a, Z)$-\textit{good} set if
\begin{itemize}
\item $\a\in Y$, and
\item $Z[Y]\in \its(\nu_\a)$.
\end{itemize}
Given a ds-generator $x$, we say $x$ is an $(\a, Z)$-\textit{ds-generator} if 
\begin{itemize}
\item $\eta^x=\eta$, $\l^x=\l$, $\k^x=\k$, $\nu^x=\nu_\a$, and
\item for some $(\a, Z)$-good $Y$, $Z[Y]=Z^x$.
\end{itemize}
Notice that for every $\a<\eta$, there is an $(\a, Z)$-ds-generator. Notice that we have the following.

\vspace{0.3cm}
\begin{enumerate}[label=(7), itemsep=0.3cm]
    \item[(7)] Suppose $x$ and $y$ are two $(\a, Z)$-ds-generators. Then \[W_x=W_y\ \text{and}\ \sigma_{x, \infty}=\sigma_{y, \infty}.\]
\end{enumerate}
\vspace{0.3cm}  

(7) holds because $\mathcal{I}_x$ has a dense subset contained in $\mathcal{I}_y$ and vice versa. Indeed, the set of $Y\in \its(x)\cap \its(y)$ is dense in both $\mathcal{I}_x$ and $\mathcal{I}_y$.
We then set \[W_\a=W_x\ \text{and}\ \sigma_\a= \sigma_{x, \infty}\] where $x$ is any $(\a, Z)$-ds-generator. 

Suppose now that $\a<\b<\eta$. It follows from \rlem{lem: agreement of hulls} that 

\vspace{0.3cm}
\begin{enumerate}[label=(8), itemsep=0.3cm]
    \item[(8)] $W_\a\in W_\b$, $\sigma_\b\rest W_\a=id$ and $\sigma_{\a}\in \rge(\sigma_{\b})$.
\end{enumerate}
\vspace{0.3cm} 

Indeed, we can fix $Y\in \powerset_{\omega_1}(\mH|\eta)$ such that $Y$ is both an $(\a, Z)$-good and a $(\b, Z)$-good set. We can then apply \rlem{lem: agreement of hulls} to \[x_0=(\eta, \l, \k, \nu_\a, Z[Y]\cap \nu_\a)\] and \[x_1=(\eta, \l, \k, \nu_\b, Z[Y]\cap \nu_\b).\] $\k$ of our current theorem plays the role of $\k'$ in \rlem{lem: agreement of hulls}, and $\k'$ of our current theorem plays the role of $\k$ in \rlem{lem: agreement of hulls}. Because \[W_\a=W_{x_0}\ \text{and}\ \sigma_\a=\sigma_{x_0, \infty}\] and \[W_\b=W_{x_1}\ \text{and}\ \sigma_\b=\sigma_{x_1, \infty},\] we get (8).\footnote{$\sigma_\b\rest W_\a=id$ follows from the fact that $\sigma_\b\rest \eta^\omega=id$ and $W_\a\in N|(\eta^+)^N$, and so $\card{W_\a}^{W_{\b}}=\eta^\omega$.}

For $\a<\b<\eta$, set \[\sigma_{\a, \b}=\sigma_\b^{-1}\circ \sigma_\a.\] Because $\sigma_\b\rest W_\a=id$ and $\sigma_\a\in \rge(\sigma_\b)$, we indeed have that \[\sigma_\a=\sigma_\b\circ \sigma_{\a, \b}.\] It now follows that the sequence \[(W_\a, \sigma_\a, \sigma_{\a, \b}: \a<\b<\eta)\in N\] and satisfies (1.1)-(1.6). Let then $W$ be the direct limit of $(W_\a, \sigma_{\a, \b}: \a<\b<\eta)$ and $\sigma:W\rightarrow N|\l$ be defined as above. Because $Z\cap \l\in \rge(\sigma_0)$, it follows that $X\in \rge(\sigma)$. The rest of the clauses of \rthm{thm: closed hulls above theta} follow from \rcl{cl: finishing proof}.
\end{proof}
\section{The $\Theta$-sequence of $N$}

Let $(\k_\a: \a<\varsigma)$ be the sequence defined as follows:
\begin{enumerate}
\item $\k_0=\Theta_{\omega^\omega}$,
\item for $\a<\varsigma$, $\k_{\a+1}=\Theta_{\k_\a^\omega}$,
\item for a limit ordinal $\a<\varsigma$, $\k_\a=\sup_{\b<\a}\k_\b$.
\end{enumerate}
\cite[Theorem 10.23]{blue2025nairian} establishes something very close to the next theorem.

\begin{theorem}\label{thm: the theta seq} The sequence $(\k_\a: \a<\varsigma)$ is the increasing enumeration of the strong cardinals of $\mH|\varsigma$ and their limits.
\end{theorem}
\begin{proof} Let $(\nu_\a: \a<\varsigma)$ be the increasing enumeration of the strong cardinals of $\mH|\varsigma$ and their limits. We prove by induction that for all $\a<\varsigma$, $\nu_\a=\k_\a$. We have that $\k_0=\Theta^N$, and it is the least strong cardinal of $\mH|\varsigma$. So $\k_0=\nu_0$. Assume now that $\a>0$ and for all $\b<\a$, $\k_\b=\nu_\b$. If $\a$ is a limit ordinal, then clearly $\k_\a=\nu_\a$. 

Assume now that the equality holds for $\a$, and we prove it for $\a+1$. Set $\k=\k_\a=\nu_\a$ and $\nu=\nu_{\a+1}$. First notice that \cite[Theorem 10.13]{blue2025nairian} implies that $\k_{\a+1}\leq \nu$. Thus, it is enough to prove that $\nu\leq \k_{\a+1}$, and to show this, it is enough to show that if $\eta'<\nu$, then there is a surjection $f:\k^\omega\rightarrow \eta'$ such that $f\in N$. Fix then $\eta'<\nu$. 

Let $\hq$ be a complete iterate of $\hp$ such that $\k_\hq$ is defined and $\sup(\pi_{\hq, \infty}[\nu_\hq])>\eta'$. Let 
\begin{itemize}
\item $\eta=\sup(\pi_{\hq, \infty}[\nu_\hq])$, 
\item $\hr=\hq^\k=(\R, \Psi)$, 
\item $\l=\pi_{\hq, \hr}(\nu_\hq)$ and 
\item $Z=\pi_{\hq, \hr}[\M^\hq|\nu_\hq]$. 
\end{itemize}
We now have that $\hh|\eta$ is a complete iterate of $\hr|\l$ such that $\T_{\hr|\l, \hh|\eta}$ is above $\l$.

For each $Y\in \its(Z)$ let $\hs_Y=\hq_{Z[Y]}$. Set $\hs_Y=(\S_Y, \Sigma_Y)$. For each $Y$, let $A_Y$ be the set of reals that code a complete $\Sigma_Y$-iterate of $\S_Y$. For $a\in A_Y$, let $\hs_{Y, a}$ be the last model of the iteration tree coded by $a$, and set \[\hs_{Y, a}=(\S_{Y, a}, \Sigma_{Y, a}).\] 
For $a\in A_Y$, we let $\sigma_{Y, a}$ be the unique embedding such that \[(\hs_{Y, a}, \sigma_{Y, a})\in \rel(Z[Y]).\] Notice that the function $(Y, a)\mapsto (\hs_{Y, a}, \sigma_{Y, a})$ is in $N$ (see \rcor{cor: uniqueness of realizability witnesses}).

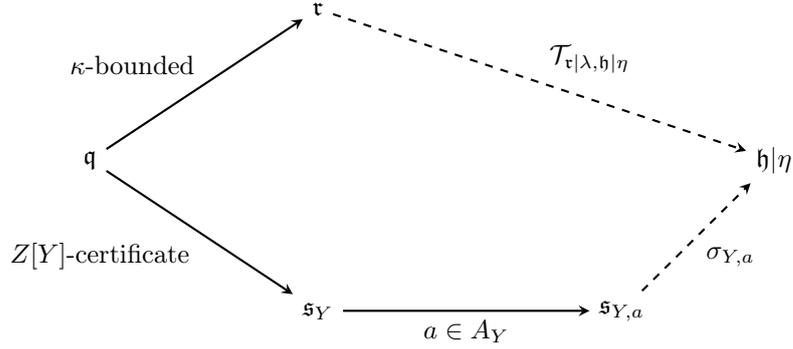
\begin{figure}[htbp]
    \centering
    \begin{tikzpicture}[>=stealth, thick]
        \node (hq) at (0, 0) {$\hq$};
        \node (hr) at (3, 2) {$\hr$};
        \node (sY) at (3, -2) {$\hs_Y$};
        \node (sYa) at (7, -2) {$\hs_{Y, a}$};
        \node (heta) at (9, 0) {$\hh|\eta$};

        \draw[->] (hq) -- node[above left] {$\k$-bounded} (hr);
        \draw[->] (hq) -- node[below left] {$Z[Y]$-certificate} (sY);
        \draw[->] (sY) -- node[below] {$a \in A_Y$} (sYa);
        \draw[->, dashed] (sYa) -- node[below right] {$\sigma_{Y, a}$} (heta);
        \draw[->, dashed] (hr) -- node[above right] {$\T_{\hr|\l, \hh|\eta}$} (heta);

        \node[draw=black, thick, fill=gray!10, rounded corners, inner sep=6pt, anchor=north] 
        at (4.5, -3.5) {
        \begin{minipage}{9cm}
            \centering
            \textbf{\underline{Key Facts for the Surjection $f$}}
            \begin{itemize}[label={\tiny$\bullet$}, itemsep=1.2pt, topsep=3pt, leftmargin=1em, font=\scriptsize]
                \item $f(Y, a, b) = \sigma_{Y, a}(\a_b) \in \eta$,
                \item Domain of $f$ bounded by $\card{\its(Z) \times A \times B} \leq \k^\omega$,
                \item Uniqueness of $\sigma_{Y, a}$ ensures $f \in N$.
            \end{itemize}
        \end{minipage}
        };
    \end{tikzpicture}
    \caption{Iteration trees and the embedding $\sigma_{Y, a}$ defining the surjection $f$.}
    \label{fig:theta_sequence_surjection}
\end{figure}

Finally, let $A=\bigcup_{Y\in \its(Z)}A_Y$ and $B\in N$ be a set of reals coding countable ordinals. For $b\in B$, let $\a_b$ be the ordinal coded by $b$. We now define a surjection \[f: \its(Z)\times A\times B\rightarrow \eta\] and show that $f\in N$. Suppose \[(Y, a, b)\in \its(Z)\times A\times B.\] If $a\not \in A_Y$ or $\a_b\not \in \S_{Y, a}$, then set $f(Y, a, b)=0$. Suppose now that $a\in A_Y$ and $\a_b\in \S_{Y, a}$. Then set \[f(Y, a, b)=\sigma_{Y, a}(\a_b).\] It now follows from the results of \cite[Section 9.2]{blue2025nairian} that $f$ is a surjection onto $\eta$,\footnote{This is because if $\zeta<\eta$ then there is a complete $\k$-bounded iterate $\hw$ of $\hq$ and a complete iterate $\hw'$ of $\hw$ such that $\T_{\hw, \hw'}$ is based on $\hw|\nu_\hw$ and $\zeta_{\hw'}$ is defined.} and since the function $(Y, a)\mapsto (\hs_{Y, a}, \sigma_{Y, a})$ is in $N$, we have that $f\in N$. Because $\card{\its(Z) \times A \times B} \leq \k^\omega$ in $N$, this induces the required surjection from $\k^\omega$ onto $\eta'$, which completes the proof.
\end{proof}

\section{Minimal Nairian models satisfy \(\NTbase\)}\label{sec: the proof of ntbase}

$\NTbase$ was introduced in \rdef{nairian theory}. The first clause of $\NTbase$ is $\ZF+\AH$, where $\AH$ is introduced in \rdef{def: ah}. Thus, to verify $\AH$ we need to show that for every $\a$, $\nN_\a$ is $H$-like. Recall that $M$ is $H$-like if 
\begin{enumerate}
\item $M$ is full (see clause 7 of \rdef{largest cardinal}) and hierarchical (see \rdef{hierarchical}), and
\item If $M$ is of successor type (see clause 2 of \rdef{largest cardinal}), then $M$ is strongly regular (see clause 8 of \rdef{largest cardinal}).
\end{enumerate}

Now $\nN_0$ is the set of hereditarily finite sets, and so clearly it is of limit type, it is full, and it is hierarchical. 

$\nN_1$ is the set of hereditarily countable sets. 
It is of successor type and is full. We need to show that it is hierarchical. For $\a<\omega_1$, let $W_\a$ be the set of all $x\in \nN_1$ such that $\ord\cap \tc(x) < \a$. $W_\a$ is transitive, since if $x\in W_\a$ and $y\in x$ then $\tc(y)\subseteq \tc(x)$, and so $\ord\cap \tc(y) < \a$, implying that $y\in W_\a$. 

Next let $(\k_\a: \a<\varsigma)$ be the sequence defined just before \rthm{thm: the theta seq}. Applying \rthm{thm: the theta seq} we get that in $N$, for every $\a<\varsigma$, $\nN_{\a+1}=H_{\kappa_\a^\omega}$. Since each $H_{\kappa_\a^\omega}$ is hierarchical as witnessed by the sequence $(N|\b: \b<\k_{\a+1})$, and limit stages trivially inherit their $H$-like properties from previous stages, we have that $N\models \AH$. 

Finally, \rthm{thm: basic reflection} and \rthm{thm: closed hulls above theta} imply that for every $\a<\varsigma$, $\rref(X, \nN_\a, \nN_{\a+1})$ holds in $N$. This finishes the proof of \rthm{thm: part 2}.

\bibliographystyle{plain}
\bibliography{main}

\bls

\noindent\address{Department of Philosophy, University of Pittsburgh, Pittsburgh, PA 15260}

\noindent\email{doug.blue@pitt.edu}

\bls

\noindent\address{Department of Mathematics, Miami University, Oxford, Ohio 45056}

\noindent\email{larsonpb@miamioh.edu}

\bls

\noindent\address{IMPAN, ul. Śniadeckich 8, 00-656 Warszawa}

\noindent\email{gsargsyan@impan.pl}

\end{document}